\documentclass[10pt]{article}
\usepackage{graphicx} 
\usepackage[utf8]{inputenc}

\usepackage{amsmath, mathtools, amsfonts, mathrsfs, hyperref}
\usepackage{amsthm, amssymb}
\usepackage{bbm}
\usepackage{tikz-cd}
\usepackage{tikz}
\usetikzlibrary{shapes.geometric, arrows}
\usetikzlibrary{positioning}

\tikzstyle{process} = [rectangle, minimum width=3cm, minimum height=1cm, text centered, text width=3cm, draw=black, fill=green!30]
\tikzstyle{decision} = [diamond, minimum width=3cm, minimum height=1cm, text centered, text width = 3cm, draw=black, fill=orange!30]
\tikzstyle{arrow} = [thick,->,>=stealth]

\usepackage{soul}
\usepackage{comment}
\usepackage[labelfont=bf]{caption}
\usepackage{xcolor}
\usepackage{authblk}
\usepackage{stmaryrd}

\theoremstyle{definition}

\theoremstyle{plain}
\newtheorem{theorem}{Theorem}[section]
\newtheorem{lemma}[theorem]{Lemma}
\newtheorem{definition}[theorem]{Definition}
\newtheorem{Proposition}[theorem]{Proposition}
\newtheorem{remark}[theorem]{Remark}

\newcommand{\bq}{\mathbb{Q}}
\newcommand{\qp}{\mathbb{Q}_p}
\newcommand{\zp}{\mathbb{Z}_p}

\newcommand{\norm}[1]{\vert #1 \vert}

\newcommand{\br}[1]{\overline{#1}}

\newcommand{\co}{\mathcal{O}}

\newcommand{\sL}{\mathcal{L}}
\newcommand{\brqp}{\br{\mathbb{Q}}_p}
\newcommand{\brFp}{\br{\mathbb{F}}_p}

\newcommand{\rmss}{{\mathrm{ss}}}

\newcommand{\bZ}{\mathbb{Z}}
\newcommand{\bQ}{\mathbb{Q}}
\newcommand{\Zp}{\bZ_p}
\newcommand{\Qp}{\bQ_p}
\newcommand{\bP}{\mathbb{P}}
\newcommand{\cL}{\mathcal{L}}
\newcommand{\IZind}{\mathrm{ind}_{IZ}^G\>}

\newcommand{\Oe}{\mathcal{O}_E}
\newcommand{\Fq}{\mathbb{F}_q}
\newcommand{\Fp}{\mathbb{F}_p}
\newcommand{\id}{\mathrm{id}}
\newcommand{\GL}{\mathrm{GL}}
\newcommand{\SL}{\mathrm{SL}}
\newcommand{\im}{\mathrm{Im}\>}

\newcommand{\SymE}[1]{\underline{\mathrm{Sym}}^{#1}E^2}
\newcommand{\SymO}[1]{\underline{\mathrm{Sym}}^{#1}\Oe^2}
\newcommand{\SymF}[1]{\mathrm{Sym}^{#1}\Fq^2}
\newcommand{\St}{\mathrm{St}}
\newcommand{\matUp}{\beta\begin{pmatrix}
1 & \lambda \\
0 & 1
\end{pmatrix}w}
\newcommand{\tWp}{\tilde{W}_p}
\newcommand{\tUp}{\tilde{U}_p}
\newcommand{\Bind}{\mathrm{ind}_B^G \>}
\newcommand{\C}[1]{\mathscr{C}^{#1}}
\newcommand{\Bnorm}[2]{\vert\vert #1 \vert\vert_{#2}}
\newcommand{\tB}{\tilde{B}}
\newcommand{\logL}{\log_{\cL}}
\newcommand{\sm}{\mathrm{smooth}}
\newcommand{\lattice}[1]{\tilde{\Theta}(#1)}
\newcommand{\latticeL}[1]{\tilde{\Theta}(#1, \cL)}
\newcommand{\ind}{\mathrm{ind}\>}
\newcommand{\sC}{\mathscr{C}}
\newcommand{\mth}{\mathrm{th}}

\title{Reductions of semi-stable representations using the Iwahori mod $p$ Local Langlands Correspondence}
\author{Anand Chitrao}

\author{Eknath Ghate}
\affil{School of Mathematics, Tata Institute of Fundamental Research \\ Homi Bhabha Road, Mumbai - 400005, India.}

\begin{document}

\maketitle

\begin{abstract}
  We determine the mod $p$ reductions of all two-dimensional semi-stable representations $V_{k,\sL}$ of
  the Galois group of $\qp$ of weights $3 \leq k \leq p+1$ and $\sL$-invariants $\sL$
  for primes $p \geq 5$.
  In particular, we describe the constants
  appearing in the unramified characters completely. The proof involves computing the reduction
  of Breuil's $\mathrm{GL}_2(\qp)$-Banach space $\tB(k,\sL)$, by studying certain logarithmic
  functions using background  
  material developed by Colmez, and then applying an Iwahori theoretic version of the mod $p$ Local Langlands
  Correspondence.
\end{abstract}  
  
\section{Introduction}
Let $p \geq 5$ be a prime and $E$ be a finite extension of $\qp$ containing $\sqrt{p}$.
In this paper, we compute the semi-simplification of the reduction mod $p$ of
the irreducible two-dimensional semi-stable representation $V_{k, \cL}$ of $\mathrm{Gal}(\brqp/\qp)$
over $E$ with Hodge-Tate weights $(0,k-1)$ for $k \in [3, p + 1]$ and 
$\cL$-invariant $\cL \in E$. 
We use the compatibility
with respect to reduction mod $p$ between the $p$-adic Local Langlands Correspondence and an
Iwahori theoretic version of the mod $p$ Local Langlands Correspondence. To the best of our knowledge,
this is the first time that such
techniques have been used to compute the reduction mod $p$ of semi-stable representations. In principle, this
method allows one to treat arbitrarily large weights $k$, and in particular, allows us to go beyond weight $k = p-1$, a
natural boundary when solving the reduction problem using strongly divisible modules from
integral $p$-adic Hodge theory. 

Let $r = k-2$ so that $1 \leq r \leq p-1$ and let $G = \GL_2(\qp)$,
$I$ be the Iwahori subgroup of $G$ and $Z$ be the center of $G$.
We choose a lattice $\latticeL{k}$ inside the $p$-adic $G$-Banach space $\tB(k, \cL)$ over $E$
associated to $V_{k, \cL}$ by Breuil \cite{Bre04, Bre10}. 
Then, using some background material from Colmez \cite{Col10a}, by a delicate analysis of $E$-valued
logarithmic functions in the variable $z \in \qp$ of the form
    \[
        \sum \lambda_i(z - z_i)^n\logL(z - z_i),
    \]
    where $i$ varies through a finite indexing set, $\lambda_i \in E$, $z_i \in \Qp$, $r/2 < n \leq r$,
    $\logL$ is the branch of the $p$-adic logarithm taking $p$ to $\sL \in E$,
     we compute
    (the semi-simplification of) the reduction $\br{\latticeL{k}}$ in terms of twists of quotients of
    representations of $G$ compactly induced from $IZ$ by images of Iwahori-Hecke operators.
    The (semi-simplification of the) reduction $\br{V}_{k, \cL}$ of the Galois representation $V_{k, \cL}$ is then
    determined by its mod $p$ Langlands correspondent $\br{\latticeL{k}}$ using the Iwahori theoretic
    reformulation of the mod $p$ Local
    Langlands Correspondence worked out in \cite{Chi23}.

    To state our main result describing the shape of $\br{V}_{k, \cL}$ we need some notation.
    Let $v_-$ and $v_+$ be
    the largest and smallest integers, respectively, such that $v_- < r/2 < v_+$ for $r \geq 1$.
    For $n \geq 1$, let $H_n = \sum_{i = 1}^{n}\frac{1}{i}$ be the $n$-th partial harmonic sum.
    Set $H_0 = 0$ and $H_{\pm} = H_{v_{\pm}}$. Let $v_p$ be the $p$-adic valuation on $\brqp$
    normalized such that $v_p(p) = 1$.
    Let $$\nu = v_p(\sL - H_{-} - H_{+})$$
    be the $p$-adic valuation of $\sL$ shifted by the
    partial harmonic sums $H_{-}$ and $H_{+}$. 
    The parameter $\nu$ plays a central role throughout. Let $\omega$ be the fundamental character of
    $\mathrm{Gal}(\br{{\mathbb Q}}_p/\qp)$ of
    level $1$. 
    Similarly, for an integer $c$ (with
    $p + 1 \nmid c$), let $\omega_2^c$ be an extension from the inertia subgroup $I_{\qp}$ to $\mathrm{Gal}(\brqp/\bq_{p^2})$
    of the $c$-th power of the
    fundamental character $\omega_2$ 
    of level $2$ chosen 
    so that the (irreducible) representation $\ind (\omega_2^c)$ obtained
    by inducing this character from $\mathrm{Gal}(\brqp/\bq_{p^2})$ to $\mathrm{Gal}(\brqp/\qp)$ has
    determinant $\omega^c$. Let $\mu_{\lambda}$ be the unramified character of $\mathrm{Gal}(\brqp/\qp)$ 
    sending geometric Frobenius to $\lambda \in \br{\mathbb{F}}_p^{*}$. 
    Our main theorem is the following result.

    \begin{theorem}\label{Main theorem in the second part of my thesis}
      For $k \in [3, p + 1]$ and for primes $p \geq 5$, the semi-simplification of the reduction mod $p$ of the
      semi-stable representation $V_{k, \cL}$ on $\mathrm{Gal}(\brqp/\qp)$ is given by an alternating
      sequence of irreducible and reducible representations:
    \[
    \br{V}_{k, \cL} \sim
    \begin{cases}
        \ind (\omega_2^{r+1 + (i-1)(p-1)}), & \text{ if $(i-1) - r/2 < \nu < i - r/2$} \\
        \mu_{\lambda_i}\omega^{r+1-i} \oplus \mu_{\lambda_i^{-1}}\omega^{i}, & \text{ if $\nu = i - r/2$},
    \end{cases}
\]
where $1 \leq i \leq (r+1)/2$ if $r$ is odd and $1 \leq i \leq (r+2)/2$ if $r$ is even\footnote{We adopt
  the following conventions:
   \begin{itemize}
    \item [-] The first interval $- r/2 < \nu < 1 - r/2$ should be interpreted as $\nu < 1 - r/2$.
    \item [-] If $r$ is odd, then the last case $\nu = 1/2$ should be interpreted as $\nu \geq 1/2$. If $r$ is even,
          then the interval $0 < \nu < 1$ should be interpreted as $\nu > 0$ and we drop the case $\nu = 1$.
    \end{itemize}},
  and the mod $p$ constants $\lambda_i$ are determined by the formulas:
    \begin{eqnarray*}
        \lambda_i & = & \br{(-1)^i \> i {r+1-i \choose i}p^{r/2-i}(\cL - H_{-} - H_{+})}, \quad \text{ if } 1 \leq i < \dfrac{r + 1}{2} \\
        \lambda_{i} + \lambda_i^{-1} & = & \br{(-1)^i \> i{r + 1 - i \choose i}p^{r/2 - i}(\cL - H_{-} - H_{+})}, \quad \text{ if } i = \dfrac{r + 1}{2} \text{ and } r \text{ is odd}.
    \end{eqnarray*}
\end{theorem}

Here are some features of the theorem:
\begin{enumerate}
\item It recovers results of Breuil-M\'ezard \cite[Theorem 1.2]{BM} and Guerberoff-Park \cite[Theorem 5.0.5]{GP} by giving a more or less uniform treatment for all weights $k \in [3,p-1]$. That the theorem matches with
  the results of these authors on the inertia subgroup $I_{\qp}$ was already checked in \cite[Section 1.4]{CGY21}.  

\item It computes the reduction $\br{V}_{k, \cL}$ on the entire Galois group $\mathrm{Gal}(\brqp/\qp)$ not just on $I_{\qp}$.
  
  For even $k$ in the interval $[3, p - 1]$, Breuil-M\'ezard \cite[Theorem 1.2]{BM} computed the reduction on
  $\mathrm{Gal}(\brqp/\qp)$. We check that Theorem \ref{Main theorem in the second part of my thesis} for even $r$ matches with their result.
  In view of the checks done in \cite[Section 1.4]{CGY21}, it suffices to check our unramified characters $\mu_{\lambda_i}$
      match with those in \cite[Theorem 1.2]{BM}. 
    \begin{itemize}
    \item First assume that $\nu < 0$ or equivalently $v_p(\cL) < 0$. We compare $\mu_{\lambda_i}$ for $1 \leq i < r/2$ in Theorem \ref{Main theorem in the second part of my thesis} with the character $\lambda\left(\br{b}^{-1}\br{\left(a_p/p^{\frac{k}{2} - 1}\right)}\right)$ appearing in \cite[Theorem 1.2 (iii)]{BM}, where 
    \[
        b = (-1)^{\frac{k}{2} - \nu}\left(\frac{k}{2} - \nu\right){\frac{k}{2} - 1 - \nu \choose -2\nu + 1}\left(\frac{-\cL}{p^{\nu}}\right)
    \]
    (and is denoted by $b(f)$ in \cite{BM}). Substituting $\nu = i - r/2$ for $1 \leq i < r/2$, we get
    \begin{eqnarray*}
        b & = & (-1)^{r - i}(r + 1 - i){r - i \choose r - 2i + 1}p^{r/2 - i}\cL \\
        & = & (-1)^i \> i \> {r + 1 - i \choose i}p^{r/2 - i}\cL \\
        & \equiv & (-1)^i \> i \> {r + 1 - i \choose i}p^{r/2 - i}(\cL - H_{-} - H_{+}) \mod \pi \\
        & = & \lambda_i,
    \end{eqnarray*}
    where $\pi$ is a uniformizer of $E$.
    Assume for simplicity that the nebentypus character of the modular form in \cite[Theorem 1.2]{BM} is trivial
    (not just trivial at $p$). Then the $p$-th Fourier coefficient $a_p$ of the form satisfies $a_p^2 = p^{r}$.
    Assume that $a_p = + p^{r/2}$ as opposed to $a_p = -p^{r/2}$.
    Recall that our $\mu_{a}$ sends a geometric Frobenius to $a \in E$ whereas
    $\lambda(a)$ in \cite[Theorem 1.2]{BM} sends an arithmetic Frobenius to $a \in E$. Therefore our unramified characters $\mu_{\lambda_i}$ match with those in \cite[Theorem 1.2 (iii)]{BM} for $\nu < 0$. 
    
    \item Next assume $\nu = 0$. There are two subcases in \cite[Theorem 1.2]{BM}, namely, $v_p(\cL - 2H_{k/2 - 1}) < 1$ and $v_p(\cL - 2H_{k/2 - 1}) \geq 1$. Assume we are in the first case. We compare $\mu_{\lambda_{r/2}}$ in Theorem \ref{Main theorem in the second part of my thesis} with $\lambda\left(\br{a}^{-1}\br{\left(a_p/p^{k/2 - 1}\right)}\right)$ in \cite[Theorem 1.2 (i), part 1]{BM}. Note
    \begin{eqnarray*}
        a & = & (-1)^{\frac{k}{2}}\left(-1 + \frac{k}{2}\left(\frac{k}{2} - 1\right)(-\cL + 2H_{k/2 - 1})\right) \\
        & = & (-1)^{\frac{r}{2}}\frac{k}{2}\left(\frac{k}{2} - 1\right)\left(\frac{1}{\frac{k}{2}\left(\frac{k}{2} - 1\right)} - (-\cL + 2H_{k/2 - 1})\right) \\
        & = & (-1)^{\frac{r}{2}}\frac{r}{2}\left(\frac{r}{2} + 1\right)\left(\cL - H_{-} - H_{+}\right) \\
        & = & \lambda_{r/2}.
    \end{eqnarray*}
    Assuming $a_p = +p^{r/2}$ for simplicity, we see that our unramified character $\mu_{\lambda_{r/2}}$ matches with the one in \cite[Theorem 1.2 (i), part 1]{BM}. Now assume $v_p(\cL - 2H_{k/2 - 1}) \geq 1$ and $\nu = 0$. Then again making the simplifying assumption $a_p = +p^{r/2}$, we have to check that our $\mu_{\lambda_{r/2}}$ is equal to $\lambda((-1)^{r/2})$ in \cite[Theorem 1.2 (i), part 2]{BM}. Indeed, 
    the assumption $v_p(\cL - 2H_{k/2 - 1}) \geq 1$ implies that
    \begin{eqnarray*}
        \lambda_{r/2} & \equiv & (-1)^{r/2}\frac{r}{2}\left(\frac{r}{2} + 1\right)\left(2H_{r/2} - H_{-} - H_{+}\right) \\
        & = & (-1)^{r/2}\frac{r}{2}\left(\frac{r}{2} + 1\right)\left(\frac{1}{r/2} - \frac{1}{r/2 + 1}\right) \\
        & = & (-1)^{r/2} \mod \pi.
    \end{eqnarray*}
    \end{itemize}
    Note that if $a_p = -p^{r/2}$, then the local Galois representation associated to the modular form
    in \cite{BM} is $V_{k, \cL}\otimes\mu_{-1}$. Therefore to compare
    Theorem \ref{Main theorem in the second part of my thesis} with \cite[Theorem 1.2]{BM}, we have to
    twist the latter by $\mu_{-1}$. This is equivalent to checking these two results match when
    $a_p = +p^{r/2}$, which we have just done.

    We remark that some partial information on the unramified characters
    $\mu_{\lambda_i}$ in odd weights
    $k$ in the interval $[3,p-2]$ was recently
    obtained by Lee-Park \cite{LP22}.

  \item It determines the \emph{self-dual} constant $\lambda_{\frac{r + 1}{2}}$ for odd weights $k$
    in the interval $[3,p]$ appearing at the end of the theorem which is
    missing in \cite{LP22}.
    The computation of this constant
    is tricky and partly explains the added size of this paper. 

  \item It treats completely
    the weights $k = p$ and $p + 1$ which were not treated in \cite{GP, LP22, BM}.
    These weights were previously inaccessible using methods involving strongly divisible
    modules\footnote {Although in the case $k = p$ there is a possibility that this may yet be
      carried out using the work of Gao on unipotent strongly divisible modules \cite{Gao17},
      and for even weights $4 \leq k \leq p+1$ there is an alternative cohomological approach
      for $\nu \geq 0$ due to Breuil-M\'ezard \cite[Theorem 1.1]{BM10}.}
    We remark that for even weights $k$ in the interval $[3,p+1]$, while we need to work with
    Iwahori-Hecke operators in a non-commutative Hecke algebra when analyzing the behavior
    of the reduction $\br{\latticeL{k}}$ around the last point
    $\nu = 0$, these operators also appear when treating the behavior around the first point
    $\nu = 1 - r/2$ for the weight $k = p+1$. 

    We also remark that, in principle, the methods introduced in this paper to compute $\br{V}_{k,\sL}$ (by
    computing $\br{\latticeL{k}}$ using
      logarithmic functions and applying the Iwahori mod $p$ LLC) work for {\it all} weights $k$.
      The computations though naturally become more complex as $k$
      increases.\footnote{In this context, we remark that a generalization of the first case 
      of the conclusion of Theorem~\ref{Main theorem in the second part of my thesis}
      (so $i = 1$ and $\nu < 1-r/2$)
      was recently proved for all weights $k \geq 4$ (and odd primes $p$) by
      Bergdall-Levin-Liu \cite[Theorem 1.1]{BLL22}
      using Breuil and Kisin modules.}
      
    \item It allows the second author to extend his proof  of the zig-zag conjecture on $I_{\qp}$ 
      for slopes $\frac{1}{2} \leq v_p(a_p) \leq \frac{p - 3}{2}$ outlined in \cite{Gha22}
      to all of $\mathrm{Gal}(\brqp/\qp)$ and
      all slopes $\frac{1}{2} \leq v(a_p) \leq \frac{p - 1}{2}$.
      
    \item It complements the recent computation by Bergdall-Pollack \cite{BP24}
     of $\sL$-invariants attached to cuspforms
     with trivial character in the sense that the $\sL$-invariants they obtain for small primes $p$ and (even)    
     weights $k$ (not necessarily
     bounded by $p+1$) have mostly negative valuation,  whereas most of the action in 
     Theorem~\ref{Main theorem in the second part of my thesis} also happens for $\nu < 0$.
     Note that if  $k \leq p+1$, then $\nu = v_p(\sL)$ if either quantity is negative.
\end{enumerate}

\section{Notation}
\begin{itemize}
        \item $p \geq 5$ is a prime.
        \item $E$ is a finite extension of $\qp$ containing $\sqrt{p}$ and $\cL$.
              $\co_E$ is the ring of integers in $E$ with a uniformizer $\pi$ and residue field
$\Fq$. Note $\sqrt{p} \equiv 0 \mod \pi$. 
        \item $k$ denotes the weight of a semi-stable
representation and $r = k - 2$.
        \item $v_-$, $v_+$ are the largest, smallest integers,
respectively, such that $v_- < r/2 < v_+$ for $r \geq 1$.
        \item For $n \geq 1$, $H_n = \sum\limits_{i = 1}^{n}\frac{1}{i}$
and $H_0 = 0, H_{\pm} = H_{v_{\pm}}$.
        \item $v_p$ is the $p$-adic valuation normalized such that $v_p(p)
= 1$.
        \item $\nu = v_p(\sL - H_{-} - H_{+})$ is the valuation
of $\sL$ shifted by the partial
harmonic sums $H_{-}$ and $H_{+}$. 
        \item $I_{\qp}$ is the inertia subgroup of $\mathrm{Gal}(\brqp/\qp)$.
        \item $\omega$ is the fundamental character of $I_{\qp}$ of level
$1$;
        it has a canonical extension to $\mathrm{Gal}(\br{{\mathbb
Q}}_p/\qp)$.
        \item $\omega_2$ is the fundamental character of $I_{\qp}$ of
level $2$; for an integer $c$ with $p
+ 1 \nmid c$,
        choose an extension of $\omega_2^c$ to
$\mathrm{Gal}(\brqp/\bq_{p^2})$ so that the irreducible representation
        $\mathrm{ind}(\omega_2^c)$ obtained
        by inducing this extension from $\mathrm{Gal}(\brqp/\bq_{p^2})$ to
$\mathrm{Gal}(\brqp/\qp)$ has determinant $\omega^c$.
        \item $G$ is the group $\mathrm{GL}_2(\qp)$.
        \item $I$ is the Iwahori subgroup of $G$ consisting of matrices
that are upper triangular $\!\!\!\!\mod p$.
        \item $B$ is the Borel subgroup of $G$ consisting of upper
triangular matrices.
        \item $\alpha = \begin{pmatrix}1 & 0 \\ 0 & p\end{pmatrix}$,
$\beta = \begin{pmatrix}0 & 1 \\ p & 0\end{pmatrix}$ and $w =
\begin{pmatrix}0 & 1 \\ 1 & 0\end{pmatrix}$. Note that $\beta =
\alpha w$.
        \item $I_n = \{[a_0] + [a_1]p + \cdots + [a_{n - 1}]p^{n - 1}
\> \vert \> a_i \in \Fp\}$ for $n \geq 1$, where $[\quad]$ denotes Teichm\"uller representative. $I_0 = \{0\}$.
        \item $V_{r} = \SymF{r}$ and 
        $\SymE{k - 2} := \norm{\det}^{\frac{k - 2}{2}} \otimes
\mathrm{Sym}^{k - 2}E^2$ for $k \geq 2$.
        \item $d^r$  denotes the character $IZ \to
\Fp^*$ sending $\left(\begin{smallmatrix}a & b \\ c & d\end{smallmatrix}\right) \in
I$ to $d^r \!\!\! \mod p$ and the scalar
matrix $p$ to $1$.
        \item For a representation $(\rho, V)$ of $IZ$ over $E$ or $\Fq$,
let $\IZind \rho$ be the vector space of functions $f : G \to V$
that are compactly supported modulo $IZ$ such that $f(hg) =
\rho(h)
        \cdot f(g)$, for all $h \in IZ$ and $g \in G$. The vector space
$\IZind \rho$ has a $G$ action:
        $g \cdot f(g') = f(g' g)$, for all $g, g' \in G$ and $f \in \IZind
\rho$. For $g \in G$ and $v \in V$, define the function $\llbracket g, v \rrbracket
\in \IZind \rho$ by
    \[
        \llbracket g, v\rrbracket(g') =
        \begin{cases}
            \rho(g'g) \cdot v, & \text{ if }g'g \in IZ \\
            0, & \text{ otherwise}.
        \end{cases}
    \]
        \item Let $(\Bind E)^{\sm}$ be the $E$-vector space of locally
constant functions from $G$ to $E$, with the action of $G$ given
by $g \cdot f(g') = f(g'g)$ for any $g, g' \in G$ and $f \in
(\Bind E)^{\sm}$.
        \item Let $\St$ be the Steinberg representation of $G$ over $E$,
i.e., the space of all locally constant functions
$H : \bP^1(\Qp) \to E$ modulo constant functions with $G$-action: $\left(\left(\begin{smallmatrix} a & b \\ c & d
\end{smallmatrix}\right)\cdot H\right)(z) = H(\frac{az + c}{bz +
d})$.
        \item $[a] \in \{0, \ldots, p - 2\}$ denotes the class of $a \!\!
\mod p - 1$.
        \item $\delta_{a, b} = 1$ if $a = b$ and is $0$ otherwise.
        \item $\br{e}$ denotes the edge $e$ of a graph with its orientation reversed.
        \item $o(e)$ and $t(e)$ are the origin and terminal vertices,
respectively, of an edge $e$.
\end{itemize}

\section{Overview}
We sketch the proof of Theorem \ref{Main theorem in the second part
  of my thesis}
(Theorem \ref{Main final theorem in the second part of my thesis}) and describe
the organization of this paper. 
    
In Section~\ref{section Iwahori mod p LLC}, 
we recall the definitions of some Iwahori-Hecke operators on (functions on) the
Bruhat-Tits tree and reformulate
Breuil's (semi-simple) mod $p$ Local Langlands Correspondence Iwahori theoretically following
\cite{Chi23}.
In Section \ref{Definition of the Banach space}, we recall the definition of Breuil's Banach space $\tB(k, \cL)$ attached to $V_{k, \cL}$ by the $p$-adic Local Langlands Correspondence. 
In Section \ref{Section defining the lattices}, we introduce the standard lattice $\latticeL{k}$ in $\tB(k, \cL)$ and prove two results which allow us to recognize some elements in it (Lemma \ref{Integers in the lattice} and Lemma \ref{stronger bound for polynomials of large degree with varying radii}). In Section \ref{Section for log functions}, we collect together some elementary properties of the logarithmic functions mentioned in the
introduction. From the definition of the lattice $\latticeL{k}$ in $\tB(k, \cL)$, we get a surjection 
    \[
        \IZind \SymF{k - 2} \twoheadrightarrow \br{\latticeL{k}}
    \]
    to its reduction mod $p$ (more precisely, mod $\pi$). As an $IZ$-module, $\SymF{k - 2}$ decomposes as a successive extension of characters. We prove that most of the resulting subquotients of $\IZind \SymF{k - 2}$ map to $0$ under this surjection and that the remaining ones factor through linear (or possibly quadratic in the self-dual case for odd weights) expressions involving Iwahori-Hecke operators. Which subquotients die depends on the size of 
    $\nu = v_p(\cL - H_{-} - H_{+})$. Since one or two (or possibly three in the case $\nu = 0$ for even weights) Jordan-H\"older factors of $\IZind \SymF{k - 2}$ remain on the right and $\br{\latticeL{k}}$ is in the image of the Iwahori
    mod $p$ Local Langlands Correspondence, we are able to deduce the structure of $\br{V}_{k, \cL}$. For instance, if only one Jordan-H\"older factor 
    survives, then 
    $\br{V}_{k, \cL}$ is irreducible.
    
    Here is a very rough sketch of the method used
    to show that the subquotients above map to $0$:

    \begin{enumerate}
    \item We first carefully choose a logarithmic
      function $g(z) = \sum_{i \in I}\lambda_i(z - z_i)^n\logL(z - z_i)$ with $\lambda_i \in E$, $z_i \in \Qp$ for $i$ in a finite indexing set $I$ and $r/2 < n \leq r = k - 2$ an integer, depending on the subquotient mentioned above. We require the vanishing of certain
      sums $\sum_i \lambda_i z_i^j$, which implies that
      $g(z)$ (lies in and) is equal to $0$ in $\tB(k, \cL)$.
      

    \item We write $g(z) = g(z)\mathbbm{1}_{\Zp} + g(z)\mathbbm{1}_{\Qp \setminus \Zp}$ and
      approximate $g(z)\mathbbm{1}_{\Zp} = g_h(z)$ in $\br{\latticeL{k}}$ for large $h$
      with $g_h(z)$ a partial Taylor expansion of $g$ as in
      \cite{Col10a} 
      (see Proposition~\ref{Convergence proposition} and Lemma \ref{Congruence lemma}).

    \item Using a radius reduction argument, we prove inductively that
      $g_h(z) = g_2(z)$ in $\br{\latticeL{k}}$ (see Lemma \ref{telescoping lemma} and its variant Lemma \ref{Telescoping Lemma in self-dual}).

        \item Using a simpler argument, we prove $g(z)\mathbbm{1}_{\Qp \setminus \Zp} = 0$ in $\br{\latticeL{k}}$ (see Lemma \ref{qp-zp part is 0} and its variant Lemma \ref{Telescoping Lemma in self-dual for qp-zp}).

        \item Thus, $g_2(z) = g(z) = 0$ in $\br{\latticeL{k}}$. Using a `shallow' inductive step (Proposition \ref{inductive steps}), we write $g_2(z)$ as the image under the surjection above of a linear expression in an Iwahori-Hecke operator applied to a generator of the subquotient in $1.$ It follows that the image of this expression in this subquotient maps to $0$ in $\br{\latticeL{k}}$. If $\nu$ is in an appropriate infinite half-interval, then this image is full, showing that the entire subquotient maps to $0$ in $\br{\latticeL{k}}$.

          When $\nu$ takes values at the boundary of these infinite
          half-intervals,
then sometimes we can only conclude that the image of the subquotient above  
under the surjection above factors through a representation of the form 
$\pi(r,\lambda,\eta)$ defined in Section~\ref{section Iwahori mod p LLC} using Iwahori induction.
\end{enumerate}

All of this is explained in more detail in Section \ref{Common section} which is
the heart of the paper. However, two tricky, important, missing cases, namely,
the self-dual case for odd weights
and a non-commutative Hecke algebra case for even weights, need extra work.
These occur around the last points
$\nu = \frac{1}{2}$ and $\nu = 0$ in the theorem, respectively, and are
treated in Section \ref{SpecOdd} and Section \ref{SpecEven}, respectively.
In particular, in the former case, $g_2(z)$ is sometimes replaced by $g_3(z)$, we require an additional `deep' inductive step (Proposition \ref{Refined hard inductive steps for self-dual}), and the linear expression above is sometimes replaced by a quadratic expression. Moreover, some tricky binomial identities 
involving harmonic sums in these two sections are proven in the Appendix.
          Here is the leitfaden for these sections:

{
    \vspace{0.1cm}
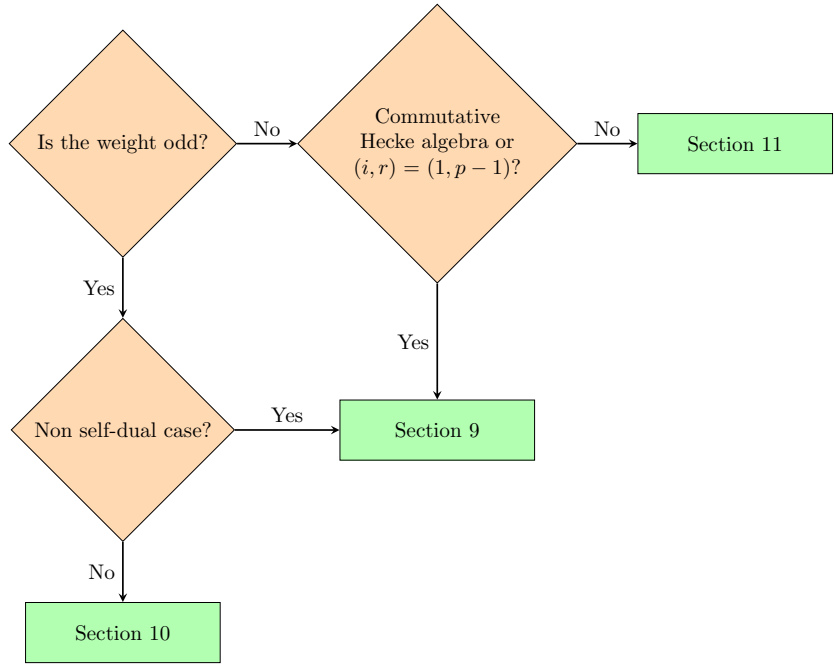
    \begin{center}
    \scalebox{0.8}{
    \begin{tikzpicture}[node distance=1cm]
    \node (Odd?)[decision]{Is the weight odd?};
    \node (ComHecAlg?)[decision, right = of Odd?]{Commutative Hecke algebra or $(i, r) = (1, p - 1)$?};
    \node (Non self-dual?)[decision, below = of Odd?]{Non self-dual case?};
    \node (SpecEven)[process, right = of ComHecAlg?]{Section \ref{SpecEven}};
    \node (Common)[process, below = of ComHecAlg?, yshift=-0.93cm]{Section \ref{Common section}};
    \node (SpecOdd)[process, below = of Non self-dual?]{Section \ref{SpecOdd}};
    \draw [arrow] (Odd?) -- node[anchor=south] {No} (ComHecAlg?);
    \draw [arrow] (Odd?) -- node[anchor=east] {Yes}(Non self-dual?);
    \draw [arrow] (ComHecAlg?) -- node[anchor=south] {No}(SpecEven);
    \draw [arrow] (ComHecAlg?) -- node[anchor=east] {Yes} (Common);
    \draw [arrow] (Non self-dual?) -- node[anchor=south] {Yes} (Common);
    \draw [arrow] (Non self-dual?) -- node[anchor=east] {No} (SpecOdd);
    \end{tikzpicture}
    }
  \end{center}
}


\section{Iwahori  mod $p$ Local Langlands Correspondence}
  \label{section Iwahori mod p  LLC}
    
    In this section, we recall the definitions of some Iwahori-Hecke operators and
    reformulate Breuil's (semi-simple) mod $p$ Local Langlands Correspondence using
    Iwahori induction. For details, we refer to \cite{Chi23}.
    
    \subsection{Operators on the Bruhat-Tits tree of $\SL_2(\Qp)$}
    \label{BT tree section}

    Let $E$ be a finite extension of $\qp$. In this section, we recall the definitions of some operators on
    compactly supported $E$-valued functions on oriented edges of the Bruhat-Tits tree of $\SL_2(\Qp)$.
    
    Let $\Delta$ be the Bruhat-Tits tree of $\SL_2(\Qp)$. Recall that vertices of $\Delta$ correspond to
    homothety classes of lattices in $\Qp^2$. Two vertices in $\Delta$ are joined by an edge in $\Delta$
    if they are represented by lattices $L$ and $L'$ such
    that $pL \subsetneq L' \subsetneq L$. We recall two elementary lemmas
    (cf. \cite[p. 6]{BL95}).
        
        \begin{lemma}
            The vertices of $\Delta$ are in a one-to-one correspondence with $G/KZ$.
          \end{lemma}

By an oriented edge of $\Delta$, we mean an ordered pair of adjacent vertices in $\Delta$.
If $e$ is an edge in $\Delta$, set $o(e)$ to be the origin of $e$ and $t(e)$ to be the target of $e$.
        \begin{lemma}
            The oriented edges of $\Delta$ are in a one-to-one correspondence with $G/IZ$.
          \end{lemma}

          We recall the definitions of two operators on compactly supported functions on the oriented
          edges of the tree. If $F : \text{oriented edges of } \Delta \to E$ is a function with compact support,
     define     
    \[
        W_p(F)(a) = F(\br{a}), \quad  U_p(F)(a) = \sum_{\substack{o(a') = t(a) \\ a' \neq \br{a}}} F(a'),
    \]
    where $\br{a}$ is the edge opposite to the edge $a$.

Let $G = \GL_2(\qp)$, $I$ be the Iwahori subgroup of $G$ and $Z$ be the center of $G$.
We can identify the elements in the compactly induced representation $\IZind E$ with  functions on the
oriented edges of $\Delta$ with finite support: the element $\llbracket g, 1\rrbracket$ is identified with the function
that maps the edge $g (\zp^2, \alpha \zp^2)$ to $1$ and all other edges to $0$.
Using this identification, we may rewrite the definitions of
the two operators  $W_p$ and $U_p$ on $\IZind E$. We need the following lemmas.

    \begin{lemma}\label{Small operator in the non-commutative Hecke algebra}
        If $e$ is the edge corresponding to the function $\llbracket g, 1\rrbracket$, then $\br{e}$ is the edge corresponding to the function $\llbracket g\beta, 1\rrbracket$.
      \end{lemma}      

    \begin{lemma}\label{Big operator in the non-commutative Hecke algebra}
        If $e$ is the edge corresponding to $\llbracket g, 1\rrbracket$, then the edges $e'$ with $o(e') = t(e)$ and $e' \neq \br{e}$ correspond to $\left\llbracket g \matUp, 1 \right\rrbracket$, where $\lambda \in I_1$.
      \end{lemma}
\noindent    Using these lemmas, we get the following  formulas for the operators $W_p$ and $U_p$
    on $\IZind E$:
    \begin{eqnarray}\label{Formulae for Up and Wp in char 0}
        W_p\llbracket g, 1\rrbracket = \llbracket g\beta, 1\rrbracket, \quad U_p\llbracket g, 1\rrbracket = \sum_{\lambda \in I_1} \left\llbracket g \matUp , 1\right\rrbracket.
    \end{eqnarray}
    The operators $W_p = T_{1,0}$ and $U_p = T_{1,2}$ are actually Iwahori-Hecke operators (if one replaces $E$ by
    $\Fq$) in the non-commutative Iwahori-Hecke algebra recalled below. 

    We now extend the definition of $W_p$ and $U_p$. We have
    \begin{lemma}\label{Projection formula}
    For a representation $\sigma$ of $G$, the map $\IZind \sigma \to \sigma \otimes \IZind E$ sending $\llbracket g, v\rrbracket$ to $gv \otimes \llbracket g, 1\rrbracket$ is an isomorphism of $G$-representations.
\end{lemma}
\noindent Let $\SymE{r} = |\det|^\frac{r}{2} \otimes \mathrm{Sym}^r E^2$ for $r \geq 0$. 
Using this lemma, we can define $\tWp$ and $\tUp$ on $\IZind \SymE{r}$
as follows:
    \[
        \tWp\llbracket g, P\rrbracket = \llbracket g\beta, \beta^{-1}P\rrbracket, \>\> \tUp\llbracket g, P\rrbracket = \sum_{\lambda \in I_1} \left\llbracket g\matUp, \left(\matUp\right)^{-1}P\right\rrbracket.
    \]

    \subsection{Iwahori-Hecke operators}
    \label{Comparison theorem section}

    The Iwahori-Hecke algebra is by definition the endomorphism algebra $\mathrm{End}_G (\IZind d^r)$
    for $0 \leq r \leq p-1$. We will sometimes loosely refer to it as the Hecke algebra.
    In this section, we recall formulas for the generators of the Iwahori-Hecke algebra
    following \cite{BL94,BL95}.

  \subsubsection{Commutative Hecke algebra}

  It is well known (cf. \cite{BL95}) that if $0 < r < p - 1$, then
  the Iwahori-Hecke algebra of the representation $\IZind d^r$ is \emph{commutative} and
  generated by two operators $T_{-1, 0}$ and $T_{1, 2}$ satisfying
    \[
        T_{-1, 0}T_{1, 2} = 0 = T_{1, 2}T_{-1, 0}.
    \]
    These operators are defined by the following formulas (cf. \cite{AB15})
    \begin{eqnarray}\label{Formulae for T-10 and T12 in the commutative case}
      T_{-1, 0}\llbracket \id, v\rrbracket = \sum_{\lambda \in I_1}\left\llbracket \begin{pmatrix}p & \lambda \\ 0 & 1\end{pmatrix}, v\right\rrbracket \text{ and } T_{1, 2}\llbracket \id, v\rrbracket = \sum_{\lambda \in I_1}\left\llbracket \begin{pmatrix}1 & 0 \\ p\lambda & p\end{pmatrix}, v\right\rrbracket,
    \end{eqnarray}
    for $v \in \Fq(d^r)$.
    The Hecke algebra does not change if we twist $\IZind d^r$ by a character of $G$ (cf. \cite{BL94}).
    Using an analog of Lemma \ref{Projection formula}, we see that the formulas for $T_{1, 2}$ and $T_{-1, 0}$ given above do not change if we replace $\IZind d^r$ by $\IZind (d^r\otimes \omega^s)$ for any $0 \leq s \leq p - 2$.

\subsubsection{Non-commutative Hecke algebra}
\label{subsubsection non-commutative}

If $r = 0$, $p - 1$, then the Iwahori-Hecke algebra of $\IZind d^r = \IZind 1$ is a \emph{non-commutative} algebra
generated by two operators $T_{1, 0}$ and $T_{1, 2}$ satisfying
    \begin{eqnarray}\label{Relations in the non-commutative Hecke algebra}
        T_{1, 0}^2 = 1 \text{ and } T_{1, 2}T_{1, 0}T_{1, 2} = -T_{1, 2}.
    \end{eqnarray}
    Using \cite[Lemma 6]{BL95} and analogs of Lemmas \ref{Small operator in the non-commutative Hecke algebra}, \ref{Big operator in the non-commutative Hecke algebra} the formulas for these operators are\footnote{Compare with \eqref{Formulae for Up and Wp in char 0}.}
    \begin{eqnarray*}
        T_{1, 0}\llbracket \id, v\rrbracket = \llbracket \beta, v\rrbracket \text{ and } T_{1, 2}\llbracket \id, v\rrbracket = \sum_{\lambda \in I_1}\left\llbracket \begin{pmatrix}1 & 0 \\ p\lambda & p\end{pmatrix}, v\right\rrbracket,
    \end{eqnarray*}
    for $v \in \Fq$. Using an analog of Lemma~\ref{Projection formula}, we see that the formula for $T_{1, 0}$
    changes slightly  if we replace $\IZind 1$ by $\IZind \omega^s$ for any $0 \leq s \leq p - 2$ but the formula for $T_{1, 2}$ remains the same. Indeed, we have
    \begin{eqnarray}\label{Formulae for T10 and T12 in the non-commutative case}
        T_{1, 0}\llbracket \id, v\rrbracket = \llbracket \beta, (-1)^{s}v\rrbracket \text{ and } T_{1, 2}\llbracket \id, v\rrbracket = \sum_{\lambda \in I_1}\left\llbracket \begin{pmatrix}1 & 0 \\ p\lambda & p\end{pmatrix}, v\right\rrbracket.
    \end{eqnarray}
    In this non-commutative Hecke algebra, we also have the operator $T_{-1, 0}$ given by 
    \begin{eqnarray}\label{T-10 in the non-commutative case}
        T_{-1, 0} = T_{1, 0}T_{1, 2}T_{1, 0}.
    \end{eqnarray}
    The formula for this operator continues to be given
    by \eqref{Formulae for T-10 and T12 in the commutative case}.

\subsection{Iwahori mod $p$ LLC}

For $0 \leq r \leq p-1$, $\lambda \in \brFp$ and $\eta: \Qp^* \to \brFp^*$ a smooth character, define
the smooth mod $p$ representation of $G$
as the following twisted quotient of compactly supported Iwahori induction:
\begin{eqnarray*}
    \label{def of pi}
    \pi(r,\lambda, \eta) & :=      & \frac{\IZind d^r}{(T_{-1, 0} + \delta_{r, p - 1}T_{1, 0})
                                       + (T_{1, 2} + \delta_{r, 0}T_{1, 0} - \lambda)} \otimes \eta,
\end{eqnarray*}
where $\delta_{a,b} = 1$ if $a = b$ and is $0$ otherwise.
  This representation coincides with the well-known representation with the same name involving spherical induction
  by the work of \cite{AB15} in the case $0 < r < p-1$ and \cite{Chi23} in the case $r = 0$, $p-1$. In particular,
  the above definition is really a theorem (cf. \cite{Chi23}). 
  Using it, one may reinterpret Breuil's (semi-simple) mod $p$ LLC \cite[Definition 4.2.4]{Bre03}
  Iwahori theoretically as was done in \cite[Theorem 2.2]{Chi23}:

  \begin{theorem}[Iwahori mod $p$ LLC]\label{Iwahori mod p LLC}
    For $r \in \{0, \ldots, p - 1\}$, $\lambda \in \brFp$ and a smooth
character $\eta: \Qp^* \to \brFp^*$, we have
\begin{itemize}
            \item If $\lambda = 0$:
                \[
                    (\ind \omega_2^{r + 1}) \otimes \eta \>\> 
                    \xleftrightarrow{\quad} \>\>  \pi(r,0,\eta) \quad \qquad \qquad \qquad \qquad \quad 
                \]
            \item If $\lambda \neq 0$:
                \begin{eqnarray*}
                    \> \> (\mu_{\lambda}\omega^{r + 1} \oplus
\mu_{\lambda^{-1}})\otimes \eta &
\xleftrightarrow{\quad } &
       \pi(r,\lambda,\eta)^{\rmss} \oplus     \pi([p-3-r],\lambda^{-1},\eta \omega^{r+1})^{\rmss},                             
\\
            \end{eqnarray*}
where $[a] \in \{0,\ldots,p-2\}$ represents the class of $a$ modulo $p-1$.    
            \end{itemize}
\end{theorem}

The Banach space $\tilde{B}(k,\sL)$ recalled in the next section 
was defined by Breuil using Iwahori induction. We have therefore (mostly for fun)
decided to write this paper using Iwahori induction. No spherical Hecke operators or Hecke algebras
appear explicitly in this paper.
In particular, we use the
above Iwahori theoretic avatar of the mod $p$ LLC above to
prove Theorem~\ref{Main theorem in the second part of my thesis} in
Section~\ref{Section containing the proof of the main theorem}.

However, sometimes we need to use the following alternative symmetric presentation of
$\pi$ which follows
from \cite[Theorem 5.1]{AB15} for $0 < r < p-1$ and \cite[Remark 3.11]{Chi23} for $r = 0,p-1$:
\begin{eqnarray}
    \label{alt def of pi}
    \pi(r,\lambda, \eta)  & \simeq  & \frac{\IZind a^r}{(T_{1, 2} + \delta_{r, p - 1}T_{1, 0})
                                       + (T_{-1, 0} + \delta_{r, 0}T_{1, 0} - \lambda)} \otimes \eta. 
\end{eqnarray}

A final remark. Let $i = \frac{r+1}{2}$ for $1 \leq r \leq p-2$ odd
and $c := \lambda_i + \lambda_i^{-1}$ be the quantity at the end of the
statement of Theorem~\ref{Main theorem in the second part of my thesis}. By the Chinese Remainder Theorem
for the ideal generated by the quadratic polynomial $T_{-1, 0}^2 - cT_{-1, 0} + 1 = (T_{-1, 0} - \lambda_i)(T_{-1, 0} - \lambda_i^{-1})$, we have
\begin{eqnarray}
  \label{pi quadratic}
     \frac{\IZind a^{\frac{r - 1}{2}}d^{\frac{r + 1}{2}}}{T_{1, 2} + (T_{-1, 0}^2 - cT_{-1, 0} + 1)}
  & \simeq & \left( \frac{\IZind a^{-1}}{T_{1,2} + (T_{-1, 0} - \lambda_i)} \oplus \frac{\IZind a^{-1}}{T_{1,2} + (T_{-1, 0} - \lambda_i^{-1})} \right) \otimes (ad)^{\frac{r + 1}{2}} \nonumber \\
  & \simeq & \pi(p - 2, \lambda_i, \omega^{\frac{r + 1}{2}}) \oplus \pi(p - 2, \lambda_i^{-1}, \omega^{\frac{r + 1}{2}}),
\end{eqnarray}
where the second isomorphism follows from  \eqref{alt def of pi} (with $r = p-2$).
If $\lambda_i = \pm 1$, then CRT no longer applies since the roots are repeated,
but, as is easily checked, the first isomorphism above still holds up to semi-simplification. 
  
\section{The Banach space $\tB(k, \cL)$}\label{Definition of the Banach space}

In \cite[Section 4.2]{Bre04}, Breuil has defined the Local Langlands correspondent $B(k, \cL)$ of the semi-stable representation $V_{k, \cL}$ for $k \geq 3$. In this section, we recall the presentation of $B(k, \cL)$ given in \cite[Corollary 3.3.4]{Bre10}. 

We begin by proving a few easy but tedious lemmas. Define $f_{\id} \in (\Bind E)^{\sm}$ by 
    \[
        f_{\id}(g) = \begin{cases}
            1 & \text{ if $g \in BIZ$} \\
            0 & \text{ otherwise.}
        \end{cases}
    \]
    To check that $f_{\id} \in (\Bind E)^{\sm}$, we note that for any $b \in B$ and $g \in G$ the condition $g \in BIZ$ is equivalent to $bg \in BIZ$ and that $IZ$ is an open subset of $G$. The following lemma makes
    an isomorphism from \cite[Section 4.6]{Bre04} explicit.


\begin{lemma}\label{Steinberg as a quotient}
    The map
    \[
        \psi : \IZind E \to (\Bind E)^{\sm}
    \]
    sending $\llbracket \id, 1\rrbracket$ to $f_{\id}$ induces an isomorphism $\dfrac{\IZind E}{(W_p + 1, U_p - 1)} \to \St$.
\end{lemma}
\begin{proof}
  We remark that $\psi$ arises by Frobenius reciprocity from
  the $IZ$-map $E \rightarrow (\Bind E)^{\sm}$ sending $1$ to $f_{id}$, noting that    
    $hf_{\id} = f_{\id}$ for all $h \in IZ$.
    
    $\psi$ induces a surjection $\dfrac{\IZind E}{(W_p + 1, U_p - 1)} \twoheadrightarrow \St$:
    
    By definition, $\St = (\Bind E)^{\sm}/E$. Since $f_{\id}$ is not a constant function, $\psi$ followed by the projection $(\Bind E)^{\sm} \twoheadrightarrow (\Bind E)^{\sm}/E$ is not the zero map. Since $\St$ is irreducible, we see that $\psi$ induces a surjection $(\Bind E)^{\sm} \twoheadrightarrow \St$.
    
    Now we check that the images of $W_p + 1$ and $U_p - 1$ map to $0$ under $\IZind E \to \St$. Note that $(W_p + 1)\llbracket \id, 1\rrbracket = \llbracket \beta, 1\rrbracket + \llbracket \id, 1\rrbracket$. Therefore for $g \in G$, we have
    \[
        \psi((W_p + 1)\llbracket \id, 1\rrbracket)(g) = f_{\id}(g \beta) + f_{\id}(g).
    \]
    We claim that $f_{\id}(g\beta) + f_{\id}(g)$ is the constant function taking every $g \in G$ to $1 \in E$. Note that the matrices $w = \begin{pmatrix}0 & 1 \\ 1 & 0\end{pmatrix}$ and $g_{\mu} := \begin{pmatrix}1 & 0 \\ \mu & 1\end{pmatrix}$ for $\mu \in \Qp$ form a system of representatives for $B\backslash G$. So we only need to check that $f_{\id}(g \beta) + f_{\id}(g) = 1$ for these matrices. 
    
    Now $w$ belongs to $BIZ = BI$ if and only if there is a matrix $\begin{pmatrix}x & y \\ 0 & z\end{pmatrix} \in B$ such that
    \[
        \begin{pmatrix}x & y \\ 0 & z\end{pmatrix}
        w =  \begin{pmatrix}y & x \\ z  & 0\end{pmatrix} \in I.
    \]
    This last matrix does not belong to $I$ for any $x, y, z \in \Qp$. This proves that $f_{\id}(w) = 0$.
    Furthermore, $w\beta = \begin{pmatrix}p & 0 \\ 0 & 1\end{pmatrix}$ clearly belongs to $B \subset BIZ$. So $f_{\id}(w\beta) = 1$. This proves that $f_{\id}(w\beta) + f_{\id}(w) = 1$.
    
    Next, consider the matrix $g_{\mu}$ for $\mu \in p\Zp$. Then clearly $g_{\mu} \in I \subset BIZ$. So $f_{\id}(g_{\mu}) = 1$. Again, $g_{\mu}\beta = \begin{pmatrix}0 & 1 \\ p & \mu\end{pmatrix}$ belongs to $BIZ = BI$ if and only if there is a matrix $\begin{pmatrix}x & y \\ 0 & z\end{pmatrix} \in B$ such that
    \[
        \begin{pmatrix}x & y \\ 0 & z\end{pmatrix}g_{\mu}\beta = \begin{pmatrix}py & x + \mu y \\ pz & \mu z\end{pmatrix} \in I.
    \]
    This forces $\mu z \in \Zp^*$ and $pz \in p\Zp$, which is not possible as $\mu \in p\Zp$. So $f_{\id}(g_\mu \beta) = 0$. This proves that $f_{\id}(g_\mu) + f_{\id}(g_\mu \beta) = 1$ for $\mu \in p\Zp$.
    
    Finally, consider $g_\mu$ for $\mu \in \Qp \setminus p\Zp$. Then
    \[
        \begin{pmatrix}-\mu/p & 1/p \\ 0 & 1/\mu\end{pmatrix}g_\mu \beta =  \begin{pmatrix}1 & 0 \\ p/\mu & 1\end{pmatrix} \in I.
    \]
    So $f_\id(g_\mu \beta) = 1$. Moreover, $g_\mu$ belongs to $BIZ = BI$ if and only if there is a matrix $\begin{pmatrix}x & y \\ 0 & z\end{pmatrix} \in B$ such that
    \[
        \begin{pmatrix}x & y \\ 0 & z\end{pmatrix} g_\mu = \begin{pmatrix}x + \mu y & y \\ \mu z & z\end{pmatrix} \in I.
    \]
    This forces $\mu z \in p\Zp$ and $z \in \Zp^*$, which is not possible as $\mu \in \Qp \setminus p\Zp$. So $f_\id(g_\mu) = 0$. This proves that $f_\id(g_\mu) + f_\id(g_\mu \beta) = 1$ for $\mu \in \Qp \setminus p\Zp$.
    
    This computation shows that $(W_p + 1)\llbracket \id, 1\rrbracket$ maps to the constant function $1$ in $(\Bind E)^{\sm}$. A similar computation shows that $(U_p - 1)\llbracket \id, 1\rrbracket$ maps to the $0$ function in $(\Bind E)^{\sm}$.
    
    These two facts show that $\psi$ induces a map $\dfrac{\IZind E}{(W_p + 1, U_p - 1)} \twoheadrightarrow \St$.
    
    It turns out that $\psi : \dfrac{\IZind E}{(W_p + 1, U_p - 1)} \twoheadrightarrow \St$ is injective (\cite[Section 4.6]{Bre04}).
\end{proof}
Using the techniques used to prove the lemma above, we get the following lemma.
\begin{lemma}
    \label{H}
    Let $H : \Qp \to E$ be a function defined by $H(z) = f_\id\begin{pmatrix}1 & 0 \\ z & 1\end{pmatrix}$. Then $H(z) = \mathbbm{1}_{p\Zp}(z)$ for all $z \in \Qp$.
\end{lemma}

In \cite[Section 4.2]{Bre04}, Breuil defines two Banach space $B(k)$ and $B(k,\sL)$ for $\sL \in E$.
We do not use the original definitions in this paper. However we note the following.
Using Lemma \ref{Projection formula} and Lemma \ref{Steinberg as a quotient}, we have the following composition of maps $\psi_k : \IZind \SymE{k - 2} \xrightarrow{\sim} \SymE{k - 2} \otimes \IZind E \twoheadrightarrow \SymE{k - 2} \otimes \St$. Since $\IZind \SymO{k - 2}$ is a lattice in the space $\IZind \SymE{k - 2}$, we see that $\psi_k(\IZind \SymO{k - 2})$ is a lattice in the space $\SymE{k - 2} \otimes \St$ \cite[Proposition 4.6.1]{Bre04}.
The completion of $\SymE{k - 2} \otimes \St$ with respect to this lattice is isomorphic to $B(k)$ \cite[Proposition 4.3.5(i)]{Bre04}. 

In \cite{Bre10}, the Banach space $B(k)$ is presented in a different way as a space of locally analytic functions from $\Qp$ to $E$. Moreover, it is proved that $B(k, \cL)$ is a quotient of $B(k)$. Let us recall these two facts.

For a real number $r \geq 0$, a continuous function $h : \Zp \to E$ belongs to the space $\C{r}(\Zp, E)$ if in its Mahler's expansion $h(z) = \sum_{n = 0}^{\infty}a_n(h){z \choose n}$, the coefficients $a_n(h)$ satisfy $n^r\norm{a_n(h)} \to 0$ as $n \to \infty$. The space $\C{r}(\Zp, E)$ is a Banach space with the norm $\Bnorm{h}{r} = \sup_n (n + 1)^r\norm{a_n(h)}$.

\begin{definition}
    Let $D(k)$ be the $E$-vector space of functions $h : \Qp \to E$ such that $h_1 : (z \mapsto h(z))\vert_{\Zp}$ belongs to $\C{\frac{k - 2}{2}}(\Zp, E)$ and $h_2 : (z \mapsto z^{k - 2}h(1/z))\vert_{\Zp \setminus \{0\}}$ extends to $\Zp$ as a function belonging to $\C{\frac{k - 2}{2}}(\Zp, E)$.
\end{definition}

This is a Banach space\footnote{We remark that we have used the definition of $h_1$ from \cite[Corollary 3.2.3]{Bre10} and not from where it is first defined in \cite[Section 3.2]{Bre10}, where there is an extra `$p$'.} for the norm $\Bnorm{h}{} = \max(\Bnorm{h_1}{\frac{k - 2}{2}}, \Bnorm{h_2}{\frac{k - 2}{2}})$. The action of $G$ on $D(k)$ is as follows\footnote{We have used the original action of $G$ defined in \cite{Bre04} (cf. \cite[Proposition 2.2.1]{Bre04}) as opposed to the one in \cite[Section 3.2]{Bre10}. This is done to ensure the $G$-equivariance of the map in Lemma \ref{Map from sym tensor st to D(k)} and the surjection \eqref{Main surjection on mod p representations} below, noting that in this paper we use the original standard action on $\SymE{k - 2}$ and $\mathrm{St}$.}:

\begin{eqnarray}
    \label{G-action}
  \left(\begin{pmatrix}a & b \\ c & d\end{pmatrix}h\right)(z) = \norm{ad - bc}^{\frac{k - 2}{2}}(bz + d)^{k - 2}h\left(\frac{az + c}{bz + d}\right).
\end{eqnarray}

Note that polynomials of degree less than or equal to $k - 2$ belong to $D(k)$. Indeed, if $h$ is a polynomial of degree less than or equal to $k - 2$, then both $h_1$ and $h_2$ are polynomials and therefore belong to $\C{\frac{k - 2}{2}}(\Zp, E)$.

\begin{definition}
    Set $\tB(k)$ to be the quotient of $D(k)$ by polynomials of degree less than or equal to $k - 2$.
\end{definition}
\noindent The Banach spaces $B(k)$ and $\tB(k)$ are isomorphic (cf. \cite[Corollary 3.2.3]{Bre10}).

Next we show that the space $\SymE{k - 2} \otimes \St$ is contained in $\tB(k)$. 

\begin{lemma}\label{Map from sym tensor st to D(k)}
    The map $\SymE{k - 2} \otimes (\Bind E)^{\sm} \to D(k)$ defined by sending $P(X, Y) \otimes f$ to $P(z, 1) \cdot f\left(\begin{pmatrix}1 & 0 \\ z & 1\end{pmatrix}\right)$ is an injective $G$-homomorphism.
\end{lemma}
\begin{proof}
    The map is well-defined since $f\left(\begin{pmatrix}1 & 0 \\ 1/z & 1\end{pmatrix}\right) = f\left(\begin{pmatrix}0 & 1 \\ 1 & z\end{pmatrix}\right)$. The $G$-equivariance of this map is clear. We prove injectivity. Suppose $\sum_{i = 1}^n P_i(X, Y) \otimes f_i$ maps to $0$. Since there are finitely many $f_i$, we may assume without loss of generality that $f_i$ have pairwise disjoint supports. This means that for every $z \in \Qp$, there is at most one $i = 1, 2, \ldots, n$ such that $f_i\left(\begin{pmatrix}1 & 0 \\ z & 1\end{pmatrix}\right) \neq 0$. Therefore restricting the function $\sum_{i = 1}^n P_i(z, 1) \cdot f\left(\begin{pmatrix}1 & 0 \\ z & 1\end{pmatrix}\right)$ to $z$ such that $\begin{pmatrix}1 & 0 \\ z & 1\end{pmatrix}$ belongs to the support of $f_i$, we see that the polynomial $P_i(z, 1)$ vanishes on infinitely many elements in $\Qp$. Therefore $P_i(X, Y) = 0$. This proves injectivity.
\end{proof}

Note that the image of $\SymE{k - 2} \otimes E \subset \SymE{k - 2} \otimes (\Bind E)^{\sm}$ under the map defined in the lemma above is the subspace of $D(k)$ consisting of polynomials of degree less than or equal to $k - 2$.
Therefore we see that $\SymE{k - 2} \otimes \St$ injects into $\tB(k)$.

Now we define $\tB(k, \cL)$. 
Let $L(k, \cL)$ be the subspace of $D(k)$ generated by polynomials of degree less than or equal to $k - 2$ and functions of the form 
\[
    h(z) = \sum_{i \in I} \lambda_i(z - z_i)^{n_i} \logL(z - z_i),
\]
where $I$ is a finite set, $\lambda_i \in E$, $z_i \in \Qp$, $n_i \in \{\lfloor\frac{k - 2}{2}\rfloor + 1, \ldots, k - 2\}$ and $\deg \sum_{i \in I} \lambda_i (z - z_i)^{n_i} < \frac{k - 2}{2}$ (\cite[Lemme 3.3.2]{Bre10}).

\begin{definition}
  Define  $\tB(k, \cL)$ as the quotient of $D(k)$ by the closure of $L(k,\cL)$ in $D(k)$.
\end{definition}

In other words, $\tB(k, \cL)$ is the quotient of $\tB(k)$ by the closure of the subspace of $\tB(k)$ consisting of functions
    \[
    h(z) = \sum_{i \in I} \lambda_i(z - z_i)^{n_i} \logL(z - z_i),
\]
where $I$ is a finite set, $\lambda_i \in E$, $z_i \in \Qp$, $n_i \in \{\lfloor\frac{k - 2}{2}\rfloor + 1, \ldots, k - 2\}$ and $\deg \sum_{i \in I} \lambda_i (z - z_i)^{n_i} < \frac{k - 2}{2}$.
Moreover, $\tB(k, \cL)$ is isomorphic to $B(k, \cL)$, the Local Langlands correspondent of $V_{k, \cL}$
(cf. \cite[Corollaire 3.3.4]{Bre10}). This shows that  $B(k, \cL)$ is a quotient of $B(k)$.

In this
paper, we shall only work with $\tB(k)$ and especially $\tB(k,\cL)$ (which have explicit presentations)
rather than $B(k)$ and $B(k,\sL)$ (which are defined somewhat abstractly). 

\section{The lattices $\lattice{k}$ and $\latticeL{k}$}\label{Section defining the lattices}

In this section, we recall the definitions of two important lattices $\lattice{k}$ and $\latticeL{k}$
in the Banach spaces $\tB(k)$ and $\tB(k,\cL)$ introduced in the previous section. We also describe
some explicit elements of these lattices.
    
Recall that there is a surjection $\IZind \SymE{k - 2} \twoheadrightarrow \SymE{k - 2}\otimes \St$. Breuil defines $\Theta(k)$ to be the image of $\IZind \SymO{k - 2}$ under this surjection \cite[Section 4.6]{Bre04}. Since $B(k)$ is the completion of $\SymE{k - 2} \otimes \St$ with respect to the lattice $\Theta(k)$, the closure $\hat{\Theta}(k)$ of the lattice $\Theta(k)$ in $B(k)$ is a lattice in $B(k)$. Note that by definition, $\Theta(k)$ is dense in $\hat{\Theta}(k)$. Therefore they have the same reduction mod $p$ (more precisely mod $\pi$).  Let $\tilde{\Theta}(k)$ be the image of $\hat{\Theta}(k)$ under the identification $B(k) \xrightarrow{\sim} \tB(k)$.
Since the image of $\IZind \SymO{k - 2}$ under $\IZind \SymE{k - 2} \twoheadrightarrow \SymE{k - 2} \otimes \St \subset \tB(k)$ is dense in $\tilde{\Theta}(k)$,
the reduction mod $p$ of the image of $\IZind \SymO{k - 2}$ in $\tB(k)$ is equal to the reduction
mod $p$ of $\tilde{\Theta}(k)$.

Let $\Fq = \co_E/\pi$ be the residue field of $E$.
Now the map $\IZind \SymO{k - 2} \to \lattice{k}$ 
described in the paragraph above gives us a surjection 
\begin{eqnarray}\label{Main surjection on mod p representations}
    \IZind \SymF{k - 2} \twoheadrightarrow \br{\lattice{k}}.
\end{eqnarray}
We prove the following lemmas about 
the map \eqref{Main surjection on mod p representations}.

\begin{lemma}\label{Bottom half factors die}
  The surjection \eqref{Main surjection on mod p representations}
  factors as
    $
        \IZind \SymF{k - 2} \twoheadrightarrow \IZind \frac{\SymF{k - 2}}{\oplus_{j > \frac{k - 2}{2}}\Fq X^jY^{k - 2 - j}} \twoheadrightarrow \br{\lattice{k}}.
    $
\end{lemma}
\begin{proof}
  The subspace $\IZind \oplus_{j > \frac{k - 2}{2}}\Fq X^jY^{k - 2 - j} \subseteq \IZind \SymF{k - 2}$ is generated by the functions $\llbracket \id, X^jY^{k - 2 - j}\rrbracket$ for $\dfrac{k - 2}{2} < j \leq k - 2$. Therefore to prove the lemma, we have to prove that for $j > \frac{k - 2}{2}$, the function $\llbracket \id, X^jY^{k - 2 - j}\rrbracket$ maps to $0$ under $\IZind \SymF{k - 2} \twoheadrightarrow \br{\lattice{k}}$. Under this map, the function $\llbracket \id, X^jY^{k - 2- j}\rrbracket$ goes to $z^j\mathbbm{1}_{p\Zp} \mod \pi \lattice{k}$, by Lemmas~\ref{Projection formula}, \ref{Steinberg as a quotient}, \ref{Map from sym tensor st to D(k)} and \ref{H}.
  
    By \eqref{G-action}, we have the following identity in $\lattice{k}$
    \[
        -\beta p^{j - \frac{k - 2}{2}} z^{k - 2 - j}\mathbbm{1}_{p\Zp}(z) = -\frac{1}{p^{\frac{k - 2}{2}}} z^{k - 2} p^{j - \frac{k - 2}{2}}\left(\frac{p}{z}\right)^{k - 2 - j}\mathbbm{1}_{p\Zp}\left(\frac{p}{z}\right) = -z^{j}\mathbbm{1}_{p\Zp}\left(\frac{p}{z}\right).
    \]
    Now $v_p\left(\dfrac{p}{z}\right) \geq 1 \iff v_p(z) \leq 0$. Therefore $\mathbbm{1}_{p\Zp}\left(\dfrac{p}{z}\right) = \mathbbm{1}_{\Qp \setminus p\Zp}(z)$. Moreover, $z^j\mathbbm{1}_{\Qp \setminus p\Zp} = -z^j\mathbbm{1}_{p\Zp}$ in $\lattice{k}$ since $z^j$ is a polynomial of degree $\leq k-2$. Therefore we see that the last term in the display above is equal to $z^j\mathbbm{1}_{p\Zp}$. Since $j > \frac{k - 2}{2}$, we see that the leftmost term in the display above belongs to $\pi\lattice{k}$. Therefore $z^j\mathbbm{1}_{p\Zp} \equiv 0 \mod \pi\lattice{k}$.
\end{proof}

\begin{lemma}\label{Socle of the top slice dies}
  Under the surjection $\IZind d^{k - 2} \cong \IZind \dfrac{\SymF{k - 2}}{\oplus_{j < k - 2}\Fq X^jY^{k - 2 - j}} \twoheadrightarrow \dfrac{\br{\lattice{k}}}{S}$,
  where $S$ is the image of $\IZind \oplus_{j < k - 2}\Fq X^jY^{k - 2 - j}$ under \eqref{Main surjection on mod p representations},
  the subspace $\mathrm{Im}\>T_{1, 2} \subseteq \IZind d^{k - 2}$ maps to $0$.
\end{lemma}
\begin{proof}
    Under the surjection above, the function
    $
      T_{1,2}\llbracket \id, Y^{k - 2}\rrbracket
       =  \sum_{\lambda \in I_1}\left\llbracket \begin{pmatrix}1 & 0 \\ -p\lambda & p\end{pmatrix}, Y^{k-2}\right\rrbracket$        maps to 
  \begin{eqnarray*}
    \sum_{\lambda \in I_1} \begin{pmatrix}1 & 0 \\ -p\lambda & p\end{pmatrix} \mathbbm{1}_{p\Zp}
    =    \sum_{\lambda \in I_1} \frac{1}{p^{r/2}} p^r  \mathbbm{1}_{p\lambda + p^2\Zp}                                                  =    p^{r/2} \mathbbm{1}_{p\Zp}      
  \end{eqnarray*}
  which vanishes mod $\pi\lattice{k}$ since $r > 0$.
\end{proof}

We prove two key lemmas which identify some (integral) elements in the lattice $\lattice{k}$. 
        
    
    \begin{lemma}\label{Integers in the lattice}
        For $r \geq 1$, $0 \leq j \leq r$, $h \in \bZ$ and $z_0 \in \Qp$, the elements 
        \[
            p^{(h - 1)\left(r/2 - j \right)}(z - z_0)^j\mathbbm{1}_{z_0 + p^h\Zp}\in \lattice{k}.
        \]
    \end{lemma}
    \begin{proof}
        Note that
        \[
          \begin{pmatrix} 1 & 0 \\ -z_0 & p^{h - 1}\end{pmatrix} z^j\mathbbm{1}_{p\Zp} =
          p^{(h - 1)\left(r/2 - j \right)}(z - z_0)^j\mathbbm{1}_{z_0 + p^h\Zp}.
        \]
        The lemma follows since $z^j\mathbbm{1}_{p\Zp} \in \lattice{k}$.
    \end{proof}
    

    \begin{lemma}\label{stronger bound for polynomials of large degree with varying radii}
        For $r \geq 1$, $r/2 \leq j \leq r$, $h \in \bZ$ and $z_0 \in \Qp$, the elements 
        \[
            p^{h\left(r/2 - j\right)}(z - z_0)^j\mathbbm{1}_{z_0 + p^h\Zp} \in \lattice{k}.
        \]
    \end{lemma}
    \begin{proof}
        Note that
        \[
          -\begin{pmatrix}0 & 1 \\ p^h & -z_0\end{pmatrix}z^{r - j}\mathbbm{1}_{p\Zp} =
          -p^{h\left(r/2 - j\right)}(z - z_0)^j\mathbbm{1}_{p\Zp}\left(\frac{p^h}{z - z_0}\right).
        \]
        Now
        \[
          \mathbbm{1}_{p\Zp}\left(\frac{p^h}{z - z_0}\right) = 1 \iff \frac{p^h}{z - z_0} \in p\Zp
          \iff v_p(z - z_0) < h \iff \mathbbm{1}_{\Qp \setminus \left(z_0 + p^h\Zp\right)}(z) = 1.
        \]
        Using the fact that polynomials of degree less than or equal to $r$ are equal to $0$ in $\tB(k)$, we see that
        \[
            -(z - z_0)^j\mathbbm{1}_{p\Zp}\left(\frac{p^h}{z - z_0}\right) = -(z - z_0)^j\left(1 - \mathbbm{1}_{z_0 + p^h\Zp}(z)\right) = (z - z_0)^j\mathbbm{1}_{z_0 + p^h\Zp}.
        \]
        The lemma then follows since $z^{r - j}\mathbbm{1}_{p\Zp} \in \lattice{k}$.
      \end{proof}

      Since $\tB(k)$ surjects onto $\tB(k, \cL)$, the image of $\tilde{\Theta}(k)$
      under $\tB(k) \twoheadrightarrow \tB(k, \cL)$ is a lattice, say  $\tilde{\Theta}(k, \cL)$,
      in $\tB(k, \cL)$. As mentioned in the introduction, by the Iwahori mod $p$ LLC
      (Theorem~\ref{Iwahori mod p LLC}), to determine
      the reduction $\br{V}_{k,\cL}$, it suffices to determine the reduction
      $\br{\tilde{\Theta}(k, \cL)}$ of $\tilde{\Theta}(k, \cL)$.

      Due to the surjection $\lattice{k} \twoheadrightarrow \latticeL{k}$,
      we may replace $\lattice{k}$ in the four lemmas above by $\latticeL{k}$.
      The module $\IZind \frac{\SymF{k - 2}}{\oplus_{j > \frac{k - 2}{2}}\Fq X^jY^{k - 2 - j}}$
      has a filtration where the successive subquotients are given by
      $\IZind d^r, \IZind ad^{r-1}, \ldots, \IZind a^{\lfloor r/2 \rfloor}d^{\lceil r/2 \rceil}$.
      Using Lemma~\ref{Bottom half factors die}, this induces a filtration on $\latticeL{k}$.
      We denote the subquotients of this filtration on
      $\latticeL{k}$ by $$F_{0, 1}, \> F_{2, 3}, \> \ldots, \> F_{2\lfloor r/2 \rfloor, \>  2\lfloor r/2 \rfloor + 1}.$$
      We therefore have surjections
  \[
    \IZind a^ld^{r-l} \twoheadrightarrow F_{2l,\> 2l + 1},
  \]
  for $0 \leq l \leq \lfloor r/2 \rfloor$. Theorem~\ref{Main theorem in the second part of my thesis}
  is proved by a detailed analysis of the subquotients $F_{2l,\> 2l + 1}$ of $\latticeL{k}$.

\section{Logarithmic functions}\label{Section for log functions}
    In this section, we write down important classical properties of certain logarithmic functions.
    The logarithmic functions that we are interested in are $E$-valued functions of $z \in \qp$
    of the form
    \[
        g(z) = \sum_{i \in I}\lambda_i (z - z_i)^{n}\logL(z - z_i),
    \]
    where $I$ is a finite set, $\lambda_i \in E, z_i \in \Zp, n \geq 1$. We require
    $\sum_{i \in I}\lambda_iz_i^j = 0$ for $0 \leq j \leq n$. However, in one case
    we will only have this identity for $0 \leq j \leq n - 1$.

    Let us first define the function $\logL : \Qp^* \to E$. Given an element $z \in \Qp^*$, write $z = p^{v_{p}(z)}\zeta u$ for some $(p - 1)^{\mathrm{th}}$ root of unity $\zeta$ and some $u \in 1 + p\Zp$. Set $\logL(z) = v_p(z)\cL + \log(u)$, where $\log(u)$ is given by the usual expansion for $\log(1 + T)$.

    Note that $\logL$ is a continuous function on $\Qp^*$ since $\Qp^*$ can be covered by pairwise disjoint open sets $U_i = \{z \in \Qp^* \mid v_p(z) = i\}$ for $i \in \bZ$ on which $\logL(z)$ is given by $i\cL + \log(z/p^i)$. The function $\logL$ is also infinitely differentiable on $\qp^*$.
    Moreover, $\logL$ does not extend to a continuous function on $\Qp$. Indeed, the limits of $\logL(p^{1 + p^n})$ and $\logL(p^{p^n})$ are $\cL$ and $0$, respectively.

    However, the function $z^n\logL(z)$ extends to a continuous function on $\Qp$ for any integer $n \geq 1$. Indeed, for any $z \in \Qp^*$, the valuation of $\logL(z)$ is bounded below by $\min\{v_p(\cL), 1\}$. Therefore the limit of $z^n\logL(z)$ as $z \to 0$ is equal to $0$. So $z^n\logL(z)$ can be extended to a continuous function on $\Qp$ if we set its value at $0$ to be $0$.

    Next we check if $z^n\logL(z)$ is differentiable. Let us check this for $n = 1$.
    For $z \neq 0$, the function $z\logL(z)$ is clearly differentiable with derivative
    $\logL(z) + 1$. To check its differentiability at $0$, we have to evaluate the limit
    \[
        \lim_{z \to 0}\frac{z^n\logL(z) - 0}{z}.
    \]
    We have seen above that this limit does not exist for $n = 1$. Therefore $z\logL(z)$ is only differentiable for $z \neq 0$. For $n = 2$ however, one sees that the derivative limit exists at $0$. Therefore, the function $z^2\logL(z)$ is differentiable everywhere with derivative given by $2z\logL(z) + z$. Since the derivative involves the function $z\logL(z)$, we see that $z^2\logL(z)$ is differentiable everywhere only once. Extending
    this, we see that the function $z^n\logL(z)$ is differentiable everywhere only $n - 1$ times. Moreover, its $(n - 1)^{\mathrm{th}}$ derivative is continuous.

    The derivatives of the functions $z^n\logL(z)$ on $\qp^*$
    are connected with the $n^{\mathrm{th}}$ harmonic sum $H_n$. Indeed, for $z \neq 0$, we have
    \begin{eqnarray}\label{Small derivative formula for polynomial times logs}
        \frac{d^j}{dz^j}(z^n\logL(z)) = \frac{n!}{(n - j)!}\left[z^{n - j}\logL(z) + (H_n - H_{n - j})z^{n - j}\right], 
    \end{eqnarray}
    for $0 \leq j \leq n$. We remind the reader of the convention that $H_0 = 0$.

    Next we prove that for $n \geq 1$, the function $g(z) = z^{n}\logL(z)$ belongs to $\sC^s(\Zp, E)$
    for any $0 \leq s < n$ in the sense of \cite[Section I.5]{Col10a}. To show this we need to check
    that for
    \[
        \epsilon_s(x,y) = g(x + y) - \sum_{j = 0}^{\lfloor s \rfloor}g^{(j)}(x)\frac{y^j}{j!},
    \]
    we have
    \[
            \inf_{x \in \Zp, \> y \in p^h\Zp} v_p(\epsilon_s(x,y)) - sh \to \infty.
    \]
    as $h \to \infty$.
    
    Fix $h$. Let $x \in \Zp$. First assume that $h > v_p(x)$. 
    Therefore $x \neq 0$. For $y \in p^h\Zp$, we have the Taylor expansion for $\logL(1 + y/x)$. So
    \begin{eqnarray*}
        \epsilon_s(x,y) & = & \sum_{j = \lfloor s \rfloor + 1}^{\infty} g^{(j)}(x)\frac{y^j}{j!}.
    \end{eqnarray*}
    If $j \leq n$, then by \eqref{Small derivative formula for polynomial times logs},
    $g^{(j)}(x) = \frac{n!}{(n - j)!}x^{n - j}\logL(x) + t_jx^{n - j}$ for some $t_j \in \Zp$. We have seen above that the valuation of $\logL(x)$ is bounded below by $\min\{v_p(\cL), 1\}$. Therefore the valuation of the $j^{\mathrm{th}}$ summand above is bounded below by $\min\{v_p(\cL), 0\} + jh - v_p(j!)$. Using the formula
    \[
        v_p(j!) = \frac{j - \sigma_p(j)}{p - 1},
    \]
    where $\sigma_p(j)$ is the sum of the $p$-adic digits of $j$, we see that
    \[
        v_p\left(g^{(j)}(x)\frac{y^j}{j!}\right) \geq \min\{v_p(\cL), 0\} + jh - \frac{j - \sigma_p(j)}{p - 1} \geq \min\{v_p(\cL), 0\} + j\left(h - \frac{1}{p - 1}\right).
    \]
    Next if $j > n$, then 
    \[
        g^{(j)}(x) = n!\frac{(-1)^{j - n - 1}(j - n - 1)!}{x^{j - n}}.
    \]
    So similarly
    \[
        v_p\left(g^{(j)}(x)\frac{y^j}{j!}\right) \geq -(j - n)v_p(x) + jh - \frac{j - \sigma_p(j)}{p - 1} \geq n\left(h - \frac{1}{p - 1}\right) + (j - n)\left(h - \frac{1}{p - 1} - v_p(x)\right).
    \]
    Since $h > v_p(x)$, this is greater than or equal to $n\left(h - \frac{1}{p - 1}\right)$.
    Putting these cases together, we see that the valuation of the $j^{\mathrm{th}}$ term in
    $\epsilon_s(x,y)$ is greater than or equal to
    \[
        \min\{v_p(\cL), 0\} + (\lfloor s \rfloor + 1)\left(h - \frac{1}{p - 1}\right).
    \]

    Now assume that $h \leq v_p(x)$. Then for $y \in p^h \zp$, then by a similar computation
    one checks that the valuation of $g(x + y)$, $g(x)$, $g^{(1)}(x)y, \ldots, g^{(\lfloor s \rfloor)}(x)\frac{y^{\lfloor s \rfloor}}{\lfloor s \rfloor !}$, and therefore that of $\epsilon_s(x,y)$, is greater than or
    equal to
    \[
      \min\{v_p(\cL), 0\} + nh - v_p(\lfloor s \rfloor!).
    \]

    Both of these estimates then show that
    \[
        \inf_{x \in \Zp, y \in p^h\Zp} v_p(\epsilon_s(x,y)) - sh \to \infty
    \]
    as $h \to \infty$.
    Therefore, we have proved that the function $z^n\logL(z)$ belongs to $\sC^s(\Zp, E)$
    for $0 \leq s < n$.
    
    Next, we make a remark on the behavior of $g(z) = z^n\logL(z)$ near $\infty$. For any $r > n$,
    the function $z^rg(1/z)$ defined over $\Zp \setminus \{0\}$ extends to a function in $\sC^{s}(\Zp, E)$,
    for any $0 \leq s < r-n$. Indeed,
    \[
        z^rg\left(\frac{1}{z}\right) = -z^{r-n}\logL(z).
    \]

    In this paper, we are interested in functions of the form
    \begin{eqnarray}\label{polynomial log functions}
        g(z) = \sum_{i \in I} \lambda_i(z - z_i)^n\logL(z - z_i),
    \end{eqnarray}
    where $\lambda_i \in E$, $z_i \in \Zp$, $n \geq 1$ is an integer.
    Computations similar to the ones done above or \cite[Proposition I.5.6]{Col10a} (which
    shows that composition of functions in $\sC^{s}(\Zp, E)$ is in $\sC^{s}(\Zp, E)$)
    show that $g(z) \in \sC^{s}(\Zp, E)$ for any $0 \leq s < n$.
    More precisely, we shall only analyze functions of the form
    \eqref{polynomial log functions} 
    with
        \begin{enumerate}
            \item $I = \{0, \> 1\> , \ldots, \> p - 1\}$ and $z_i = [i]$ (here $[\>  \>]$ denotes Teichm\"uller lift)
            \item $I = \{0, \> 1, \> \ldots, n, \> p\}$  and $z_i = i$
            \item $I = \{0, \> 1, \ldots, \> p - 1\}$ and $z_i = i$  (this case is rarely
                used in this paper), 
        \end{enumerate}
     and $r/2 < n \leq r \leq p - 1$.

     We shall additionally impose the condition that
     \begin{eqnarray}
       \label{vanishing sums}
       \sum_{i \in I}\lambda_iz_i^j = 0 \text{ for }   0 \leq j \leq n
     \end{eqnarray}
     (in one exceptional case, namely the function in Proposition~\ref{nu geq}, we only have the
     vanishing for $0 \leq j \leq n - 1$ when $n = r = p-1$).
    One of the consequences of the condition $\sum_{i \in I}\lambda_iz_i^j = 0$ for $0 \leq j \leq n-1$
    is that the derivatives of $g(z)$ have a simple form. Indeed, for any $0 \leq j \leq n - 1$,
    we have (as opposed to formula \eqref{Small derivative formula for polynomial times logs})
    \begin{eqnarray}\label{Simple derivative formula}
        g^{(j)}(z) = \sum_{i \in I}\frac{n!}{(n - j)!}\lambda_i(z - z_i)^{n - j}\logL(z - z_i).
    \end{eqnarray}
    The other consequence of these conditions is that  $\deg \sum_{i \in I} \lambda_i(z-z_i)^{n} \leq 0 <
    \frac{k-2}{2}$, so that $g(z) \in L(k,\cL) \subset D(k)$  by \cite[Lemme 3.3.2]{Bre10}.
    In fact, this follows even if $\deg \sum_{i \in I} \lambda_i(z-z_i)^{n} < \frac{r}{2}$. Indeed 
    $$z^r g(1/z) = \sum_{i \in I} \lambda_i z^{r-n}  (1 - zz_i)^n \logL (1 - zz_i)  
                      - (\sum_{i \in I} \lambda_i z^{r-n} (1-zz_i)^n) \logL (z).$$
    (Each of the summands in) the first term is in $\sC^{r/2}(\Zp, E)$ since $z^n \logL (z) \in
    \sC^{r/2}(\Zp, E)$ and we may use \cite[Proposition I.5.6]{Col10a}, and one checks
    that the degree condition 
    implies that the second
    term is of the form $z^{\lfloor r/2 \rfloor +1} \logL(z)$ times a polynomial so is also
    in $\sC^{r/2}(\Zp, E)$.
    
    For each of the three cases above, we may pick $\lambda_i$  by the following lemma
    so that the condition \eqref{vanishing sums} holds.
        \begin{lemma}\label{Main coefficient identites}
            If
            \begin{enumerate}
                \item $I = \{0, \> 1\> , \ldots, \> p - 1\}$ and $z_i = [i]$ 
                \item $I = \{0, \> 1, \> \ldots, n, \> p\}$ for some $r/2 < n \leq r \leq p - 1$, and $z_i = i$
                \item $I = \{0, \> 1, \ldots, \> p - 1\}$ and $z_i = i$,
            \end{enumerate}
            then there are $\lambda_i \in \zp$ not all zero such that $\sum\limits_{i \in I}\lambda_i z_i^j = 0$ for $0 \leq j \leq \vert I \vert - 2$. Moreover, these $\lambda_i$ satisfy
            \begin{enumerate}
                \item $\lambda_0 = 1 - p$ and $\lambda_i = 1$ for all $1 \leq i \leq p - 1$
                \item $\lambda_0 = 1 \!\! \mod p$, $\lambda_p = -1$ and $\lambda_i = 0 \!\! \mod p$ for $1 \leq i \leq n$
                \item $\lambda_0 = 1$ and $\lambda_i = 1 \!\! \mod p$ for all $1 \leq i \leq p - 1$.
            \end{enumerate}
        \end{lemma}
        \begin{proof}
            We prove this lemma in three cases.
            \begin{enumerate}
                \item  Recall the standard fact:
    \begin{eqnarray}\label{sum of powers of roots of unity}
            \sum_{i = 0}^{p - 1}[i]^j = 
            \begin{cases}
                0 & \text{ if } p - 1 \nmid j \\
                p - 1 & \text{ if } p - 1 \mid j,
                \end{cases}
                & \text{ for $j \geq 1$.}
    \end{eqnarray}
    Using \eqref{sum of powers of roots of unity}, we see that if $\lambda_0 = 1 - p$ and $\lambda_i = 1$ for $1 \leq i \leq p - 1$ then $\sum_{i \in I}\lambda_i [i]^j = 0$ for all $0 \leq j \leq p - 2$.

  \item Taking $\lambda_p = -1$ and separating the $i = p$ summand from
    the equations $\sum\limits_{i \in I}\lambda_i i^j = 0$, we see that for $0 \leq j \leq n$, we have
                \[
                    \sum_{i = 0}^{n}\lambda_i i^j = p^j.
                \]
                Since $i \not \equiv i' \!\! \mod p$ for $0 \leq i, \> i' \leq n \leq p - 1$, there is a unique solution to the system of equations above with $\lambda_i \in \Zp$. Moreover, reducing the equations in the display above modulo $p$, we get
                \[
                    \sum_{i = 0}^{n}\br{\lambda_i}\> \br{i}^j = \begin{cases}
                    1 & \text{ if } j = 0 \\
                    0 & \text{ if } 1 \leq j \leq n.
                    \end{cases}
                \]
                Since $\br{\lambda_0} = 1$ and $\br{\lambda_i} = 0$ for $1 \leq j \leq n$ is the unique solution to the mod $p$ system of equations in the display above, we obtain the lemma.

                \item Since $i \not \equiv i' \!\! \mod p$ for distinct $0 \leq i, \> i' \leq p - 1$, we see that there is a unique integral solution to the system of equations
                \[
                    \lambda_0 = 1 \text{ and } \sum_{i = 0}^{p - 1}\lambda_i i^j = 0
                \]
                for $0 \leq j \leq p - 2$. Reducing these equations modulo $p$, we get
                \[
                    \br{\lambda_0} = 1 \text{ and } \sum_{i = 0}^{p - 1}\br{\lambda_i} \> \br{i}^j = 0.
                \]
                Since $i \equiv [i] \!\! \mod p$, by \eqref{sum of powers of roots of unity}
                and uniqueness, $\br{\lambda_i} = 1$ for $1 \leq i \leq p - 1$. \qedhere
            \end{enumerate}
        \end{proof}

\begin{remark}
        \label{r=1,2}
        We will also sometimes need $\sum_{i \in I}\lambda_iz_i^{n + 1} = 0$
        (see the proofs of Proposition~\ref{Final theorem for the hardest case in self-dual} and
        Proposition~\ref{nu in between -0.5 and 0.5}, which both use
        Lemma~\ref{Telescoping Lemma in self-dual for qp-zp} (2)).
        The vanishing in these cases will follow from the
        hypothesis $p \geq 5$, noting that in the above propositions $n = \dfrac{r+1}{2}$ for
        $1 \leq r \leq p-2$ odd, $\lambda_i$ and $z_i$ are chosen as in Lemma~\ref{Main coefficient identites}
        (3) and $\dfrac{r+1}{2} + 1 \leq p-2$. 
\end{remark}

While analyzing the functions \eqref{polynomial log functions},
we will come across the following sums of $l^{\mathrm{th}}$ powers
    \begin{eqnarray}\label{Definition of S_l}
        S_l := \sum_{i = 0}^{p - 1} i^l, \text{ for } 0 \leq l \leq p - 1.
    \end{eqnarray}
    We adopt the convention that $0^0 = 1$. One interesting property of these functions is that
    \begin{eqnarray}\label{S_l is divisible by p}
        \frac{S_l}{p} \in \Zp, \text{ for } 0 \leq l \leq p - 2 \text{ and } \frac{S_{p - 1} + 1}{p} \in \Zp.
    \end{eqnarray}
    Indeed for $0 \leq i \leq p - 1$, we have $i^j \equiv [i]^j \mod p$ for every $0 \leq j \leq p - 1$. Therefore equation \eqref{S_l is divisible by p} follows from equation \eqref{sum of powers of roots of unity}. The first few values of these sums are
    \begin{eqnarray}\label{First few values of S_l}
        \frac{S_0}{p} = 1, \frac{S_1}{p} \equiv -\frac{1}{2} \!\!\!\!\mod p, \frac{S_2}{p} \equiv \frac{1}{6} \!\!\!\! \mod p.
    \end{eqnarray}

\section{Some results concerning functions in $\sC^{r/2}(\Zp, E)$}
The aim of this section is to state and prove some generalities on functions in the space $\sC^{s}(\Zp, E)$ that will be needed in the sections that follow.

One of the tools that we will use to prove various congruences in the lattice $\latticeL{k}$ is \cite[Proposition I.5.13]{Col10a}:
\begin{Proposition}\label{Convergence proposition}
    Let $s \geq 0$. If $g \in \sC^{s}(\Zp, E)$, then $\tilde{g}_h \to g$ as $h \to \infty$ where
    \[
        \tilde{g}_h(z) := \sum_{m = 0}^{p^h - 1} \left[\sum_{j = 0}^{\lfloor s \rfloor}\frac{g^{(j)}(m)}{j!}(z - m)^j\right]\mathbbm{1}_{m + p^h\Zp}.
    \]
\end{Proposition}

Let now $g(z) \in D(k)$. Then $g(z) \mathbbm{1}_{\zp} \in D(k)$. Indeed, the function
$$z^r g(1/z) \mathbbm{1}_{\zp}(1/z) = z^rg(1/z) \mathbbm{1}_{\zp \setminus p\zp}(z) = z^rg(1/z) - z^rg(1/z) \mathbbm{1}_{p\zp}(z)$$ on $\zp \setminus 0$ extends to a function in $\sC^{r/2}(\Zp, E)$, since
$z^rg(1/z)$ does by definition of $D(k)$, and an easy further check shows that multiplication by $\mathbbm{1}_{p\zp}$
preserves the space $\sC^{r/2}(\Zp, E)$. It follows that $g(z) \mathbbm{1}_{\qp \setminus \zp} \in D(k)$.

By the above proposition we see that for large $h$, we have
$g(z) \mathbbm{1}_{\zp}  \equiv \tilde{g}_h(z) \mod \pi \latticeL{k}$.
We now modify $\tilde{g}_h$ slightly using the following lemma.

    \begin{lemma}\label{Congruence lemma}
        Suppose $g \in \sC^{r/2}(\Zp, E)$ is a function that has continuous derivatives of order $\lfloor r/2 \rfloor + 1, \ldots, t$ for an integer $r/2 < t \leq r$. Define
        \[
            g_h(z) := \sum_{m = 0}^{p^h - 1}\left[\sum_{j = 0}^{t}\frac{g^{(j)}(m)}{j!}(z - m)^j\right]\mathbbm{1}_{m + p^h\Zp}.
        \]
        Then for large $h$, we have
        \[
            \tilde{g}_h(z) \equiv g_h(z) \mod \pi \latticeL{k}.
        \]
    \end{lemma}
    \begin{proof}
      Note that for large $h$, $\lfloor r/2 \rfloor + 1 \leq j \leq t$ and $0 \leq m \leq p^h - 1$, we have
        \begin{eqnarray}\label{Excessive derivatives are 0}
            \frac{g^{(j)}(m)}{j!}(z - m)^j \mathbbm{1}_{m + p^h\Zp} \equiv 0 \mod \pi \latticeL{k}.
        \end{eqnarray}
        Indeed, by the continuity of the derivatives we see that the valuation of $\frac{g^{(j)}(m)}{j!}$ is bounded below by some finite rational number $M$ for all $m \in \Zp$ and all $\lfloor r/2 \rfloor + 1 \leq j \leq t$. We choose $h$ large enough so that $M > (h - 1)(r/2 - j)$, noting $r/2-j$ is negative.
        Then by Lemma \ref{Integers in the lattice}, we see that \eqref{Excessive derivatives are 0} is true.
    \end{proof}

We give an outline of the general strategy used in this paper to establish some congruences. We choose a
function $g(z)$ as in equation \eqref{polynomial log functions}. Using Lemma \ref{Congruence lemma}, we see that $g(z) \mathbbm{1}_{\Zp} \equiv g_h(z) \mod \pi\latticeL{k}$ for large $h$. Then we use Lemma~\ref{Truncated Taylor expansion} and Lemma~\ref{telescoping lemma} below to descend from large $h$ to $h = 2$
(or $h = 3$, see the variants Lemma \ref{strong truncated Taylor expansion for self-dual} and Lemma \ref{Telescoping Lemma in self-dual}). We then prove that\footnote{In an exceptional case, we prove that $g(z)\mathbbm{1}_{\Qp \setminus \Zp} \equiv -p^{-1}\cL\mathbbm{1}_{\Zp} \mod \pi \latticeL{k}$.} $g(z)\mathbbm{1}_{\Qp \setminus \Zp} \equiv 0 \mod \pi \latticeL{k}$ using Lemma \ref{qp-zp part is 0} (and its variant Lemma \ref{Telescoping Lemma in self-dual for qp-zp}). Since $g(z)$ is equal to $0$ in $\tB(k, \cL)$, we get a congruence $g_2(z) \equiv 0 \mod \pi\latticeL{k}$ (or $g_3(z) \equiv 0 \mod \pi \latticeL{k}$).

\begin{lemma}\label{Truncated Taylor expansion}
    Let $k/2 \leq n \leq r \leq p - 1$ and $z_0 \in \Zp$. Fix $x \in \bQ$ such that $x \geq - 1$ and $x + v_p(\cL) \geq r/2 - n$. Define
    \begin{enumerate}
        \item $g(z) = p^x(z - z_0)^n\logL(z - z_0)\mathbbm{1}_{\Zp}$
        \item $g(z) = p^xz^r\logL(z)\mathbbm{1}_{p\Zp}$
        \item $g(z) = p^x z^{r - n} (1 - zz_0)^n\logL(1 - zz_0)\mathbbm{1}_{p\Zp}.$
    \end{enumerate}
    Then for $h \geq 3$, $0 \leq a \leq p^{h - 1} - 1$, $0 \leq \alpha \leq p - 1$ and $0 \leq j \leq n - 1$, we have
    \[
        g^{(j)}(a + \alpha p^{h - 1}) \equiv g^{(j)}(a) + \alpha p^{h - 1}g^{(j + 1)}(a) + \cdots + \frac{(\alpha p^{h - 1})^{n - 1 - j}}{(n - 1 - j)!}g^{(n - 1)}(a) \mod (p^{h - 1})^{r/2 - j}\pi.
    \]
    Moreover, we have
    \begin{eqnarray*}
        && g^{(j)}(a + \alpha p^{h - 1})(z - a - \alpha p^{h - 1})^j\mathbbm{1}_{a + \alpha p^{h - 1} + p^h\Zp} \\
        && \equiv \left[g^{(j)}(a) + \alpha p^{h - 1}g^{(j + 1)}(a) + \cdots + \frac{(\alpha p^{h - 1})^{n - 1 - j}}{(n - 1 - j)!}g^{(n - 1)}(a)\right](z - a - \alpha p^{h - 1})^{j}\mathbbm{1}_{a + \alpha p^{h - 1} + p^h \Zp}
    \end{eqnarray*}
    modulo $\pi \latticeL{k}$.
\end{lemma}
\begin{proof}
  We prove the first congruence in the conclusion of the lemma.
  First assume that $g(z) = p^x(z - z_0)^n\logL(z - z_0)\mathbbm{1}_{\Zp}$.
  The proof in this case
  is similar to the proof that $g(z) \in \sC^{n - 1}(\Zp, E)$. The difference is that we need to keep track of the estimates. We consider two cases.
    \begin{enumerate}
        \item $v_p(a - z_0) < h - 1$:
Since $g$ is analytic in the neighborhood $a + p^{h - 1}\Zp$, using \eqref{Small derivative formula for polynomial times logs}, we write
\begin{eqnarray}\label{Far away from z0 in Truncated Taylor expansion}
    g^{(j)}(a + \alpha p^{h - 1}) & = & g^{(j)}(a) + \cdots + \frac{(\alpha p^{h - 1})^{n - 1 - j}}{(n - 1 - j)!}g^{(n - 1)}(a) \nonumber\\
    & & + \frac{(\alpha p^{h - 1})^{n - j}}{(n - j)!}n!p^{x}[\logL(a - z_0) + H_n] \nonumber \\
    & & + \sum_{l \geq 1}\frac{(\alpha p^{h - 1})^{n - j + l}}{(n - j + l)!}n! p^{x} \frac{(-1)^{l - 1}}{(a - z_0)^l}(l - 1)!.
\end{eqnarray}
Since $x + v_p(\cL) \geq r/2 - n$ and $x \geq -1 \geq r/2 - n$, we see that the valuation of the term in the second line above is at least $(n - j)(h - 1) + r/2 - n$. Now
\begin{equation}\label{h greater than 2}
    (n - j)(h - 1) + r/2 - n > (h - 1)(r/2 - j) \iff h > 2.
\end{equation}
Therefore for $h \geq 3$, the term in the second line in equation \eqref{Far away from z0 in Truncated Taylor expansion} is $0$ modulo $(p^{h - 1})^{r/2 - j}\pi$. Next, write the general term in the last sum in equation \eqref{Far away from z0 in Truncated Taylor expansion} as
\[
    p^{x}(\alpha p^{h - 1})^{n - j}\frac{n!}{(n - j)!}\left(\frac{\alpha p^{h - 1}}{a - z_0}\right)^{l}(-1)^{l - 1}\frac{1}{l{n - j + l \choose l}}.
\]
Using the estimate $v_p{n - j + l \choose l} \leq \lfloor \log_p(n - j + l)\rfloor - v_p(l)$ obtained using Kummer's theorem, the valuation of this term is greater than or equal to
\[
    -1 + (n - j)(h - 1) + l - \log_p(n - j + l).
\]
Moreover,
\begin{eqnarray*}
    &&-1 + (n - j)(h - 1) + l - \log_p(n - j + l) > (r/2 - j)(h - 1) \\
    && \qquad \stackrel{h \geq 3}{\iff} 2n - r - 1 + l > \log_p(n - j + l),
\end{eqnarray*}
which is true for all $l \geq 1$. Therefore we have proved the lemma for the first function when $v_p(a - z_0) < h - 1$.

\item $v_p(a - z_0) \geq h - 1$: Recall the derivative formula \eqref{Small derivative formula for polynomial times logs}:
\begin{eqnarray}\label{Derivative formula for g(z)}
    g^{(j)}(z) = \left[\frac{n!}{(n - j)!}p^x(z - z_0)^{n - j}\logL(z - z_0) + p^x t_j (z - z_0)^{n - j}\right]\mathbbm{1}_{\Zp},
\end{eqnarray}
where $t_j \in \Zp$. The valuations of the terms $g^{(j)}(a + \alpha p^{h - 1}), \> g^{(j)}(a), \> \alpha p^{h - 1}g^{(j+1)}(a), \> \ldots,$ $(\alpha p^{h - 1})^{n - 1 - j}g^{(n - 1)}(a)$ are greater than or equal to $(n - j)(h - 1) + r/2 - n$. So equation \eqref{h greater than 2} implies that these terms are congruent to $0$ modulo $(p^{h - 1})^{r/2 - j}\pi$. This proves the lemma for the first function when $v_p(a - z_0) \geq h - 1$.
\end{enumerate}

The proof of the first congruence in the conclusion of the lemma for the second function $g(z)$ is similar to that for the first function. The difference is that it holds vacuously
if $p \nmid a$ since both sides vanish.

Finally assume that $g(z) = p^x z^{r - n} (1 - zz_0)^n\logL(1 - zz_0)\mathbbm{1}_{p\Zp}$. For $j \geq 0$, its $j^{\mathrm{th}}$ derivative is given by the product rule
\begin{eqnarray}\label{Derivative formula for the qp-zp part}
    g^{(j)}(z) & = & p^x\sum_{m = 0}^{j}{j \choose m}\left.\begin{cases} \dfrac{(r - n)!z^{r - n - j + m}}{(r - n - j + m)!} & \!\!\!\! \text{ if } m \geq j - r + n \\ 0 & \!\!\!\!\text{ if } m < j - r + n\end{cases}\right\} \> \cdot \>  (-z_0)^m  \\
    && \left.\begin{cases}\dfrac{n!}{(n - m)!}\left[(1 - zz_0)^{n - m}\logL(1 - zz_0) + *(1 - zz_0)^{n - m}\right] & \!\!\!\! \text{ if } m \leq n \\ \dfrac{n!(-1)^{m - n - 1}(m - n - 1)!}{(1 - zz_0)^{m - n}} & \!\!\!\! \text{ if } m > n\end{cases}\right\}\mathbbm{1}_{p\Zp}(z), \nonumber
\end{eqnarray}
for some $* \in \Zp$. Since $g^{(j)}(z)$ has $\mathbbm{1}_{p\Zp}$ as a factor, the lemma is vacuously true if $p \nmid a$. So assume that $p \mid a$. We expand
\begin{eqnarray*}
    g^{(j)}(a + \alpha p^{h - 1}) & = & g^{(j)}(a) + \cdots + \frac{(\alpha p^{h - 1})^{n - 1 - j}}{(n - 1 - j)!}g^{(n - 1)}(a) \\
    && + \sum_{l \geq 1}\frac{(\alpha p^{h - 1})^{n - 1 - j + l}}{(n - 1 - j + l)!}g^{(n - 1 + l)}(a).
\end{eqnarray*}
Now using \eqref{Derivative formula for the qp-zp part}, we see that the valuation of the $l^{\mathrm{th}}$ summand on the second line of the right side of the equation above is greater than or equal to $x + (h - 1)(n - 1 - j + l) - v_p((n - 1 - j + l)!)$. This is greater than $(h - 1)(r/2 - j)$. Indeed, this is clear
for $l =1$, and follows for $l > 1$ using $v_p(t!) \leq \frac{t}{p - 1}$.

The second congruence in the conclusion of the lemma follows from the first
using Lemma~\ref{Integers in the lattice}.
\end{proof}

\begin{lemma}\label{telescoping lemma}
    Let $k/2 \leq n \leq r \leq p - 1$ and $z_0 \in \Zp$. Fix $x \in \bQ$ such that $x \geq - 1$ and $x + v_p(\cL) \geq r/2 - n$. Then for 
    \begin{enumerate}
        \item $g(z) = p^x(z - z_0)^n\logL(z - z_0)\mathbbm{1}_{\Zp}$
        \item $g(z) = p^xz^r\logL(z)\mathbbm{1}_{p\Zp}$
        \item $g(z) = p^x z^{r - n} (1 - zz_0)^n\logL(1 - zz_0)\mathbbm{1}_{p\Zp},$
    \end{enumerate}
    we have $g(z) \equiv g_2(z) \mod \pi\latticeL{k}$.
\end{lemma}
\begin{proof}
    An easy check shows that each of the functions $g(z)$ above is in $D(k)$.
    Using Lemma \ref{Congruence lemma}, we see that $g(z) \equiv g_h(z) \mod \pi\latticeL{k}$ for large $h$. In particular, assume that $h \geq 3$. Making the substitution $m = a + \alpha p^{h - 1}$, we write
    \[
        g_h(z) = \sum_{a = 0}^{p^{h - 1} - 1}\sum_{\alpha = 0}^{p - 1}\left[ \sum_{j = 0}^{n - 1} \frac{g^{(j)}(a + \alpha p^{h - 1})}{j!}(z - a - \alpha p^{h - 1})^j\right]\mathbbm{1}_{a + \alpha p^{h - 1} + p^h \Zp}.
    \]
    %
By Lemma \ref{Truncated Taylor expansion}
    \begin{eqnarray*}
        g_h(z) \equiv \sum_{a = 0}^{p^{h - 1} - 1}\sum_{\alpha = 0}^{p - 1} && \!\!\!\! \!\!\!\! \left\{\left[\frac{1}{0! \> 0!}g(a) + \frac{\alpha p^{h - 1}}{0! \> 1!} g^{(1)}(a) + \cdots + \frac{(\alpha p^{h - 1})^{n - 2}}{0! \> (n - 2)!}g^{(n - 2)}(a) + \frac{(\alpha p^{h - 1})^{n - 1}}{0! \> (n - 1)!}g^{(n - 1)}(a)\right] \right.\\
        && \!\!\!\!\!\!\!\! + \left[\frac{1}{1! \> 0!}g^{(1)}(a) + \frac{\alpha p^{h - 1}}{1! \> 1!} g^{(2)}(a) + \cdots + \frac{(\alpha p^{h - 1})^{n - 2}}{1! \> (n - 2)!}g^{(n - 1)}(a)\right](z - a - \alpha p^{h - 1}) \\ 
        && \quad \quad \vdots \qquad \qquad \qquad \vdots \\
        && \!\!\!\!\!\!\!\! + \left.\left[\frac{1}{(n - 1)! \> 0!}g^{(n - 1)}(a)\right](z - a - \alpha p^{h - 1})^{n - 1}\right\} \mathbbm{1}_{a + \alpha p^{h - 1} + p^h\Zp} \mod \pi\latticeL{k}.
    \end{eqnarray*}
    Expanding $[(z - a) - \alpha p^{h - 1}]^j$ for $0 \leq j \leq n - 1$ and collecting like terms of $(z - a)$ diagonally, the terms involving $\alpha$ disappear
    \begin{eqnarray*}
        g_h(z) & \equiv & \sum_{a = 0}^{p^{h - 1} - 1}\sum_{\alpha = 0}^{p - 1}\left[\sum_{j = 0}^{n - 1}\frac{g^{(j)}(a)}{j!}(z - a)^{j}\right]\mathbbm{1}_{a + \alpha p^{h - 1} + p^h\Zp}  \\
        & \equiv & \sum_{a = 0}^{p^{h - 1} - 1}\left[\sum_{j = 0}^{n - 1} \frac{g^{(j)}(a)}{j!}(z - a)^j\right]\mathbbm{1}_{a + p^{h - 1}\Zp} \equiv g_{h - 1}(z) \mod \pi \latticeL{k}.
    \end{eqnarray*}
    Iterating this process, we get $g_h(z) \equiv g_2(z) \mod \pi \latticeL{k}$. 
    Therefore the claim follows.
\end{proof}

    \begin{lemma}\label{qp-zp part is 0}
      Let $k/2 \leq n \leq r \leq p - 1$. For $i$ in a finite indexing set $I$, let $\lambda_i \in E$ and $z_i \in \Zp$ be such that $\sum_{i \in I}\lambda_i z_i^j = 0$ for all $0 \leq j \leq n$.
      Fix $x \in \bQ$ such that $x \geq -1$. Then for
        \[
            f(z) = \sum_{i \in I}p^x\lambda_i z^{r - n}(1 - zz_i)^n \logL(1 - zz_i)\mathbbm{1}_{p\Zp},
        \]
        we have $f(z) \equiv 0 \mod \pi \latticeL{k}$.
    \end{lemma}
    \begin{proof}
        First note that $f(z) \equiv f_2(z) \mod \pi \latticeL{k}$ by Lemma \ref{telescoping lemma} (3). Next, we show that $f_2(z) \equiv 0 \mod \pi \latticeL{k}$.
        Recall that
        \[
            f_2(z) = \sum_{a = 0}^{p^2 - 1}\sum_{j = 0}^{n - 1}\frac{f^{(j)}(a)}{j!}(z - a)^j\mathbbm{1}_{a + p^2\Zp}.
        \]
        Since $f(z)$ has $\mathbbm{1}_{p\Zp}$ as a factor, we see that the $p \nmid a$ terms in the expression above are $0$. Therefore we can write
        \[
            f_2(z) = \sum_{\alpha = 0}^{p - 1}\sum_{j = 0}^{n - 1}\frac{f^{(j)}(\alpha p)}{j!}(z - \alpha p)^j\mathbbm{1}_{\alpha p + p^2\Zp}.
        \]
        Since $j \leq n - 1$, the $j^{\mathrm{th}}$ derivative of $f$ is given by (cf. \eqref{Derivative formula for the qp-zp part})
        \begin{eqnarray*}
            f^{(j)}(z) = \sum_{i \in I}p^x\lambda_i \sum_{m = 0}^{j}{j \choose m}\left.\begin{cases}\dfrac{(r - n)!z^{r - n - j + m}}{(r - n - j + m)!} & \text{ if } m \geq j - r + n \\ 0 & \text{ if }m < j - r + n\end{cases}\right\} \cdot (-z_i)^m \\
            \frac{n!}{(n - m)!}[(1 - zz_i)^{n - m}\logL(1 - zz_i) + *(1 - zz_i)^{n - m}]\mathbbm{1}_{p\Zp}(z),
        \end{eqnarray*}
        where $* \in \Zp$. We substitute $z = \alpha p$ in the equation above. Expanding
        $$(1 - \alpha pz_i)^{n - m}\logL(1 - \alpha pz_i) = \sum_{l \geq 1} c_l (\alpha pz_i)^l$$
        for some $c_l$, we have $v_p(c_l(\alpha pz_i)^l) \geq l - \lfloor \log_p(l)\rfloor$. Using the identities $\sum_{i \in I}\lambda_i z_i^j = 0$, we see that
        \[
            v_p\left(\frac{f^{(j)}(\alpha p)}{j!}\right) \geq \min_{l}\{-1 + r - n - j + m + l - \lfloor \log_p(l) \rfloor \},
        \]
        where the minimum is taken over $l \geq n - m + 1$. 
        A short computation shows that\footnote{This is false for odd primes in
        the base case $l = n-m+1$ when $n = r= 2$, $m = 0$
        and $p = 3$ but this does not concern us because of our blanket assumption $p \geq 5$.}
        $$-1 + r - n - j + m + l - \log_p(l) > r/2 - j$$ for $l \geq n - m + 1$.
        Therefore using Lemma \ref{Integers in the lattice}, we see that $f_2(z) \equiv 0 \mod \pi \latticeL{k}$.
    \end{proof}

    We remark that $r = 1$ is not considered in the lemma because of the condition $k/2 \leq r$.
    
\section{Analysis of $\br{\latticeL{k}}$ around all points but the last}\label{Common section}
This section is the heart of the paper. We use the relation $g_2(z) \equiv 0 \mod \pi\latticeL{k}$
proved in the last section for certain $E$-valued functions $g$ on $\qp$ as in
\eqref{polynomial log functions} to show that most subquotients of $\br{\latticeL{k}}$ are
quotients of cokernels of linear expressions involving Iwahori-Hecke operators acting on
compactly induced representations from
$IZ$ to $G$. We give a uniform treatment for all the points appearing in
Theorem~\ref{Main theorem in the second part of my thesis} except for the last point.

\subsection{The $\nu \geq i - r/2$ cases}
    In this section, we prove that if $\nu \geq i - r/2$ for $i = 1, 2, \ldots ,\lceil r/2 \rceil - 1$, then the map $\IZind a^{i - 1}d^{r - i + 1} \twoheadrightarrow F_{2i-2, \> 2i - 1}$ factors as
    \[
            \IZind a^{i - 1}d^{r- i + 1} \twoheadrightarrow \left.\begin{cases}\dfrac{\IZind a^{i - 1}d^{r - i + 1}}{\im(\lambda_iT_{-1, 0} - 1)} & \!\!\!\! \text{ if } (i, r) \neq (1, p - 1) \\ \dfrac{\IZind a^{i - 1}d^{r - i + 1}}{\im(\lambda_i(T_{-1, 0} + T_{1, 0}) - 1)} & \!\!\!\!\text{ if } (i, r) = (1, p - 1)\end{cases}\right\} \twoheadrightarrow F_{2i - 2, \> 2i - 1},
        \]
    where $\lambda_i = (-1)^i\> i {r - i + 1\choose i}p^{r/2 - i}\cL$. Moreover for $\nu = i - r/2$, the second map in the display above induces a surjection
    \[
        \pi([2i - 2 - r], \lambda_i^{-1}, \omega^{r - i + 1}) \twoheadrightarrow F_{2i - 2, \> 2i - 1},
    \]
    where $[2i - 2 - r] = 0, 1, \ldots, \> p - 2$ represents the class of $2i - 2 - r$ modulo $p - 1$.

    We need the following preparatory lemma.
\begin{lemma}\label{Sum of the i neq a terms}
    Let $0 \leq j \leq r/2$ and $k/2 < n \leq p - 1$. For $0 \leq a, i \leq p - 1$, let $a_i \in \Zp$ be such that $a - [i] = [a - i] + pa_i$. For $\lambda_0 = 1 - p$ and $\lambda_i = 1$ when $1 \leq i \leq p - 1$, we have
    \[
        \sum_{\substack{i = 0 \\ i \neq a}}^{p - 1}p^{-1}\lambda_i[a - i]^{n - j - 1}pa_i \equiv \frac{a^{n - j}}{n - j} \mod \pi.
    \]
\end{lemma}
\begin{proof}
    Since $[0] = 0$ and $n - j - 1 > 0$, we see that
    \[
        \sum_{\substack{i = 0 \\ i \neq a}}^{p - 1}p^{-1}\lambda_i[a - i]^{n - j - 1}pa_i = \sum_{i = 0}^{p - 1}p^{-1}\lambda_i[a - i]^{n - j - 1}pa_i.
    \]
    Consider
    \[
        \sum_{i = 0}^{p - 1}p^{-1}\lambda_i(a - [i])^{n - j} \equiv \sum_{i = 0}^{p - 1}p^{-1}\lambda_i\left([a - i]^{n - j} + (n - j)[a - i]^{n - j - 1}pa_i\right) \mod \pi.
    \]
    Expanding $(a - [i])^{n - j}$ and using Lemma~\ref{Main coefficient identites} (1),
    we can rewrite the above equation as 
    \[
        \sum_{i = 0}^{p - 1}p^{-1}[i]^{n - j} \equiv \left(\sum_{i = 0}^{p - 1}p^{-1}[i]^{n - j}\right) - [a]^{n - j} + \sum_{i = 0}^{p - 1}p^{-1}\lambda_i(n - j)[a - i]^{n - j - 1}pa_i \mod \pi.
    \]
    Therefore
    \[
        \sum_{i = 0}^{p - 1}p^{-1}\lambda_i[a - i]^{n - j - 1}pa_i \equiv \frac{[a]^{n - j}}{n - j} \equiv \frac{a^{n - j}}{n - j} \mod \pi. \qedhere
    \]
\end{proof}

\begin{Proposition}\label{nu geq}
    Let $k/2 < n \leq k - 2 \leq p - 1$. For $\nu \geq 1 + r/2 - n$, set
\[
    g(z) = p^{-1} \sum_{i = 0}^{p - 1} \lambda_i (z - [i])^n\logL(z - [i]),
\]
where the $\lambda_i$ are given by Lemma~\ref{Main coefficient identites} (1).
Then, we have
    \begin{eqnarray*}
        g(z) & \equiv & \sum_{a = 0}^{p - 1} p^{-1}\cL(z - a)^n \mathbbm{1}_{a + p\Zp} - \delta_{n, \> p - 1}p^{-1}\cL\mathbbm{1}_{\Zp}\\
        & & + \sum_{a = 0}^{p - 1}\left[\sum_{j = 0}^{\lfloor r/2 \rfloor}{n \choose j} \frac{a^{n - j}}{n - j}(z - a)^j\right]\mathbbm{1}_{a + p\Zp} \\ 
        & & + \sum_{a = 0}^{p - 1} \left[\sum_{j = \lceil r/2 \rceil}^{n - 1} \cL{n \choose j}p^{-1}(a - [a])^{n - j}(z - a)^j\right]\mathbbm{1}_{a + p\Zp} \mod \pi\latticeL{k}.
    \end{eqnarray*}
\end{Proposition}
\begin{proof}
  Write $g(z) = g(z)\mathbbm{1}_{\Zp} + g(z)\mathbbm{1}_{\Qp \setminus \Zp}$. Using
  the fact that the sum $\sum_{i=0}^{p-1} \lambda_i [i]^j$ vanishes for $0 \leq j \leq p-2$ and
  is equal to $p-1$ for $j = p-1$ (thus, in the exceptional case $n = r = p-1$ we have only
  have vanishing for $0 \leq j \leq n-1$), we have
    \[
        w\cdot g(z)\mathbbm{1}_{\Qp \setminus \Zp} = \left[\sum_{i = 0}^{p - 1}p^{-1}\lambda_i z^{r - n}(1 - z[i])^n\logL(1 - z[i]) - \delta_{n, \> p - 1}p^{-1}(p - 1)z^r\logL(z)\right]\mathbbm{1}_{p\Zp}.
    \]
    We claim that this is congruent to $\delta_{n, \> p - 1}p^{-1}\cL z^r\mathbbm{1}_{p\Zp}$ modulo $\pi \latticeL{k}$. First set
    \[
        f(z) = \sum_{i = 0}^{p - 1}p^{-1}\lambda_iz^{r - n}(1 - z[i])^n\logL(1 - z[i])\mathbbm{1}_{p\Zp}(z).
    \]
    Using Lemma \ref{telescoping lemma} (3), we see that $f(z) \equiv f_2(z) \mod \pi \latticeL{k}$. For $0 \leq j \leq n - 1$, recall that
    \[
        f_2(z) = \sum_{a = 0}^{p^2 - 1}\sum_{j = 0}^{n - 1}\frac{f^{(j)}(a)}{j!}(z - a)^j\mathbbm{1}_{a + p^2\Zp},
    \]
    and
    \begin{eqnarray*}
        f^{(j)}(a) & = & \sum_{i = 0}^{p - 1}p^{-1}\lambda_i\sum_{m = 0}^{j}{j \choose m}\left.\begin{cases}\dfrac{(r - n)!a^{r - n - j + m}}{(r - n - j + m)!} & \!\!\!\! \text{ if } m \geq j - r + n \\ 0 & \!\!\!\! \text{ if } m < j - r + n\end{cases}\right\} \cdot (-[i])^m \\
        && \frac{n!}{(n - m)!}\left[(1 - a[i])^{n - m}\logL(1 - a[i]) + *(1 - a[i])^{n - m}\right]\mathbbm{1}_{p\Zp}(a),
    \end{eqnarray*}
    where $* \in \Zp$. Since $f^{(j)}(z)$ has $\mathbbm{1}_{p\Zp}$ as a factor, we see that $f^{(j)}(a) = 0$ if $p \nmid a$. So we assume that $p \mid a$. We write the terms inside the square brackets as $\sum_{l \geq 0} c_l(a[i])^l$ and kill the $l < n - m$ terms using
    Lemma~\ref{Main coefficient identites} (1),
    \[
        v_p\left(\frac{f^{(j)}(a)}{j!}\right) \geq \min_{l \geq n - m}\{- 1 + r - n - j + m + l - \lfloor \log_p(l) \rfloor\}.
    \]
    Since $k/2 < r$, we see that $r > 2$. A short computation involving two cases: $l = n - m$ and $l \geq n - m + 1$ shows that $-1 + r - n - j + m + l - \lfloor \log_p(l) \rfloor > r/2 - j$. Lemma \ref{Integers in the lattice} then implies that $f_2(z) \equiv 0 \mod \pi \latticeL{k}$.

    Next, set $f(z) = p^{-1}z^r\logL(z)\mathbbm{1}_{p\Zp}$. By Lemma~\ref{telescoping lemma} (2),
    we see that $f(z) \equiv f_2(z) \!\! \mod \pi \latticeL{k}$. Now
    \[
        f_2(z) = \sum_{a = 0}^{p^2 - 1}\sum_{j = 0}^{r - 1}\frac{f^{(j)}(a)}{j!}(z - a)^j\mathbbm{1}_{a + p^2\Zp},
    \]
    and
    \[
        f^{(j)}(a) = p^{-1}\frac{r!}{(r - j)!}\left[z^{r - j}\logL(z) + t_jz^{r - j}\right]\mathbbm{1}_{p\Zp}(a),
    \]
    where $t_j \in \Zp$. Since $\mathbbm{1}_{p\Zp}(a) = 0$ if $p \nmid a$, we see that
    \[
        f_2(z) = \sum_{\alpha = 0}^{p - 1}\sum_{j = 0}^{r - 1}p^{-1}{r \choose j}[(\alpha p)^{r - j}\logL(\alpha p) + t_j(\alpha p)^{r - j}](z - \alpha p)^j\mathbbm{1}_{\alpha p + p^2\Zp}.
    \]
    Since $r > 2$, we see that $-1 + r - j > r/2 - j$. Therefore we delete the second term in the brackets above. Moreover, note that $p^{-1}(\alpha p)^{r - j}\logL(\alpha p) \equiv p^{-1}(\alpha p)^{r - j}\cL \mod p^{r/2 - j}\pi$. Indeed, this is clear if $\alpha = 0$ and if $\alpha \neq 0$, then this is true because $p \mid \logL(\alpha)$. Therefore
    \begin{eqnarray*}
        f_2(z) & \equiv & \sum_{\alpha = 0}^{p - 1}\sum_{j = 0}^{r - 1}p^{-1}{r \choose j}(\alpha p)^{r - j}\cL(z - \alpha p)^{j}\mathbbm{1}_{\alpha p + p^2\Zp} \\
        & = & \sum_{\alpha = 0}^{p - 1}p^{-1}\cL[z^r - (z - \alpha p)^r]\mathbbm{1}_{\alpha p + p^2\Zp} \mod \pi \latticeL{k}.
    \end{eqnarray*}
    Using Lemma \ref{stronger bound for polynomials of large degree with varying radii}, we see that the second term in the brackets above is $0$ modulo $\pi\latticeL{k}$. Therefore we get
    \[
        f_2(z) \equiv p^{-1}\cL z^r\mathbbm{1}_{p\Zp} \mod \pi \latticeL{k}.
    \]
    These computations show that
    \[
        w \cdot g(z)\mathbbm{1}_{\Qp \setminus \Zp} \equiv \delta_{n, \> p - 1}p^{-1}\cL z^r\mathbbm{1}_{p\Zp} \mod \pi \latticeL{k}.
    \]
    Applying $w$ again, we get
    \begin{eqnarray}\label{qp-zp part in geq}
        g(z)\mathbbm{1}_{\Qp \setminus \Zp} \equiv -\delta_{n, \> p - 1}p^{-1}\cL \mathbbm{1}_{\Zp} \mod \pi \latticeL{k}.
    \end{eqnarray}
    This is the second term on the right side of the statement of the proposition.\footnote{This is the first and only time when $g(z)\mathbbm{1}_{\Qp \setminus \Zp}$ contributes.}

    Next, using Lemma \ref{telescoping lemma} (1), we see that $g(z)\mathbbm{1}_{\Zp} \equiv g_2(z) \mod \pi\latticeL{k}$. We analyze $g_2(z) - g_1(z)$ and $g_1(z)$ separately. Consider
    \begin{eqnarray*}
        g_2(z) - g_1(z) & = & \sum_{a = 0}^{p - 1}\sum_{\alpha = 0}^{p - 1}\left[\sum_{j = 0}^{n - 1} \frac{g^{(j)}(a + \alpha p)}{j!}(z - a - \alpha p)^j\right]\mathbbm{1}_{a + \alpha p + p^2\Zp} \\
        && - \sum_{a = 0}^{p - 1}\left[\sum_{j = 0}^{n - 1}\frac{g^{(j)}(a)}{j!}(z - a)^j\right]\mathbbm{1}_{a + p\Zp}.
    \end{eqnarray*}
    Expanding $(z - a)^j\mathbbm{1}_{a + p\Zp} = \sum_{\alpha = 0}^{p - 1}[(z - a - \alpha p) + \alpha p]^j \mathbbm{1}_{a + \alpha p + p^2\Zp}$ and collecting like powers of of $(z - a - \alpha p)$, we get
    \begin{eqnarray*}
        g_2(z) - g_1(z) \!\!\! & = & \! \!\! \sum_{a = 0}^{p - 1}\sum_{\alpha = 0}^{p - 1}\sum_{j = 0}^{n - 1}\left[g^{(j)}(a + \alpha p) - g^{(j)}(a) - \alpha p g^{(j + 1)}(a) - \cdots - \frac{(\alpha p)^{(n - 1 - j)}}{(n - 1 - j)!}g^{(n - 1)}(a)\right] \\
        & & \!\!\! \qquad \qquad \qquad \qquad \frac{(z - a - \alpha p)^j}{j!}\mathbbm{1}_{a + \alpha p + p^2 \Zp}.
    \end{eqnarray*}
    Since $\sum_{i = 0}^{p - 1} \lambda_i (z - [i])^{n - 1} = 0$ by Lemma~\ref{Main coefficient identites} (1), we may use equation \eqref{Simple derivative formula} to see that the coefficient of $(z - a - \alpha p)^j\mathbbm{1}_{a + \alpha p + p^2\Zp}$ in the equation above is
    \begin{eqnarray}\label{jth summand in g2 - g1 in geq}
        && \!\!\!\!\!\!\! {n \choose j}\sum_{i = 0}^{p - 1} p^{-1}\lambda_i\Big[(a + \alpha p - [i])^{n - j}\logL(a + \alpha p - [i]) -
        (a - [i])^{n - j}\logL(a - [i]) \\
        && \!\!\!\!\!\!\! - \left. {n - j \choose 1}\alpha p (a - [i])^{n - j - 1}\logL(a - [i]) - \ldots - {n - j \choose n - j - 1}(\alpha p)^{n - j - 1}(a - [i])\logL(a - [i])\right]. \nonumber
    \end{eqnarray}

    First assume that $i \neq a$. Writing $\logL(a + \alpha p - [i]) = \logL(1 + (a - [i])^{-1}\alpha p) + \logL(a - [i])$, we see that the $i^{\mathrm{th}}$ summand in the display above is
    \[
        p^{-1}\lambda_i\left[(a - [i] + \alpha p)^{n - j}\logL(1 + (a - [i])^{-1}\alpha p) + (\alpha p)^{n - j}\logL(a - [i])\right].
    \]
    Since $0 \leq a, i \leq p - 1$, the assumption $a \neq i$ implies that $a \not \equiv i \mod p$. Therefore the valuation of $\logL(a - [i])$ is greater than or equal to $1$. So the second term in the display above is congruent to $0$ modulo $p^{r/2 - j}\pi$. Thus, expanding, the $i \neq a$ summand
    in expression \eqref{jth summand in g2 - g1 in geq} is
    \begin{eqnarray}\label{i not equal to a in geq}
        && \!\!\!\! p^{-1}\lambda_i[(a - [i])^{n - j - 1}\alpha p + c_{n - j - 2}(a - [i])^{n - j - 2}(\alpha p)^2 + \cdots + c_1(a - [i])(\alpha p)^{n - j - 1}] \\
        && \quad \mod p^{r/2 - j}\pi, \nonumber
    \end{eqnarray}
    where $c_1, \ldots, c_{n - j - 2} \in \Zp$.

    Next assume that $i = a$. 
    We claim that for any $m \in \Zp$ with $p \mid m$ and $0 \leq l \leq n - j - 1$, we have
    \begin{eqnarray}\label{Replace logs by L}
        p^{-1}p^l m^{n - j - l}\logL(m) \equiv p^{-1}p^l m^{n - j - l}\cL \mod p^{r/2 - j}\pi.
    \end{eqnarray}
    Indeed, if $v_p(m) = 1,$ then writing $m = pu$ for some unit $u$, we have
    $\logL(m) = \cL + \log(u)$. The valuation of $p^{-1}p^lm^{n - j - l}\log(u)$ is greater than or equal to $n - j > r/2 - j$. Therefore we have proved the claim when $v_p(m) = 1$. If $v_p(m) \geq 2$, then using $v_p(\cL) \geq 1 + r/2 - n$ we see that the valuations of both the terms in the claim are greater than or equal to $l + 2(n - j - l) + r/2 - n > r/2 - j$. Therefore if $v_p(m) \geq 2$,
    then both sides of \eqref{Replace logs by L} vanish modulo $p^{r/2 - j}\pi$.
    This means that the logarithms in expression \eqref{jth summand in g2 - g1 in geq} can be replaced by $\cL$ modulo $p^{r/2 - j}\pi$. So the $i = a$ summand in expression \eqref{jth summand in g2 - g1 in geq} becomes
    \begin{eqnarray}\label{i equal to a in geq}
        p^{-1}\lambda_a(\alpha p)^{n - j}\cL \mod p^{r/2 - j}\pi.
    \end{eqnarray}

    Putting expressions \eqref{i not equal to a in geq} and \eqref{i equal to a in geq} in expression \eqref{jth summand in g2 - g1 in geq} and using the identity $\sum_{i = 0}^{p - 1}\lambda_i(z - [i])^{n - 1} = 0$, we get
    \begin{eqnarray*}
        g_2(z) - g_1(z) & \equiv & \sum_{a = 0}^{p - 1}\sum_{\alpha = 0}^{p - 1}\left[\sum_{j = 0}^{n - 1}p^{-1}\lambda_a {n \choose j}(\alpha p)^{n - j}\cL(z - a - \alpha p)^j\right]\mathbbm{1}_{a + \alpha p + p^2 \Zp} \mod \pi\latticeL{k} \\
        & \equiv & \sum_{a = 0}^{p - 1}\sum_{\alpha = 0}^{p - 1}p^{-1}\cL\lambda_a\left[(z - a)^n - (z - a - \alpha p)^n\right]\mathbbm{1}_{a + \alpha p + p^2 \Zp} \mod \pi\latticeL{k}.
    \end{eqnarray*}
    Now recall that $\lambda_i \equiv 1 \mod p$ for all $0 \leq i \leq p - 1$. Also using Lemma \ref{stronger bound for polynomials of large degree with varying radii}, we see that the second term in the second line above belongs to $\pi \latticeL{k}$. So
    \begin{eqnarray}\label{g2 - g1 in geq}
        g_2(z) - g_1(z) \equiv \sum_{a = 0}^{p - 1}p^{-1}\cL(z - a)^n \mathbbm{1}_{a + p\Zp} \mod \pi\latticeL{k}.
    \end{eqnarray}

    Next
    \begin{eqnarray}\label{Expression for g1}
        g_1(z) = \sum_{a = 0}^{p - 1}\left[\sum_{j = 0}^{n - 1} \frac{g^{(j)}(a)}{j!}(z - a)^j\right]\mathbbm{1}_{a + p\Zp}.
    \end{eqnarray}
    Note that
    \begin{eqnarray}\label{jth summand in g1 in geq}
        \frac{g^{(j)}(a)}{j!} = \sum_{i = 0}^{p - 1}p^{-1}{n \choose j}\lambda_i(a - [i])^{n - j}\logL(a - [i]).
    \end{eqnarray}

    First assume that $j > r/2$. For $i \neq a$, the $i^{\mathrm{th}}$ summand on the right side of equation \eqref{jth summand in g1 in geq} is integral. Indeed, $0 \leq a, i \leq p - 1$ and $i \neq a$ implies that $a - [i]$ is a $p$-adic unit. Therefore $p^{-1}\logL(a - [i])$ is integral. Using Lemma \ref{stronger bound for polynomials of large degree with varying radii}, we see that the $j > r/2$ and $i \neq a$ terms do not contribute to $g_1(z)$. Working as in the $g_2(z) - g_1(z)$ case we see that for $i = a$, the $i^{\mathrm{th}}$ summand in equation \eqref{jth summand in g1 in geq} is 
    \[
        p^{-1}{n \choose j}\lambda_a(a - [a])^{n - j}\cL \equiv p^{-1}{n \choose j}(a - [a])^{n - j}\cL \mod p^{r/2 - j}\pi.
    \]
    Therefore if $j > r/2$, then
    \begin{eqnarray}\label{j greater than r/2 in g1 and geq}
        \frac{g^{j}(a)}{j!}(z - a)^j\mathbbm{1}_{a + p\Zp} \equiv p^{-1}{n \choose j} (a - [a])^{n - j}\cL (z - a)^j\mathbbm{1}_{a + p\Zp} \mod \pi\latticeL{k}.
    \end{eqnarray}

    Next assume that $j < r/2$. For $i \neq a$, 
    write $a - [i] = [a - i] + pa_i$ for some $a_i \in \Zp$. Then $\logL(a - [i]) = \logL(1 + [a - i]^{-1}pa_i)$ since $[a - i]$ is a root of unity. Therefore the $i^{\mathrm{th}}$ summand on the right side of equation \eqref{jth summand in g1 in geq} is
    \begin{eqnarray*}
        p^{-1}{n \choose j}\lambda_i(a - [i])^{n - j}\logL(a - [i]) 
        & \equiv & p^{-1}{n \choose j}\lambda_i[a - i]^{n - j - 1}pa_i \mod \pi.
    \end{eqnarray*}
    Next, the $i = a$ term on the right side of equation \eqref{jth summand in g1 in geq} is congruent to $0$ modulo $\pi$. Indeed, the valuation of $\logL(a - [a])$ is greater than or equal to $\min\{v_p(\cL), 1\} \geq 1 + r/2 - n$, the valuation of $(a - [a])^{n - j}$ is greater than or equal to $n - j$ and $j < r/2$.

    Finally assume that $j = r/2$. Working as in the $j > r/2, \> i = a$ case and $j < r/2, \> i \neq a$ case, we see that the $i^{\mathrm{th}}$ summand on the right side of equation \eqref{jth summand in g1 in geq} is
    \begin{eqnarray}\label{j equal to r/2 in g1 in geq}
        \begin{cases}
            p^{-1}{n \choose j}(a - [a])^{n - r/2}\cL \mod \pi & \text{ if } i = a, \\
            p^{-1}{n \choose j}\lambda_i[a - i]^{n - r/2 - 1}pa_i \mod \pi & \text{ if } i \neq a.
        \end{cases}
    \end{eqnarray}

    For $0 \leq j \leq r/2$, Lemma \ref{Sum of the i neq a terms} gives us
    \begin{eqnarray}\label{j leq r/2 in g1 and geq}
        \sum_{\substack{i = 0 \\ i \neq a}}^{p - 1}p^{-1}{n \choose j}\lambda_i [a - i]^{n - j - 1}pa_i \equiv {n \choose j}\frac{a^{n - j}}{n - j} \mod \pi.
    \end{eqnarray}

    Putting equations \eqref{j greater than r/2 in g1 and geq}, \eqref{j equal to r/2 in g1 in geq} and \eqref{j leq r/2 in g1 and geq} in equation \eqref{Expression for g1}, and using
    Lemma~\ref{Integers in the lattice}, we get
    \begin{eqnarray}\label{g1 in geq}
        g_1(z) & \equiv & \sum_{a = 0}^{p - 1}\left[\sum_{j = 0}^{\lfloor r/2 \rfloor} {n \choose j}\frac{a^{n - j}}{n - j}(z - a)^j\right]\mathbbm{1}_{a + p\Zp} \nonumber\\
        & & + \sum_{a = 0}^{p - 1}\left[\sum_{j = \lceil r/2 \rceil}^{n - 1}\cL{n \choose j}p^{-1}(a - [a])^{n - j}(z - a)^j\right]\mathbbm{1}_{a + p\Zp} \mod \pi \latticeL{k}.
    \end{eqnarray}

    Equations \eqref{qp-zp part in geq}, \eqref{g2 - g1 in geq} and \eqref{g1 in geq} immediately prove the proposition.
\end{proof}

\begin{remark}
    The term $\delta_{n, p - 1}p^{-1}\cL\mathbbm{1}_{\Zp}$ is present only when $r$ takes the boundary value $p - 1$ and $n = r$ is the highest possible value of $n$.
\end{remark}
    
    To prove the claim made at the beginning of this section concerning $F_{2i - 2, \> 2i - 1}$, we need certain inductive steps coming from Proposition \ref{nu geq} for $\nu > 1 - r/2, 2 - r/2, \ldots, (i - 1) - r/2$. We state and prove these (shallow) inductive steps now.

    \begin{Proposition}\label{inductive steps}
        For $k/2 < n \leq r$, assume that $\nu > 1 + r/2 - n$. Then
        \[
            z^{r - n}\mathbbm{1}_{p\Zp} \equiv \sum_{b = 1}^{p - 1}\sum_{l = 1}^{\lfloor r/2 \rfloor}c(n, l)b^{r - l - n}(z - b)^l\mathbbm{1}_{b + p\Zp} \mod \pi \latticeL{k},
        \]
        where $c(n, l) \in \Zp$ and
        \[
            c(n, l) = 
            \begin{cases}
                \dfrac{(-1)^l(r - n)!n}{(r - l)\ldots(n - l)} & \text{ if } l \geq r - n + 1,\\
                -1 + \dfrac{(-1)^l(r - n)!n}{(r - l)\ldots(n - l)} & \text{ if } l = r - n.
            \end{cases}
        \]
    \end{Proposition}
    \begin{proof}
      If $\nu > 1 + r/2 - n$, then using Proposition \ref{nu geq}, Lemma~\ref{Integers in the lattice} with $h = 0$, Lemma~\ref{stronger bound for polynomials of large degree with varying radii} with
      $h = 1$ and noting that $g(z) = 0$ in $\latticeL{k}$, we get
      \begin{eqnarray}
            \label{hard inductive step}
            0 \equiv \sum_{a = 0}^{p- 1}\left[\sum_{j = 0}^{\lfloor r/2 \rfloor} {n \choose j}\frac{a^{n - j}}{n - j}(z - a)^j\right]\mathbbm{1}_{a + p\Zp} \mod \pi \latticeL{k}.
        \end{eqnarray}
        Similarly, expanding binomially $z^{n} = [(z - a) + a]^{n}$ and using
        Lemma \ref{stronger bound for polynomials of large degree with varying radii} with $h = 1$
        gives
        \begin{eqnarray}\label{Easy inductive step}
            z^{n}\mathbbm{1}_{\Zp} \equiv \sum_{a = 0}^{p - 1}\left[\sum_{j = 0}^{\lfloor r/2 \rfloor}{n \choose j}a^{n - j}(z - a)^j\right]\mathbbm{1}_{a + p\Zp} \mod \pi \latticeL{k}.
        \end{eqnarray}
        Subtracting $n$ times equation \eqref{hard inductive step} from equation \eqref{Easy inductive step}, we get
        \[
            z^n\mathbbm{1}_{\Zp} \equiv \sum_{a = 0}^{p - 1}\left[\sum_{j = 1}^{\lfloor r/2 \rfloor}-\frac{j}{n - j}{n \choose j}a^{n - j}(z - a)^j\right]\mathbbm{1}_{a + p\Zp} \mod \pi \latticeL{k}.
          \]
        We have dropped the $j = 0$ term since the factor $\dfrac{j}{n - j}$ is $0$.  
        Dropping the $a = 0$ summand (since $j < n$)
        and applying $-w$ to both sides,
        we get
        \[
            z^{r-n}\mathbbm{1}_{p\Zp} \equiv \sum_{a = 1}^{p - 1}\left[\sum_{j = 1}^{\lfloor r/2 \rfloor}\frac{j}{n - j}{n \choose j}a^{n - j}z^{r - j}(1 - az)^j\right]\mathbbm{1}_{a^{-1} + p\Zp} \mod \pi \latticeL{k}.
        \]
        Pulling $-a$ out of $(1 - az)^j$, we get
        \[
            z^{r-n}\mathbbm{1}_{p\Zp} \equiv \sum_{a = 1}^{p - 1}\left[\sum_{j = 1}^{\lfloor r/2 \rfloor}\frac{(-1)^jj}{n - j}{n \choose j}a^nz^{r - j}(z - a^{-1})^j\right]\mathbbm{1}_{a^{-1} + p\Zp} \mod \pi \latticeL{k}.
        \]
        Expanding $z^{r - j} = \sum_{l = 0}^{r - j}{r - j \choose l}a^{-l}(z - a^{-1})^{r - j - l}$, using
        Lemma~\ref{Integers in the lattice} to replace $a^{-1}$ by $b$, and dropping
        the $0 \leq l < \lceil r/2 \rceil$ terms by Lemma \ref{stronger bound for polynomials of large degree with varying radii}, we get
        \begin{eqnarray}
        \label{intermediate inductive step}
            z^{r - n}\mathbbm{1}_{p\Zp} \equiv \sum_{b = 1}^{p - 1}\left[\sum_{j = 1}^{\lfloor r/2 \rfloor}\sum_{l = \lceil r/2 \rceil}^{r - j}\frac{(-1)^j j}{n - j}{n \choose j}{r - j \choose l}b^{l - n}(z - b)^{r - l}\right]\mathbbm{1}_{b + p\Zp} \!\! \mod \pi \latticeL{k}.
        \end{eqnarray}
        Exchanging the order of the sums inside the brackets and replacing $l$ by $r - l$, we get
        \[
            z^{r - n}\mathbbm{1}_{p\Zp} \equiv \sum_{b = 1}^{p - 1}\left[\sum_{l = 1}^{\lfloor r/2 \rfloor}\left(\sum_{j = 1}^{l}\frac{(-1)^j j}{n - j}{n \choose j}{r - j \choose r - l}\right)b^{r - l - n}(z - b)^{l}\right]\mathbbm{1}_{b + p\Zp} \mod \pi \latticeL{k}.
        \]
        We see that 
        \begin{eqnarray*}
          c(n, l) & :=  &\sum_{j = 1}^{l}\frac{(-1)^j j}{n - j}{n \choose j}{r - j \choose r - l}  \quad \in \zp\\
                & = & \frac{n!}{(r - l)!}\sum_{j = 1}^{l} (-1)^j \frac{(r - j)(r - j - 1) \cdots (n - j + 1)}{(j - 1)!(n - j)(l - j)!} \nonumber \\
                & = & \frac{n!}{(r - l)!}\sum_{j = 1}^{l} \frac{(-1)^j}{(j - 1)!(l - j)!}\left[(n - j)^{r - n - 1} + s_1 (n - j)^{r - n - 2} + \cdots + s_{r-n-1} + \frac{(r - n)!}{n - j}\right], \nonumber
            \end{eqnarray*}
            where $s_1, s_2, \ldots, s_{r - n - 1} \in \Zp$.
            Using the fact that the alternating sum of binomial coefficients is equal to $0$, we see that if $l \geq r - n + 1$, then for any $0 \leq i \leq r - n - 1$, we have
            \[
                \sum_{j = 1}^{l}\frac{(-1)^j(n - j)^{i}}{(j - 1)!(l - j)!} = 0.
            \]
            If $l = r - n$, then writing $l  = (r-1)-n+1$, we see that the above
            sums vanish for $0 \leq i \leq r-n-2$. However, if $l = r-n$, then  
            \[
                \sum_{j = 1}^{l}\frac{(-1)^j(n - j)^{l - 1}}{(j - 1)!(l - j)!} = -1.
            \]
            Therefore
            \begin{eqnarray}\label{Penultimate c(n, l)}
                c(n, l) = 
                \begin{cases}
                \dfrac{n!}{(r - l)!}\sum \limits_{j = 1}^{l} \dfrac{(-1)^j (r - n)!}{(j - 1)!(l - j)!(n - j)} & \text{ if } l \geq r - n + 1 \\
                \dfrac{n!}{(r - l)!}\left[-1 + \sum\limits_{j = 1}^{l}\dfrac{(-1)^j(r - n)!}{(j - 1)!(l - j)!(n - j)}\right] & \text{ if } l = r - n.
                \end{cases}
            \end{eqnarray}
            By partial fraction decomposition, we have
            \[
                \frac{1}{(n - 1)(n - 2)\cdots(n - l)} = \sum_{j = 1}^{l}\frac{a_j}{n - j}
            \text{ with }   
                a_j = \frac{(-1)^{l - j}}{(j - 1)!(l - j)!}.
            \]
            Substituting this in equation \eqref{Penultimate c(n, l)}, we get
            \[
                c(n, l) = 
                \begin{cases}
                    \dfrac{(-1)^{l}(r - n)! n}{(r - l)\cdots(n - l)} & \text{ if } l \geq r - n + 1 \\
                    -1 + \dfrac{(-1)^l(r - n)! n}{(r - l) \cdots (n - l)} & \text{ if } l = r - n.
                 \end{cases}  
                    \qedhere       
            \]
    \end{proof}

    We now prove the claim made at the beginning of this section.
    \begin{theorem}\label{Final theorem for geq}
        Let $i = 1, 2, \ldots, \lceil r/2 \rceil - 1$. If $\nu \geq i - r/2$, then the map $\IZind a^{i - 1}d^{r - i + 1} \twoheadrightarrow F_{2i - 2, \> 2i - 1}$ factors as
        \[
            \IZind a^{i - 1}d^{r- i + 1} \twoheadrightarrow \left.\begin{cases}\dfrac{\IZind a^{i - 1}d^{r - i + 1}}{\im(\lambda_iT_{-1, 0} - 1)} & \!\!\!\! \text{ if } (i, r) \neq (1, p - 1) \\ \dfrac{\IZind a^{i - 1}d^{r - i + 1}}{\im(\lambda_i(T_{-1, 0} + T_{1, 0}) - 1)} & \!\!\!\!\text{ if } (i, r) = (1, p - 1)\end{cases}\right\} \twoheadrightarrow F_{2i - 2, \> 2i - 1},
        \]
        where
        \[
            \lambda_i = (-1)^i i {r - i + 1 \choose i}p^{r/2 - i}\cL.
        \]
        Moreover for $\nu = i - r/2$, the second map above induces a surjection
            \[
                \pi([2i - 2 - r], \lambda_i^{-1}, \omega^{r - i + 1}) \twoheadrightarrow F_{2i - 2, \> 2i - 1},
            \]
            where $[2i - 2 - r] = 0, 1, \ldots p - 2$ represents the class of $2i - 2 - r$ modulo $p - 1$.
    \end{theorem}
    \begin{proof}
      All congruences in this proof are in the space $\br{\latticeL{k}}$ modulo the image of the subspace $\IZind \oplus_{j < r - i + 1} \Fq X^{r - j}Y^j$ under $\IZind \SymF{k - 2} \twoheadrightarrow \br{\latticeL{k}}$. In particular, we work modulo the $G$-translates of the functions
      $z^i \mathbbm{1}_{p\zp}$, $z^{i+1} \mathbbm{1}_{p\zp}, \ldots,
      z^{\lfloor r/2 \rfloor} \mathbbm{1}_{p\zp} \in \latticeL{k}$.
    
        Fix an $i = 1, 2, \ldots, \lceil r/2 \rceil - 1$. First assume that $(i, r) \neq (1, p - 1)$. Applying Proposition \ref{nu geq} with $n = r - i + 1$, we get
        \begin{eqnarray}\label{Main equation in geq}
            0  \equiv  \sum_{a = 0}^{p - 1}p^{-1}\cL(z - a)^{r - i + 1}\mathbbm{1}_{a + p\Zp} 
                  + \sum_{a = 0}^{p - 1}\left[\sum_{j = 0}^{i - 1}{r - i + 1 \choose j}{\frac{a^{r - i + 1 - j}}{r - i + 1 - j}}(z - a)^j\right]\mathbbm{1}_{a + p\Zp}. 
        \end{eqnarray}
        The second function on the right side of the congruence in the statement of the proposition
        dies since $n = p-1$ if and only if $(i,r)= (1,p-1)$. The last function there
        dies
        since the $(a,j)$-th summand there is an integer multiple of
        $\begin{pmatrix}0 & 1 \\ p & -a \end{pmatrix} z^{r-j} \mathbbm{1}_{p\Zp}$ for
        $\lceil r/2 \rceil \leq j \leq r-i$. 
        
        Next we note that $\nu \geq i - r/2$ implies that $\nu > 1 - r/2, \> 2 - r/2, \> \ldots, \> (i - 1) - r/2$. Therefore using Proposition \ref{inductive steps} with $n = r, \> r - 1, \> \ldots, \> r - i + 2$, we get
        \begin{eqnarray*}
            z^{r - n}\mathbbm{1}_{p\Zp} \equiv \sum_{b = 1}^{p - 1}\sum_{l = 1}^{i - 1} c(n, l)b^{r - l - n}(z - b)^l\mathbbm{1}_{b + p\Zp}.
        \end{eqnarray*}
        Applying $\sum\limits_{a = 0}^{p - 1}a^{n - i + 1}\begin{pmatrix}1 & 0 \\ -a & 1\end{pmatrix}$, we get
        \[
            \sum_{a = 0}^{p - 1}a^{n - i + 1}(z - a)^{r - n}\mathbbm{1}_{a + p\Zp} \equiv \sum_{a = 0}^{p - 1}\sum_{b = 1}^{p - 1}\sum_{l = 1}^{i - 1}a^{n - i + 1}c(n, l)b^{r - l - n}(z - (a + b))^l\mathbbm{1}_{a + b + p\Zp}.
        \]
        Making the substitution $a + b = \lambda$ and using 
        \begin{eqnarray}\label{sum over b identity}
            \sum_{b = 1}^{p - 1}(\lambda - b)^{n - i + 1}b^{r - l - n} \equiv 
            \begin{cases}
                (-1)^{n - r + l + 1}{n - i + 1 \choose n - r + l}\lambda^{r - l - i + 1} & \text{ if } l \geq r - n \\
                0 & \text{ if } l < r - n 
            \end{cases} \mod p
        \end{eqnarray}
        and Lemma~\ref{Integers in the lattice}, we get
        \[
            \sum_{a = 0}^{p - 1}a^{n - i + 1}(z - a)^{r - n}\mathbbm{1}_{a + p\Zp} \equiv \sum_{\lambda = 0}^{p - 1}\sum_{l = r - n}^{i - 1}(-1)^{n - r + l + 1}\lambda^{r - l - i + 1}{n - i + 1 \choose n - r + l}c(n, l)(z - \lambda)^l\mathbbm{1}_{\lambda + p\Zp}.
        \]
        Renaming $\lambda$ as $a$ on the right side of the equation above and plugging the value of $c(n, l)$ obtained in Proposition \ref{inductive steps}, we get one equation:
        \begin{eqnarray}\label{Refined inductive steps in geq}
            0 \equiv \sum_{a = 0}^{p - 1}\sum_{l = r - n}^{i - 1} {n - i + 1 \choose n - r + l}\frac{(r - n)! n}{(r - l) \cdots (n - l)}a^{r - l - i + 1}(z - a)^{l}\mathbbm{1}_{a + p\Zp},
        \end{eqnarray}
        for each $n = r, \> r - 1, \> \ldots, \> r - i + 2$. We note that the extra
        $-1$ part of $c(n, l)$ for $l = r - n$ cancels with the left side.

        Finally using the binomial expansion, we have
        \begin{eqnarray}\label{Trivial equation in geq}
            z^{r - i + 1}\mathbbm{1}_{\Zp} \equiv \sum_{a = 0}^{p - 1}\sum_{l = 0}^{i - 1}{r - i + 1 \choose l}a^{r - i + 1 - l}(z - a)^l\mathbbm{1}_{a + p\Zp}.
        \end{eqnarray}
        Using equations \eqref{Refined inductive steps in geq} and \ref{Trivial equation in geq}, we wish to modify the expression in the second term on the right side of equation \eqref{Main equation in geq} to a multiple of $z^{r - i + 1}\mathbbm{1}_{\Zp}$. To this end, we consider the matrix equation
        \begin{eqnarray}\label{Matrix equation in geq}
            \!\!\!\!{\tiny\begin{pmatrix}
                {r - i + 1 \choose 0}\frac{r}{r} & 0 & 0 & \cdots & {r - i + 1 \choose 0} \\
                {r - i + 1 \choose 1}\frac{r}{r - 1} & {r - i \choose 0}\frac{r - 1}{(r - 1)(r - 2)} & 0 & \cdots & {r - i + 1 \choose 1} \\
                {r - i + 1 \choose 2}\frac{r}{r - 2} & {r - i \choose 1}\frac{r - 1}{(r - 2)(r - 3)} & {r - i - 1 \choose 0}\frac{2!(r - 2)}{(r - 2)(r - 3)(r - 4)} & \cdots & {r - i + 1 \choose 2} \\
                \vdots & & & & \vdots \\
                {r - i + 1 \choose i - 1}\frac{r}{r - i + 1} & {r - i \choose i - 2}\frac{r - 1}{(r - i + 1)(r - i)} & {r - i - 1 \choose i - 3}\frac{2!(r - 2)}{(r - i + 1)(r - i)(r - i - 1)} & \cdots & {r - i + 1 \choose i - 1}
            \end{pmatrix}
            \!\!\!\begin{pmatrix}
                x_1 \\ x_2 \\ x_3 \\ \vdots \\ x_{i - 1} \\ x_i
            \end{pmatrix} \!\!\!
            = \!\!\!\begin{pmatrix}
                {r - i + 1 \choose 0}\frac{1}{r - i + 1} \\ {r - i + 1 \choose 1}\frac{1}{r - i} \\ {r - i + 1 \choose 2}\frac{1}{r - i - 1} \\ \vdots \\ {r - i + 1 \choose i - 2}\frac{1}{r - 2i + 3} \\ {r - i + 1 \choose i - 1}\frac{1}{r - 2i + 2}
            \end{pmatrix}\!\!.}
        \end{eqnarray}
        Once we have the solution to this equation, we multiply the $n = r, \> r - 1, \> r - 2, \> \ldots, \> r - i + 2$ instances of \eqref{Refined inductive steps in geq} by $x_1, \> x_2, \> x_3, \ldots, \> x_{i - 1}$, respectively, and add them to $x_i$ times equation \eqref{Trivial equation in geq} to get
        \[
             x_i z^{r - i + 1}\mathbbm{1}_{\Zp} \equiv \sum_{a = 0}^{p - 1}\left[\sum_{l = 0}^{i - 1}{r - i + 1 \choose l}\frac{a^{r-i+1 - j}}{r-i+1 - l}(z - a)^l\right]\mathbbm{1}_{a + p\Zp}.
        \]
        Substituting this in equation \eqref{Main equation in geq}, we get
        \begin{eqnarray}\label{Crude final congruence in geq}
            0 \equiv \left[\sum_{a = 0}^{p - 1}p^{-1}\cL(z - a)^{r - i + 1}\mathbbm{1}_{a + p\Zp}\right] + x_i z^{r - i + 1}\mathbbm{1}_{\Zp}.
        \end{eqnarray}
        We first find the value of $x_i$. This is done using Cramer's rule. Let $A$ be the $i \times i$ matrix on the left side of equation \eqref{Matrix equation in geq}. Let $A_i$ be the same $i \times i$ matrix, but with the last column replaced by the the column vector on the right side of equation \eqref{Matrix equation in geq}. Then 
        \[
            x_i = \frac{\det A_i}{\det A}.
        \]

        First consider the matrix $A$. We number the rows and columns as $0, 1, \ldots, i - 1$. Perform the operation $C_{i - 1} \to C_{i - 1} - C_0$ to make the $(0, i - 1)^{\mathrm{th}}$ entry $0$. Then for $j \geq 1$, the $(j, i - 1)^{\mathrm{th}}$ entry in the resulting matrix is
        \[
            {r - i + 1 \choose j} - {r - i + 1 \choose j}\frac{r}{r - j} = -{r - i \choose j - 1}\frac{r - i + 1}{r - j}.
        \]
        Next, perform the operation $C_{i - 1} \to C_{i - 1} + {r - 2 \choose 1}\left(\frac{r - i + 1}{r - 1}\right)C_1$ to make the $(1, i - 1)^{\mathrm{th}}$ entry $0$. Then for $j \geq 2$, the $(j, i - 1)^{\mathrm{th}}$ entry in the resulting matrix is
        \[
            -{r - i \choose j - 1}\frac{r - i + 1}{r - j} + {r - 2 \choose 1}\frac{r - i + 1}{r - 1}{r - i \choose j - 1}\frac{r - 1}{(r - j){r - j - 1 \choose 1}} = {r - i - 1 \choose j - 2}\frac{(r - i + 1)(r - i)}{(r - j)(r - j - 1)}.
        \]
        Continuing this process, we see that after making $(i - 1)$ entries in the last column $0$, the $(j, i - 1)^{\mathrm{th}}$ entry for $j \geq i - 1$ is
        \[
            (-1)^{i - 1}{r - 2i + 2 \choose j - (i - 1)}\frac{(r - i + 1)(r - i) \cdots (r - 2i + 3)}{(r - j)(r - j - 1) \cdots (r - j - (i - 2))}.
        \]
        Therefore after reducing $A$ to a lower triangular matrix, we see that the $(i - 1, i - 1)^{\mathrm{th}}$ entry is $(-1)^{i - 1}$.

        Next consider the matrix $A_i$. Perform the operation $C_{i - 1} \to C_{i - 1} - \frac{1}{r - i + 1}C_0$ to make the $(0, i - 1)^{\mathrm{th}}$ entry $0$. Then for $j \geq 1$, the $(j, i - 1)^{\mathrm{th}}$ entry is
        \[
            {r - i + 1 \choose j}\frac{1}{r - i + 1 - j} - \frac{1}{r - i + 1}{r - i + 1 \choose j}\frac{r}{r - j} = {r - i \choose j - 1}\frac{i - 1}{(r - i + 1 - j)(r - j)}.
        \]
        To make the $(1, i - 1)^{\mathrm{th}}$ entry $0$, perform $C_{i - 1} \to C_{i - 1} - \frac{(r - 2)(i - 1)}{(r - i)(r - 1)}C_1$. Then for $j \geq 2$, the $(j, i - 1)^{\mathrm{th}}$ entry is
        \begin{eqnarray*}
            && {r - i \choose j - 1}\frac{i - 1}{(r - i + 1 - j)(r - j)} - \frac{(r - 2)(i - 1)}{(r - i)(r - 1)}{r - i \choose j - 1}\frac{r - 1}{(r - j)(r - j - 1)} \\
            && = {r - i - 1 \choose j - 2}\frac{(i - 1)(i - 2)}{(r - j)(r - j - 1)(r - i + 1 - j)}.
        \end{eqnarray*}
        Continuing this process, we see that after making the $(i - 1)$ entries in the last column $0$, the $(j, i - 1)^{\mathrm{th}}$ entry for $j \geq i - 1$ is
        \[
            {r - 2i + 2 \choose j - (i - 1)}\frac{(i - 1)!}{(r - j)(r - j - 1) \cdots (r - j - (i - 2))(r - i + 1 - j)}.
        \]
        Therefore the $(i - 1, i - 1)^{\mathrm{th}}$ entry is
        \[
            \frac{(i - 1)!}{(r - i + 1)(r - i)\cdots (r - 2i + 2)}.
        \]

        Putting these determinant computations together in Cramer's formula, we get
        \begin{equation}\label{value for c in geq}
            x_i = (-1)^{i - 1}\frac{(i - 1)!}{(r - i + 1)(r - i)\cdots (r - 2i + 2)} = (-1)^{i - 1} \frac{1}{i{r - i + 1 \choose i}}.
        \end{equation}
        Substituting this value of $x_i$ in equation \eqref{Crude final congruence in geq}, we get
        \[
            0 \equiv (-1)^{i}i {r - i + 1 \choose i}\left[\sum_{a = 0}^{p - 1}p^{-1}\cL(z - a)^{r - i + 1}\mathbbm{1}_{a + p\Zp}\right] - z^{r - i + 1}\mathbbm{1}_{\Zp}.
        \]
        We show that this equation is the image of $-w(\lambda_i T_{-1, 0} - 1)\llbracket \id, X^{i - 1}Y^{r - i + 1}\rrbracket$ under the surjection $\IZind a^{i - 1}d^{r - i + 1} \twoheadrightarrow F_{2i - 2, \> 2i - 1}$.
        Clearly, $w\cdot z^{i - 1}\mathbbm{1}_{p\Zp} = -z^{r - i + 1}\mathbbm{1}_{\Zp}$. Next, using
        \eqref{Formulae for T-10 and T12 in the commutative case} we have        
        \begin{eqnarray*}
            -wT_{-1, 0}\llbracket \id, X^{i - 1}Y^{r - i + 1}\rrbracket & = & -\sum_{\lambda \in I_1}\left\llbracket \begin{pmatrix}0 & 1 \\ 1 & 0\end{pmatrix}\begin{pmatrix}p & -\lambda \\ 0 & 1\end{pmatrix}, X^{i - 1}Y^{r - i + 1}\right\rrbracket \\
            & \mapsto & -\sum_{\lambda \in I_1}\begin{pmatrix}0 & 1 \\ p & -\lambda\end{pmatrix}z^{i - 1}\mathbbm{1}_{p\Zp} \\
            & = & -\sum_{\lambda \in I_1}p^{i - 1 - r/2}(z - \lambda)^{r - i + 1}\mathbbm{1}_{p\Zp}\left(\frac{p}{z - \lambda}\right) \\
            & = & \sum_{\lambda \in I_1} p^{i - r/2 - 1}(z - \lambda)^{r - i + 1}\mathbbm{1}_{\lambda + p\Zp}.
        \end{eqnarray*}
        Therefore we have proved that the surjection $\IZind a^{i - 1}d^{r - i + 1} \twoheadrightarrow F_{2i - 2, \> 2i - 1}$ factors as
        \[
            \IZind a^{i - 1}d^{r - i + 1} \twoheadrightarrow \frac{\IZind a^{i - 1}d^{r - i + 1}}{\im(\lambda_iT_{-1, 0} - 1)} \twoheadrightarrow F_{2i - 2, \> 2i - 1}.
        \]

        Next assume that $\nu = i - r/2$. Consider the following exact sequences
        \begin{eqnarray}
          \label{rectangle}
          \begin{tikzcd}
                0 \arrow[r] & \im T_{-1, 0} \arrow[r] \arrow[d, two heads] & \IZind a^{i - 1}d^{r - i + 1} \arrow[r]\arrow[d, two heads] & \dfrac{\IZind a^{i - 1}d^{r - i + 1}}{\im T_{-1, 0}} \arrow[r]\arrow[d, two heads] & 0 \\
                0 \arrow[r] & S \arrow[r] & F_{2i - 2, \> 2i - 1} \arrow[r] & Q \arrow[r] & 0,
          \end{tikzcd}
        \end{eqnarray}
        where $S$ (for submodule) is the image of $\im T_{-1, 0}$ under the surjection $\IZind a^{i - 1}d^{r - i + 1} \twoheadrightarrow F_{2i - 2, \> 2i - 1}$ and $Q$ (for quotient) is $F_{2i - 2, \> 2i - 1}/S$. We know that the subspace $\im (\lambda_i T_{-1, 0} - 1)$ maps to $0$ under the middle vertical map in the display above. Since $\im (\lambda_i T_{-1, 0} - 1)$ maps surjectively onto $\dfrac{\IZind a^{i - 1}d^{r - i + 1}}{\im T_{-1, 0}}$ the right vertical map is $0$. Therefore we get a surjection $\im T_{-1, 0} \twoheadrightarrow F_{2i - 2, \> 2i - 1}$. First assume that $a^{i - 1}d^{r - i + 1}$ does not factor through $\omega$. For $1 \leq r \leq p - 1$ and $1 \leq i \leq \lceil r/2 \rceil - 1$, the character $a^{i - 1}d^{r - i + 1}$ factors through $\omega$ only when $i = 1$ and $r = p - 1$. Using the first isomorphism theorem and \cite[Proposition 3.1]{AB15}, we get a surjection
        \[
            \frac{\IZind a^{i - 1}d^{r - i + 1}}{\im T_{1, 2}} \twoheadrightarrow F_{2i - 2, \> 2i - 1}
          \]
          obtained by pre-composing the first surjection above with $T_{-1,0}$. Since
          $T_{-1,0} ( \im (\lambda_i T_{-1, 0} - 1) ) \subset   \im (\lambda_i T_{-1, 0} - 1)$        
        the last surjection factors through
        \[
           \frac{\IZind a^{i - 1}d^{r - i + 1}}{\im T_{1, 2} + \im (\lambda_i T_{-1, 0} - 1)} \twoheadrightarrow F_{2i - 2, \> 2i - 1}.
         \]
         Since $\lambda_i$ is a unit, the space on the left is the same
         as
         \[
           \frac{\IZind a^{i - 1}d^{r - i + 1}}{\im T_{1, 2} + \im (T_{-1, 0} - \lambda_i^{-1})}
               \simeq \pi([2i - 2 - r], \lambda_i^{-1}, \omega^{r - i + 1}), 
         \]
        where $[2i - 2 - r] = 0, 1, \ldots, p - 2$ represents the class of $2i - 2 - r$ modulo $p - 1$, by
        \eqref{alt def of pi}, noting that
        $[2i - 2 - r] \neq 0$ since $(i,r) \neq (1,p-1)$.
        
        Now assume that $(i, r) = (1, p - 1)$. We remark that in this case $k = p+1$ and so this case was not considered in \cite{BM}. Even in our approach involving the
        Iwahori mod $p$ LLC a new complication arises: we need to use the Hecke operator $T_{1,0}$ in the non-commutative Iwahori-Hecke algebra to analyze the subquotient $F_{0,1}$
        Applying Proposition $\ref{nu geq}$ with $n = r = p - 1$, we get
        \begin{eqnarray*}
            0 & \equiv & \sum_{a = 0}^{p - 1}p^{-1}\cL(z - a)^r\mathbbm{1}_{a + p\Zp} - p^{-1}\cL \mathbbm{1}_{\Zp} + \sum_{a = 0}^{p - 1} \frac{a^r}{r}\mathbbm{1}_{a + p\Zp}.
        \end{eqnarray*}
        Here the third function on the right side of the congruence in the proposition simplifies to the
        last term above by the remarks made in the first paragraph of the proof.
        Multiplying this equation by $-r$ and using the definition of $\lambda_i$ in the statement of the
        theorem, 
        we get
        \begin{eqnarray}\label{Crude final congruence for i,r = 1, p-1}
            0 \equiv \lambda_i\left[\sum_{a = 0}^{p - 1}p^{-r/2}(z - a)^r\mathbbm{1}_{a + p\Zp} - p^{-r/2}\mathbbm{1}_{\Zp}\right] - \sum_{a = 0}^{p - 1}a^r\mathbbm{1}_{a + p\Zp}.
        \end{eqnarray}
        By the remarks at the beginning of the proof
        \[
            z^r\mathbbm{1}_{\Zp} = \sum_{a = 0}^{p - 1}\sum_{j = 0}^{r}{r \choose j}a^j(z - a)^{r - j}\mathbbm{1}_{a + p\Zp} \equiv \sum_{a = 0}^{p - 1}\left[a^r + (z - a)^r\right]\mathbbm{1}_{a + p\Zp} \equiv \sum_{a = 0}^{p - 1}a^r\mathbbm{1}_{a + p\Zp}.
        \]
        For the final congruence above, we have used Lemma \ref{stronger bound for polynomials of large degree with varying radii}. Substituting this in equation \eqref{Crude final congruence for i,r = 1, p-1}, we get
        \[
            0 \equiv \lambda_i\left[\sum_{a = 0}^{p - 1}p^{-r/2}(z - a)^r\mathbbm{1}_{a + p\Zp} - p^{-r/2}\mathbbm{1}_{\Zp}\right] - z^r\mathbbm{1}_{\Zp}.
        \]
        Using \eqref{Formulae for T10 and T12 in the non-commutative case}, we see exactly as before that the right side of the equation above is the image of the function $-w[\lambda_i(T_{-1, 0} + T_{1, 0}) - 1]\llbracket \id, Y^{r}\rrbracket$ under the surjection $\IZind d^{p - 1} \twoheadrightarrow F_{0, 1}$. Therefore the surjection $\IZind 1 \twoheadrightarrow F_{0, 1}$ factors as
        \[
            \IZind 1 \twoheadrightarrow \frac{\IZind 1}{\im(\lambda_i(T_{-1, 0} + T_{1, 0}) - 1)} \twoheadrightarrow F_{0, 1}.
        \]
        Now assume that $\nu = 1 - r/2$. Using Lemma \ref{Socle of the top slice dies}, we see that this surjection further gives us
        \[
            \frac{\IZind 1}{\im(T_{1, 2}T_{1, 0}) + \im ((T_{-1, 0} + T_{1, 0}) - \lambda_{i}^{-1})} \twoheadrightarrow F_{0, 1}.
        \]
        By \eqref{alt def of pi}, and noting $T_{1,0}$ is an automorphism, the space on the left is 
        \[
          \pi(0, \lambda_{i}^{-1}, 1) = \pi([2i - 2 - r], \lambda_i^{-1}, \omega^{r - i + 1}).
        \]
        Therefore we have proved the theorem. 
    \end{proof}

\subsection{The $\nu \leq i - r/2$ cases}
    In this section, we prove that if $\nu \leq i - r/2$ for $i = 1, \> 2, \> \ldots, \> \lceil r/2 \rceil - 1$, then the map $\IZind a^i d^{r - i} \twoheadrightarrow F_{2i, \> 2i + 1}$ factors as
    \[
        \IZind a^i d^{r - i} \twoheadrightarrow \frac{\IZind a^i d^{r - i}}{\im(\lambda_i^{-1} T_{1, \> 2} - 1)} \twoheadrightarrow F_{2i, \> 2i + 1}, 
    \]
    where $\lambda_i = (-1)^i i {r - i + 1 \choose i}p^{r/2 - i}\cL$. Moreover for $\nu = i - r/2$, the second map in the display above induces a surjection 
    \[
        \pi(r - 2i, \lambda_i, \omega^i) \twoheadrightarrow F_{2i, \> 2i + 1}.
    \]

    \begin{Proposition}\label{nu leq}
        Let $k/2 < n \leq k - 2 \leq p - 1$. For $\nu \leq 1 + r/2 - n$, set
        \[
            g(z) = p^x\left[\sum_{i \in I} \lambda_i (z - i)^n\logL(z - i)\right],
        \]
        where $x \in \bQ$ with $x + \nu = r/2 - n$, $I = \{0, 1, \ldots, n, p\}$ and the $\lambda_i \in \zp$
        are as in Lemma~\ref{Main coefficient identites} (2).
        Then, we have
        \[
            g(z) \equiv p^{1 + x}\sum_{a = 1}^{p - 1}a^{-1}z^n\mathbbm{1}_{a + p\Zp} + \sum_{j = \lceil r/2 \rceil}^{n - 1}(-1)^{n - j + 1}{n \choose j}p^{x + n - j}\cL z^j \mathbbm{1}_{p\Zp} \mod \pi \latticeL{k}.
        \]
    \end{Proposition}
    \begin{proof}

      Write $g(z) = g(z)\mathbbm{1}_{\Zp} + g(z)\mathbbm{1}_{\Qp \setminus \Zp}$. By Lemma~\ref{Main coefficient identites} (2), we have
      \begin{eqnarray}\label{Identities concerning the coefficients in leq}
        \sum_{i \in I}\lambda_i i^{j} = 0, \text{ for } 0 \leq j \leq n,
        \> \lambda_0 \equiv 1 \!\! \mod p, \> \lambda_i \equiv 0 \!\! \mod p \text{ for }
                1 \leq i \leq n  \text{ and } \lambda_p = -1. 
      \end{eqnarray}
      Since the summation identity in \eqref{Identities concerning the coefficients in leq} is also true
      for $j = n$, the $\logL(z)$ term dies and so
        \[
            w\cdot g(z)\mathbbm{1}_{\Qp \setminus \Zp} = \sum_{i \in I}p^x\lambda_iz^{r - n}(1 - zi)^n\logL(1 - zi)\mathbbm{1}_{p\Zp}.
        \]
        Lemma \ref{qp-zp part is 0} implies that $w\cdot g(z) \mathbbm{1}_{\Qp \setminus \Zp} \equiv 0 \mod \pi \latticeL{k}$.
    
        Next, using Lemma \ref{telescoping lemma} (1) we see that $g(z)\mathbbm{1}_{\Zp} \equiv g_2(z) \mod \pi \latticeL{k}$. We first claim that $g_2(z) - g_1(z) \equiv 0 \mod \pi \latticeL{k}$. Working as in the proof of Proposition \ref{nu geq}, we see that the coefficient of $(z - a - \alpha p)^j \mathbbm{1}_{a + \alpha p + p^2 \Zp}$ in $g_2(z) - g_1(z)$ is
        \begin{eqnarray}\label{jth summand in g2 - g1 in leq}
          && {n \choose j}
             \sum_{i \in I}p^x \lambda_i  \Big[(a + \alpha p - i)^{n - j}\logL(a + \alpha p - i) - (a - i)^{n - j}\logL(a - i) \\
            && - {n - j \choose 1}(\alpha p)(a - i)^{n - j - 1}\logL(a - i) - \cdots - {n - j \choose n - j - 1}(\alpha p)^{n - j - 1}(a - i)\logL(a -  i)\Big]. \nonumber
        \end{eqnarray}
        By Lemma~\ref{Integers in the lattice}, to prove the claim it suffices to show
        that this coefficient is congruent to $0$ modulo $p^{r/2 - j}\pi$.

        First assume that $a \neq 0$. For $i \neq a$, the $i^{\mathrm{th}}$ summand in \eqref{jth summand in g2 - g1 in leq} 
        can be written as
        \[
          p^x\lambda_i\big[(a - i + \alpha p)^{n - j}\logL(1 + (a - i)^{-1}\alpha p) + (\alpha p)^{n - j}\logL(a - i)\big].
        \]
        The condition $\nu \leq 1 + r/2 - n$ and the equation $x + \nu = r/2 - n$ implies that $x \geq - 1$. Therefore the valuation of the last term in the expression above is greater than or equal to $n - j > r/2 - j$. So we drop the last term modulo $p^{r/2 - j}\pi$. Moreover, expanding $(a - i + \alpha p)^{n - j}$ and the logarithm, we get
        \begin{eqnarray}\label{a neq 0 and zi neq a term}
          p^x\lambda_i\big[(a - i)^{n - j - 1}\alpha p + c_{n - j - 2}(a - i)^{n - j - 2}(\alpha p)^2 + \cdots + c_1(a - i)(\alpha p)^{n - j - 1}\big] 
            \!\!\!\! \mod p^{r/2 - j}\pi,
        \end{eqnarray}
        for some $c_1,\ldots, c_{n - j - 2} \in \Zp$.
        Next assume that $i = a$. Then the $i^{\mathrm{th}}$ summand
        in expression \eqref{jth summand in g2 - g1 in leq} can be written as
        \[
          p^x\lambda_i(\alpha p)^{n - j}\logL(\alpha p).
        \]
        The valuation of this term is greater than or equal to $1 + r/2 - j > r/2 - j$ since $p \mid \lambda_a$ by \eqref{Identities concerning the coefficients in leq}. Therefore the $i=a$ summand
        in expression \eqref{jth summand in g2 - g1 in leq} is congruent to $0$ modulo $p^{r/2 - j}\pi$.

        Using the summation identities in \eqref{Identities concerning the coefficients in leq}
        and expression \eqref{a neq 0 and zi neq a term}, we see that the coefficient of $(z - a - \alpha p)^j\mathbbm{1}_{a + \alpha p + p^2 \Zp}$ in $g_2(z) - g_1(z)$ is $0$ modulo $p^{r/2 - j}\pi$ if $a \neq 0$.

        Now assume that $a = 0$. For $i \neq 0, \> p$, the $i^{\mathrm{th}}$ summand in expression \eqref{jth summand in g2 - g1 in leq} is also given by \eqref{a neq 0 and zi neq a term}. Making a simplification similar to the one made in claim \eqref{Replace logs by L}, we see that the sum of the $i = 0$ and $i = p$ summands
        in expression \eqref{jth summand in g2 - g1 in leq} is
        \begin{eqnarray*}
          p^x\left\{\left[(\alpha p)^{n - j}\cL\right] - \left[(\alpha p - p)^{n - j}\cL - (-p)^{n - j}\cL - {n - j \choose 1}(\alpha p)(-p)^{n - j - 1}\cL \right.\right.\\
            \left.\left. - \cdots - {n - j \choose n - j - 1}(\alpha p)^{n - j - 1}(-p)\cL\right]\right\} \mod p^{r/2 - j}\pi.
        \end{eqnarray*}
        This expression is clearly $0$. Now using the identity concerning the sum in \eqref{Identities concerning the coefficients in leq}, we see that the coefficient of $(z - a - \alpha p)^j\mathbbm{1}_{a + \alpha p + p^2 \Zp}$ in $g_2(z) - g_1(z)$ is
        \[
            -{n \choose j}p^x\lambda_p\left[(-p)^{n - j - 1}\alpha p + c_{n - j - 2}(-p)^{n - j - 2}(\alpha p)^2 + \cdots + c_1(-p)(\alpha p)^{n - j - 1}\right] \mod p^{r/2 - j}\pi.
        \]
        Each term in the expression above has valuation greater than or equal to $n - j - 1 > r/2 - j$. Therefore we have proved that the coefficient of $(z - a - \alpha p)^j\mathbbm{1}_{a + \alpha p + p^2\Zp}$ in $g_2(z) - g_1(z)$ is $0$ modulo $p^{r/2 - j}\pi$ if $a = 0$.

        This shows that $g(z) \equiv g_1(z) \!\! \mod \pi \latticeL{k}$. Next, we simplify $g_1(z)$. Recall that
        \begin{eqnarray}\label{g1 in leq}
            g_1(z) = \sum_{a = 0}^{p - 1}\left[\sum_{j = 0}^{n - 1}\frac{g^{(j)}(a)}{j!}(z - a)^j\right]\mathbbm{1}_{a + p\Zp},
        \end{eqnarray}
        where
        \begin{eqnarray}\label{jth summand in g1 in leq}
            \frac{g^{j}(a)}{j!} = \sum_{i \in I}{n \choose j}p^x\lambda_i(a - i)^{n - j}\logL(a - i).
        \end{eqnarray}
        First assume that $a \neq 0$. The $i \neq 0, a, p$ summands on the right side of equation \eqref{jth summand in g1 in leq} are congruent to $0$ modulo $\pi$ since $a \not \equiv i \!\! \mod p$ and $p \mid \lambda_i$ for $i \neq 0, p$. The $i = a$ term is equal to $0$. The sum of the $i = 0$ and $i = p$ terms
        in \eqref{jth summand in g1 in leq}
        is congruent to 
        \[
            {n \choose j}p^x\left[a^{n - j}\logL(a) - (a - p)^{n - j}\logL(a - p)\right] \mod \pi
          \]
        since, by \eqref{Identities concerning the coefficients in leq}, $\lambda_0 = 1 \!\! \mod p$
        and $\lambda_p = -1$.  
        Expanding $(a - p)^{n - j}$ and dropping the terms divisible by $p$ since $x \geq -1$ and $p \> \vert \> \logL(a - p)$, this equals
        \[
            {n \choose j}p^x(-a^{n - j})\logL(1 - a^{-1}p) \mod \pi.
        \]
        Expanding $\logL(1 - a^{-1}p)$ using the usual Taylor series and dropping the terms that are congruent to $0$ mod $\pi$, we see that the sum of the $i = 0$ and $p$ summands in \eqref{jth summand in g1 in leq} is
        \[
            {n \choose j}p^{1 + x}a^{n - j - 1} \mod \pi.
        \]
        Therefore for $a \neq 0$, we get
        \begin{eqnarray}\label{jth term in a neq 0 in g1 in leq}
            \frac{g^{(j)}(a)}{j!} \equiv {n \choose j}p^{1 + x}a^{n - j - 1} \mod \pi.
        \end{eqnarray}

        Next, assume that $a = 0$. Then again the $i \neq 0, \> p$ summands in
        \eqref{jth summand in g1 in leq} are congruent to $0$ modulo $\pi$ by the same reasoning as in the $a \neq 0$ case above. The $i = 0$ summand is $0$. The $i = p$ summand is
        \[
            -{n \choose j}p^x(-p)^{n - j}\cL.
        \]
        Therefore for $a = 0$, we have
        \begin{eqnarray}\label{jth term in a = 0 in g1 in leq}
            \frac{g^{(j)}(0)}{j!} \equiv (-1)^{n - j + 1}{n \choose j}p^{x + n - j}\cL \mod \pi.
        \end{eqnarray}
        Putting equations \eqref{jth term in a neq 0 in g1 in leq} and \eqref{jth term in a = 0 in g1 in leq} in equation \eqref{g1 in leq} and using Lemma \ref{Integers in the lattice}, we get
        \begin{eqnarray*}
            g_1(z) & \equiv & \sum_{a = 1}^{p - 1}\left[\sum_{j = 0}^{n - 1}{n \choose j}p^{1 + x}a^{n - j - 1}(z - a)^j\right]\mathbbm{1}_{a + p\Zp} + \left[\sum_{j = 0}^{n - 1}(-1)^{n - j + 1}{n \choose j}p^{x + n - j}\cL z^j\right]\mathbbm{1}_{p\Zp} \\
            & \equiv & \sum_{a = 1}^{p - 1}p^{1 + x}a^{-1}\left[z^n - (z - a)^n\right]\mathbbm{1}_{a + p\Zp} + \left[\sum_{j = 0}^{n - 1}(-1)^{n - j + 1}{n \choose j}p^{x + n - j}\cL z^j\right]\mathbbm{1}_{p\Zp} \\
            && \qquad \qquad \mod \pi \latticeL{k}.
        \end{eqnarray*}
        Using Lemma~\ref{stronger bound for polynomials of large degree with varying radii} for the first sum and Lemma \ref{Integers in the lattice} for the second sum above, we get
        \[
            g_1(z) \equiv \sum_{a = 1}^{p - 1}p^{1 + x}a^{-1}z^n\mathbbm{1}_{a + p\Zp} + \left[\sum_{j = \lceil r/2 \rceil}^{n - 1}(-1)^{n - j + 1}{n \choose j}p^{x + n - j}\cL z^j\right]\mathbbm{1}_{p\Zp} \mod \pi \latticeL{k}.
        \]
        This proves the proposition since $g(z) \equiv g_1(z) \!\! \mod \pi \latticeL{k}.$
    \end{proof}
    
    \begin{theorem}\label{Final theorem for leq}
        Let $i = 1, \> 2, \> \ldots, \> \lceil r/2 \rceil - 1$. If $\nu \leq i - r/2$, then the map $\IZind a^i d^{r - i} \twoheadrightarrow F_{2i, \> 2i + 1}$ factors as
        \[
            \IZind a^i d^{r - i} \twoheadrightarrow \frac{\IZind a^i d^{r - i}}{\im(\lambda_i^{-1} T_{1, 2} - 1)} \twoheadrightarrow F_{2i, \> 2i + 1},
        \]
        where $\lambda_i = (-1)^i i {r - i + 1 \choose i}p^{r/2 - i}\cL$. Moreover for $\nu = i - r/2$, the second map in the display above induces a surjection 
    \[
        \pi(r - 2i, \lambda_i, \omega^i) \twoheadrightarrow F_{2i, \> 2i + 1}.
    \]
    \end{theorem}
    \begin{proof}
        All congruences in this proof are in the space $\br{\latticeL{k}}$ modulo the image of the subspace $\IZind \oplus_{j < r - i}\Fq X^{r - j}Y^j$ under $\IZind \SymF{k - 2} \twoheadrightarrow \br{\latticeL{k}}$.

        Fix an $i = 1, 2, \ldots, \lceil r/2 \rceil - 1$. Applying Proposition~\ref{nu leq} with $n = r - i + 1$ and the remark above, we get
        \begin{eqnarray}\label{Main equation in leq}
            0 \equiv p^{1 + x}\sum_{a = 1}^{p - 1}a^{-1}z^{r - i + 1}\mathbbm{1}_{a + p\Zp} + (r - i + 1)p^{x + 1}\cL z^{r - i}\mathbbm{1}_{p\Zp}.
        \end{eqnarray}
        First assume that $\nu < i - r/2$. Then $x > -1$. So write
        \[
            p^{1 + x}\sum_{a = 1}^{p - 1}a^{-1}z^{r - i + 1}\mathbbm{1}_{a + p\Zp} = \sum_{a = 1}^{p - 1}\sum_{j = 0}^{r - i + 1}p^{1 + x}a^{-1}{r - i + 1 \choose j}a^{r - i + 1 - j}(z - a)^{j}\mathbbm{1}_{a + p\Zp}.
        \]
        Using Lemma \ref{Integers in the lattice}, we see that the terms on the right side of the equation above are congruent to $0$ modulo $\pi \latticeL{k}$. Therefore we get
        \[
            0 \equiv (r - i + 1)p^{1 + x}\cL z^{r - i}\mathbbm{1}_{p\Zp}.
        \]
        Applying $\beta = \begin{pmatrix}0 & 1 \\ p & 0\end{pmatrix}$, we get
        \[
            0 \equiv -(r - i + 1)p^{1 + x + r/2 - i}\cL z^i \mathbbm{1}_{p\Zp}.
        \]
        Since $x + \nu = r/2 - n = i - 1 - r/2$ and $1 \leq r - i + 1 \leq p - 1$, we see that the coefficient $-(r - i+1)p^{1 + x + r/2 - i}\cL$ is a $p$-adic unit. Cancelling this coefficient, we get
        \begin{eqnarray}\label{nu strictly less than i - r/2}
            0 \equiv z^i\mathbbm{1}_{p\Zp}.
        \end{eqnarray}
        This shows that if $\nu < i - r/2$, then the map $\IZind a^i d^{r - i} \twoheadrightarrow F_{2i, \> 2i + 1}$ factors as
        \[
            \IZind a^i d^{r - i} \twoheadrightarrow \frac{\IZind a^i d^{r - i}}{1} \twoheadrightarrow F_{2i,\> 2i+1}.
        \]
        Therefore to prove the theorem, we must show that $\lambda_i^{-1} \equiv 0 \mod \pi$. This is immediate since the valuation of $\cL$ is strictly less than $i - r/2$.

        Now assume that $\nu = i - r/2$. Therefore $x = -1$.
        Applying $\beta$
        to \eqref{Main equation in leq}, we get
        \[
            0 \equiv \sum_{a = 1}^{p - 1}p^{r/2 - i + 1}a^{-1}z^{i - 1}\mathbbm{1}_{a^{-1}p + p^2\Zp} - (r - i + 1)p^{r/2 - i}\cL z^i \mathbbm{1}_{p\Zp}.
        \]
        Expanding $z^{i - 1} = (z - a^{-1}p + a^{-1}p)^{i - 1}$ and replacing $a^{-1}$ by $a$, we get
        \[
            0 \equiv \sum_{a = 1}^{p - 1}\sum_{j = 0}^{i - 1} {i - 1 \choose j}p^{r/2 - i}(ap)^{i - j}(z - ap)^{j} \mathbbm{1}_{ap + p^2\Zp} - (r - i + 1)p^{r/2 - i}\cL z^i\mathbbm{1}_{p\Zp}.
        \]
        Since $i - j > 0$, we may also include the $a = 0$ term in the first sum. So write
        \begin{eqnarray}\label{Main congruence in leq in the equality case}
            0 \equiv \sum_{a = 0}^{p - 1}\sum_{j = 0}^{i - 1} {i - 1 \choose j}p^{r/2 - i}(ap)^{i - j}(z - ap)^{j} \mathbbm{1}_{ap + p^2\Zp} - (r - i + 1)p^{r/2 - i}\cL z^i\mathbbm{1}_{p\Zp}.
        \end{eqnarray}

        Next we note that $\nu = i - r/2$ implies that $\nu > 1 - r/2, \> 2 - r/2, \> \ldots, \> (i - 1) - r/2$. Therefore using Proposition \ref{inductive steps} and the remark in the first paragraph of the proof, we get
        \[
            z^{r - n}\mathbbm{1}_{p\Zp} \equiv \sum_{b = 1}^{p - 1}\sum_{l = 1}^{i}c(n, l)b^{r - l - n}(z - b)^{l} \mathbbm{1}_{b + p\Zp}.
        \]
        Applying $\sum\limits_{a = 0}^{p - 1}a^{i - r + n}\begin{pmatrix}1 & 0 \\ -pa & p\end{pmatrix}$, we get
        \begin{eqnarray*}
            && \sum_{a = 0}^{p - 1}p^{n - r/2}a^{i - r + n}(z - ap)^{r - n}\mathbbm{1}_{ap + p^2 \Zp} \\
            && \equiv \sum_{a = 0}^{p - 1}\sum_{b = 1}^{p - 1}\sum_{l = 1}^{i} p^{r/2 - l}c(n, l)a^{i - r + n}b^{r - l - n}(z - (a + b)p)^l\mathbbm{1}_{(a + b)p + p^2\Zp}.
        \end{eqnarray*}
        Making the substitution $a + b = \lambda$ and using 
        \[
            \sum_{b = 1}^{p - 1}(\lambda - b)^{i - r + n}b^{r - l - n} \equiv 
            \begin{cases}
                (-1)^{l - r + n + 1}{i - r + n \choose l - r + n}\lambda^{i - l} & \text{ if } l \geq r - n, \\
                0 & \text{ if } l < r - n 
            \end{cases} \mod p
        \]
        and Lemma~\ref{Integers in the lattice}, we get
        \begin{eqnarray*}
            && \sum_{a = 0}^{p - 1}p^{n - r/2}a^{i - r + n}(z - ap)^{r - n}\mathbbm{1}_{ap + p^2\Zp} \\
            && \equiv \sum_{\lambda = 0}^{p - 1}\sum_{l = r - n}^{i}p^{r/2 - l}(-1)^{l - r + n + 1}\lambda^{i - l}{i - r + n \choose l - r + n}c(n, l)(z - \lambda p)^l\mathbbm{1}_{\lambda p + p^2\Zp}.
        \end{eqnarray*}
        Renaming $\lambda$ as $a$ and plugging the value of $c(n, l)$ obtained in Proposition \ref{inductive steps}, we get one equation
        \begin{eqnarray}\label{Refined inductive steps in leq}
            0 \equiv \sum_{a = 0}^{p - 1}\sum_{l = r - n}^{i}{i - r + n \choose l - r + n}\frac{(r - n)! n}{(r - l) \cdots (n - l)}p^{r/2 - l}a^{i - l}(z - ap)^{l}\mathbbm{1}_{ap + p^2\Zp},
        \end{eqnarray}
        for each $n = r, \> r - 1, \> \ldots, \> r - i + 2$.
        
        Finally note that $p^{r/2 - i}z^i\mathbbm{1}_{p\Zp} \equiv 0$ since $r/2 - i > 0$. Therefore expanding $z^i = (z - ap + ap)^i$, we get
        \begin{equation}\label{Trivial equation in leq}
            0 \equiv\sum_{j = 0}^{i}p^{r/2 - i}{i \choose j}(ap)^{i - j}(z - ap)^j\mathbbm{1}_{ap + p^2\Zp}.
        \end{equation}
        Using equations \eqref{Main congruence in leq in the equality case}, \eqref{Refined inductive steps in leq} and \eqref{Trivial equation in leq}, we wish to obtain terms in $(z - ap)$ involving only the $i^{\mathrm{th}}$ power. To this end, consider the matrix equation
        \begin{eqnarray}\label{Matrix equation in leq}
            \begin{pmatrix}
                {i \choose 0}\frac{r}{r} & 0 & 0 & \cdots & {i \choose 0} & {i - 1 \choose 0} \\
                {i \choose 1}\frac{r}{r - 1} & {i - 1 \choose 0}\frac{r - 1}{(r - 1)(r - 2)} & 0 & \cdots & {i \choose 1} & {i - 1 \choose 1} \\
                {i \choose 2}\frac{r}{r - 2} & {i - 1 \choose 1}\frac{r - 1}{(r - 2)(r - 3)} & {i - 2 \choose 0}\frac{2!(r - 2)}{(r - 2)(r - 3)(r - 4)} & \cdots & {i \choose 2} & {i - 1 \choose 2} \\
                \vdots & & & & \vdots & \vdots \\
                {i \choose i}\frac{r}{r - i} & {i - 1 \choose i - 1}\frac{r - 1}{(r - i)(r - i - 1)} & {i - 2 \choose i - 2}\frac{2!(r - 2)}{(r - i)(r - i - 1)(r - i - 2)} & \cdots & {i \choose i} & 0
            \end{pmatrix}
            \begin{pmatrix}x_1 \\ x_2 \\ x_3 \\ \vdots \\ x_i \\ x_{i + 1}\end{pmatrix} = \begin{pmatrix}0 \\ 0 \\ 0 \\ \vdots \\ 0 \\ 1 \end{pmatrix}.
        \end{eqnarray}
        Once we have the solution to this matrix equation, we multiply the $n = r, \> r - 1, \> r - 2, \> \ldots, \> r - i + 2$ instances of equation \eqref{Refined inductive steps in leq} by $x_1, \> x_2, \> x_3, \> \ldots, \> x_{i - 1}$, equation \eqref{Trivial equation in leq} by $x_i$, equation \eqref{Main congruence in leq in the equality case} by $x_{i + 1}$ and add them all to get
        \begin{eqnarray}\label{Crude final congruence in leq}
            0 \equiv \sum_{a = 0}^{p - 1}p^{r/2 - i}(z - ap)^{i}\mathbbm{1}_{ap + p^2\Zp} - x_{i + 1}(r - i + 1)p^{r/2 - i}\cL z^i\mathbbm{1}_{p\Zp}.
        \end{eqnarray}

        We first find the value of $x_{i + 1}$. We again use Cramer's rule. Let $A$ be the $(i + 1)\times (i + 1)$ matrix on the left side of equation \eqref{Matrix equation in leq}. Let $A_{i + 1}$ be the same matrix, but with the last column replaced by the column vector on the right side of equation \eqref{Main equation in leq}. Then
        \[
            x_{i + 1} = \frac{\det A_{i + 1}}{\det A}.
        \]

        First consider the matrix $A$. Recall that we number the rows and columns starting with $0$. Perform the operation $C_{i - 1} \to C_{i - 1} - C_0$ to make the $(0, i - 1)^{\mathrm{th}}$ entry $0$. Then for $j \geq 1$, the $(j, i - 1)^{\mathrm{th}}$ entry in the resulting matrix is
        \[
            {i \choose j} - {i \choose j}\frac{r}{r - j} = -{i - 1 \choose j - 1}\frac{i}{r - j}.
        \]
        Next, perform the operation $C_{i - 1} \to C_{i - 1} + (r - 2)\frac{i}{r - 1}C_1$ to make the $(1, i - 1)^{\mathrm{th}}$ entry $0$. Then for $j \geq 2$, the $(j, i - 1)^{\mathrm{th}}$ entry in the resulting matrix is
        \[
            -{i - 1 \choose j - 1}\frac{i}{r - j} + (r - 2)\frac{i}{r - 1}{i - 1 \choose j - 1}\frac{r - 1}{(r - j)(r - j - 1)} = {i - 2 \choose j - 2}\frac{i(i - 1)}{(r - j)(r - j - 1)}.
        \]
        Continuing this process, after making $(i - 1)$ entries in the second-to-last column $0$, the $(j, i - 1)^{\mathrm{th}}$ entry for $j \geq i - 1$ is
        \[
            (-1)^{i - 1}{1 \choose j - i + 1}\frac{i(i - 1)\cdots 2}{(r - j)(r - j - 1) \cdots (r - j - i + 2)}.
        \]

        Now we reduce the last column. Perform the operation $C_i \to C_i - C_0$ to make the $(0, i)^{\mathrm{th}}$ entry $0$. Then for $j \geq 1$, the $(j, i)^{\mathrm{th}}$ entry in the resulting matrix is
        \[
            {i - 1 \choose j} - {i \choose j}\frac{r}{r - j} = -{i - 1 \choose j - 1}\frac{r + i - j}{r - j}.
        \]
        Next, perform the operation $C_i \to C_i + \frac{(r - 2)(r + i - 1)}{r - 1}C_1$ to make the $(1, i)^{\mathrm{th}}$ entry $0$. Then for $j \geq 2$, the $(j, i)^{\mathrm{th}}$ entry in the resulting matrix is
        \begin{eqnarray*}
            && -{i - 1 \choose j - 1}\frac{r + i - j}{r - j} + \frac{(r - 2)(r + i - 1)}{r - 1}{i - 1 \choose j - 1}\frac{r - 1}{(r - j)(r - j - 1)} \\
            && = {i - 2 \choose j - 2}\frac{(i - 1)(2(r - 1) + i - j)}{(r - j)(r - j - 1)}.
        \end{eqnarray*}
        Continuing this process, after making $i - 1$ entries in the last column $0$, the $(j, i)^{\mathrm{th}}$ entry for $j \geq i - 1$ in the resulting matrix is
        \[
            (-1)^{i}{1 \choose j - i + 1}\frac{(i - 1)![(i - 1)(r - i + 2) + (i - j)]}{(r - j)(r - j - 1)\cdots (r - j - i + 2)}.
        \]

        Now, the reduction for the second-to-last column in the matrix $A_{i + 1}$ is the same as that in the matrix $A$. Moreover, the last column in the matrix $A_{i + 1}$ has $i$ zeros and $1$ in the last entry. Therefore the ratio
        \begin{eqnarray*}
            \frac{\det A}{\det A_{i + 1}} & = & (-1)^{i}\left[\frac{(i - 1)![(i - 1)(r - i + 2)]}{(r - i)\cdots (r - 2i + 2)} - \frac{(i - 1)![(i - 1)(r - i + 2) + 1]}{(r - i)\cdots (r - 2i + 2)}\right] \\
            & = & (-1)^{i}\frac{(i - 1)!}{(r - i)\cdots (r - 2i + 2)}. 
        \end{eqnarray*}
        Therefore
        \[
            x_{i + 1} = (-1)^{i}\frac{(r - i)\cdots (r - 2i + 2)}{(i - 1)!}.
        \]
        Substituting this in equation \eqref{Crude final congruence in leq}, we get
        \[
            0 \equiv \sum_{a = 0}^{p - 1}p^{r/2 - i}(z - ap)^i\mathbbm{1}_{ap + p^2\Zp} - (-1)^{i}i {r - i + 1 \choose i}p^{r/2 - i}\cL z^i\mathbbm{1}_{p\Zp}.
        \]

        This equation is the image of $(T_{1, 2} - \lambda_i)\llbracket \id, X^iY^{r - i}\rrbracket$ under $\IZind a^i d^{r - i} \twoheadrightarrow F_{2i, \> 2i + 1}$. Indeed, by
        \eqref{Formulae for T-10 and T12 in the commutative case},
        \[
            T_{1, 2}\llbracket \id, X^iY^{r - i}\rrbracket = \sum_{\lambda \in I_1}\left\llbracket \begin{pmatrix}1 & 0 \\ -p\lambda & p\end{pmatrix}, X^iY^{r - i}\right\rrbracket
            \mapsto \sum_{\lambda \in I_1}p^{r/2 - i}(z - \lambda p)^i\mathbbm{1}_{\lambda p + p^2 \Zp}.
        \]
        Therefore we have proved that the surjection $\IZind a^{i}d^{r - i} \twoheadrightarrow F_{2i, \> 2i + 1}$ factors as
        \[
            \IZind a^{i}d^{r - i} \twoheadrightarrow \frac{\IZind a^{i}d^{r - i}}{\im(\lambda_i^{-1}T_{1, 2} - 1)} \twoheadrightarrow F_{2i, \> 2i + 1}.
        \]

        Consider the following exact sequences
            \[
                \begin{tikzcd}
                    0 \arrow[r] & \im T_{1, 2} \arrow[r] \arrow[d, two heads] & \IZind a^{i}d^{r - i} \arrow[r]\arrow[d, two heads] & \dfrac{\IZind a^{i}d^{r - i}}{\im T_{1, 2}} \arrow[r]\arrow[d, two heads] & 0 \\
                    0 \arrow[r] & S \arrow[r] & F_{2i, \> 2i + 1} \arrow[r] & Q \arrow[r] & 0,
                \end{tikzcd}
            \]
            where $S$ is the image of $\im T_{1, 2}$ under the surjection $\IZind a^id^{r - i} \twoheadrightarrow F_{2i, \> 2i + 1}$ and $Q = F_{2i, \> 2i + 1}/S$. Since $\im (\lambda_i^{-1}T_{1, 2} - 1)$ maps to $0$ under the middle vertical map in the diagram above, we see that the right vertical map is $0$. Therefore we have a surjection $\im T_{1, 2} \twoheadrightarrow F_{2i, \> 2i + 1}$. As before,
            by the first isomorphism theorem and \cite[Proposition 3.1]{AB15}, we get a surjection
            \[
                \frac{\IZind a^{i}d^{r - i}}{\im T_{-1, 0}} \twoheadrightarrow F_{2i, \> 2i + 1},
            \]
            which factors through
            \[
              \frac{\IZind a^{i}d^{r - i}}{\im T_{-1, 0} + \im (\lambda_i^{-1}T_{1, 2} - 1)}
              \twoheadrightarrow F_{2i, \> 2i + 1}.
            \]
            The space on the left is  
            \[
                \frac{\IZind a^{i}d^{r - i}}{\im T_{-1, 0} + \im (T_{1, 2} - \lambda_i)} = \pi(r - 2i, \lambda_i, \omega^{i}).
            \]
            Therefore we have proved the theorem.
    \end{proof}

\section{Analysis of $\br{\latticeL{k}}$ around the last point for odd weights}\label{SpecOdd}

    In the previous section, we gave a uniform description of $\br{\latticeL{k}}$ around all points appearing in Theorem \ref{Main theorem in the second part of my thesis} but the last point, namely, $\nu = 0$ for even $k$ and $\nu = 0.5$ for odd $k$. These two points demand separate treatments.
    
    In this section, we assume that $k$ and $r$ are odd, and describe the behavior of $\br{\latticeL{k}}$ around the point $\nu = 0.5$. The self-dual case of the mod $p$ local Langlands correspondence arises, making the analysis trickier.
\subsection{The $\nu \geq 0.5$ case for odd weights}\label{geq section for odd weights}

In this section, we prove that if $\nu \geq 0.5$, then the map $\IZind a^{\frac{r-1}{2}}d^{\frac{r+1}{2}} \twoheadrightarrow F_{r-1, \> r}$ factors as
\[
    \IZind a^{\frac{r-1}{2}}d^{\frac{r+1}{2}} \twoheadrightarrow \frac{\IZind a^{\frac{r-1}{2}}d^{\frac{r+1}{2}}}{\im(T_{-1, 0}^2 - c T_{-1, 0} + 1)} \twoheadrightarrow F_{r-1, \> r},
\]
where
\[
    c = (-1)^{\frac{r+1}{2}}\dfrac{r+1}{2}\left(\dfrac{\cL - H_{-} - H_{+}}{p^{0.5}}\right).
\]
Moreover, we claim that for $i = \dfrac{r+2}{2}$ there is a surjection
\[
    \pi(p - 2, \lambda_i, \omega^{\frac{r + 1}{2}}) \oplus \pi(p - 2, \lambda_i^{-1}, \omega^{\frac{r + 1}{2}}) \twoheadrightarrow F_{r - 1, \> r},
\]
where $\lambda_i + \lambda_i^{-1} = c.$

To prove this, we need inductive steps similar to equation \eqref{Refined inductive steps in leq}, but with an extra $\sqrt{p}$ in the denominator. This is provided by Proposition \ref{Refined hard inductive steps for self-dual}. Let us start by proving the following lemmas.  We remark that part 3.
of the next two lemmas will only be used in Section~\ref{leq section for odd weights}.

    \begin{lemma}\label{strong truncated Taylor expansion for self-dual}
    Let $\dfrac{r + 1}{2} \leq n \leq r \leq p - 2$ and $z_0 \in \Zp$. Fix $x \in \bQ$. Define
    \[
        g(z) = p^x(z - z_0)^n\logL(z - z_0)\mathbbm{1}_{\Zp}.
    \]
    For $0 \leq a \leq p^{h - 1} - 1$, $0 \leq \alpha \leq p - 1$ and $0 \leq j \leq n - 1$, if
    \begin{enumerate}
        \item $\nu > 0, \> \dfrac{r + 3}{2} \leq n \leq r$, $x \geq - 1.5$ and $h \geq 3$, or
        \item $ \nu > 0, \> n = \dfrac{r + 1}{2}$, $x \geq -1$ and $h \geq 4$, or
        \item $-0.5 < \nu < 0.5, \> n = \dfrac{r + 1}{2}$, $x + \nu = -0.5$ and $h \geq 3$,
    \end{enumerate}
    then we have
    \[
        g^{(j)}(a + \alpha p^{h - 1}) \equiv g^{(j)}(a) + \alpha p^{h - 1}g^{(j + 1)}(a) + \cdots + \frac{(\alpha p^{h - 1})^{n - 1 - j}}{(n - 1 - j)!}g^{(n - 1)}(a) \mod (p^{h - 1})^{r/2 - j}\pi.
    \]
    Moreover, we have
    \begin{eqnarray*}
        && g^{(j)}(a + \alpha p^{h - 1})(z - a - \alpha p^{h - 1})^j\mathbbm{1}_{a + \alpha p^{h - 1} + p^h\Zp} \\
        && \equiv \left[g^{(j)}(a) + \alpha p^{h - 1}g^{(j + 1)}(a) + \cdots + \frac{(\alpha p^{h - 1})^{n - 1 - j}}{(n - 1 - j)!}g^{(n - 1)}(a)\right](z - a - \alpha p^{h - 1})^{j}\mathbbm{1}_{a + \alpha p^{h - 1} + p^h \Zp}
    \end{eqnarray*}
    modulo $\pi \latticeL{k}$.
    \end{lemma}
    \begin{proof}
        The proof of this lemma is analogous to the proof of Lemma \ref{Truncated Taylor expansion}. Consider two cases.
        \begin{enumerate}
            \item $v_p(a - z_0) < h - 1$: Recall equation \eqref{Far away from z0 in Truncated Taylor expansion} from Lemma \ref{Truncated Taylor expansion}:
            \begin{eqnarray}\label{Far away from z0 in Truncated Taylor expansion in self-dual}
                g^{(j)}(a + \alpha p^{h - 1}) & = & g^{(j)}(a) + \cdots + \frac{(\alpha p^{h - 1})^{n - 1 - j}}{(n - 1 - j)!}g^{(n - 1)}(a) \nonumber\\
                & & + \frac{(\alpha p^{h - 1})^{n - j}}{(n - j)!}n!p^{x}[\logL(a - z_0) + H_n] \nonumber\\
                & & + \sum_{l \geq 1}\frac{(\alpha p^{h - 1})^{n - j + l}}{(n - j + l)!}n! p^{x} \frac{(-1)^{l - 1}}{(a - z_0)^l}(l - 1)!.
            \end{eqnarray}
            Using $\nu > 0$, we see that the valuation of the term in the second line
            on the right side of the equation above is greater than or equal to $(n - j)(h - 1) + x$ and using $-0.5 < \nu < 0.5$ and $x + \nu = -0.5$, it is strictly greater than $(n - j)(h - 1) - 1$. The first estimate is strictly greater than $(h - 1)(r/2 - j)$ in the first two cases stated in this lemma. The second estimate is greater than or equal to $(h - 1)(r/2 - j)$ in the third case.

            Next, the valuation of a term in the third line
            on the right side of equation \eqref{Far away from z0 in Truncated Taylor expansion in self-dual} is greater than or equal to $x + (n - j)(h - 1) + l - \log_p(n - j + l)$. This estimate is again strictly greater than $(h - 1)(r/2 - j)$ in all the cases stated in the lemma.

            \item $v_p(a - z_0) \geq h - 1$: Using the derivative formula \eqref{Derivative formula for g(z)} for $g(z)$, we see that the valuations of $g^{(j)}(a + \alpha p^{h - 1}), \> g^{(j)}(a), \>\ldots, (\alpha p^{h - 1})^{n - 1 - j}g^{(n - 1)}(a)$ are greater than or equal to $(n - j)(h - 1) + x$ in the first two cases and strictly greater than $(n - j)(h - 1) - 1$ in the third case. The first estimate is strictly greater than $(h - 1)(r/2 - j)$ in the first two cases stated in the lemma as seen above. The second estimate is greater than or equal to $(h - 1)(r/2 - j)$ in the third case again as seen above.
        \end{enumerate}
        The final claim made in this lemma is a consequence of the first claim made in this lemma and Lemma \ref{Integers in the lattice}.
    \end{proof}

    \begin{lemma}\label{Telescoping Lemma in self-dual}
        Let $\dfrac{r + 1}{2} \leq n \leq r \leq p - 2$ and $z_0 \in \Zp$. Fix $x \in \bQ$. Then for
        \[
            g(z) = p^x(z - z_0)^n\logL(z - z_0)\mathbbm{1}_{\Zp},
        \]
        we have
        \begin{enumerate}
            \item $g(z) \equiv g_2(z) \mod \pi \latticeL{k}$ if $\nu > 0, \> \dfrac{r + 3}{2} \leq n \leq r$ and $x \geq -1.5$,
            \item $g(z) \equiv g_3(z) \mod \pi \latticeL{k}$ if $\nu > 0, \> n = \dfrac{r + 1}{2}$ and $x \geq -1$, or
            \item $g(z) \equiv g_2(z) \mod \pi \latticeL{k}$ if $-0.5 < \nu < 0.5$, $n = \dfrac{r + 1}{2}$ and $x + \nu = -0.5$.
        \end{enumerate}
    \end{lemma}
    \begin{proof}
        The proof of this lemma is exactly the same as that of Lemma \ref{telescoping lemma} but with Lemma \ref{Truncated Taylor expansion} replaced by Lemma \ref{strong truncated Taylor expansion for self-dual}.
      \end{proof}

The following lemma is used in the proof of the next lemma.
      
\begin{lemma}\label{Strong truncated Taylor expansion for qp-zp for self-dual}
        Let $\dfrac{r + 1}{2} \leq n \leq r \leq p - 2$ and $z_0 \in \Zp$. Fix $x \in \bQ$. Define
        \[
            f(z) = p^xz^{r - n}(1 - zz_0)^n\logL(1 - zz_0)\mathbbm{1}_{p\Zp}.
        \]
        For $0 \leq a \leq p^{h - 1} - 1$, $0 \leq \alpha \leq p - 1$ and $0 \leq j \leq n - 1$, if
        \begin{enumerate}
            \item $\dfrac{r + 3}{2} \leq n \leq r$, $x \geq -1.5$, and $h \geq 3$, or
            \item $n = \dfrac{r + 1}{2}$, $x \geq - 1$ and $h \geq 4$,
        \end{enumerate}
        then we have
        \[
            f^{(j)}(a + \alpha p^{h - 1}) \equiv f^{(j)}(a) + \alpha p^{h - 1}f^{(j + 1)}(a) + \cdots + \frac{(\alpha p^{h - 1})}{(n - 1 - j)!}f^{(n - 1)}(a) \mod (p^{h - 1})^{r/2 - j}\pi.
        \]
        Moreover, we have
        \begin{eqnarray*}
            && f^{(j)}(a + \alpha p^{h - 1})(z - a - \alpha p^{h - 1})^j\mathbbm{1}_{a + \alpha p^{h - 1}+ p^h\Zp} \\
            && \equiv \left[f^{(j)}(a) + \alpha p^{h - 1}f^{(j + 1)}(a) + \cdots + \frac{(\alpha p^{h - 1})}{(n - 1 - j)!}f^{(n - 1)}(a)\right](z - a - \alpha p^{h - 1})^j\mathbbm{1}_{a + \alpha p^{h - 1} + p^h \Zp}
        \end{eqnarray*}
        modulo $\pi \latticeL{k}$.
    \end{lemma}
    \begin{proof}
        Since $f(z)$ has $\mathbbm{1}_{p\Zp}$ as a factor, this lemma is vacuously true if $p \nmid a$. So assume that $p \mid a$. The $j^{\mathrm{th}}$ derivative of $f$ is given in \eqref{Derivative formula for the qp-zp part}. We expand $f^{(j)}(a + \alpha p^{h - 1})$ as a Taylor series:
        \begin{eqnarray*}
            f^{(j)}(a + \alpha p^{h - 1}) & = & f^{(j)}(a) + \cdots + \frac{(\alpha p^{h - 1})^{n - 1 - j}}{(n - 1 - j)}f^{(n - 1)}(a) \\
            && + \sum_{l \geq 1}\frac{(\alpha p^{h - 1})^{n - 1 - j + l}}{(n - 1 - j + l)!}f^{(n - 1 + l)}(a).
        \end{eqnarray*}
        Using the derivative formula \eqref{Derivative formula for the qp-zp part}, we see that the valuation of the $l^{\mathrm{th}}$ summand in the second line above is greater than or equal to $x + (h - 1)(n - 1 - j + l) - v_p((n - 1 - j + l)!)$.
        \begin{enumerate}
            \item Note that $-1.5 + (h - 1)(n - 1 - j + l) - v_p((n - 1 - j + l)!) > (h - 1)(r/2 - j)$ is true for any $h \geq 3$ if and only if it is true for $h = 3$. We check this inequality for $h = 3$. Indeed, the inequality is clearly true using the estimate $v_p(t!) \leq \frac{t}{p - 1}.$
            \item Note that $-1 + (h - 1)(n - 1 - j + l) - v_p((n - 1 - j + l)!)$ is not greater than $(h - 1)(r/2 - j)$ for $h = 3$ and $l = 1$. Therefore we have to take $h \geq 4$. We then check the inequality $-1 + (h - 1)(n - 1 - j + l) - v_p((n - 1 - j + l)!) > (h - 1)(r/2 - j)$ for $h = 4$. Indeed, this is clear for $l = 1$ and for $l \geq 2$, we see this again using the estimate $v_p(t!) \leq \frac{t}{p - 1}$.
        \end{enumerate}

        The second claim in this lemma follows from the first claim by using Lemma \ref{Integers in the lattice}.
    \end{proof}

    The following lemma is used to show that $g(z) \mathbbm{1}_{\qp \setminus \zp} \equiv 0 \mod \pi \latticeL{k}$ for the functions $g(z)$ appearing (twice in) Proposition~\ref{Refined hard inductive steps for self-dual}, Proposition~\ref{Final theorem for the hardest case in self-dual}, Proposition~\ref{nu in between -0.5 and 0.5}.
    
        \begin{lemma}\label{Telescoping Lemma in self-dual for qp-zp}
            Let $\dfrac{r + 1}{2} \leq n \leq r \leq p - 2$. For $i$ in a finite indexing set $I$, let $\lambda_i \in E$ and $z_i \in \Zp$ be such that $\sum_{i \in I}\lambda_i z_i^j = 0$ for all $0 \leq j \leq n$. If $n = \dfrac{r + 1}{2}$, then we also assume that $\sum_{i \in I}\lambda_i z_i^{n + 1} = 0$. Fix an $x \in \bQ$. Then for
            \[
                f(z) = \sum_{i \in I}p^x \lambda_i z^{r - n}(1 - zz_i)^n\logL(1 - zz_i)\mathbbm{1}_{p\Zp},
            \]
            we have
            \begin{enumerate}
                \item $f(z) \equiv 0 \mod \pi \latticeL{k}$ if $\dfrac{r + 3}{2} \leq n \leq r$ and $x \geq -1.5$.
                \item $f(z) \equiv 0 \mod \pi \latticeL{k}$ if $n = \dfrac{r + 1}{2}$ and $x \geq -1$.
            \end{enumerate}
        \end{lemma}
        \begin{proof}
        We have
            \begin{enumerate}
                \item As in the proof of Lemma \ref{telescoping lemma} with Lemma \ref{Truncated Taylor expansion} replaced by Lemma \ref{Strong truncated Taylor expansion for qp-zp for self-dual}, we see that $f(z) \equiv f_2(z) \mod \pi \latticeL{k}$. We now prove that $f_2(z) \equiv 0 \mod \pi \latticeL{k}$.
                This case does not occur if $r = 1$ thanks to the inequality $\dfrac{r + 3}{2} \leq r$. So we assume that $r \geq 3$ (since $r$ is odd). Recall that
                \[
                    f_2(z) = \sum_{a = 0}^{p^2 - 1}\sum_{j = 0}^{n - 1}\frac{f^{(j)}(a)}{j!}(z - a)^j\mathbbm{1}_{a + p^2\Zp}.
                \]
                Working as in the proof of Lemma \ref{qp-zp part is 0}, we see that
                \[
                    v_p\left(\frac{f^{(j)}(\alpha p)}{j!}\right) \geq \min_{l \geq n - m + 1}\{-1.5 + r - n - j + m + l - \lfloor \log_p(l)\rfloor\}.
                \]
                A short computation shows that $-1.5 + r - n - j + m + l - \lfloor \log_p(l) \rfloor > r/2 - j$ for $l \geq n - m + 1$. Therefore Lemma \ref{Integers in the lattice} implies that $f_2(z) \equiv 0 \mod \pi \latticeL{k}$.
                \item In this case, $f(z) \equiv f_3(z) \mod \pi \latticeL{k}$ again using Lemma \ref{Strong truncated Taylor expansion for qp-zp for self-dual}. We show that $f_3(z) \equiv 0 \mod \pi \latticeL{k}$. Recall that
                \[
                    f_3(z) = \sum_{a = 0}^{p^3 - 1}\sum_{j = 0}^{n - 1}\frac{f^{(j)}(a)}{j!}(z - a)^j\mathbbm{1}_{a + p^3\Zp}.
                \]
                We similarly see that
                \[
                    v_p\left(\frac{f^{j}(\alpha p)}{j!}\right) \geq \min_{l \geq n - m + 2}\{-1 + r - n - j + m + l - \lfloor \log_p(l)\rfloor\},
                \]
                A short computation using $p > 3$ shows that $-1 + r - n - j + m + l -\lfloor \log_p(l) \rfloor > 2(r/2 - j)$ for all $l \geq n - m + 2$. So $f_3(z) \equiv 0 \mod \pi \latticeL{k}$ by Lemma \ref{Integers in the lattice}. \qedhere
            \end{enumerate}
        \end{proof}

\begin{lemma}\label{Sum of the i neq a terms in self-dual}
    Let $\dfrac{r + 3}{2} < n \leq r$ and $0 \leq j \leq \dfrac{r + 1}{2}$. For $0 \leq a, \> i \leq p - 1$, let $a_i \in \Zp$ be such that $a - [i] = [a - i] + pa_i$. Set $\lambda_0 = 1 - p$ and $\lambda_i = 1$ when $1 \leq i \leq p - 1$. Then we have
    \[
        \sum_{\substack{i = 0 \\ i \neq a}}\frac{\lambda_i}{p^{1.5}}[a - i]^{n - j - 1}pa_i = \frac{a^{n - j}}{p^{0.5}(n - j)} \mod \pi.
    \]
\end{lemma}
\begin{proof}
  The proof is similar to the proof of Lemma~\ref{Sum of the i neq a terms}, noting that
  $\pi \mid p^{0.5}$.
\end{proof}

We now state and prove the following (deep) inductive steps. The word `deep' is meant to signify that the power of $p$ appearing here is $p^{1/2}$ deeper than the power of $p$ appearing in the earlier inductive steps in Proposition \ref{inductive steps} (more precisely, equation \eqref{Refined inductive steps in leq}). 
\begin{Proposition}\label{Refined hard inductive steps for self-dual}
     Let $\nu > 0$. Then for each $0 \leq a \leq p - 1$ and $\dfrac{r + 3}{2} \leq n \leq r$, we have
    \[
        0 \equiv \sum_{\alpha = 0}^{p - 1}\sum_{l = r-n}^{{\frac{r+1}{2}}}p^{\frac{r-1}{2} - l}{n - \frac{r-1}{2} \choose n - r + l}\frac{(r-n)! n}{(r - l)\cdots (n - l)}\alpha^{\frac{r+1}{2} - l}(z - a - \alpha p)^l \mathbbm{1}_{a + \alpha p + p^2\Zp} \mod \pi \latticeL{k}.
    \]
\end{Proposition}
\begin{proof}
  We remark that $\dfrac{r + 3}{2} \leq r$ implies $r \neq 1$.
  First assume that $\dfrac{r+3}{2} < n \leq r$. Set
    \[
        g(z) = \sum_{i = 0}^{p - 1}\frac{\lambda_i}{p^{1.5}}(z - [i])^n\logL(z - [i]),
    \]
    where $\lambda_0 = 1 - p$ and $\lambda_i = 1$ for $1 \leq i \leq p - 1$. Write $g(z) = g(z)\mathbbm{1}_{\Zp} + g(z)\mathbbm{1}_{\Qp \setminus \Zp}$. Using the identity $\sum_{i = 0}^{p - 1}\lambda_i [i]^j = 0$ for $0 \leq j \leq n \leq p - 2$ proved in Lemma~\ref{Main coefficient identites} (1), we see that
    \[
        w \cdot g(z)\mathbbm{1}_{\Qp \setminus \Zp} = \sum_{i = 0}^{p - 1}\frac{\lambda_i}{p^{1.5}}z^{r - n}(1 - z[i])^n\logL(1 - z[i])\mathbbm{1}_{p\Zp}.
    \]
    A priori, there is an extra term involving $\logL(z)$ on the right side of the equation above
    as in the proof of Proposition~\ref{nu geq}.
    However, due to the identity $\sum_{i = 0}^{p - 1}\lambda_i[i]^{n} = 0$, this term becomes $0$. Applying Lemma~\ref{Telescoping Lemma in self-dual for qp-zp} (1), we see that this function, hence also
    $g(z)\mathbbm{1}_{\qp \setminus\Zp}$,  is $0$ modulo $\pi \latticeL{k}$.
    So $g(z) \equiv g(z)\mathbbm{1}_{\Zp} \mod \pi \latticeL{k}$.
    
    Lemma~\ref{Telescoping Lemma in self-dual} (1) implies that $g(z)\mathbbm{1}_{\Zp} \equiv g_2(z) \mod \pi \latticeL{k}$. Next, we analyze $g_2(z) - g_1(z)$ and $g_1(z)$ separately.

    For $0 \leq j \leq n - 1$, the coefficient of $(z - a - \alpha p)^j\mathbbm{1}_{a + \alpha p + p^2 \Zp}$ in $g_2(z) - g_1(z)$ is
    \begin{eqnarray}\label{jth summand in g2 - g1 in self-dual}
        && \!\!\!\!\!\!\!\! {n \choose j}\sum_{i = 0}^{p - 1} \frac{\lambda_i}{p^{1.5}}\Big[(a + \alpha p - [i])^{n - j}\logL(a + \alpha p - [i]) -
        (a - [i])^{n - j}\logL(a - [i]) \\
        && \!\!\!\!\!\!\!\! - \left. {n - j \choose 1}\alpha p (a - [i])^{n - j - 1}\logL(a - [i]) - \ldots - {n - j \choose n - j - 1}(\alpha p)^{n - j - 1}(a - [i])\logL(a - [i])\right]. \nonumber
    \end{eqnarray}

    First assume that $i \neq a$. Writing $\logL(a + \alpha p - [i]) = \logL(1 + (a - [i])^{-1}\alpha p) + \logL(a - [i])$, we see that the $i^{\mathrm{th}}$ summand in the display above is
    \[
        \frac{\lambda_i}{p^{1.5}}[(a - [i] + \alpha p)^{n - j}\logL(1 + (a - [i])^{-1}\alpha p) + (\alpha p)^{n - j}\logL(a - [i])].
    \]
    Since $0 \leq a \leq p - 1$, the condition $i \neq a$ implies that $i \not \equiv a \!\! \mod p$. Therefore the valuation of $\logL(a - [i])$ is greater than or equal to $1$. Since $n > \dfrac{r+3}{2} > \dfrac{r + 1}{2}$, the second term in the display above is congruent to $0$ modulo $p^{r/2 - j}\pi$. Moreover, expanding the first term, we see that the $i \neq a$ summand in expression \eqref{jth summand in g2 - g1 in self-dual} is
    \begin{eqnarray}\label{i not equal to a in self-dual}
        \frac{\lambda_i}{p^{1.5}}[(a - [i])^{n - j - 1}\alpha p + c_{n - j - 2}(a - [i])^{n - j - 2}(\alpha p)^2 + \cdots + c_0(\alpha p)^{n - j}] \mod p^{r/2 - j}\pi,
    \end{eqnarray}
    where $c_{n - j - 2}, \ldots, c_0 \in \Zp$.

    Next assume that $i = a$. Using $\nu > 0$, we see that each summand in expression \eqref{jth summand in g2 - g1 in self-dual} has valuation greater than or equal to $n - j - 1.5$. Since $n > \dfrac{r + 3}{2}$, we see that the $i = a$ summand in expression \eqref{jth summand in g2 - g1 in self-dual} is $0$ modulo $p^{r/2 - j}\pi$.

    Using \eqref{i not equal to a in self-dual} and the identity $\sum_{i = 0}^{p - 1}\lambda_i(z - [i])^{n - 1} = 0$, we see that expression \eqref{jth summand in g2 - g1 in self-dual} becomes
    \begin{eqnarray}\label{Final expression for the jth term in g2 - g1 in the self-dual case}
        -\frac{\lambda_a}{p^{1.5}}[(a - [a])^{n - j - 1}(\alpha p) + \cdots + c_0(\alpha p)^{n - j}] \mod p^{r/2 - j}\pi.
    \end{eqnarray}
    Each term in \eqref{Final expression for the jth term in g2 - g1 in the self-dual case} has valuation greater than or equal to $n - j - 1.5$. Since $n > \dfrac{r + 3}{2}$, the display above is congruent to $0$ modulo $p^{r/2 - j}\pi$. Therefore
     $  g_2(z) - g_1(z) \equiv 0 \mod \pi \latticeL{k}$.

    Next we analyze
    \begin{eqnarray}\label{g1 in inductive steps in self-dual}
        g_1(z) = \sum_{a = 0}^{p - 1}\left[\sum_{j = 0}^{n - 1}\frac{g^{(j)}(a)}{j!}(z - a)^j\right]\mathbbm{1}_{a + p\Zp}.
    \end{eqnarray}
    Note that
    \begin{eqnarray}\label{jth summand in g1 in inductive steps in self-dual}
        \frac{g^{(j)}(a)}{j!} = {n \choose j}\sum_{i = 0}^{p - 1}\frac{\lambda_i}{p^{1.5}}(a - [i])^{n - j}\logL(a - [i]).
    \end{eqnarray}
    First assume that $j > \dfrac{r + 1}{2}$. For $i \neq a$, the $i^{\mathrm{th}}$ summand on the right side of equation \eqref{jth summand in g1 in inductive steps in self-dual} has valuation greater than or equal to $-0.5$. 
    Using Lemma \ref{stronger bound for polynomials of large degree with varying radii}, we see that the $j > \dfrac{r + 1}{2}$ and $i \neq a$ terms do not contribute to $g_1(z)$. Next, the valuation of the $i = a$ summand in \eqref{jth summand in g1 in inductive steps in self-dual} is greater than or equal to $n - j - 1.5$. Since $n > \dfrac{r + 3}{2}$, Lemma \ref{stronger bound for polynomials of large degree with varying radii} implies that the $j > \dfrac{r + 1}{2}$ and $i = a$ terms do not contribute to $g_1(z)$ as well.
    
    Now assume that $j \leq \dfrac{r + 1}{2}$. The valuation of the $i = a$ summand on the right side of equation \eqref{jth summand in g1 in inductive steps in self-dual} is greater than or equal to $n - j - 1.5$. Since $n > \dfrac{r + 3}{2}$, we see that $n - j - 1.5 > 0$. Therefore using Lemma \ref{Integers in the lattice}, we see that the $i = a$ terms do not contribute to $g_1(z)$. Next, for $i \neq a$ write $a - [i] = [a - i] + pa_i$ for some $a_i \in \Zp$. Then $\logL(a - [i]) = \logL(1 + [a - i]^{-1}pa_i)$ since $[a - i]$ is a root of unity. Therefore the $i \neq a$ summand on the right side of \eqref{jth summand in g1 in inductive steps in self-dual} is
    \begin{eqnarray}\label{i not equal to a in g1 in self-dual}
        \frac{\lambda_i}{p^{1.5}}([a - i] + pa_i)^{n - j}\logL(1 + [a - i]^{-1}pa_i) & = & \frac{\lambda_i}{p^{1.5}}[a - i]^{n - j - 1}pa_i \mod \pi.
    \end{eqnarray}
    Plugging \eqref{i not equal to a in g1 in self-dual} into \eqref{jth summand in g1 in inductive steps in self-dual} and using Lemma \ref{Sum of the i neq a terms in self-dual}, we get
    \[
        \frac{g^{(j)}(a)}{j!} \equiv {n \choose j} \frac{a^{n - j}}{p^{0.5}(n - j)} \mod \pi.
    \]

    Therefore
    \[
        g_1(z) \equiv \sum_{a = 0}^{p - 1}\sum_{j = 0}^{\frac{r + 1}{2}}{n \choose j}\frac{a^{n - j}}{p^{0.5}(n - j)}(z - a)^j \mathbbm{1}_{a + p\Zp} \mod \pi \latticeL{k}.
    \]

    This shows that
    \begin{eqnarray*}\label{hard inductive steps for large n}
        0 \equiv \sum_{a = 0}^{p - 1} \left[ \sum_{j = 0}^{\frac{r + 1}{2}}{n \choose j}\frac{a^{n - j}}{p^{0.5}(n - j)}(z - a)^j \right] \mathbbm{1}_{a + p\Zp} \mod \pi \latticeL{k}.
    \end{eqnarray*}
    This is equation \eqref{hard inductive step} in the proof of the `shallow' inductive steps
    except that the sum over $j$ goes up to
    $\dfrac{r+1}{2}$ and there is an extra denominator $p^{0.5}$. 
    Proceeding exactly as in the proof of in Proposition~\ref{inductive steps}), we get
    \[
        \frac{1}{p^{0.5}}z^{r - n}\mathbbm{1}_{p\Zp} \equiv \sum_{b = 1}^{p - 1}\left[\sum_{l = 1}^{\frac{r + 1}{2}}\left(\sum_{j = 1}^{l} \frac{(-1)^j j}{n - j}{n \choose j}{r - j \choose r - l}\right)\frac{b^{r - l - n}}{p^{0.5}}(z - b)^{l}\right]\mathbbm{1}_{b + p\Zp} \mod \pi \latticeL{k}.
      \]
      We remark that because of the slightly deeper denominator,
      Lemma~\ref{stronger bound for polynomials of large degree with varying radii} can
      only be used in the intermediate step \eqref{intermediate inductive step}
      to drop the $0 \leq l < \dfrac{r - 1}{2}$ terms there.

      The expression inside the parentheses is $c(n, l)$ from Proposition \ref{inductive steps}. Applying the operator $\sum\limits_{\alpha = 0}^{p - 1}\alpha^{n - \frac{r - 1}{2}}\begin{pmatrix}1 & 0 \\ -a - \alpha p & p\end{pmatrix}$, we get
    \begin{eqnarray*}
        && \sum_{\alpha = 0}^{p - 1}\alpha^{n - \frac{r - 1}{2}}p^{n - \frac{r + 1}{2}}(z - a - \alpha p)^{r - n}\mathbbm{1}_{a + \alpha p + p^2\Zp} \\
        && \equiv \sum_{\alpha = 0}^{p - 1}\sum_{b = 1}^{p - 1}\sum_{l = 1}^{\frac{r + 1}{2}}\alpha^{n - \frac{r - 1}{2}} p^{\frac{r - 1}{2} - l}c(n, l)b^{r - l - n}(z - a - (\alpha + b)p)^l\mathbbm{1}_{a + (\alpha + b)p + p^2\Zp} \mod \pi \latticeL{k}.
    \end{eqnarray*}
    Making the substitution $\alpha + b = \lambda$ and summing over $b$ using \eqref{sum over b identity} (but with $i = \lceil {r/2} \rceil$), we get
    \begin{eqnarray*}
        && \!\!\!\!\!\!\!\! \sum_{\lambda = 0}^{p - 1}\alpha^{n - \frac{r - 1}{2}}p^{n - \frac{r + 1}{2}}(z - a - \alpha p)^{r - n}\mathbbm{1}_{a + \alpha p + p^2\Zp} \\
        && \!\!\!\!\!\!\!\! \equiv \sum_{\lambda = 0}^{p - 1}\sum_{l = r - n}^{\frac{r + 1}{2}}p^{\frac{r - 1}{2} - l}\lambda^{\frac{r + 1}{2} - l}c(n, l)(-1)^{n - r + l + 1}{n - \frac{r - 1}{2} \choose n - r + l}(z - a - \lambda p)^l\mathbbm{1}_{a + \lambda p + p^2\Zp} \!\!\!\!\! \mod \pi \latticeL{k}.
    \end{eqnarray*}
    Renaming $\lambda$ as $\alpha$ and substituting the value of $c(n, l)$ obtained in
    Proposition~\ref{inductive steps},
    proves the proposition for $\dfrac{r + 3}{2} < n \leq r$.

    Next, assume that $n = \dfrac{r + 3}{2}$. Set
    \[
        g(z) = \sum_{i = 0}^{p - 1}\frac{\lambda_i}{p^{1.5}}(z - i)^{\frac{r + 3}{2}}\logL(z - i),
    \]
    where the $\lambda_i$ are as in Lemma~\ref{Main coefficient identites} (3), so that
    $\lambda_0 = 1$ and $\sum\limits_{i = 0}^{p - 1}\lambda_i i^j = 0$
    for $0 \leq j \leq p - 2$. Note that this function is slightly different from the one we chose for $\dfrac{r + 3}{2} < n \leq r$. This change is made to eliminate the terms analogous to the $(a-[a])$ terms in equation \eqref{Final expression for the jth term in g2 - g1 in the self-dual case}. We have to make this change since the bound $n - j - 1.5 > r/2 - j$ is no longer true for $n = \dfrac{r + 3}{2}$. 
    
    We write $g(z) = g(z)\mathbbm{1}_{\Zp} + g(z)\mathbbm{1}_{\Qp \setminus \Zp}$. Note that
    \[
        w \cdot g(z)\mathbbm{1}_{\Qp \setminus \Zp}  = \sum_{i = 0}^{p - 1}\frac{\lambda_i}{p^{1.5}}z^{\frac{r - 3}{2}}(1 - zi)^{\frac{r + 3}{2}}\logL(1 - zi)\mathbbm{1}_{p\Zp}.
    \]
    Applying Lemma~\ref{Telescoping Lemma in self-dual for qp-zp} (1), we see
    that $g(z)\mathbbm{1}_{\Qp \setminus \Zp} \equiv 0 \mod \pi \latticeL{k}$. Therefore as before, $g(z) \equiv g(z)\mathbbm{1}_{\Zp} \mod \pi \latticeL{k}$. Now Lemma~\ref{Telescoping Lemma in self-dual} (1) implies that $g(z)\mathbbm{1}_{\Zp} \equiv g_2(z) \mod \pi \latticeL{k}$. We analyze $g_2(z) - g_1(z)$ and $g_1(z)$ separately.

    The coefficient of $(z - a - \alpha p)^j\mathbbm{1}_{a + \alpha p + p^2\Zp}$ in $g_2(z) - g_1(z)$ is
    \begin{eqnarray}\label{jth summand in g2 - g1 in self-dual for n = r+3/2}
        && \!\!\!\!\!\!\!\!\!\!\!\! {\frac{r + 3}{2} \choose j}\sum_{i = 0}^{p - 1} \frac{\lambda_i}{p^{1.5}}\Big[(a + \alpha p - i)^{\frac{r + 3}{2} - j}\logL(a + \alpha p - i) -
        (a - i)^{\frac{r + 3}{2} - j}\logL(a - i) \\
        && \!\!\!\!\!\!\!\!\!\! - \!\!\left. {\frac{r + 3}{2} - j \choose 1}\alpha p (a - i)^{\frac{r + 3}{2} - j - 1}\!\logL(a - i) \! - \! \ldots - \! {\frac{r + 3}{2} - j \choose \frac{r + 3}{2} - j - 1}\!(\alpha p)^{\frac{r + 3}{2} - j - 1}(a - i)\logL(a - i)\right]. \nonumber
    \end{eqnarray}
    First assume $i \neq a$. Writing $\logL(a + \alpha p - i) = \logL(1 + (a - i)^{-1}\alpha p) + \logL(a - i)$, we see that the $i^{\mathrm{th}}$ summand in the display above is
    \[
        \frac{\lambda_i}{p^{1.5}}\left[(a - i + \alpha p)^{\frac{r + 3}{2} - j}\logL(1 + (a - i)^{-1}\alpha p) + (\alpha p)^{\frac{r + 3}{2} - j}\logL(a - i)\right].
    \]
    As before $p \mid \logL(a - i)$. So the second summand above is $0$ modulo $p^{r/2 - j}\pi$. Moreover, expanding the first term, we see that the $i \neq a$ summand in expression \eqref{jth summand in g2 - g1 in self-dual for n = r+3/2} is
    \begin{eqnarray}\label{i neq a summand in g2 - g1 in self-dual for n = r+3/2}
        \frac{\lambda_i}{p^{1.5}}[(a - i)^{\frac{r + 1}{2} - j}\alpha p + c_{\frac{r - 1}{2} - j}(a - i)^{\frac{r - 1}{2} - j}(\alpha p)^2 + \cdots + c_0(\alpha p)^{\frac{r + 3}{2} - j}] \mod p^{r/2 - j}\pi,
    \end{eqnarray}
    where $c_{\frac{r - 1}{2}-j}, \ldots, c_0 \in \Zp$ with 
    $c_0 = H_{\frac{r + 3}{2} - j}$, by \eqref{Small derivative formula for polynomial times logs}.
    
    Next, assume that $i = a$. Then the $i^{\mathrm{th}}$ summand in expression \eqref{jth summand in g2 - g1 in self-dual for n = r+3/2} is
    \[
        \frac{\lambda_a}{p^{1.5}}(\alpha p)^{\frac{r + 3}{2} - j}\logL(\alpha p) = \frac{\lambda_a}{p^{1.5}}(\alpha p)^{\frac{r + 3}{2} - j}\cL + \frac{\lambda_a}{p^{1.5}}(\alpha p)^{\frac{r + 3}{2} - j}\logL(\alpha).
    \]
    The second term in the display above is $0$ modulo $p^{r/2 - j}\pi$. This is clear if $\alpha = 0$. If $\alpha \neq 0$, then the valuation of the second term in the right side of the display above is greater than or equal to $(r + 2)/2 - j > r/2 - j$ since $\alpha$ is a $p$-adic unit. So the $i = a$ summand in \eqref{jth summand in g2 - g1 in self-dual for n = r+3/2} is
    \begin{eqnarray}\label{i = a summand in g2 - g1 in self-dual for n = r+3/2}
        \frac{\lambda_a}{p^{1.5}}(\alpha p)^{\frac{r + 3}{2} - j}\cL \mod p^{r/2 - j}\pi.
    \end{eqnarray}
    Putting \eqref{i neq a summand in g2 - g1 in self-dual for n = r+3/2} and \eqref{i = a summand in g2 - g1 in self-dual for n = r+3/2} in \eqref{jth summand in g2 - g1 in self-dual for n = r+3/2} and using the identity $\sum\limits_{i = 0}^{p - 1}\lambda_i(z - i)^{p - 2} = 0$, we see that the coefficient of $(z - a - \alpha p)^j\mathbbm{1}_{a + \alpha p + p^2\Zp}$ in $g_2(z) - g_1(z)$ is
    \[
        {\frac{r + 3}{2} \choose j}\frac{\lambda_a}{p^{1.5}}(\alpha p)^{\frac{r + 3}{2} - j}(\cL - H_{\frac{r + 3}{2} - j}) \mod p^{r/2 - j}\pi.
    \]
    Note $\lambda_a \equiv 1 \!\! \mod p$. Also, the
    valuation of the coefficient of $(z - a - \alpha p)^{\frac{r + 1}{2}}$ is greater than
    or equal to $0.5$, so  we may drop the $j = \dfrac{r + 1}{2}$ term, by
    Lemma \ref{stronger bound for polynomials of large degree with varying radii}. We get
    \begin{eqnarray}\label{g2 - g1 in self-dual for n = r+3/2}
        g_2(z) - g_1(z) & \equiv & \sum_{a = 0}^{p - 1}\sum_{\alpha = 0}^{p - 1}\left[\sum_{j = 0}^{\frac{r - 1}{2}}{\frac{r + 3}{2} \choose j}\frac{1}{p^{1.5}}(\alpha p)^{\frac{r + 3}{2} - j}\left(\cL - H_{\frac{r + 3}{2} - j}\right)(z - a - \alpha p)^j\right]\mathbbm{1}_{a + \alpha p + p^2\Zp} \nonumber \\
        && \qquad \qquad \mod \pi \latticeL{k}.
    \end{eqnarray}

    We claim that $g_2(z) - g_1(z) \equiv 0 \!\!\mod \pi \latticeL{k}$. Since $\nu > 0$, Proposition \ref{inductive steps} gives us
    \[
        z^{r - n}\mathbbm{1}_{p\Zp} \equiv \sum_{b = 1}^{p - 1}\sum_{j = 1}^{\frac{r - 1}{2}}c(n, j)b^{r - j - n}(z - b)^j\mathbbm{1}_{b + p\Zp} \mod \pi \latticeL{k}
    \]
    for every $\dfrac{r + 3}{2} \leq n \leq r$. Applying $\sum\limits_{\alpha = 0}^{p - 1}\alpha^{n - \frac{r - 3}{2}}\begin{pmatrix}1 & 0 \\ -a - \alpha p & p\end{pmatrix}$ to both sides,
    making the substitution $\alpha + b = \lambda$ and summing over $b$ using \eqref{sum over b identity},
    renaming $\lambda$ as $\alpha$ and substituting the value of $c(n, j)$ from Proposition \ref{inductive steps}, 
    as before,  we get 
    \begin{eqnarray}
      \label{Applied inductive steps in the hard inductive steps for g2-g1}
      0 & \equiv & \sum_{\alpha = 0}^{p - 1}\sum_{j = r - n}^{\frac{r - 1}{2}}\frac{1}{p^{1.5}}\frac{(r - n)! n}{(r - j)\cdots (n - j)}{n - \frac{r - 3}{2} \choose n - r + j}(\alpha p)^{\frac{r + 3}{2} - j}(z - a -\alpha p)^j \mathbbm{1}_{a + \alpha p + p^2\Zp} \\
      &  &  \qquad \qquad \mod \pi \latticeL{k} \nonumber
    \end{eqnarray}
    for each $\dfrac{r + 3}{2} \leq n \leq r$.

    We wish to write the $a^{\mathrm{th}}$ summand on the right side of equation \eqref{g2 - g1 in self-dual for n = r+3/2} as the sum of $x_n \in E$ times \eqref{Applied inductive steps in the hard inductive steps for g2-g1}  for  $\dfrac{r + 3}{2} \leq n \leq r$ and $x_{\frac{r+1}{2}} \in E$ times 
    \[
        \sum_{\alpha = 0}^{p - 1}\frac{1}{p^{1.5}}(\alpha p)^{2}(z - a - \alpha p)^{\frac{r - 1}{2}}\mathbbm{1}_{a + \alpha p + p^2\Zp}.
      \]
    To solve for $x_{\frac{r+1}{2}}$, we compare the coefficients of $\sum\limits_{\alpha = 0}^{p - 1} \dfrac{1}{p^{1.5}}(\alpha p)^{\frac{r+3}{2}-j} (z - a - \alpha p)^{j}\mathbbm{1}_{a + \alpha p + p^2\Zp}$
    for $0 \leq j \leq \dfrac{r-1}{2}$ to get the following matrix equation
    \[
        \! \begin{pmatrix}
                {\frac{r + 3}{2} \choose 0}\frac{r}{r} & 0 & 0 & \cdots & 0 \\
                {\frac{r + 3}{2} \choose 1}\frac{r}{r - 1} & {\frac{r + 1}{2} \choose 0}\frac{r - 1}{(r - 1)(r - 2)} & 0 & \cdots & 0 \\
                {\frac{r + 3}{2} \choose 2}\frac{r}{r - 2} & {\frac{r + 1}{2} \choose 1}\frac{r - 1}{(r - 2)(r - 3)} & {\frac{r - 1}{2} \choose 0}\frac{2!(r - 2)}{(r - 2)(r - 3)(r - 4)} & \cdots & 0 \\
                \vdots & \vdots & \vdots & & \vdots \\
                \binom{\frac{r + 3}{2}}{\frac{r - 1}{2}}\frac{r}{\frac{r + 1}{2}} & {\frac{r + 1}{2} \choose \frac{r - 3}{2}}\frac{r - 1}{(\frac{r + 1}{2})(\frac{r - 1}{2})} & {\frac{r - 1}{2} \choose \frac{r - 5}{2}}\frac{2!(r - 2)}{(\frac{r + 1}{2})(\frac{r - 1}{2})(\frac{r - 3}{2})} & \cdots & 1
            \end{pmatrix} \!\! 
            \begin{pmatrix}x_r \\ x_{r - 1} \\ x_{r - 2} \\ \vdots \\ x_{\frac{r+1}{2}} \end{pmatrix} \!\! = \!\! \begin{pmatrix}{\frac{r + 3}{2} \choose 0}\left(\cL - H_{\frac{r + 3}{2}}\right) \\ {\frac{r + 3}{2} \choose 1}\left(\cL - H_{\frac{r + 1}{2}}\right) \\ {\frac{r + 3}{2} \choose 2}\left(\cL - H_{\frac{r - 1}{2}}\right) \\ \vdots \\ {\frac{r + 3}{2} \choose \frac{r - 1}{2}}(\cL - H_2) \end{pmatrix}.
    \]
    It can be checked\footnote{\label{space-time 10}See Appendix~\ref{Footnote 10}.} that
    $x_{\frac{r+1}{2}} = (-1)^{\frac{r - 1}{2}}\dfrac{r + 3}{4}\left(\cL - H_{\frac{r - 1}{2}} - H_{\frac{r + 1}{2}}\right)$ (for $r > 1$). Therefore \eqref{g2 - g1 in self-dual for n = r+3/2} becomes
    \begin{eqnarray*}
        g_2(z) - g_1(z) \!\! & \!\! \equiv \!\! & \!\! \sum_{a = 0}^{p - 1}\sum_{\alpha = 0}^{p - 1}(-1)^{\frac{r - 1}{2}}\frac{r + 3}{4}\left(\cL - H_{\frac{r - 1}{2}} - H_{\frac{r + 1}{2}}\right)\frac{1}{p^{1.5}}(\alpha p)^2(z - a - \alpha p)^{\frac{r - 1}{2}}\mathbbm{1}_{a + \alpha p + p^2\Zp} \\
        && \qquad \qquad \mod \pi \latticeL{k}.
    \end{eqnarray*}
    Since $\nu > 0$, we see that the valuation of the coefficient of $(z - a - \alpha p)^{\frac{r - 1}{2}}\mathbbm{1}_{a + \alpha p + p^2\Zp}$ is strictly greater than $0.5$. Therefore Lemma \ref{Integers in the lattice} implies that
    \[
        g_2(z) - g_1(z) \equiv 0 \mod \pi \latticeL{k}.
    \]

    Next, consider
    \begin{eqnarray}\label{g1 in self-dual inductive step for n = r+3/2}
        g_1(z) = \sum_{a = 0}^{p - 1}\left[\sum_{j = 0}^{\frac{r + 1}{2}}\frac{g^{(j)}(a)}{j!}(z - a)^j\right]\mathbbm{1}_{a + p\Zp}.
    \end{eqnarray}
    Recall
    \[
        \frac{g^{(j)}(a)}{j!} = {\frac{r + 3}{2} \choose j}\sum_{i = 0}^{p - 1}\frac{\lambda_i}{p^{1.5}}(a - i)^{\frac{r + 3}{2} - j}\logL(a - i).
    \]
    The $i = a$ term is $0$. For $i \neq a$, write $a - i = [a - i] + pa_i$ for some $a_i \in \Zp$. Note that $\logL(a - i) = \logL(1 + [a - i]^{-1}pa_i)$ since $[a - i]$ is a root of unity. Expanding $\logL(1 + [a - i]^{-1}pa_i)$, we get
    \begin{eqnarray*}
        \frac{g^{(j)}(a)}{j!} & \equiv & {\frac{r + 3}{2} \choose j}\sum_{\substack{i = 0  \\ i \neq a}}^{p - 1}\frac{\lambda_i}{p^{1.5}}[a - i]^{\frac{r + 1}{2} - j}pa_i \mod \pi \\
        & \equiv & {\frac{r + 3}{2} \choose j}\sum_{\substack{i = 0 \\ i \neq a}}^{p - 1}\frac{\lambda_i}{p^{1.5}}\left(\frac{(a - i)^{\frac{r + 3}{2} - j} - [a - i]^{\frac{r + 3}{2} - j}}{\frac{r + 3}{2} - j}\right) \mod \pi.
    \end{eqnarray*}
    Recall that $\lambda_i = 1 \mod p$. Since the expression in the parentheses is divisible by $p$, we can replace $\lambda_i$ in the last line above by $1$. Moreover, using \eqref{sum of powers of roots of unity}, we get
    \[
        \frac{g^{(j)}(a)}{j!} \equiv {\frac{r + 3}{2} \choose j}\sum_{i = 0}^{p - 1}\frac{(a - i)^{\frac{r + 3}{2} - j}}{p^{1.5}\left(\frac{r + 3}{2} - j\right)} \mod \pi.
    \]
    Expanding $(a - i)^{\frac{r + 3}{2} - j} = \sum\limits_{l = 0}^{\frac{r + 3}{2} - j}{\frac{r + 3}{2} - j \choose l}(-1)^l i^la^{\frac{r + 3}{2} - j - l}$, we get
    \[
        \frac{g^{(j)}(a)}{j!} \equiv {\frac{r + 3}{2} \choose j}\sum_{l = 0}^{\frac{r + 3}{2} - j}\frac{(-1)^{l}}{\frac{r + 3}{2} - j}\frac{S_l}{p}{\frac{r + 3}{2} - j \choose l}\frac{a^{\frac{r + 3}{2} -j - l}}{p^{0.5}} \mod \pi,
    \]
    where $S_l$ is defined in equation \eqref{Definition of S_l}. Putting this equation in \eqref{g1 in self-dual inductive step for n = r+3/2}, we get
    \[
        g_1(z) \equiv \sum_{a = 0}^{p - 1}\sum_{j = 0}^{\frac{r + 1}{2}}\sum_{l = 0}^{\frac{r + 3}{2} - j}\frac{(-1)^l}{\frac{r + 3}{2} - j} \frac{S_l}{p} {\frac{r + 3}{2} \choose j}{\frac{r + 3}{2} - j \choose l}\frac{a^{\frac{r + 3}{2} - j - l}}{p^{0.5}}(z - a)^j\mathbbm{1}_{a + p\Zp} \mod \pi \latticeL{k}.
    \]

    Since we have shown that $g(z) \equiv g_1(z) \mod \pi \latticeL{k}$, we get
    \begin{eqnarray}\label{primitive hard inductive step for n = r+3/2}
        \!\!\!\!\!\!\!\! 0 \equiv \sum_{a = 0}^{p - 1}\sum_{j = 0}^{\frac{r + 1}{2}}\sum_{l = 0}^{\frac{r + 3}{2} - j}\frac{(-1)^l}{\frac{r + 3}{2} - j} \frac{S_l}{p} {\frac{r + 3}{2} \choose j}{\frac{r + 3}{2} - j \choose l}\frac{a^{\frac{r + 3}{2} - j - l}}{p^{0.5}}(z - a)^j\mathbbm{1}_{a + p\Zp} \!\!\!\! \mod \pi \latticeL{k}.
    \end{eqnarray}
    
    Next, for $l = 0, 1, 2, \ldots, \dfrac{r + 3}{2}$, using the binomial expansion we have
    \begin{eqnarray*}
      \frac{1}{p^{0.5}}z^{\frac{r + 3}{2}-l}\mathbbm{1}_{\Zp} & \equiv & \sum_{a = 0}^{p - 1}\sum_{j = 0}^{\frac{r + 3}{2}-l}{\frac{r + 3}{2} -l \choose j} \frac{1}{p^{0.5}}a^{\frac{r + 3}{2}-l- j}(z - a)^j\mathbbm{1}_{a + p\Zp}
   \end{eqnarray*}
Multiplying the $l^{\mathrm{th}}$ equation above by 
     $   \frac{(-1)^{l}}{\frac{r + 3}{2}}\frac{S_l}{p}{\frac{r + 3}{2} \choose l}, $
    adding all the equations, 
    interchanging the sums over $l$ and $j$ and dropping the $j = \dfrac{r + 3}{2}$ term
    using Lemma~\ref{stronger bound for polynomials of large degree with varying radii}, we get
    \begin{eqnarray*}
        \sum_{l = 0}^{\frac{r + 3}{2}}\frac{(-1)^l}{\frac{r + 3}{2}}\frac{S_l}{p}{\frac{r + 3}{2} \choose l}\frac{1}{p^{0.5}}z^{\frac{r + 3}{2} - l}\mathbbm{1}_{\Zp} & \equiv & \sum_{a = 0}^{p - 1}\sum_{j = 0}^{\frac{r + 1}{2}}\sum_{l = 0}^{\frac{r + 3}{2} - j}\frac{(-1)^l}{\frac{r + 3}{2}}\frac{S_l}{p}{\frac{r + 3}{2} \choose l}{\frac{r + 3}{2} - l \choose j}\frac{a^{\frac{r + 3}{2} - l - j}}{p^{0.5}} \\
        && \qquad \qquad \qquad \qquad (z - a)^j\mathbbm{1}_{a + p\Zp} \mod \pi \latticeL{k}.
    \end{eqnarray*}
    Subtracting \eqref{primitive hard inductive step for n = r+3/2} from the
    equation above using 
    $    {\frac{r + 3}{2} \choose l}{\frac{r + 3}{2} - l \choose j} = {\frac{r + 3}{2} \choose j}{\frac{r + 3}{2} - j \choose l}$,    
    cancelling $\dfrac{r + 3}{2}$ on both sides, dropping the superfluous $j = 0$ summand on the
    right side, and bringing the $1 \leq l \leq \dfrac{r + 3}{2}$ terms from the left side
    to the right side, we get 
    \begin{eqnarray*}
        \frac{1}{p^{0.5}}z^{\frac{r + 3}{2}}\mathbbm{1}_{\Zp} & \equiv & \sum_{l = 1}^{\frac{r + 3}{2}}(-1)^{l + 1}{\frac{r + 3}{2} \choose l}\frac{S_l}{p}\frac{1}{p^{0.5}}z^{\frac{r + 3}{2} - l}\mathbbm{1}_{\Zp} \\
        && + \sum_{a = 0}^{p - 1}\sum_{j = 1}^{\frac{r + 1}{2}}\sum_{l = 0}^{\frac{r + 3}{2} - j}(-1)^{l + 1}\frac{S_l}{p}\frac{j}{\frac{r + 3}{2} - j}{\frac{r + 3}{2} \choose j}{\frac{r + 3}{2} - j \choose l} \frac{a^{\frac{r + 3}{2} - j - l}}{p^{0.5}}(z - a)^j\mathbbm{1}_{a + p\Zp} \\
        && \qquad \qquad \mod \pi \latticeL{k}.
    \end{eqnarray*}
    Separating the $a = 0$ and $a \neq 0$ terms from the second sum above,
    dropping the superfluous $a = 0$ terms
    and applying $w$ to both sides,
    we get
    \begin{eqnarray*}
        -\frac{1}{p^{0.5}}z^{\frac{r - 3}{2}}\mathbbm{1}_{p\Zp} & \equiv & \sum_{l = 1}^{\frac{r + 3}{2}}(-1)^{l}{\frac{r + 3}{2} \choose l}\frac{S_l}{p}\frac{1}{p^{0.5}}z^{\frac{r - 3}{2} + l}\mathbbm{1}_{p\Zp} \\
        && + \sum_{j = 1}^{\frac{r + 1}{2}} (-1)^{\frac{r + 3}{2} - j}\frac{S_{\frac{r + 3}{2} - j}}{p}\frac{j}{\frac{r + 3}{2} - j}{\frac{r + 3}{2} \choose \frac{r + 3}{2} - j}\frac{1}{p^{0.5}}z^{r - j}\mathbbm{1}_{\Zp} \\
        && + \sum_{a = 1}^{p - 1}\sum_{j = 1}^{\frac{r + 1}{2}}\sum_{l = 0}^{\frac{r + 3}{2} - j}(-1)^{l + 1}\frac{S_l}{p}\frac{j}{\frac{r + 3}{2} - j}{\frac{r + 3}{2} \choose j}{\frac{r + 3}{2} - j \choose l}\frac{a^{\frac{r + 3}{2} - j - l}}{p^{0.5}} \\
        && \qquad \qquad \qquad \qquad \qquad z^{r - j}(1 - az)^j\mathbbm{1}_{a^{-1} + p\Zp} \mod \pi \latticeL{k}.
    \end{eqnarray*}
    Splitting $\mathbbm{1}_{\Zp}$ in the second summand as a sum of $\mathbbm{1}_{p\Zp}$ and $\mathbbm{1}_{a^{-1} + p\Zp}$ for $1 \leq a \leq p - 1$, we get
    \begin{eqnarray*}
        -\frac{1}{p^{0.5}}z^{\frac{r - 3}{2}}\mathbbm{1}_{p\Zp} & \equiv & \sum_{l = 1}^{\frac{r + 3}{2}}(-1)^{l}{\frac{r + 3}{2} \choose l}\frac{S_l}{p}\frac{1}{p^{0.5}}z^{\frac{r - 3}{2} + l}\mathbbm{1}_{p\Zp} \\
        && + \sum_{j = 1}^{\frac{r + 1}{2}} (-1)^{\frac{r + 3}{2} - j}\frac{S_{\frac{r + 3}{2} - j}}{p}\frac{j}{\frac{r + 3}{2} - j}{\frac{r + 3}{2} \choose \frac{r + 3}{2} - j}\frac{1}{p^{0.5}}z^{r - j}\mathbbm{1}_{p\Zp} \\
        && + \sum_{a = 1}^{p - 1}\sum_{j = 1}^{\frac{r + 1}{2}}(-1)^{\frac{r + 3}{2} - j}\frac{S_{\frac{r + 3}{2} - j}}{p}\frac{j}{\frac{r + 3}{2} - j}{\frac{r + 3}{2} \choose \frac{r + 3}{2} - j}\frac{1}{p^{0.5}}z^{r - j}\mathbbm{1}_{a^{-1} + p\Zp} \\
        && + \sum_{a = 1}^{p - 1}\sum_{j = 1}^{\frac{r + 1}{2}}\sum_{l = 0}^{\frac{r + 3}{2} - j}(-1)^{l + 1}\frac{S_l}{p}\frac{j}{\frac{r + 3}{2} - j}{\frac{r + 3}{2} \choose j}{\frac{r + 3}{2} - j \choose l}\frac{a^{\frac{r + 3}{2} - j - l}}{p^{0.5}} \\
        && \quad \quad \quad \quad \quad \quad \quad \quad \qquad z^{r - j}(1 - az)^j\mathbbm{1}_{a^{-1} + p\Zp}\mod \pi \latticeL{k}.
    \end{eqnarray*}
    Using Lemma \ref{stronger bound for polynomials of large degree with varying radii}, we drop the $l > 2$ terms and the $j < \dfrac{r - 1}{2}$ terms from the first and the second sum on the right side, respectively, and combine them into the first sum below. Moreover, we expand $z^{r - j} = [(z - a^{-1}) + a^{-1}]^{r - j}$ in the third and the fourth sum and replace $a^{-1}$ by $b$ to get
    \begin{eqnarray*}
        -\frac{1}{p^{0.5}}z^{\frac{r - 3}{2}}\mathbbm{1}_{p\Zp} & \equiv & \sum_{l = 1}^{2}(-1)^l\frac{S_l}{p}{\frac{r + 3}{2} \choose l}\frac{r + 3}{2l}\frac{1}{p^{0.5}}z^{\frac{r - 3}{2} + l}\mathbbm{1}_{p\Zp} \\
        && + \sum_{b = 1}^{p - 1}\sum_{j = 1}^{\frac{r + 1}{2}}\sum_{m = 0}^{r - j}(-1)^{\frac{r + 3}{2} - j}\frac{S_{\frac{r + 3}{2} - j}}{p}\frac{j}{\frac{r + 3}{2} - j}{\frac{r + 3}{2} \choose \frac{r + 3}{2} - j}\frac{1}{p^{0.5}}{r - j \choose m}b^{m} \\
        && \quad \quad \quad \quad \quad \quad \>\> (z - b)^{r - j - m}\mathbbm{1}_{b + p\Zp} \\
        && + \sum_{b = 1}^{p - 1}\sum_{j = 1}^{\frac{r + 1}{2}}\sum_{l = 0}^{\frac{r + 3}{2} - j}\sum_{m = \frac{r - 1}{2}}^{r - j}(-1)^{l + j + 1}\frac{S_l}{p}\frac{j}{\frac{r + 3}{2} - j}{\frac{r + 3}{2} \choose j}{\frac{r + 3}{2} - j \choose l}{r - j \choose m} \\
        && \quad \quad \quad \quad \quad \quad \quad \quad \quad \qquad \frac{b^{l + m - \frac{r + 3}{2}}}{p^{0.5}}(z - b)^{r - m}\mathbbm{1}_{b + p\Zp} \mod \pi \latticeL{k}.
    \end{eqnarray*}
    We have dropped the $m < \dfrac{r - 1}{2}$ terms in the third line on the right side of the equation above using Lemma \ref{stronger bound for polynomials of large degree with varying radii}.
    Applying the operator $\sum_{\alpha = 0}^{p - 1}\alpha^2\begin{pmatrix}1 & 0 \\ - a - \alpha p & p\end{pmatrix}$, we get

    \begin{eqnarray}\label{Intermediate step in the hard inductive steps for n = r + 3/2}
        & - & \sum_{\alpha = 0}^{p - 1}\alpha^2 p(z - a - \alpha p)^{\frac{r - 3}{2}}\mathbbm{1}_{a + \alpha p + p^2\Zp} \nonumber \\
        & \equiv & \sum_{\alpha = 0}^{p - 1}\sum_{l = 1}^{2}\alpha^2p^{r/2}(-1)^l\frac{S_l}{p}{\frac{r + 3}{2}  \choose l}\frac{r + 3}{2l}\frac{1}{p^{0.5}}\left(\frac{z - a - \alpha p}{p}\right)^{\frac{r - 3}{2} + l}\mathbbm{1}_{a + \alpha p + p^2\Zp} \nonumber \\
        && + \sum_{\alpha = 0}^{p - 1}\sum_{b = 1}^{p - 1}\sum_{j = 1}^{\frac{r + 1}{2}}\sum_{m = 0}^{r - j}\alpha^2 p^{r/2}(-1)^{\frac{r + 3}{2} - j}\frac{S_{\frac{r + 3}{2} - j}}{p}\frac{j}{\frac{r + 3}{2} - j}{\frac{r + 3}{2} \choose \frac{r + 3}{2} - j}\frac{1}{p^{0.5}}{r - j \choose m}b^m \nonumber \\
        && \qquad\qquad\qquad\qquad\qquad \left(\frac{z - a - (\alpha + b)p}{p}\right)^{r - j - m}\mathbbm{1}_{a + (\alpha + b)p + p^2\Zp} \nonumber \\
        && + \sum_{\alpha = 0}^{p - 1}\sum_{b = 1}^{p - 1}\sum_{j = 1}^{\frac{r + 1}{2}}\sum_{l = 0}^{\frac{r + 3}{2} - j}\sum_{m = \frac{r - 1}{2}}^{r - j}\alpha^2 p^{r/2}(-1)^{l + j + 1}\frac{S_l}{p}\frac{j}{\frac{r + 3}{2} - j}{\frac{r + 3}{2} \choose j}{\frac{r + 3}{2} - j \choose l}{r - j \choose m}\frac{b^{l + m - \frac{r + 3}{2}}}{p^{0.5}} \nonumber \\
        && \qquad\qquad\qquad\qquad\qquad\qquad \left(\frac{z - a - (\alpha + b)p}{p}\right)^{r - m}\mathbbm{1}_{a + (\alpha + b)p + p^2\Zp} \!\!\mod \pi \latticeL{k}.
    \end{eqnarray}

    The first term on the right side of equation \eqref{Intermediate step in the hard inductive steps for n = r + 3/2} is the same in equation \eqref{six term identity for inductive steps for n = r+3/2 in self-dual}. Next we evaluate the sums over $b$ in the remaining two sums on the right side of equation \eqref{Intermediate step in the hard inductive steps for n = r + 3/2} by substituting $\lambda = \alpha + b$ and using
    \begin{eqnarray}\label{Sum over b identities for alpha^2}
      \sum_{b = 1}^{p - 1}(\lambda - b)^2b^t =
      \begin{cases}
          0    & \text{ if }  1 - p < t < -2 \text{ or } 0 < t < p - 3\\
          (-1)^{1-t}{2 \choose -t}\lambda^{2 + t}  & \text{ if } -2 \leq t \leq 0 \\
          -1   & \text{ if } t = p - 3
      \end{cases} \mod p.
    \end{eqnarray}
    Note that it is legal to do this by Lemma~\ref{Integers in the lattice}. This lemma 
    also allows us to replace $\lambda$ by its first $p$-adic digit.
    
    First consider the summands in the second sum on the right for $0 \leq m \leq p - 3$. Note that the summands with $0 < m < p - 3$ are $0$ modulo $\pi\latticeL{k}$ by the first case in \eqref{Sum over b identities for alpha^2}. The third case in \eqref{Sum over b identities for alpha^2} implies that the $m = p - 3$ summand is
    \[
        -\sum_{\lambda = 0}^{p - 1} \! p^{\frac{r - 1}{2}}\!(-1)^{\frac{r + 1}{2}}\!\frac{S_{\frac{r + 1}{2}}}{p}\frac{2}{r + 1}\frac{r + 3}{2}\mathbbm{1}_{a + \lambda p + p^2\Zp} \! \equiv \! -p^{\frac{r - 1}{2}}\!(-1)^{\frac{r + 1}{2}}\!\frac{S_{\frac{r + 1}{2}}}{p}\frac{r + 3}{r + 1}\mathbbm{1}_{a + p\Zp}\!\!\!\mod \pi \latticeL{k},
    \]
    which is $0$ modulo $\pi\latticeL{k}$ by Lemma \ref{Integers in the lattice} noting $r > 1$.
    Finally using the second case in \eqref{Sum over b identities for alpha^2}, we simplify the $m = 0$ summand and write it as the second line on the right side of equation \eqref{six term identity for inductive steps for n = r+3/2 in self-dual}.
    Next we analyze the summands in the third sum on the right. Note that the power of $b$ in this summand is in the range $[-2, p - 4]$. If $l > 2$, then it is in the range $[1, p - 4]$ so equation \eqref{Sum over b identities for alpha^2} implies that these terms are zero modulo $\pi \latticeL{k}$. We consider the $l = 1, \> 2$ and $l = 0$ cases separately since retrospectively the $l = 0$ terms will be the important ones (cf. \ref{l = 0 are the important summands}). First assume that $l = 1$ or $2$. The terms with the power of $b$ in the range $[1, p - 4]$ are $0$ modulo $\pi \latticeL{k}$ as seen above. The power of $b$ in the third sum on the right side of equation \eqref{Intermediate step in the hard inductive steps for n = r + 3/2} is $0$ only for the pairs $(l = 1, m = (r + 1)/2)$ and $(l = 2, m = (r - 1)/2)$ and is $-1$ only for $(l = 1, m = (r - 1)/2)$. The power of $b$ cannot be $-2$. Using equation \eqref{Sum over b identities for alpha^2}, we simplify these and write them as the third, fourth and fifth lines on the right below. Finally, using the first two cases of  \eqref{Sum over b identities for alpha^2}, we write the $l = 0$ summand of the third sum on the
    right side of equation \eqref{Intermediate step in the hard inductive steps for n = r + 3/2} as the last two lines in equation \eqref{six term identity for inductive steps for n = r+3/2 in self-dual}: 
    \begin{eqnarray}\label{six term identity for inductive steps for n = r+3/2 in self-dual}
        & - & \sum_{\alpha = 0}^{p - 1}\alpha^2 p (z - a - \alpha p)^{\frac{r - 3}{2}}\mathbbm{1}_{a + \alpha p + p^2\Zp} \nonumber \\
        & \equiv & \sum_{\alpha = 0}^{p - 1}\sum_{l = 1}^{2}\alpha^2p^{1 - l}(-1)^l\frac{S_l}{p}{\frac{r + 3}{2}  \choose l}\frac{r + 3}{2l}\left(z - a - \alpha p\right)^{\frac{r - 3}{2} + l}\mathbbm{1}_{a + \alpha p + p^2\Zp} \nonumber \\
        && - \sum_{\lambda = 0}^{p - 1}\sum_{j = 1}^{\frac{r + 1}{2}}p^{j - \frac{r + 1}{2}}(-1)^{\frac{r + 3}{2} - j}\frac{S_{\frac{r + 3}{2} - j}}{p}\frac{j}{\frac{r + 3}{2} - j}{\frac{r + 3}{2} \choose \frac{r + 3}{2} - j}\lambda^2(z - a - \lambda p)^{r - j}\mathbbm{1}_{a + \lambda p + p^2\Zp} \nonumber \\
        && - \sum_{\lambda = 0}^{p - 1}\sum_{j = 1}^{\frac{r - 1}{2}}\lambda^2 (-1)^j\frac{S_1}{p}j{\frac{r + 3}{2} \choose j}{r - j \choose \frac{r + 1}{2}}(z - a - \lambda p)^{\frac{r - 1}{2}}\mathbbm{1}_{a + \lambda p + p^2\Zp} \nonumber \\
        && - \sum_{\lambda = 0}^{p - 1}\sum_{j = 1}^{\frac{r - 1}{2}}\lambda^2 p^{-1} (-1)^{j + 1}\frac{S_2}{p}\frac{j\left(\frac{r + 1}{2} - j\right)}{2}{\frac{r + 3}{2} \choose j}{r - j \choose \frac{r - 1}{2}}(z - a - \lambda p)^{\frac{r + 1}{2}}\mathbbm{1}_{a + \lambda p + p^2\Zp} \nonumber \\
        && + \sum_{\lambda = 0}^{p - 1}\sum_{j = 1}^{\frac{r + 1}{2}}2\lambda p^{-1}(-1)^j \frac{S_1}{p}j{\frac{r + 3}{2} \choose j}{r - j \choose \frac{r - 1}{2}}(z - a - \lambda p)^{\frac{r + 1}{2}}\mathbbm{1}_{a + \lambda p + p^2\Zp} \nonumber \\
        && + \sum_{\lambda = 0}^{p - 1}\sum_{j = 1}^{\frac{r + 1}{2}}\sum_{m = \frac{r - 1}{2}}^{r - j}{2 \choose \frac{r + 3}{2} - m}\lambda^{m - \frac{r - 1}{2}}p^{m - \frac{r + 1}{2}}(-1)^{j + \frac{r + 3}{2} - m}\frac{S_0}{p}\frac{j}{\frac{r + 3}{2} - j}{\frac{r + 3}{2} \choose j}{r - j \choose m} \nonumber \\
        && \qquad \qquad \qquad \qquad\qquad (z - a - \lambda p)^{r - m}\mathbbm{1}_{a + \lambda p + p^2\Zp} \mod \pi \latticeL{k}.
    \end{eqnarray}

        We replace $\lambda$ by $\alpha$ in the last five sums on the right side of the equation above. We now check that the first four sums on the right side of \eqref{six term identity for inductive steps for n = r+3/2 in self-dual} are $0$ modulo $\pi \latticeL{k}$. First, using Lemma \ref{stronger bound for polynomials of large degree with varying radii} we see that the terms with $j < \dfrac{r - 1}{2}$ in the second sum on the right side of \eqref{six term identity for inductive steps for n = r+3/2 in self-dual} are $0$ modulo $\pi\latticeL{k}$. Therefore the first four terms above only involve the $\dfrac{r - 1}{2}$ and $\dfrac{r + 1}{2}$ powers of $(z - a - \alpha p)$. Using equation \eqref{First few values of S_l}, we see that the sum of the coefficients of $\alpha^2(z - a - \alpha p)^{\frac{r - 1}{2}}\mathbbm{1}_{a + \alpha p + p^2\Zp}$ in these four terms is
        \begin{eqnarray*}
            && \frac{1}{2}\left(\frac{r + 3}{2}\right)^2 - \frac{1}{2}\left(\frac{r + 1}{2}\right)\left(\frac{r + 3}{2}\right) + \sum_{j = 1}^{\frac{r - 1}{2}}(-1)^j\left(\frac{1}{2}\right)j{\frac{r + 3}{2} \choose j}{r - j \choose \frac{r + 1}{2}} \\
            && = \frac{1}{2}\left(\frac{r + 3}{2}\right)\left[1 + \sum_{j = 1}^{\frac{r - 1}{2}}(-1)^j{\frac{r + 1}{2} \choose j - 1}{r - j \choose \frac{r + 1}{2}}\right] = 0,
        \end{eqnarray*}
        by \cite[Equation 1.83]{Gou10} with $k = \dfrac{r - 3}{2}$, $n = \dfrac{r + 1}{2}$ and $i = j - 1$ there. 
        Next, using equation \eqref{First few values of S_l}, the sum of coefficients of $\alpha^2 p^{-1}(z - a - \alpha p)^{\frac{r + 1}{2}}\mathbbm{1}_{a + \alpha p + p^2\Zp}$ in the same four terms is
        \begin{eqnarray*}
            && \left(\frac{1}{6}\right)\left(\frac{(r + 3)(r + 1)}{8}\right)\left(\frac{r + 3}{4}\right) - \left(\frac{1}{6}\right)\left(\frac{r - 1}{4}\right)\left(\frac{(r + 3)(r + 1)}{8}\right)  \\
            && + \sum_{j = 1}^{\frac{r - 1}{2}}(-1)^j\left(\frac{1}{6}\right)\frac{j\left(\frac{r + 1}{2} - j\right)}{2}{\frac{r + 3}{2} \choose j}{r - j \choose \frac{r - 1}{2}} \\
            && = \left(\frac{1}{6}\right)\left(\frac{(r + 3)(r + 1)}{8}\right) \left[1 + \sum_{j = 1}^{\frac{r - 1}{2}}(-1)^j\frac{\frac{r + 1}{2} - j}{\frac{r + 3}{2} - j}{\frac{r - 1}{2} \choose j - 1}{r - j \choose \frac{r - 1}{2}}\right] = 0.
        \end{eqnarray*}
Indeed, consider
        \begin{eqnarray*}
         &  & \sum_{j = 1}^{\frac{r - 1}{2}}(-1)^j\frac{\frac{r + 1}{2} - j}{\frac{r + 3}{2} - j}{\frac{r - 1}{2} \choose j - 1}{r - j \choose \frac{r - 1}{2}} \\
         &  & = \sum_{j = 1}^{\frac{r - 1}{2}}(-1)^j{\frac{r - 1}{2} \choose j - 1}{r - j \choose \frac{r - 1}{2}} - \frac{1}{\frac{r + 1}{2}}\sum_{j = 1}^{\frac{r - 1}{2}}(-1)^j{\frac{r + 1}{2} \choose j - 1}{r - j \choose \frac{r - 1}{2}}.
        \end{eqnarray*}
        Putting $k = \dfrac{r - 1}{2}, \> n = \dfrac{r - 1}{2}$, and $i = j - 1$ in \cite[Equation 1.83]{Gou10}, the first sum on the right side of the equation above is $(-1)^{\frac{r - 1}{2}} - 1$. Putting $k = j - 1, \> n = \dfrac{r - 1}{2}, \> x = -1$ and $y = \dfrac{r - 1}{2}$ in \cite[Equation 6.39]{Gou10}, the second sum is $-(-1)^{\frac{r + 1}{2}}\left(\dfrac{r + 1}{2}\right)$.
        
        Next, the fifth term on the right side of equation \eqref{six term identity for inductive steps for n = r+3/2 in self-dual} vanishes because
        \[
            \sum_{j = 1}^{\frac{r + 1}{2}}(-1)^jj{\frac{r + 3}{2} \choose j}{r - j \choose \frac{r - 1}{2}} = \left(\frac{r + 3}{2}\right)\sum_{j = 1}^{\frac{r + 1}{2}}(-1)^j{\frac{r + 1}{2} \choose j - 1}{r - j \choose \frac{r - 1}{2}} = 0,
        \]
        by \cite[Equation 6.39]{Gou10} with $k = j - 1$, $n = \dfrac{r - 1}{2}$, $x = -1$ and $y = \dfrac{r - 1}{2}$.
    
    Therefore equation \eqref{six term identity for inductive steps for n = r+3/2 in self-dual} becomes
    \begin{eqnarray}\label{l = 0 are the important summands}
        &&-\sum_{\alpha = 0}^{p - 1}\alpha^2 p(z - a - \alpha p)^{\frac{r - 3}{2}}\mathbbm{1}_{a + \alpha p + p^2\Zp} \\
        && \equiv \sum_{\alpha = 0}^{p - 1}\sum_{j = 1}^{\frac{r + 1}{2}}\sum_{m = \frac{r - 1}{2}}^{r - j}{2 \choose \frac{r + 3}{2} - m}\alpha^{m - \frac{r - 1}{2}}p^{m - \frac{r + 1}{2}}(-1)^{j + \frac{r + 3}{2} - m}\frac{S_0}{p}\frac{j}{\frac{r + 3}{2} - j}{\frac{r + 3}{2} \choose j}{r - j \choose m} \nonumber \\
        && \qquad \qquad \qquad \qquad \qquad (z - a - \alpha p)^{r - m}\mathbbm{1}_{a + \alpha p + p^2\Zp} \mod \pi \latticeL{k}. \nonumber
    \end{eqnarray}
    Interchanging the $j$ and $m$ sums and replacing $m$ by $r - l$, we get
    \begin{eqnarray*}
        &&-\sum_{\alpha = 0}^{p - 1}\alpha^2 p(z - a - \alpha p)^{\frac{r - 3}{2}}\mathbbm{1}_{a + \alpha p + p^2\Zp} \\
        && \equiv \sum_{\alpha = 0}^{p - 1}\sum_{l = 1}^{\frac{r + 1}{2}}{2 \choose l - \frac{r - 3}{2}}\alpha^{\frac{r + 1}{2} - l}p^{\frac{r - 1}{2} - l}(-1)^{l - \frac{r - 3}{2}}\left(\sum_{j = 1}^{l}(-1)^{j}\frac{j}{\frac{r + 3}{2} - j}{\frac{r + 3}{2} \choose j}{r - j \choose r - l}\right) \\
        && \qquad\qquad\qquad\qquad (z - a -\alpha p)^{l}\mathbbm{1}_{a + \alpha p + p^2\Zp} \mod \latticeL{k}. 
    \end{eqnarray*}
    The expression inside the parentheses is $c(n, l)$
    appearing in the proof of Proposition \ref{inductive steps}  with $n = \dfrac{r + 3}{2}$.
    Substitute its value obtained in that proposition. The $-1$ part of $c(n, l)$ with
    $l = r - n$ cancels with the term on the left side. Therefore we obtain the proposition for $n = \dfrac{r + 3}{2}$.
\end{proof}

The following lemma is used in the proof of Proposition \ref{Final theorem for the hardest case in self-dual}.

\begin{lemma}\label{T12 is 0 in self-dual} Let $1 \leq r \leq p-2$ be odd. 
    For $\nu > -0.5$, we have
    \[
        \sum_{a = 0}^{p - 1}(z - a)^{\frac{r - 1}{2}}\mathbbm{1}_{a + p\Zp} \equiv 0 \mod \pi \latticeL{k}.
    \]
\end{lemma}
\begin{proof}
  The statement follows from Lemma~\ref{Integers in the lattice} for $r = 1$.
  If $3 \leq r \leq p-2$, the proof is similar to the proof of Proposition \ref{nu leq}
  and is left as an exercise (or see the first author's thesis): one uses 
    \[
        g(z) = p^{-1}\sum_{i \in I} \lambda_i (z - i)^{\frac{r + 3}{2}}\logL(z - i),
    \]
    where $I = \left\{0, \> 1, \> \ldots, \> \dfrac{r + 3}{2}, \> p\right\}$ and $\lambda_i \in \zp$
    are as in Lemma~\ref{Main coefficient identites} (2).
\end{proof}

Now we turn to the claim made at the beginning of this section.

\begin{Proposition}
    \label{Final theorem for the hardest case in self-dual}
    Let $1 \leq r \leq p-2$ be odd. For $\nu \geq 0.5$, set
    \[
        g(z) = p^{-1}\sum_{i = 0}^{p - 1}\lambda_i (z - i)^{\frac{r + 1}{2}}\logL(z - i),
    \]
    where the $\lambda_i \in \zp$ are as in Lemma~\ref{Main coefficient identites} (3).
    Then, up to multiplication by $(-1)^{\frac{r-1}{2}} (\frac{r+1}{2})$, 
    \begin{eqnarray*}
        g(z) & \equiv & \sum_{a = 0}^{p^2 - 1}p^{-1}(z - a)^{\frac{r + 1}{2}}\mathbbm{1}_{z + p^2\Zp} \\
        && + \sum_{a = 0}^{p - 1}(-1)^{\frac{r - 1}{2}}\frac{r + 1}{2}\left(\frac{\cL - H_{-} - H_{+}}{p^{0.5}}\right)p^{-0.5}(z - a)^{\frac{r + 1}{2}}\mathbbm{1}_{a + p\Zp} \\
        && + z^{\frac{r + 1}{2}}\mathbbm{1}_{\Zp} \mod \pi \latticeL{k}.
    \end{eqnarray*}
\end{Proposition}
\begin{proof}
  
  Write $g(z) = g(z)\mathbbm{1}_{\Zp} + g(z)\mathbbm{1}_{\Qp \setminus \Zp}$. By Lemma \ref{Main coefficient identites} (3)  $\lambda_0 = 1$ and $\lambda_i = 1 \mod p$ for $1 \leq i \leq p - 1$. Moreover,
  the assumption $p \geq 5$ implies that $\sum_{i = 0}^{p - 1}\lambda_i i^j = 0$ for $0 \leq j \leq \dfrac{r + 1}{2} + 1 \leq p - 2$. As usual
    \[
        w \cdot g(z)\mathbbm{1}_{\Qp \setminus \Zp} = \sum_{i = 0}^{p - 1}p^{-1}\lambda_iz^{\frac{r - 1}{2}}(1 - zi)^{\frac{r + 1}{2}}\logL(1 - zi)\mathbbm{1}_{p\Zp}.
    \]
    Moreover, using Lemma \ref{Telescoping Lemma in self-dual for qp-zp} (2),
    this function vanishes modulo $\pi \latticeL{k}$. Therefore
    $g(z)\mathbbm{1}_{\Qp \setminus \Zp} \equiv 0 \mod \pi \latticeL{k}$ and 
    $g(z) \equiv g(z)\mathbbm{1}_{\Zp} \mod \pi \latticeL{k}$.
    
    Using Lemma \ref{Telescoping Lemma in self-dual} (2), we see that $g(z)\mathbbm{1}_{\Zp} \equiv g_3(z) \mod \pi \latticeL{k}$. We analyze $g_3(z) - g_2(z)$, $g_2(z) - g_1(z)$ and $g_1(z)$ separately. First consider
    \begin{eqnarray*}
        g_3(z) - g_2(z) & = & \sum_{a = 0}^{p^2 - 1}\sum_{\alpha = 0}^{p - 1}\left[\sum_{j = 0}^{\frac{r - 1}{2}}\frac{g^{(j)}(a + \alpha p^2)}{j!}(z - a - \alpha p^2)^j\right]\mathbbm{1}_{a + \alpha p^2 + p^3\Zp} \\
        && - \sum_{a = 0}^{p^2 - 1}\left[\sum_{j = 0}^{\frac{r - 1}{2}}\frac{g^{(j)}(a)}{j!}(z - a)^j\right]\mathbbm{1}_{a + p^2\Zp}.
    \end{eqnarray*}
    The usual argument shows that
    the coefficient of $(z - a - \alpha p^2)^j\mathbbm{1}_{a + \alpha p^2 + p^3\Zp}$ in $g_3(z) - g_2(z)$ is
    \begin{eqnarray}\label{jth summand in g3 - g2 in main self-dual}
        && \!\!\!\! {\frac{r + 1}{2} \choose j}\sum_{i = 0}^{p - 1}p^{-1}\lambda_i\Big[(a + \alpha p^2 - i)^{\frac{r + 1}{2} - j}\logL(a + \alpha p^2 - i) - (a - i)^{\frac{r + 1}{2} - j}\logL(a - i) \\
        && \!\!\!\! - {\frac{r + 1}{2} - j \choose 1}(\alpha p^2)(a - i)^{\frac{r - 1}{2} - j}\logL(a - i) - \cdots - {\frac{r + 1}{2} - j \choose \frac{r - 1}{2} - j}(\alpha p^2)^{\frac{r - 1}{2} - j}(a - i)\logL(a - i)\Big]\nonumber.
    \end{eqnarray}
    
    First assume that $a \neq i$. Writing $\logL(a - i + \alpha p^2) = \logL(1 + (a - i)^{-1}\alpha p^2) + \logL(a - i)$, we see that the $i \neq a$ summand in expression \eqref{jth summand in g3 - g2 in main self-dual} is
    \[
        p^{-1}\lambda_i\left[(a - i + \alpha p^2)^{\frac{r + 1}{2} - j}\logL(1 + (a - i)^{-1}\alpha p^2) + (\alpha p^2)^{\frac{r + 1}{2} - j}\logL(a - i)\right].
    \]
    Expanding the first term in the display above, we see that the $i \neq a$ summand in \eqref{jth summand in g3 - g2 in main self-dual} is
    \begin{eqnarray}\label{i neq a summands in g3 - g2 in self-dual}
        && p^{-1}\lambda_i\Big[(a - i)^{\frac{r - 1}{2} - j}\alpha p^2 + c_{\frac{r - 3}{2} - j}(a - i)^{\frac{r - 3}{2} - j}(\alpha p^2)^2 + \cdots + c_0(\alpha p^2)^{\frac{r + 1}{2} - j} \\
        && \qquad\qquad\qquad\qquad\qquad\qquad  + (\alpha p^2)^{\frac{r + 1}{2} - j}\logL(a - i)\Big] \mod p^{r - 2j}\pi, \nonumber
    \end{eqnarray}
    where $c_0, \> \ldots, \> c_{\frac{r - 3}{2} - j} \in \Zp$ and $c_0 = H_{\frac{r + 1}{2} - j}$.

    Next assume that $i = a$. Then the $i^{\mathrm{th}}$ summand in expression \eqref{jth summand in g3 - g2 in main self-dual} is
    \begin{eqnarray*}
        p^{-1}\lambda_a(\alpha p^2)^{\frac{r + 1}{2} - j}\logL(\alpha p^2) = p^{-1}\lambda_a(\alpha p^2)^{\frac{r + 1}{2} - j}(2\cL) + p^{-1}\lambda_a(\alpha p^2)^{\frac{r + 1}{2} - j}\logL(\alpha).
    \end{eqnarray*}
    The second term on the right side of the equation above is $0$ modulo $p^{r - 2j}\pi$. Indeed, this is clear if $\alpha = 0$ and if $\alpha \neq 0$, then the valuation of the second term on the right side of the display above is $r + 1 - 2j > r - 2j$. Therefore the $i = a$ summand in expression \eqref{jth summand in g3 - g2 in main self-dual} is
    \begin{eqnarray}\label{i = a summand in g3 - g2 in self-dual}
        p^{-1}\lambda_a(\alpha p^2)^{\frac{r + 1}{2} - j}(2\cL) \mod p^{r - 2j}\pi.
    \end{eqnarray}

    If $a$ is in the range $[0, p - 1]$, then both the cases, $i \neq a$ and $i = a$ arise. If $a$ is in the range $[p, p^2 - 1]$, then only the first case, $i \neq a$ arises. Therefore if $0 \leq a \leq p - 1$, then using the identity $\sum_{i = 0}^{p - 1}\lambda_i i^{j} = 0$ for $0 \leq j \leq p - 2$ and Lemma~\ref{Integers in the lattice}, we see that the $a^{\mathrm{th}}$ summand in $g_3(z) - g_2(z)$ is
    \begin{eqnarray}\label{0 leq a leq p - 1 summands in g3 - g2 in self-dual}
        \!\!\sum_{\alpha = 0}^{p - 1}\sum_{j = 0}^{\frac{r - 1}{2}}\!{\frac{r + 1}{2} \choose j}p^{-1}\lambda_a\!\!\left(2\cL - H_{\frac{r + 1}{2} - j}\right)\!\!(\alpha p^2)^{\frac{r + 1}{2} - j}(z - a - \alpha p^2)^{j}\mathbbm{1}_{a + \alpha p^2 + p^3\Zp} \!\!\! \mod \pi \latticeL{k}.
    \end{eqnarray}
    Next if $p \leq a \leq p^2 - 1$, then expression \eqref{jth summand in g3 - g2 in main self-dual} is
        \[
            {\frac{r + 1}{2} \choose j}p^{-1}\lambda_{i_0}(\alpha p^2)^{\frac{r + 1}{2} - j}\logL(a - i_0) \mod p^{r - 2j}\pi,
        \]
        where $0 \leq i_0 \leq p - 1$ is the unique integer such that $i_0 \equiv a \!\!\mod p$. Since $0 \leq i_0 \leq p - 1$ and $p \leq a \leq p^2 - 1$, the valuation of the difference $a - i_0$ is $1$. So $\logL(a - i_0) = \cL \mod p\Zp$. This means that $\logL(a - i_0)$ in the expression above can be replaced by $\cL$ modulo $p^{r/2 - j}\pi$. Using Lemma \ref{Integers in the lattice}, we therefore see that for $p \leq a \leq p^2 - 1$, the $a^{\mathrm{th}}$ summand in $g_3(z) - g_2(z)$ is
    \begin{eqnarray}\label{p leq a leq p^2 - 1 summands in g3 - g2 in self-dual}
        \sum_{\alpha = 0}^{p - 1}\sum_{j = 0}^{\frac{r - 1}{2}}{\frac{r + 1}{2} \choose j}p^{-1}\lambda_{i_0}\cL(\alpha p^2)^{\frac{r + 1}{2} - j}(z - a - \alpha p^2)^j\mathbbm{1}_{a + \alpha p^2 + p^3\Zp} \mod \pi \latticeL{k},
    \end{eqnarray}
    Combining expressions \eqref{0 leq a leq p - 1 summands in g3 - g2 in self-dual} and \eqref{p leq a leq p^2 - 1 summands in g3 - g2 in self-dual}, using the fact that $\lambda_{i} \equiv 1 \mod p$ for $0 \leq i \leq p - 1$ and separating out one $\cL$ from expression \eqref{0 leq a leq p - 1 summands in g3 - g2 in self-dual}, we see that
    \begin{eqnarray}\label{g3 - g2 in self-dual}
        g_3(z) - g_2(z) \!\! & \equiv & \!\! \sum_{a = 0}^{p - 1}\sum_{\alpha = 0}^{p - 1}\sum_{j = 0}^{\frac{r - 1}{2}}{\frac{r + 1}{2} \choose j}p^{-1}\!\!\left(\cL - H_{\frac{r + 1}{2} - j}\right)(\alpha p^2)^{\frac{r + 1}{2} - j} 
                                    (z - a - \alpha p^2)^j\mathbbm{1}_{a + \alpha p^2 + p^3\Zp} \\ 
        && \!\! + \!\! \sum_{a = 0}^{p^2 - 1}\!\sum_{\alpha = 0}^{p - 1}\!\sum_{j = 0}^{\frac{r - 1}{2}}\!{\frac{r + 1}{2} \choose j}p^{-1}\cL(\alpha p^2)^{\frac{r + 1}{2} - j}(z - a - \alpha p^2)^j\mathbbm{1}_{a + \alpha p^2 + p^3\Zp} \!\!\! \mod \pi \latticeL{k}. \nonumber
    \end{eqnarray}

    Now $\nu \geq 0.5$ implies that $\nu > 1 + r/2 - n$ for all $k/2 < n \leq r$, so by
    Proposition \ref{inductive steps}, we get
    \[
        z^{r - n}\mathbbm{1}_{p \Zp} \equiv \sum_{b = 1}^{p - 1}\sum_{j = 1}^{\frac{r - 1}{2}}c(n, j)b^{r - j - n}(z - b)^j\mathbbm{1}_{b + p\Zp} \mod \pi \latticeL{k}.
    \]
    Applying the operator $\sum\limits_{\alpha = 0}^{p - 1}\alpha^{n - \frac{r - 1}{2}}\begin{pmatrix}1 & 0 \\ -a - \alpha p^2 & p^2\end{pmatrix}$ to both sides, 
    making the substitution $\alpha + b = \lambda$ on the right side and summing over $b$
    using equation \eqref{sum over b identity} with $i = \dfrac{r + 1}{2}$,
    replacing $\lambda$ by $\alpha$ and substituting the value of $c(n, j)$ obtained in Proposition \ref{inductive steps},
    we get
    \begin{eqnarray}\label{Applied inductive steps for g3 - g2 in self-dual}
        0 & \equiv & \sum_{\alpha = 0}^{p - 1}\sum_{j = r - n}^{\frac{r - 1}{2}}{n - \frac{r - 1}{2} \choose n - r + j}\frac{(r - n)! n}{(r - j) \cdots (n - j)}p^{-1}(\alpha p^2)^{\frac{r + 1}{2} - j} \\
        && \quad \quad \quad \quad \quad \quad \qquad (z - a - \alpha p^2)^j \mathbbm{1}_{a + \alpha p^2 + p^3\Zp} \mod \pi \latticeL{k} \nonumber
    \end{eqnarray}
    for each $k/2 < n \leq r$.

    For a fixed $0 \leq a \leq p - 1$, we wish to write the $a^{\mathrm{th}}$ summand in the first term on the right side of equation \eqref{g3 - g2 in self-dual} as $x_n \in E$ times the right side of equation \eqref{Applied inductive steps for g3 - g2 in self-dual} for $k/2 < n \leq r$ and add them to $x_{\frac{r + 1}{2}} \in E$ times
    \[
        \sum_{\alpha = 0}^{p - 1}p^{-1}(\alpha p^2)(z - a - \alpha p^2)^{\frac{r - 1}{2}}\mathbbm{1}_{a + \alpha p^2 + p^3\Zp}.
    \]
    Comparing coefficients of like terms of the form $p^{-1}(\alpha p^2)^{\frac{r + 1}{2} - j}(z - a - \alpha p^2)^j\mathbbm{1}_{a + \alpha p^2 + p^3\Zp}$, we get the following matrix equation
    \[
        \begin{pmatrix}
            {\frac{r + 1}{2} \choose 0}\frac{r}{r} & 0 & 0 & \cdots & 0 \\
            {\frac{r + 1}{2}  \choose 1}\frac{r}{r - 1} & {\frac{r - 1}{2} \choose 0}\frac{r - 1}{(r - 1)(r - 2)} & 0 & \cdots & 0 \\
            {\frac{r + 1}{2} \choose 2}\frac{r}{r - 2} & {\frac{r - 1}{2} \choose 1}\frac{r - 1}{(r - 2)(r - 3)} & {\frac{r - 3}{2} \choose 0}\frac{2!(r - 2)}{(r - 2)(r - 3)(r - 4)} & \cdots & 0 \\
            \vdots & \vdots & \vdots & & \vdots \\
            {\frac{r + 1}{2} \choose \frac{r - 1}{2}}\frac{r}{\frac{r + 1}{2}} & {\frac{r - 1}{2} \choose \frac{r - 3}{2}}\frac{r - 1}{(\frac{r + 1}{2})(\frac{r - 1}{2})} & {\frac{r - 3}{2} \choose \frac{r - 5}{2}}\frac{2!(r - 2)}{(\frac{r + 1}{2})(\frac{r - 1}{2})(\frac{r - 3}{2})} & \cdots & 1
        \end{pmatrix}
        \begin{pmatrix}
            x_r \\ x_{r - 1} \\ x_{r - 2} \\ \vdots \\ x_{\frac{r + 1}{2}}
        \end{pmatrix}
        =
        \begin{pmatrix}
            {\frac{r + 1}{2} \choose 0}(\cL - H_{\frac{r + 1}{2}}) \\
            {\frac{r + 1}{2} \choose 1}(\cL - H_{\frac{r - 1}{2}}) \\
            {\frac{r + 1}{2} \choose 2}(\cL - H_{\frac{r - 3}{2}}) \\
            \vdots \\
            {\frac{r + 1}{2} \choose \frac{r - 1}{2}}(\cL - H_{1})
        \end{pmatrix}.
    \]
    It can be checked\footnote{\label{space-time 11}See Appendix~\ref{Footnote 11}.}
    that $x_{\frac{r + 1}{2}} = (-1)^{\frac{r - 1}{2}}\left(\cL - H_{-} - H_{+}\right)$.

    Therefore, we can write equation \eqref{g3 - g2 in self-dual} as
    \begin{eqnarray*}
        g_3(z) - g_2(z) & \equiv & \sum_{a = 0}^{p - 1}\sum_{\alpha = 0}^{p - 1}(-1)^{\frac{r - 1}{2}}\left(\cL - H_{-} - H_{+}\right)p^{-1}(\alpha p^2)(z - a - \alpha p^2)^{\frac{r - 1}{2}}\mathbbm{1}_{a + \alpha p^2 + p^3\Zp} \\
        && + \sum_{a = 0}^{p^2 - 1}\sum_{\alpha = 0}^{p - 1}\sum_{j = 0}^{\frac{r - 1}{2}}{\frac{r + 1}{2} \choose j}p^{-1}\cL(\alpha p^2)^{\frac{r + 1}{2} - j}(z - a - \alpha p^2)^j\mathbbm{1}_{a + \alpha p^2 + p^3\Zp} \\
        && \qquad \qquad \mod \pi \latticeL{k}.
    \end{eqnarray*}
    Since $\nu \geq 0.5$, the first term on the right side of the display above is $0$ modulo $\pi \latticeL{k}$ by Lemma \ref{Integers in the lattice}. Next, we simplify the second line on the right side of the display above using the binomial theorem to get
    \[
        g_3(z) - g_2(z) \equiv \sum_{a = 0}^{p^2 - 1}\sum_{\alpha = 0}^{p - 1}p^{-1}\cL\left[(z - a)^\frac{r + 1}{2} - (z - a - \alpha p^2)^{\frac{r + 1}{2}}\right]\mathbbm{1}_{a + \alpha p^2 + p^3\Zp} \mod \pi \latticeL{k}.
    \]
    Note that $\nu \geq 0.5 \geq 0$ implies $\cL \in \Zp$. Therefore using Lemma \ref{stronger bound for polynomials of large degree with varying radii} the second term on the right side of the display above is $0$ modulo $\pi \latticeL{k}$. Since the first term on the right side of the display above does not involve $\alpha$, the characteristic functions sum up to $\mathbbm{1}_{a + p^2\Zp}$. Therefore 
    \begin{eqnarray}\label{Final g3 - g2 in self-dual}
        g_3(z) - g_2(z) \equiv \sum_{a = 0}^{p^2 - 1}p^{-1}\cL(z - a)^{\frac{r + 1}{2}}\mathbbm{1}_{a + p^2\Zp} \mod \pi \latticeL{k}.
    \end{eqnarray}

    We now simplify $g_2(z) - g_1(z)$. 
    The coefficient of $(z - a - \alpha p)^j\mathbbm{1}_{a + \alpha p + p^2\Zp}$ in $g_2(z) - g_1(z)$ is
    \begin{eqnarray}\label{jth summand in g2 - g1 in main self-dual}
        && \!\!\!\!\!\!\!\! {\frac{r + 1}{2} \choose j}\sum_{i = 0}^{p - 1}p^{-1}\lambda_i\Big[(a + \alpha p - i)^{\frac{r + 1}{2} - j}\logL(a + \alpha p - i)^{\frac{r + 1}{2} - j} - (a - i)^{\frac{r + 1}{2} - j}\logL(a - i) \\
        && \!\!\!\!\!\!\!\! - {\frac{r + 1}{2} - j \choose 1}(\alpha p)(a - i)^{\frac{r - 1}{2} - j}\logL(a - i) - \cdots - {\frac{r + 1}{2} - j \choose \frac{r - 1}{2} - j}(\alpha p)^{\frac{r - 1}{2} - j}(a - i)\logL(a - i)\Big]. \nonumber
    \end{eqnarray}
    Analyzing the $i \neq a$ and $i = a$ terms separately and
  using $\sum_{i = 0}^{p - 1}\lambda_i i^{j} = 0$ for $0 \leq j \leq p - 2$,
  we get as usual 
    \begin{eqnarray}\label{g2 - g1 in self-dual}
        g_2(z) - g_1(z) & \equiv & \sum_{a = 0}^{p - 1}\sum_{\alpha = 0}^{p - 1}\sum_{j = 0}^{\frac{r - 1}{2}}{\frac{r + 1}{2} \choose j}p^{-1}\left(\cL - H_{\frac{r + 1}{2} - j}\right)(\alpha p)^{\frac{r + 1}{2} - j} \\
        && \quad \quad \quad \quad \quad \quad \qquad (z - a - \alpha p)^{j}\mathbbm{1}_{a + \alpha p + p^2 \Zp} \mod \pi \latticeL{k}. \nonumber
    \end{eqnarray}
    Assume momentarily that $r > 1$. Proposition \ref{Refined hard inductive steps for self-dual} gives us
    \begin{eqnarray}\label{Applied hard inductive steps in g2 - g1 in self-dual}
        0 & \equiv & \sum_{\alpha = 0}^{p - 1}\sum_{j = r-n}^{{\frac{r+1}{2}}}p^{\frac{r-1}{2} - j}{n - \frac{r-1}{2} \choose n - r + j}\frac{(r-n)! n}{(r - j)\cdots (n - j)}\alpha^{\frac{r+1}{2} - j} \\
        && \quad \quad \quad \quad \quad \quad \qquad (z - a - \alpha p)^j \mathbbm{1}_{a + \alpha p + p^2\Zp} \mod \pi \latticeL{k} \nonumber
    \end{eqnarray}
    for every $\dfrac{r + 3}{2} \leq n \leq r$. For each $a$ and $\alpha$ in $[0, p - 1]$, we wish to write the coefficient of $p^{-1}(\alpha p)^{\frac{r + 1}{2} - j}(z - a - \alpha p)^j\mathbbm{1}_{a + \alpha p + p^2\Zp}$ in equation \eqref{g2 - g1 in self-dual} as a sum of $x_n \in E$ times the corresponding coefficient in equation \eqref{Applied hard inductive steps in g2 - g1 in self-dual}, $x_{\frac{r + 1}{2}} \in E$ times
    \[
        \sum_{\alpha = 0}^{p - 1}\alpha (z - a - \alpha p)^{\frac{r - 1}{2}}\mathbbm{1}_{a + \alpha p + p^2 \Zp}
    \]
    and $x_{\frac{r - 1}{2}} \in E$ times
    \[
        \sum_{\alpha = 0}^{p - 1}p^{-1}(z - a - \alpha p)^{\frac{r + 1}{2}}\mathbbm{1}_{a + \alpha p + p^2\Zp}.
    \]
    Therefore we get the following matrix equation
    \begin{eqnarray}\label{First matrix equation in g2 - g1 in self-dual}
        \!\!\!\! \tiny\begin{pmatrix}
            {\frac{r + 1}{2} \choose 0}\frac{r}{r} & 0 & 0 & \cdots & 0 & 0\\
            {\frac{r + 1}{2}  \choose 1}\frac{r}{r - 1} & {\frac{r - 1}{2} \choose 0}\frac{r - 1}{(r - 1)(r - 2)} & 0 & \cdots & 0 & 0\\
            {\frac{r + 1}{2} \choose 2}\frac{r}{r - 2} & {\frac{r - 1}{2} \choose 1}\frac{r - 1}{(r - 2)(r - 3)} & {\frac{r - 3}{2} \choose 0}\frac{2!(r - 2)}{(r - 2)(r - 3)(r - 4)} & \cdots & 0 & 0\\
            \vdots & \vdots & \vdots & & \vdots & \vdots\\
            {\frac{r + 1}{2} \choose \frac{r - 1}{2}}\frac{r}{\frac{r + 1}{2}} & {\frac{r - 1}{2} \choose \frac{r - 3}{2}}\frac{r - 1}{(\frac{r + 1}{2})(\frac{r - 1}{2})} & {\frac{r - 3}{2} \choose \frac{r - 5}{2}}\frac{2!(r - 2)}{(\frac{r + 1}{2})(\frac{r - 1}{2})(\frac{r - 3}{2})} & \cdots & 1 & 0\\
            {\frac{r + 1}{2} \choose \frac{r + 1}{2}}\frac{r}{\frac{r - 1}{2}} & {\frac{r - 1}{2} \choose \frac{r - 1}{2}}\frac{r - 1}{(\frac{r - 1}{2})(\frac{r - 3}{2})} & {\frac{r - 3}{2} \choose \frac{r - 3}{2}}\frac{2!(r - 2)}{(\frac{r - 1}{2})(\frac{r - 3}{2})(\frac{r - 5}{2})} & \cdots & 0 & 1
        \end{pmatrix} \!\!
        \begin{pmatrix}
            x_r \\ x_{r - 1} \\ x_{r - 2} \\ \vdots \\ x_{\frac{r + 1}{2}} \\ x_{\frac{r - 1}{2}}
        \end{pmatrix} \!\! 
        = \!\! 
        \begin{pmatrix}
            {\frac{r + 1}{2} \choose 0}(\cL - H_{\frac{r + 1}{2}}) \\
            {\frac{r + 1}{2} \choose 1}(\cL - H_{\frac{r - 1}{2}}) \\
            {\frac{r + 1}{2} \choose 2}(\cL - H_{\frac{r - 3}{2}}) \\
            \vdots \\
            {\frac{r + 1}{2} \choose \frac{r - 1}{2}}(\cL - H_{1}) \\
            0
        \end{pmatrix}.
    \end{eqnarray}
    It can be checked\footnote{\label{space-time 12}See Appendix~\ref{Footnote 12}.} that
    \[
        x_{\frac{r + 1}{2}} = (-1)^{\frac{r - 1}{2}}\left(\cL - H_{-} - H_{+}\right) \text{ and } x_{\frac{r - 1}{2}} \equiv -\cL + (-1)^{\frac{r - 1}{2}}\frac{1}{\frac{r + 1}{2}} \mod \pi.
    \]
    Using Lemma~\ref{stronger bound for polynomials of large degree with varying radii},
    we can therefore write equation \eqref{g2 - g1 in self-dual} as
    \begin{eqnarray}\label{Intermediate g2 - g1 in self-dual}
        \!\!\!\!\!\! g_2(z) - g_1(z) \!\! & \equiv & \!\! \sum_{a = 0}^{p - 1}\sum_{\alpha = 0}^{p - 1}(-1)^{\frac{r - 1}{2}}\!\left(\frac{\cL - H_{-} - H_{+}}{p^{0.5}}\right)\!p^{0.5}\alpha (z - a - \alpha p)^{\frac{r - 1}{2}}\mathbbm{1}_{a + \alpha p + p^2\Zp} \\
        && \!\! + \sum_{a = 0}^{p - 1}\sum_{\alpha = 0}^{p - 1}\left(-\cL + (-1)^{\frac{r - 1}{2}}\frac{1}{\frac{r + 1}{2}}\right)p^{-1}(z - a - \alpha p)^{\frac{r + 1}{2}}\mathbbm{1}_{a + \alpha p + p^2\Zp} \nonumber \\
        && \!\! \qquad \qquad \mod \pi \latticeL{k}. \nonumber
    \end{eqnarray}
    The expressions \eqref{g2 - g1 in self-dual} and \eqref{Intermediate g2 - g1 in self-dual} coincide
    if $r = 1$ as well using Lemma~\ref{stronger bound for polynomials of large degree with varying radii}.
    As $0 \leq a, \> \alpha \leq p - 1$, we see that $a + \alpha p$ varies from $0$ to $p^2 - 1$. Therefore we replace $a + \alpha p$ by $a$ varying between $0$ and $p^2 - 1$ in the second term on the right side of the equation \eqref{Intermediate g2 - g1 in self-dual}.

    Fixing $0 \leq a \leq p - 1$, we write the $a^{\mathrm{th}}$ summand in the first term on the right side of the equation \eqref{Intermediate g2 - g1 in self-dual} as a multiple of
    \[
        p^{-0.5}(z - a)^{\frac{r + 1}{2}}\mathbbm{1}_{a + p\Zp}.
    \]
    We first expand
    \begin{eqnarray}\label{Binomial equation in inductive steps for g2 - g1 in self-dual}
        \!\!\!\!\!\! p^{-0.5}(z - a)^{\frac{r + 1}{2}}\mathbbm{1}_{a + p\Zp} & \equiv & \sum_{\alpha = 0}^{p - 1}\sum_{j = 0}^{\frac{r - 1}{2}}{\frac{r + 1}{2} \choose j}p^{-0.5}(\alpha p)^{\frac{r + 1}{2} - j}(z - a - \alpha p)^j \mathbbm{1}_{a + \alpha p + p^2\Zp} \\
        && \qquad \qquad \mod \pi \latticeL{k}. \nonumber
    \end{eqnarray}
    Note that we have dropped the $j = \dfrac{r + 1}{2}$ term using Lemma \ref{stronger bound for polynomials of large degree with varying radii}. Since $\nu \geq 0.5$, using Proposition \ref{inductive steps} we have the following shallow inductive steps
    \[
        z^{r - n}\mathbbm{1}_{p\Zp} \equiv \sum_{b = 1}^{p - 1}\sum_{j = 1}^{\frac{r - 1}{2}}c(n, j)b^{r - j - n}(z - b)^j\mathbbm{1}_{b + p\Zp} \mod \pi \latticeL{k}
    \]
    for each $\dfrac{r + 3}{2} \leq n \leq r$. Applying $\sum\limits_{\alpha = 0}^{p - 1}\alpha^{n - \frac{r - 1}{2}}\begin{pmatrix}1 & 0 \\ - a - \alpha p & p\end{pmatrix}$ to both sides, 
    substituting $\lambda = \alpha + b$ on the right side and summing over $b$ using equation \eqref{sum over b identity}
    with $i = \dfrac{r + 1}{2}$,
    replacing $\lambda$ by $\alpha$, and substituting the value of $c(n, j)$ from Proposition \ref{inductive steps},
    we get
    \begin{eqnarray}\label{Applied shallow inductive steps in g2 - g1 in self-dual}
        0 & \equiv & \sum_{\alpha = 0}^{p - 1}\sum_{j = r - n}^{\frac{r - 1}{2}}{n - \frac{r - 1}{2} \choose n - r + j}\frac{(r - n)!n}{(r - j) \cdots (n - j)}p^{-0.5}(\alpha p)^{\frac{r + 1}{2} - j} \\
        && \qquad \qquad \qquad \qquad (z - a - \alpha p)^j\mathbbm{1}_{a + \alpha p + p^2\Zp} \mod \pi \latticeL{k}. \nonumber
    \end{eqnarray}
    Now write the coefficient of $p^{-0.5}(\alpha p)^{\frac{r + 1}{2} - j}(z - a - \alpha p)^j\mathbbm{1}_{a + \alpha p + p^2\Zp}$ in the first summand on the right side of equation \eqref{Intermediate g2 - g1 in self-dual} as a sum of $x_n \in E$ times the corresponding coefficient in equation \eqref{Applied shallow inductive steps in g2 - g1 in self-dual} and $x_{\frac{r + 1}{2}} \in E$ times the corresponding coefficient in \eqref{Binomial equation in inductive steps for g2 - g1 in self-dual}. Therefore we get the following matrix equation
    \[
        {\tiny\begin{pmatrix}
            {\frac{r + 1}{2} \choose 0}\frac{r}{r} & 0 & 0 & \cdots & {\frac{r + 1}{2} \choose 0} \\
            {\frac{r + 1}{2}  \choose 1}\frac{r}{r - 1} & {\frac{r - 1}{2} \choose 0}\frac{r - 1}{(r - 1)(r - 2)} & 0 & \cdots & {\frac{r + 1}{2} \choose 1} \\
            {\frac{r + 1}{2} \choose 2}\frac{r}{r - 2} & {\frac{r - 1}{2} \choose 1}\frac{r - 1}{(r - 2)(r - 3)} & {\frac{r - 3}{2} \choose 0}\frac{2!(r - 2)}{(r - 2)(r - 3)(r - 4)} & \cdots & {\frac{r + 1}{2} \choose 2} \\
            \vdots & \vdots & \vdots & & \vdots \\
            {\frac{r + 1}{2} \choose \frac{r - 1}{2}}\frac{r}{\frac{r + 1}{2}} & {\frac{r - 1}{2} \choose \frac{r - 3}{2}}\frac{r - 1}{(\frac{r + 1}{2})(\frac{r - 1}{2})} & {\frac{r - 3}{2} \choose \frac{r - 5}{2}}\frac{2!(r - 2)}{(\frac{r + 1}{2})(\frac{r - 1}{2})(\frac{r - 3}{2})} & \cdots & {\frac{r + 1}{2} \choose \frac{r - 1}{2}}
        \end{pmatrix}
        \begin{pmatrix}
            x_r \\ x_{r - 1} \\ x_{r - 2} \\ \vdots \\ x_{\frac{r + 1}{2}}
        \end{pmatrix}
        =
        \begin{pmatrix}
            0 \\ 0 \\ 0 \\ \vdots \\ (-1)^{\frac{r - 1}{2}}\left(\frac{\cL - H_{-} - H_{+}}{p^{0.5}}\right)
        \end{pmatrix}.}
    \]
    The determinant of this matrix is the same as that of the matrix $A$ in the matrix equation
    \eqref{Matrix equation in geq} with $i$ there equal to $\dfrac{r+1}{2}$.
    By Cramer's rule, we see $x_{\frac{r + 1}{2}} = \dfrac{\cL - H_{-} - H_{+}}{p^{0.5}}$. We finally get
    \begin{eqnarray}\label{Final g2 - g1 in self-dual}
        g_2(z) - g_1(z) & \equiv & \sum_{a = 0}^{p - 1}\left(\frac{\cL - H_{-} - H_{+}}{p^{0.5}}\right)p^{-0.5}(z - a)^{\frac{r + 1}{2}}\mathbbm{1}_{a + p\Zp} \\
        && + \sum_{a = 0}^{p^2 - 1}\left(-\cL + (-1)^{\frac{r - 1}{2}}\frac{1}{\frac{r + 1}{2}}\right)p^{-1}(z - a)^{\frac{r + 1}{2}}\mathbbm{1}_{a + p^2\Zp} \mod \pi \latticeL{k}. \nonumber
    \end{eqnarray}

    Next we simplify $g_1(z)$. Recall that
    \begin{equation}\label{g1 in self-dual}
        g_1(z) = \sum_{a = 0}^{p - 1}\sum_{j = 0}^{\frac{r - 1}{2}}\frac{g^{(j)}(a)}{j!}(z - a)^j\mathbbm{1}_{a + p\Zp},
    \end{equation}
    where
    \begin{eqnarray}\label{jth summand in g1 in self-dual}
        \frac{g^{(j)}(a)}{j!} = {\frac{r + 1}{2} \choose j}\sum_{i = 0}^{p - 1}\frac{\lambda_i}{p}(a - i)^{\frac{r + 1}{2} - j}\logL(a - i).
    \end{eqnarray}
    The $i = a$ summand is $0$. For $i \neq a$, we write $a - i = [a - i] + pa_i$ for some $a_i \in \Zp$. Since $[a - i]$ is a root of unity, $\logL(a - i) = \logL(1 + [a - i]^{-1}pa_i)$. Expanding this logarithm, we get
    \begin{eqnarray*}
        \frac{g^{(j)}(a)}{j!} & \equiv & {\frac{r + 1}{2} \choose j}\sum_{\substack{i = 0 \\ i \neq a}}^{p - 1}\frac{\lambda_i}{p}[a - i]^{\frac{r - 1}{2} - j}pa_i \mod \pi \\
        & \equiv & {\frac{r + 1}{2} \choose j}\sum_{\substack{i = 0 \\ i \neq a}}^{p - 1}\frac{\lambda_i}{p}\left(\frac{(a - i)^{\frac{r + 1}{2} - j} - [a - i]^{\frac{r + 1}{2} - j}}{\frac{r + 1}{2} - j}\right) \mod \pi.
    \end{eqnarray*}
    The expression inside the parentheses is divisible by $p$. Since $\lambda_i \equiv 1 \mod p$, we may replace $\lambda_i$ in the last line above by $1$. Moreover, since $0 \leq j \leq \dfrac{r - 1}{2}$, the power $\dfrac{r + 1}{2} - j$ is in the range $[1, p - 2]$. Using equation \eqref{sum of powers of roots of unity}, we see that the equation above simplifies to
    \[
        \frac{g^{(j)}(a)}{j!} \equiv {\frac{r + 1}{2} \choose j}\sum_{i = 0}^{p - 1}\frac{1}{p}\frac{(a - i)^{\frac{r + 1}{2} - j}}{\frac{r + 1}{2} - j} \equiv {\frac{r + 1}{2} \choose j}\sum_{m = 0}^{\frac{r + 1}{2} - j}{\frac{r + 1}{2} - j \choose m}(-1)^m \frac{S_m}{p}\frac{a^{\frac{r + 1}{2} - j - m}}{\frac{r + 1}{2} - j} \mod \pi.
    \]
    Putting this in equation \eqref{g1 in self-dual}, and using Lemma~\ref{Integers in the lattice}, we get
    \begin{eqnarray*}
        g_1(z) \equiv \sum_{a = 0}^{p - 1}\sum_{j = 0}^{\frac{r - 1}{2}}\sum_{m = 0}^{\frac{r + 1}{2} - j}{\frac{r + 1}{2} \choose j}{\frac{r + 1}{2} - j \choose m}(-1)^m\frac{S_m}{p}\frac{a^{\frac{r + 1}{2} - j - m}}{\frac{r + 1}{2} - j}(z - a)^j\mathbbm{1}_{a + p\Zp} \mod \pi \latticeL{k}.
    \end{eqnarray*}
    The $m = 0$ part of this sum will be the important part of $g_1(z)$. We separate out the 
    $m \neq 0$ terms 
    and interchange the order of the $m$ and $j$ sums to get
    \begin{eqnarray}\label{Intermediate g1 in self-dual}
        g_1(z) \!\! & \equiv & \!\! \sum_{a = 0}^{p - 1}\sum_{j = 0}^{\frac{r - 1}{2}}{\frac{r + 1}{2} \choose j}\frac{a^{\frac{r + 1}{2} - j}}{\frac{r + 1}{2} - j}(z - a)^j\mathbbm{1}_{a + p\Zp} \\
        && \!\! + \sum_{a = 0}^{p - 1}\sum_{m = 1}^{\frac{r + 1}{2}}\sum_{j = 0}^{\frac{r + 1}{2} - m}\!\!{\frac{r + 1}{2} \choose j}{\frac{r + 1}{2} - j \choose m}(-1)^m \frac{S_m}{p}\frac{a^{\frac{r + 1}{2} - j - m}}{\frac{r + 1}{2} - j}(z - a)^j \mathbbm{1}_{a + p\Zp} \!\!\! \mod \pi \latticeL{k}. \nonumber
    \end{eqnarray}

    We show the second sum on the right side of \eqref{Intermediate g1 in self-dual} is $0$ mod $\pi \latticeL{k}$. Fix an $m$ in $\left[1, \dfrac{r - 1}{2}\right]$. Using the shallow inductive steps
    Proposition~\ref{inductive steps} for $n = r, \> r - 1, \ldots, \dfrac{r + 1}{2} + m$, we get
    \[
        z^{r - n}\mathbbm{1}_{p\Zp} \equiv \sum_{b = 1}^{p - 1}\sum_{j = 1}^{\frac{r - 1}{2}}c(n, j)b^{r - j - n}(z - b)^j\mathbbm{1}_{b + p\Zp} \mod \pi \latticeL{k}.
    \]
    Applying the operator $\sum\limits_{a = 0}^{p - 1}a^{n - \frac{r - 1}{2} - m}\begin{pmatrix}1 & 0 \\ -a & 1\end{pmatrix}$, we get
    \begin{eqnarray*}
        \sum_{a = 0}^{p - 1}a^{n - \frac{r - 1}{2} - m}(z - a)^{r - n}\mathbbm{1}_{a + p\Zp} \!\! & \equiv & \!\! \sum_{a = 0}^{p - 1}\sum_{b = 1}^{p - 1}\sum_{j = 1}^{\frac{r - 1}{2}}a^{n - \frac{r - 1}{2} - m}b^{r - j - n}c(n, j)(z - (a + b))^j\mathbbm{1}_{(a + b) + p\Zp} \\
        && \!\! \qquad \qquad \mod \pi \latticeL{k}.
    \end{eqnarray*}
    Substituting $\lambda = a + b$ and summing over $b$ using equation \eqref{sum over b identity} with
    $i = \dfrac{r+1}{2} + m$, we get
    \begin{eqnarray*}
        \sum_{a = 0}^{p - 1}a^{n - \frac{r - 1}{2} - m}(z - a)^{r - n}\mathbbm{1}_{a + p\Zp} & \equiv & \sum_{\lambda = 0}^{p - 1}\sum_{j = r - n}^{\frac{r - 1}{2}}(-1)^{n - r + j + 1}{n - \frac{r - 1}{2} - m \choose n - r + j}c(n, j)\lambda^{\frac{r + 1}{2} - m - j} \\
        && \quad \quad \quad \quad \quad \quad (z - \lambda)^j\mathbbm{1}_{\lambda + p\Zp} \mod \pi \latticeL{k}.
    \end{eqnarray*}
    For $j > \dfrac{r + 1}{2} - m$, the binomial coefficient is $0$. Therefore we reduce the upper limit on $j$ to $\dfrac{r + 1}{2} - m$. Note that $m \geq 1$ implies that $\dfrac{r + 1}{2} - m \leq \dfrac{r - 1}{2}$. We replace $\lambda$ in the equation above by $a$ and substitute the value of $c(n, j)$ obtained in Proposition \ref{inductive steps}. The $-1$ part of $c(n, l)$ cancels with the term on the left side. Therefore we get
    \begin{eqnarray}\label{Refined inductive steps to bring the power to (r+1)/2 - m}
        \!\!\! 0 \equiv \! \sum_{a = 0}^{p - 1}\sum_{j = r - n}^{\frac{r + 1}{2} - m}\!\!\!{n - \frac{r - 1}{2} - m \choose n - r + j}\!\frac{(r - n)! n}{(r - j) \cdots (n - j)}a^{\frac{r + 1}{2} - m - j}(z - a)^j\mathbbm{1}_{a + p\Zp} \!\!\!\! \mod \pi \latticeL{k}.
    \end{eqnarray}
    We wish to write the term $a^{\frac{r + 1}{2} - j - m}(z - a)^j\mathbbm{1}_{a + p\Zp}$ in the second sum on the right side of \eqref{Intermediate g1 in self-dual} as a sum of $x_n \in E$ times the corresponding term in equation \eqref{Refined inductive steps to bring the power to (r+1)/2 - m} for $\dfrac{r+1}{2} + m \leq n \leq r$
    and $x_{\frac{r - 1}{2} + m} \in E$ times
    \[
        (z - a)^{\frac{r + 1}{2} - m}\mathbbm{1}_{a + p\Zp}.
    \]
    To this end, consider the matrix equation
    \[
        \begin{pmatrix}
            {\frac{r + 1}{2} - m \choose 0}\frac{r}{r} & 0 & \cdots & 0 \\
            {\frac{r + 1}{2} - m \choose 1}\frac{r}{r - 1} & {\frac{r - 1}{2} - m \choose 0}\frac{r - 1}{(r - 1)(r - 2)} & \cdots & 0 \\
            {\frac{r + 1}{2} - m \choose 2}\frac{r}{r - 2} & {\frac{r - 1}{2} - m \choose 1}\frac{r - 1}{(r - 2)(r - 3)} & \cdots & 0 \\
            \vdots & \vdots & & \vdots \\
            {\frac{r + 1}{2} - m \choose \frac{r + 1}{2} - m}\frac{r}{\frac{r - 1}{2} + m} & {\frac{r - 1}{2} + m \choose \frac{r - 1}{2} + m}\frac{r - 1}{(\frac{r - 1}{2} + m)(\frac{r - 3}{2} + m)} & \cdots & 1
        \end{pmatrix} \!\! 
        \begin{pmatrix}
            x_r \\ x_{r - 1} \\ \vdots \\ x_{\frac{r - 1}{2} + m}
        \end{pmatrix} \!\!
        = \!\! \begin{pmatrix}
            {\frac{r + 1}{2} \choose 0}{\frac{r + 1}{2} \choose m}(-1)^m\frac{S_m}{p}\frac{1}{\frac{r + 1}{2}} \\
            {\frac{r + 1}{2} \choose 1}{\frac{r - 1}{2} \choose m}(-1)^m\frac{S_m}{p}\frac{1}{\frac{r - 1}{2}} \\
            {\frac{r + 1}{2} \choose 2}{\frac{r - 3}{2} \choose m}(-1)^m\frac{S_m}{p}\frac{1}{\frac{r - 3}{2}} \\ \vdots \\ {\frac{r + 1}{2} \choose \frac{r + 1}{2} - m}{m \choose m}(-1)^m\frac{S_m}{p}\frac{1}{m}
        \end{pmatrix}\!\!.
    \]
    Since the determinant of the coefficient matrix is a $p$-adic unit, there are solutions $x_{n} \in \Zp$ for $\dfrac{r - 1}{2} + m \leq n \leq r$. This shows that for $m \in \left[1, \dfrac{r - 1}{2}\right]$, the $m^{\mathrm{th}}$ summand in equation \eqref{Intermediate g1 in self-dual} is
    \begin{eqnarray}\label{First step in killing nonzero m terms}
        x_{\frac{r - 1}{2} + m}\sum_{a = 0}^{p - 1}(z - a)^{\frac{r + 1}{2} - m}\mathbbm{1}_{a + p\Zp} \mod \pi \latticeL{k} \nonumber
    \end{eqnarray}
    for an integer $x_{\frac{r - 1}{2} + m}$. This statement is clearly also true for $m = \dfrac{r + 1}{2}$ since then $j$ can only take the value $0 = \dfrac{r + 1}{2} - m$.

    We claim that 
    \[
        \sum_{a = 0}^{p - 1} (z - a)^{\frac{r + 1}{2} - m}\mathbbm{1}_{a + p\Zp} \equiv 0 \mod \pi \latticeL{k}.
    \]
    This is clear for $m = 1$ by Lemma \ref{T12 is 0 in self-dual}. So assume that $m$ belongs to
    $\left[2, \dfrac{r + 1}{2}\right]$. Again using the shallow inductive step in
    Proposition~\ref{inductive steps}
    for $n = \frac{r - 1}{2} + m$, we get
        \[
            z^{r - n}\mathbbm{1}_{p\Zp} \equiv \sum_{b = 1}^{p - 1}\sum_{j = 1}^{\frac{r - 1}{2}}c(n, j)b^{r - j - n}(z - b)^j\mathbbm{1}_{b + p\Zp} \mod \pi \latticeL{k}.
        \]
        Applying the operator $\sum\limits_{a = 0}^{p - 1}a^{n - \frac{r - 1}{2} - m}\begin{pmatrix}1 & 0 \\ - a & 1\end{pmatrix}$ on both sides,
        substituting $\lambda = a + b$ and summing over $b$ on the right side
        using equation \eqref{sum over b identity}
        with $i = \dfrac{r+1}{2}$, we get as just seen above
        \begin{eqnarray*}
            \sum_{a = 0}^{p - 1}a^{n - \frac{r - 1}{2} - m}(z - a)^{r - n}\mathbbm{1}_{a + p\Zp} & \equiv & \sum_{\lambda = 0}^{p - 1}\sum_{j = r - n}^{\frac{r - 1}{2}}(-1)^{n - r + j + 1}{n - \frac{r - 1}{2} - m \choose n - r + j}c(n, j)\lambda^{\frac{r + 1}{2} - m - j} \\
            && \quad \quad \quad \quad \quad \quad(z - \lambda)^j\mathbbm{1}_{\lambda + p\Zp} \mod \pi \latticeL{k}.
        \end{eqnarray*}
        If $j > \dfrac{r + 1}{2} - m$, then the binomial coefficient is $0$. So we reduce the upper limit on $j$ to $\dfrac{r + 1}{2} - m$. Moreover, the lower limit on $j$ is $r - n = \dfrac{r + 1}{2} - m$. Therefore we get
        \begin{eqnarray*}
            \sum_{a = 0}^{p - 1}a^{n - \frac{r - 1}{2} - m}(z - a)^{r - n}\mathbbm{1}_{a + p\Zp} & \equiv & \sum_{\lambda = 0}^{p - 1}-c\left(\frac{r - 1}{2} + m, \frac{r + 1}{2} - m\right)(z - \lambda)^{\frac{r + 1}{2} - m}\mathbbm{1}_{\lambda + p\Zp} \\
            && \qquad \qquad \mod \pi \latticeL{k}.
        \end{eqnarray*}
        We replace $\lambda$ by $a$. Putting the value of $c\left(\frac{r - 1}{2} + m, \frac{r + 1}{2} - m\right)$ and cancelling the $-1$ part with the left side, we get
        \[
            0 \equiv \sum_{a = 0}^{p - 1}\frac{\left(\frac{r + 1}{2} - m\right)!\left(\frac{r - 1}{2} + m\right)}{\left(\frac{r - 1}{2} + m\right)\cdots\left(2m - 1\right)}(z - a)^{\frac{r + 1}{2} - m}\mathbbm{1}_{a + p\Zp} \mod \pi \latticeL{k}.
        \]
        Now $2 \leq m \leq \dfrac{r + 1}{2}$ implies that the coefficient in the equation above is a unit.
        This proves the claim for such $m$ as well.

        Thus, the second sum on the right side of \eqref{Intermediate g1 in self-dual}
        is $0$ mod $\pi \latticeL{k}$.
    We get
    \begin{eqnarray}
      \label{Penultimate g1 in self-dual}
        g_1(z) \equiv \sum_{a = 0}^{p - 1}\sum_{j = 0}^{\frac{r - 1}{2}}{\frac{r + 1}{2} \choose j}\frac{a^{\frac{r + 1}{2} - j}}{\frac{r + 1}{2} - j}(z - a)^j\mathbbm{1}_{a + p\Zp} \mod \pi \latticeL{k}.
    \end{eqnarray}
    We wish to write all the powers of $(z - a)$ appearing on the right side of the equation above in terms of $(z - a)^{\frac{r + 1}{2}}$. So for $\dfrac{r + 3}{2} \leq n \leq r$, consider again
    the shallow inductive steps
    \[
        z^{r - n}\mathbbm{1}_{p\Zp} \equiv \sum_{b = 1}^{p - 1}\sum_{j = 1}^{\frac{r - 1}{2}}c(n, j)b^{r - j - n}(z - b)^j\mathbbm{1}_{b + p\Zp} \mod \pi \latticeL{k}
    \]
    obtained in Proposition \ref{inductive steps}. Applying the operator $\sum\limits_{a = 0}^{p - 1}a^{n - \frac{r - 1}{2}}\begin{pmatrix}1 & 0 \\ -a & 1\end{pmatrix}$ to both sides, 
    substituting $\lambda = a + b$ on the right side and summing over $b$ using
    equation \eqref{sum over b identity} with $i = \dfrac{r+1}{2}$, 
    replacing $\lambda$ by $a$, and substituting the value of $c(n, j)$ obtained in
    Proposition \ref{inductive steps},
    we get
    \begin{eqnarray}\label{Applied inductive steps for g1 in self-dual}
        0 \equiv \sum_{a = 0}^{p - 1}\sum_{j = r - n}^{\frac{r - 1}{2}}{n - \frac{r - 1}{2} \choose n - r + j}\frac{(r - n)!n}{(r - j) \cdots (n - j)}a^{\frac{r + 1}{2} - j}(z - a)^j\mathbbm{1}_{a + p\Zp} \mod \pi \latticeL{k}.
    \end{eqnarray}
    Next, consider the following equation obtained from the binomial expansion
    \begin{eqnarray}\label{Trivial equation in g1 in self-dual}
        z^{\frac{r + 1}{2}}\mathbbm{1}_{\Zp} \equiv \sum_{a = 0}^{p - 1}\sum_{j = 0}^{\frac{r - 1}{2}}{\frac{r + 1}{2} \choose j}a^{\frac{r + 1}{2} - j}(z - a)^j\mathbbm{1}_{a + p\Zp} \mod \pi \latticeL{k}.
    \end{eqnarray}
    We have dropped the $j = \dfrac{r + 1}{2}$ term on the right side of the equation above using Lemma \ref{stronger bound for polynomials of large degree with varying radii}.

    We wish to write the coefficient of $a^{\frac{r + 1}{2} - j}(z - a)^j\mathbbm{1}_{a + p\Zp}$ in equation
    \eqref{Penultimate g1 in self-dual} as a sum of $x_{r+1-n} \in E$ times the corresponding coefficient in equation \eqref{Applied inductive steps for g1 in self-dual} for $n = r, r-1, \ldots, \dfrac{r+3}{2}$ and $x_{\frac{r + 1}{2}}$ times
    the corresponding coefficient in equation \eqref{Trivial equation in g1 in self-dual}. We get following matrix equation
    \[
        {\begin{pmatrix}
                {\frac{r + 1}{2} \choose 0}\frac{r}{r} & 0 & 0 & \cdots & {\frac{r + 1}{2} \choose 0} \\
                {\frac{r + 1}{2} \choose 1}\frac{r}{r - 1} & {\frac{r - 1}{2} \choose 0}\frac{r - 1}{(r - 1)(r - 2)} & 0 & \cdots & {\frac{r + 1}{2} \choose 1} \\
                {\frac{r + 1}{2} \choose 2}\frac{r}{r - 2} & {\frac{r - 1}{2} \choose 1}\frac{r - 1}{(r - 2)(r - 3)} & {\frac{r - 3}{2} \choose 0}\frac{2!(r - 2)}{(r - 2)(r - 3)(r - 4)} & \cdots & {\frac{r + 1}{2} \choose 2} \\
                \vdots & & & & \vdots \\
                {\frac{r + 1}{2} \choose \frac{r - 1}{2}}\frac{r}{\frac{r + 1}{2}} & {\frac{r - 1}{2} \choose \frac{r - 3}{2}}\frac{r - 1}{(\frac{r + 1}{2})(\frac{r - 1}{2})} & {\frac{r - 3}{2} \choose \frac{r - 5}{2}}\frac{2!(r - 2)}{(\frac{r + 1}{2})(\frac{r - 1}{2})(\frac{r - 3}{2})} & \cdots & {\frac{r + 1}{2} \choose \frac{r - 1}{2}}
            \end{pmatrix}
            \begin{pmatrix}
                x_1 \\ x_2 \\ x_3 \\ \vdots \\ x_{\frac{r + 1}{2}}
            \end{pmatrix}
            = \begin{pmatrix}
                {\frac{r + 1}{2} \choose 0}\frac{1}{\frac{r + 1}{2}} \\ {\frac{r + 1}{2} \choose 1}\frac{1}{\frac{r - 1}{2}} \\ {\frac{r + 1}{2} \choose 2}\frac{1}{\frac{r - 3}{2}} \\ \vdots \\ {\frac{r + 1}{2} \choose \frac{r - 1}{2}}\frac{1}{1}
            \end{pmatrix}.}
    \]
    We have already solved this in the proof of Theorem~\ref{Final theorem for geq}: if
    $i = \dfrac{r + 1}{2}$, then \eqref{value for c in geq} gives
    \[
        x_{\frac{r + 1}{2}} = (-1)^{\frac{r - 1}{2}}\frac{1}{\frac{r + 1}{2}}.
    \]
    This shows that \eqref{Penultimate g1 in self-dual} becomes
    \begin{equation}\label{Final g1 in self-dual}
        g_1(z) \equiv (-1)^{\frac{r - 1}{2}}\frac{1}{\frac{r + 1}{2}}z^{\frac{r + 1}{2}}\mathbbm{1}_{\Zp} \mod \pi \latticeL{k}.
    \end{equation}

    Adding up equations \eqref{Final g3 - g2 in self-dual}, \eqref{Final g2 - g1 in self-dual} and \eqref{Final g1 in self-dual}, we get
    \begin{eqnarray*}
        g(z) & \equiv & \sum_{a = 0}^{p^2 - 1}\left(\frac{(-1)^{\frac{r - 1}{2}}}{\frac{r + 1}{2}}\right)p^{-1}(z - a)^{\frac{r + 1}{2}}\mathbbm{1}_{z + p^2\Zp} 
        + \sum_{a = 0}^{p - 1}\left(\frac{\cL - H_{-} - H_{+}}{p^{0.5}}\right)p^{-0.5}(z - a)^{\frac{r + 1}{2}}\mathbbm{1}_{a + p\Zp} \\
                                &&
        + \frac{(-1)^{\frac{r - 1}{2}}}{\frac{r + 1}{2}}z^{\frac{r + 1}{2}}\mathbbm{1}_{\Zp} \mod \pi \latticeL{k}.
    \end{eqnarray*}
    Multiplying by $(-1)^{\frac{r - 1}{2}}\left(\dfrac{r + 1}{2}\right)$ throughout, we get the claim made in the proposition.
\end{proof}

\begin{theorem}\label{Final theorem for geq 0.5 for odd weights}
    Let $1 \leq r \leq p - 2$ be odd. If $\nu \geq 0.5$, then the map $\IZind a^{\frac{r - 1}{2}}d^{\frac{r + 1}{2}} \twoheadrightarrow F_{r - 1, \> r}$ factors as
    \[
        \IZind a^{\frac{r - 1}{2}}d^{\frac{r + 1}{2}} \twoheadrightarrow \frac{\IZind a^{\frac{r - 1}{2}}d^{\frac{r + 1}{2}}}{\im(T_{-1, 0}^2 - cT_{-1, 0} + 1)} \twoheadrightarrow F_{r - 1, \> r},
    \]
    where
    \[
        c = (-1)^{\frac{r + 1}{2}}\frac{r + 1}{2}\left(\frac{\cL - H_{-} - H_{+}}{p^{0.5}}\right).
    \]
        Moreover, there is a surjection
        \[
            \pi(p - 2, \lambda_{\frac{r + 1}{2}}, \omega^{\frac{r + 1}{2}}) \oplus \pi(p - 2, \lambda_{\frac{r + 1}{2}}^{-1}, \omega^{\frac{r + 1}{2}}) \twoheadrightarrow F_{r - 1, \> r},
          \]
       where  $\lambda_{\frac{r + 1}{2}} + \lambda_{\frac{r + 1}{2}}^{-1} = c$.
\end{theorem}
\begin{proof}
  We claim that under the map $\IZind a^{\frac{r - 1}{2}}d^{\frac{r + 1}{2}} \twoheadrightarrow F_{r - 1, \> r}$,
  the element
    \[
        -w\left[T_{-1, 0}^2 - cT_{-1, 0} + 1\right]\llbracket \id, X^{\frac{r - 1}{2}}Y^{\frac{r + 1}{2}}\rrbracket
    \]
    maps to $0$. The first part of the theorem then follows since $\llbracket \id, X^{\frac{r - 1}{2}}Y^{\frac{r + 1}{2}}\rrbracket$ generates $\IZind a^{\frac{r - 1}{2}}d^{\frac{r + 1}{2}}$. Indeed, 
    Proposition~\ref{Final theorem for the hardest case in self-dual} implies that
    \begin{eqnarray*}
        0  & \equiv & \sum_{a = 0}^{p^2 - 1}p^{-1}(z - a)^{\frac{r + 1}{2}}\mathbbm{1}_{z + p^2\Zp} 
                                      - \sum_{a = 0}^{p - 1}
                                      c p^{-0.5}(z - a)^{\frac{r + 1}{2}}\mathbbm{1}_{a + p\Zp}  
                                    + z^{\frac{r + 1}{2}}\mathbbm{1}_{\Zp} \mod \pi \latticeL{k},
    \end{eqnarray*}
    and using the formula \eqref{Formulae for T-10 and T12 in the commutative case} for $T_{-1,0}$, we have
    \begin{itemize}
        \item $-wT_{-1, 0}^2\llbracket\id, X^{\frac{r - 1}{2}}Y^{\frac{r + 1}{2}}\rrbracket$ maps to
        \begin{eqnarray*}
            -\sum_{\lambda, \mu \in I_1}\begin{pmatrix}0 & 1 \\ p^2 & -\lambda - p\mu\end{pmatrix}z^{\frac{r - 1}{2}}\mathbbm{1}_{p \Zp}
            & = & -\sum_{\lambda, \mu \in I_1}\frac{1}{p^r}(z - \lambda - p\mu)^r\left(\frac{p^2}{z - \lambda -p\mu}\right)^{\frac{r - 1}{2}}\\
            && \quad \quad \quad \quad \quad \mathbbm{1}_{p\Zp}\left(\frac{p^2}{z - \lambda - p\mu}\right) \\
            & = & \sum_{a = 0}^{p^2 - 1}p^{-1}(z - a)^{\frac{r + 1}{2}}\mathbbm{1}_{a + p^2\Zp} \mod \pi \latticeL{k}
        \end{eqnarray*}

        \item $-wT_{-1, 0}\llbracket\id, X^{\frac{r - 1}{2}}Y^{\frac{r + 1}{2}}\rrbracket$ maps to
        \begin{eqnarray*}
            -\sum_{\lambda \in I_1}\begin{pmatrix}0 & 1 \\ p & -\lambda\end{pmatrix}z^{\frac{r - 1}{2}}\mathbbm{1}_{p\Zp} & = & -\sum_{\lambda \in I_1}\frac{1}{p^{r/2}}(z - \lambda)^r\left(\frac{p}{z - \lambda}\right)^{\frac{r - 1}{2}}\mathbbm{1}_{p\Zp}\left(\frac{p}{z - \lambda}\right) \\
            & = & \sum_{a = 0}^{p - 1}p^{-0.5}(z - a)^{\frac{r + 1}{2}}\mathbbm{1}_{a + p\Zp} \mod \pi \latticeL{k}
        \end{eqnarray*}

        \item $-w\llbracket \id, X^{\frac{r - 1}{2}}Y^{\frac{r + 1}{2}}\rrbracket$ maps to
        \begin{eqnarray*}
          -\begin{pmatrix}0 & 1 \\ 1 & 0\end{pmatrix}z^{\frac{r - 1}{2}}\mathbbm{1}_{p\Zp}
            & =  & - z^{r} \left(\frac{1}{z}\right)^{\frac{r - 1}{2}}\mathbbm{1}_{p\Zp}\left(\frac{1}{z}\right) \\
            & =  & z^{\frac{r + 1}{2}}\mathbbm{1}_{\Zp} \mod \pi \latticeL{k}.
        \end{eqnarray*}
    \end{itemize}
    Thus, the surjection $\IZind a^{\frac{r - 1}{2}}d^{\frac{r + 1}{2}} \twoheadrightarrow F_{r - 1, \> r}$ factors as
    \[
        \IZind a^{\frac{r - 1}{2}}d^{\frac{r + 1}{2}} \twoheadrightarrow \frac{\IZind a^{\frac{r - 1}{2}}d^{\frac{r + 1}{2}}}{\im(T_{-1, 0}^2 - cT_{-1, 0} + 1)} \twoheadrightarrow F_{r - 1, \> r}.
      \]
      
    Next consider the following exact sequences
    \[
        \begin{tikzcd}
            0 \ar[r] & \im T_{-1, 0} \ar[r]\ar[d, two heads] & \IZind a^{\frac{r - 1}{2}}d^{\frac{r + 1}{2}} \ar[r]\ar[d, two heads] & \dfrac{\IZind a^{\frac{r - 1}{2}}d^{\frac{r + 1}{2}}}{\im T_{-1, 0}} \ar[r]\ar[d, two heads] & 0 \\
            0 \ar[r] & S \ar[r] & F_{r - 1, \> r} \ar[r] & Q \ar[r] & 0,
        \end{tikzcd}
    \]
    where $S$ is the image of $\im T_{-1, 0}$ under the surjection $\IZind a^{\frac{r - 1}{2}}d^{\frac{r + 1}{2}} \twoheadrightarrow F_{r - 1, \> r}$ and $Q = F_{r - 1, \> r}/S$. We know that the subspace $\im (T_{-1, 0}^2 - cT_{-1, 0} + 1)$ maps to $0$ under the middle vertical map. Therefore the right vertical map is $0$. This shows that we have a surjection $\im T_{-1, 0} \twoheadrightarrow F_{r - 1, \> r}$. Using the first isomorphism theorem and \cite[Proposition 3.1]{AB15}, we get a second surjection 
    \[
        \frac{\IZind a^{\frac{r - 1}{2}}d^{\frac{r + 1}{2}}}{\im T_{1, 2}} \twoheadrightarrow F_{r - 1, \> r},
    \]
    by pre-composing the first one by $T_{-1,0}$. This map also factors through
    \[
        \frac{\IZind a^{\frac{r - 1}{2}}d^{\frac{r + 1}{2}}}{\im T_{1, 2} + \im (T_{-1, 0}^2 - cT_{-1, 0} + 1)} \twoheadrightarrow F_{r - 1, \> r}
    \]
    since $T_{-1,0}$ maps $\im (T_{-1, 0}^2 - cT_{-1, 0} + 1)$ into itself.
    By \eqref{pi quadratic}, the term on the left is
    \[
      \pi(p - 2, \lambda_{\frac{r+1}{2}}, \omega^{\frac{r + 1}{2}}) \oplus \pi(p - 2, \lambda_{\frac{r+1}{2}}^{-1}, \omega^{\frac{r + 1}{2}}),
    \]
    proving the second part of the theorem as well.
\end{proof}

\subsection{The $-0.5 < \nu < 0.5$ cases for odd weights}
\label{leq section for odd weights}

In this section, we continue to assume that $k$ and $r$ are odd. We show that if $-0.5 < \nu < 0.5$, then the map $\IZind a^{\frac{r - 1}{2}}d^{\frac{r + 1}{2}} \twoheadrightarrow F_{r-1, \> r}$ factors as
\[
    \IZind a^{\frac{r - 1}{2}}d^{\frac{r + 1}{2}} \twoheadrightarrow \frac{\IZind a^{\frac{r - 1}{2}}d^{\frac{r + 1}{2}}}{\im T_{-1, 0}} \twoheadrightarrow F_{r-1, \> r}.
\]
We begin by proving
\begin{Proposition}
  \label{nu in between -0.5 and 0.5}
  Let $1 \leq r \leq p-2$ be odd.
  Let $-0.5< \nu < 0.5$. Set $x \in \bQ$ such that $x + \nu = -0.5$. Define
    \[
        g(z) = \sum_{i = 0}^{p - 1}p^x\lambda_i(z - i)^{\frac{r + 1}{2}}\logL(z - i),
    \]
    where the $\lambda_i \in \zp$ are as in Lemma~\ref{Main coefficient identites} (3).
    Then
    \[
        g(z) \equiv \sum_{a = 0}^{p - 1}\left[p^{0.5 + x}(\cL - H_{-} - H_{+})\right]p^{-0.5}(z - a)^{\frac{r + 1}{2}}\mathbbm{1}_{a + p\Zp} \mod \pi \latticeL{k}.
    \]
\end{Proposition}
\begin{proof}
  Recall that by Lemma~\ref{Main coefficient identites} (3), we have $\lambda_i \equiv 1 \mod p\>$ for
  $0 \leq i \leq p-1$ and $\sum_{i = 0}^{p - 1}\lambda_i i^j = 0$ for $0 \leq j \leq p - 2$.
  Note again that $p \geq 5$ implies $\sum_{i = 0}^{p - 1}\lambda_i i^j = 0$
  for $0 \leq j \leq \dfrac{r + 1}{2} + 1$. 

  Write $g(z) = g(z)\mathbbm{1}_{\Zp} + g(z)\mathbbm{1}_{\Qp \setminus \Zp}$.
  As before, we have
    \[
        w \cdot g(z)\mathbbm{1}_{\Qp \setminus \Zp} = \sum_{i = 0}^{p - 1}p^x\lambda_iz^{\frac{r - 1}{2}}(1 - zi)^{\frac{r + 1}{2}}\logL(1 - zi)\mathbbm{1}_{p\Zp}.
    \]
    Since $-1 < x < 0$, Lemma \ref{Telescoping Lemma in self-dual for qp-zp} (2)
    implies that $g(z)\mathbbm{1}_{\Qp \setminus \Zp} \equiv 0 \mod \pi \latticeL{k}.$
    Therefore $g(z) \equiv g(z)\mathbbm{1}_{\Zp}$ modulo $\pi \latticeL{k}$.
    
    Applying Lemma~\ref{Telescoping Lemma in self-dual} (3), we see that $g(z)\mathbbm{1}_{\Zp} \equiv g_2(z) \mod \pi \latticeL{k}$. We analyze $g_2(z) - g_1(z)$ and $g_1(z)$ separately. 

    Recall that for $0 \leq j \leq \frac{r - 1}{2}$, the coefficient of $(z - a - \alpha p)^j\mathbbm{1}_{a + \alpha p + p^2\Zp}$ in $g_2(z) - g_1(z)$ is
    \begin{eqnarray}\label{jth summand in g2 - g1 in between -0.5 and 0.5}
        &&{\frac{r + 1}{2} \choose j}\sum_{i = 0}^{p - 1}p^x\lambda_i\Big[(a + \alpha p - i)^{\frac{r + 1}{2} - j}\logL(a + \alpha p - i) - (a - i)^{\frac{r + 1}{2} - j}\logL(a - i) \\
        &&- {\frac{r + 1}{2} - j \choose 1}(\alpha p)(a - i)^{\frac{r - 1}{2} - j}\logL(a - i) - \cdots - {\frac{r + 1}{2} - j \choose \frac{r - 1}{2} - j}(\alpha p)^{\frac{r - 1}{2} - j}(a - i)\logL(a - i)\Big]. \nonumber
    \end{eqnarray}

    First assume $i \neq a$. Write $\logL(a - i + \alpha p) = \logL(1 + (a - i)^{-1}\alpha p) + \logL(a - i)$. Therefore the $i \neq a$ summand in the display above is
    \[
        p^x\lambda_i\left[(a - i + \alpha p)^{\frac{r + 1}{2} - j}\logL(1 + (a - i)^{-1}\alpha p) + (\alpha p)^{\frac{r + 1}{2} - j}\logL(a - i)\right].
    \]
    As usual the second term in the expression above is $0$ modulo $p^{r/2 - j}\pi$ as $x > -1$. Expanding both the factors in the first term, we see that the $i \neq a$ summand in expression \eqref{jth summand in g2 - g1 in between -0.5 and 0.5} is
    \begin{eqnarray}\label{i neq a summand in g2 - g1 in between -0.5 and 0.5}
        && p^x\lambda_i[(a - i)^{\frac{r - 1}{2} - j}(\alpha p) + c_{\frac{r - 3}{2} - j}(a - i)^{\frac{r - 3}{2} - j}(\alpha p)^2 + \cdots + c_1(a - i)(\alpha p)^{\frac{r - 1}{2} - j} + c_0 (\alpha p)^{\frac{r + 1}{2} - j}] \nonumber\\
        && \qquad \qquad \mod p^{r/2 - j}\pi,
    \end{eqnarray}
    where $c_0, \ldots, c_{\frac{r - 3}{2} - j} \in \Zp$ and $c_0 = H_{\frac{r + 1}{2} - j}$.
    Next, if $i = a$, then as usual one sees that
    the $i^{\mathrm{th}}$ summand in expression \eqref{jth summand in g2 - g1 in between -0.5 and 0.5} is
    \begin{eqnarray}\label{i = a summand in g2 - g1 in between -0.5 and 0.5}
        p^x\lambda_a(\alpha p)^{\frac{r + 1}{2} - j}\cL \mod p^{r/2 - j}\pi.
    \end{eqnarray}
    Putting expressions \eqref{i neq a summand in g2 - g1 in between -0.5 and 0.5} and \eqref{i = a summand in g2 - g1 in between -0.5 and 0.5} in \eqref{jth summand in g2 - g1 in between -0.5 and 0.5} and using the summation identities $\sum_{i = 0}^{p - 1}\lambda_i i^j = 0$ for $0 \leq j \leq p - 2$, we see that the coefficient of $(z - a - \alpha p)^{j}\mathbbm{1}_{a + \alpha p + p^2\Zp}$ in $g_2(z) - g_1(z)$ is
    \[
        {\frac{r + 1}{2} \choose j}p^x\lambda_a(\alpha p)^{\frac{r + 1}{2} - j}\left(\cL - H_{\frac{r + 1}{2} - j}\right) \mod p^{r/2 - j}\pi.
    \]
    Since $\lambda_i \equiv 1 \!\! \mod p$ for $0 \leq i \leq p - 1$, we see, after rearranging the
    power of $p$, that
    \begin{eqnarray}\label{g2 - g1 in between -0.5 and 0.5}
        g_2(z) - g_1(z) & \equiv & \sum_{a = 0}^{p - 1}\sum_{\alpha = 0}^{p - 1}\sum_{j = 0}^{\frac{r - 1}{2}}{\frac{r + 1}{2} \choose j}p^{0.5 + x}\left(\cL - H_{\frac{r + 1}{2} - j}\right)p^{r/2 - j}\alpha^{\frac{r + 1}{2} - j}(z - a - \alpha p)^j\mathbbm{1}_{a + \alpha p + p^2\Zp} \nonumber\\
        && \qquad \qquad \mod \pi \latticeL{k}.
    \end{eqnarray}

    We first wish to prove that 
    \begin{eqnarray}
        \label{wish to prove for g2 - g1}
        g_2(z) - g_1(z) & \equiv & \sum_{a = 0}^{p - 1}\sum_{\alpha = 0}^{p - 1}(-1)^{\frac{r - 1}{2}}\left(\cL - H_{-} - H_{+}\right)p^{0.5 + x}p^{0.5}\alpha(z - a - \alpha p)^{\frac{r - 1}{2}}\mathbbm{1}_{a + \alpha p + p^2\Zp} \nonumber \\
        && \qquad \qquad \mod \pi \latticeL{k}.
    \end{eqnarray}
    This is clear for $r = 1$. So we may assume that $1 < r \leq p-2$.
    We consider two cases $0 < \nu < 0.5$ and $-0.5 < \nu \leq 0$.

    First assume that $0 < \nu < 0.5$. We may use the deep inductive steps obtained in Proposition \ref{Refined hard inductive steps for self-dual} since $r \neq 1$ and $\nu > 0$. Multiplying them
    by the integral element $p^{1 + x} \in \Oe$
    and using Lemma \ref{stronger bound for polynomials of large degree with varying radii} to
    drop the $j = \dfrac{r + 1}{2}$ summand,
    for $\dfrac{r + 3}{2} \leq n \leq r$ we get
    \begin{eqnarray}\label{Applied inductive steps in between 0 and 0.5}
        0 & \equiv & \sum_{\alpha = 0}^{p - 1}\sum_{j = r-n}^{{\frac{r-1}{2}}}{n - \frac{r-1}{2} \choose n - r + j}\frac{(r-n)! n}{(r - j)\cdots (n - j)}p^{0.5 + x}p^{r/2 - j}\alpha^{\frac{r+1}{2} - j}(z - a - \alpha p)^j \mathbbm{1}_{a + \alpha p + p^2\Zp} \nonumber\\
        && \qquad \qquad \mod \pi \latticeL{k}.
    \end{eqnarray}
    For a fixed $a$ and $\alpha$, we write the term
    \[
        \sum_{j = 0}^{\frac{r - 1}{2}}{\frac{r + 1}{2} \choose j}p^{0.5 + x}\left(\cL - H_{\frac{r + 1}{2} - j}\right) p^{r/2 - j}\alpha^{\frac{r + 1}{2} - j}(z - a - \alpha p)^j \mathbbm{1}_{a + \alpha p + p^2\Zp}
    \]
    appearing in \eqref{g2 - g1 in between -0.5 and 0.5} as the sum of $x_n \in E$ times the corresponding term in \eqref{Applied inductive steps in between 0 and 0.5} and $x_{\frac{r + 1}{2}}$ times
    \[
        p^{0.5 + x}p^{0.5}\alpha(z - a - \alpha p)^{\frac{r - 1}{2}}\mathbbm{1}_{a + \alpha p + p^2\Zp}.
    \]
    Comparing the coefficients of like terms of $p^{0.5 + x}p^{r/2 - j}\alpha^{\frac{r + 1}{2} - j}(z - a - \alpha p)^j\mathbbm{1}_{a + \alpha p + p^2\Zp}$, we have the following matrix equation
    \begin{eqnarray}\label{Matrix equation in between 0 and 0.5}
        {\tiny\begin{pmatrix}
            {\frac{r + 1}{2} \choose 0}\frac{r}{r} & 0 & 0 & \cdots & 0 \\
            {\frac{r + 1}{2} \choose 1}\frac{r}{r - 1} & {\frac{r - 1}{2} \choose 0}\frac{r - 1}{(r - 1)(r - 2)} & 0 & \cdots & 0 \\
            {\frac{r + 1}{2} \choose 2}\frac{r}{r - 2} & {\frac{r - 1}{2} \choose 1}\frac{r - 1}{(r - 2)(r - 3)} & {\frac{r - 3}{2} \choose 0}\frac{2!(r - 2)}{(r - 2)(r - 3)(r - 4)} & \cdots & 0 \\
            \vdots & \vdots & \vdots & & \vdots \\
            {\frac{r + 1}{2} \choose \frac{r - 1}{2}}\frac{r}{\frac{r + 1}{2}} & {\frac{r - 1}{2} \choose \frac{r - 3}{2}}\frac{r - 1}{(\frac{r + 1}{2})(\frac{r - 1}{2})} & {\frac{r - 3}{2} \choose \frac{r - 5}{2}}\frac{2!(r - 2)}{(\frac{r + 1}{2})(\frac{r - 1}{2})(\frac{r - 3}{2})} & \cdots & 1
        \end{pmatrix}\!\! 
        \begin{pmatrix}
            x_r \\ x_{r - 1} \\ x_{r - 2} \\ \vdots \\ x_{\frac{r + 1}{2}}
        \end{pmatrix} \!\! 
        = \!\!
        \begin{pmatrix}
            {\frac{r + 1}{2} \choose 0}\left(\cL - H_{\frac{r + 1}{2}}\right) \\ {\frac{r + 1}{2} \choose 1}\left(\cL - H_{\frac{r - 1}{2}}\right) \\
            {\frac{r + 1}{2} \choose 2}\left(\cL - H_{\frac{r - 3}{2}}\right) \\
            \vdots \\ 
            {\frac{r + 1}{2} \choose \frac{r - 1}{2}}\left(\cL - H_{1}\right)
        \end{pmatrix}\!.}
    \end{eqnarray}
    This is the same as
    \eqref{First matrix equation in g2 - g1 in self-dual} (dropping the last variable and equation there). 
    As seen there, we have $x_{\frac{r + 1}{2}} = (-1)^{\frac{r - 1}{2}}\left(\cL - H_{-} - H_{+}\right)$. Therefore
    we obtain \eqref{wish to prove for g2 - g1} if $0 < \nu < 0.5$.

    Next assume that $-0.5 < \nu \leq 0$. Then $\nu > 1 + r/2 - n$ for all $\dfrac{r + 3}{2} \leq n \leq r$.
    Using Proposition~\ref{inductive steps}, we have the following shallow inductive steps
    \[
        z^{r - n}\mathbbm{1}_{p\Zp} \equiv \sum_{b = 1}^{p - 1}\sum_{j = 1}^{\frac{r - 1}{2}}c(n, j)b^{r - j - n}(z - b)^j\mathbbm{1}_{b + p\Zp} \mod \pi \latticeL{k}
    \]
    for all $\dfrac{r + 3}{2} \leq n \leq r$. Applying the operator $\sum_{\alpha = 0}^{p - 1}\alpha^{n - \frac{r - 1}{2}}\begin{pmatrix}1 & 0 \\ - a - \alpha p & p\end{pmatrix}$ to both sides,
    substituting $\lambda = \alpha + b$ and using \eqref{sum over b identity} with $i = \frac{r+1}{2}$
    to evaluate the sum over $b$,
    replacing $\lambda$ by $\alpha$, and substituting the value of $c(n, j)$ obtained in Proposition \ref{inductive steps},
    we get
    \begin{eqnarray}\label{Applied inductive steps in between -0.5 and 0}
        0 & \equiv & \sum_{\alpha = 0}^{p - 1}\sum_{j = r - n}^{\frac{r - 1}{2}}{n - \frac{r - 1}{2} \choose n - r + j}\frac{(r - n)!n}{(r - j) \cdots (n - j)}p^{r/2 - j}\alpha^{\frac{r+ 1}{2} - j}(z - a - \alpha p)^j\mathbbm{1}_{a + \alpha p + p^2\Zp} \nonumber \\
        && \qquad \qquad \mod \pi \latticeL{k}.
    \end{eqnarray}
    We remark that to get  the equations \eqref{Applied inductive steps in between -0.5 and 0} we have not used
    $\nu \leq 0$ and \eqref{Applied inductive steps in between -0.5 and 0} holds for $\nu > -0.5$.
    Again, for a fixed $a$ and $\alpha$, we write the term
    \[
        \sum_{j = 0}^{\frac{r - 1}{2}}{\frac{r + 1}{2} \choose j}p^{0.5 + x}\left(\cL - H_{\frac{r + 1}{2} - j}\right) p^{r/2 - j}\alpha^{\frac{r + 1}{2} - j}(z - a - \alpha p)^j \mathbbm{1}_{a + \alpha p + p^2\Zp}
    \]
    appearing in \eqref{g2 - g1 in between -0.5 and 0.5} as the sum of $x_n \in E$ times the corresponding term in \eqref{Applied inductive steps in between -0.5 and 0} and $x_{\frac{r + 1}{2}}$ times
    \[
        p^{0.5}\alpha(z - a - \alpha p)^{\frac{r - 1}{2}}\mathbbm{1}_{a + \alpha p + p^2\Zp}.
    \]
    Comparing the coefficients of like terms of\footnote{We drop the $p^{0.5 + x}$ here because we want the solutions $x_n$ to be integral. Indeed, the conditions $x + \nu = -0.5$ and $\nu \leq 0$ imply that the products $p^{0.5 + x}\left(\cL - H_{\frac{r + 1}{2} - j}\right)$ in \eqref{Matrix equation in between -0.5 and 0} are integers
      for all $0 \leq j \leq \dfrac{r - 1}{2}$.} $p^{r/2 - j}\alpha^{\frac{r + 1}{2} - j}(z - a - \alpha p)^j\mathbbm{1}_{a + \alpha p + p^2\Zp}$, we have the following matrix equation
    \begin{eqnarray}\label{Matrix equation in between -0.5 and 0}
        \!\!\!\!\!\!{\tiny\begin{pmatrix}
            {\frac{r + 1}{2} \choose 0}\frac{r}{r} & 0 & 0 & \cdots & 0 \\
            {\frac{r + 1}{2} \choose 1}\frac{r}{r - 1} & {\frac{r - 1}{2} \choose 0}\frac{r - 1}{(r - 1)(r - 2)} & 0 & \cdots & 0 \\
            {\frac{r + 1}{2} \choose 2}\frac{r}{r - 2} & {\frac{r - 1}{2} \choose 1}\frac{r - 1}{(r - 2)(r - 3)} & {\frac{r - 3}{2} \choose 0}\frac{2!(r - 2)}{(r - 2)(r - 3)(r - 4)} & \cdots & 0 \\
            \vdots & \vdots & \vdots & & \vdots \\
            {\frac{r + 1}{2} \choose \frac{r - 1}{2}}\frac{r}{\frac{r + 1}{2}} & {\frac{r - 1}{2} \choose \frac{r - 3}{2}}\frac{r - 1}{(\frac{r + 1}{2})(\frac{r - 1}{2})} & {\frac{r - 3}{2} \choose \frac{r - 5}{2}}\frac{2!(r - 2)}{(\frac{r + 1}{2})(\frac{r - 1}{2})(\frac{r - 3}{2})} & \cdots & 1
        \end{pmatrix} \!\!\!
        \begin{pmatrix}
            x_r \\ x_{r - 1} \\ x_{r - 2} \\ \vdots \\ x_{\frac{r + 1}{2}}
        \end{pmatrix} \!\!\! 
        = \!\!\!
        \begin{pmatrix}
            p^{0.5 + x}{\frac{r + 1}{2} \choose 0}\left(\cL - H_{\frac{r + 1}{2}}\right) \\ 
            p^{0.5 + x}{\frac{r + 1}{2} \choose 1}\left(\cL - H_{\frac{r - 1}{2}}\right) \\
            p^{0.5 + x}{\frac{r + 1}{2} \choose 2}\left(\cL - H_{\frac{r - 3}{2}}\right) \\
            \vdots \\ 
            p^{0.5 + x}{\frac{r + 1}{2} \choose \frac{r - 1}{2}}\left(\cL - H_{1}\right)
        \end{pmatrix}\!\!.}
    \end{eqnarray}
    Since the right side of this equation is a multiple of the right side of \eqref{Matrix equation in between 0 and 0.5}, we get
    \[
        x_{\frac{r + 1}{2}} = (-1)^{\frac{r - 1}{2}}p^{0.5 + x}(\cL - H_{-} - H_{+}).
    \]
    Therefore, even if $-0.5 < \nu \leq 0$, we have \eqref{wish to prove for g2 - g1}.

    Summarizing, we have shown that for $-0.5 < \nu < 0.5$, 
    \begin{eqnarray}\label{Penultimate g2 - g1 in between -0.5 and 0.5}
        g_2(z) - g_1(z) & \equiv & \sum_{a = 0}^{p - 1}\sum_{\alpha = 0}^{p - 1}(-1)^{\frac{r - 1}{2}}p^{0.5 + x}\left(\cL - H_{-} - H_{+}\right)p^{0.5}\alpha(z - a - \alpha p)^{\frac{r - 1}{2}}\mathbbm{1}_{a + \alpha p + p^2\Zp} \nonumber \\
        && \qquad \qquad \mod \pi \latticeL{k}.
    \end{eqnarray}
    Next consider the binomial expansion
    \begin{eqnarray}\label{Trivial equation in between -0.5 and 0.5}
        p^{-0.5}(z - a)^{\frac{r + 1}{2}}\mathbbm{1}_{a + p\Zp} & \equiv & \sum_{\alpha = 0}^{p - 1}\sum_{j = 0}^{\frac{r - 1}{2}}{\frac{r + 1}{2} \choose j}p^{r/2 - j}\alpha^{\frac{r + 1}{2} - j}(z - a - \alpha p)^j\mathbbm{1}_{a + \alpha p + p^2\Zp} \nonumber \\
        && \qquad \qquad \mod \pi \latticeL{k}.
    \end{eqnarray}
    Again we have dropped the $j = \dfrac{r + 1}{2}$ term by Lemma \ref{stronger bound for polynomials of large degree with varying radii}. For fixed $a$, $\alpha$, we wish to write
    \[
        (-1)^{\frac{r - 1}{2}}p^{0.5 + x}\left(\cL - H_{-} - H_{+}\right)p^{0.5}\alpha(z - a - \alpha p)^{\frac{r - 1}{2}}\mathbbm{1}_{a + \alpha p + p^2\Zp}
    \]
    appearing in \eqref{Penultimate g2 - g1 in between -0.5 and 0.5} as $x_n \in E$ times the corresponding term in \eqref{Applied inductive steps in between -0.5 and 0} and $x_{\frac{r + 1}{2}} \in E$ times the corresponding term in \eqref{Trivial equation in between -0.5 and 0.5}. Comparing coefficients of like terms in $p^{r/2 - j}\alpha^{\frac{r + 1}{2} - j}(z - a - \alpha p)^j\mathbbm{1}_{a + \alpha p + p^2\Zp}$, we get the following matrix equation
    \[\tiny
        \!\!\begin{pmatrix}
            {\frac{r + 1}{2} \choose 0}\frac{r}{r} & 0 & 0 & \cdots & {\frac{r + 1}{2} \choose 0} \\
            {\frac{r + 1}{2}  \choose 1}\frac{r}{r - 1} & {\frac{r - 1}{2} \choose 0}\frac{r - 1}{(r - 1)(r - 2)} & 0 & \cdots & {\frac{r + 1}{2} \choose 1} \\
            {\frac{r + 1}{2} \choose 2}\frac{r}{r - 2} & {\frac{r - 1}{2} \choose 1}\frac{r - 1}{(r - 2)(r - 3)} & {\frac{r - 3}{2} \choose 0}\frac{2!(r - 2)}{(r - 2)(r - 3)(r - 4)} & \cdots & {\frac{r + 1}{2} \choose 2} \\
            \vdots & \vdots & \vdots & & \vdots \\
            {\frac{r + 1}{2} \choose \frac{r - 1}{2}}\frac{r}{\frac{r + 1}{2}} & {\frac{r - 1}{2} \choose \frac{r - 3}{2}}\frac{r - 1}{(\frac{r + 1}{2})(\frac{r - 1}{2})} & {\frac{r - 3}{2} \choose \frac{r - 5}{2}}\frac{2!(r - 2)}{(\frac{r + 1}{2})(\frac{r - 1}{2})(\frac{r - 3}{2})} & \cdots & {\frac{r + 1}{2} \choose \frac{r - 1}{2}}
        \end{pmatrix} \!\!\! 
        \begin{pmatrix}
            x_r \\ x_{r - 1} \\ x_{r - 2} \\ \vdots \\ x_{\frac{r + 1}{2}}
        \end{pmatrix} \!\!\! 
        = \!\!\! 
        \begin{pmatrix}
            0 \\ 0 \\ 0 \\ \vdots \\ (-1)^{\frac{r - 1}{2}}\left(\cL - H_{-} - H_{+}\right)p^{0.5 + x}
        \end{pmatrix}\!\!.
    \]
    We have computed the determinant of the coefficient matrix appearing above while solving \eqref{Matrix equation in geq}. Taking $i = \frac{r+1}{2}$ there and using Cramer's rule, we get
    \[
        x_{\frac{r + 1}{2}} = p^{0.5 + x}(\cL - H_{-} - H_{+}).
    \]
    Therefore summing over $a$ and $\alpha$, we get
    \begin{eqnarray*}\label{Final g2 - g1 in between -0.5 and 0.5}
        g_2(z) - g_1(z) \equiv \sum_{a = 0}^{p - 1}p^{0.5 + x}(\cL - H_{-} - H_{+})p^{-0.5}(z - a)^{\frac{r + 1}{2}}\mathbbm{1}_{a + p\Zp} \mod \pi \latticeL{k}.
    \end{eqnarray*}

    This in fact proves the proposition, since it is easy to see that
    $g_1(z) \equiv 0 \mod \pi \latticeL{k}$. Indeed, 
    \[
        g_1(z) = \sum_{a = 0}^{p - 1}\sum_{j = 0}^{\frac{r - 1}{2}}\frac{g^{(j)}(a)}{j!}(z - a)^j\mathbbm{1}_{a + p\Zp},
    \]
    where
    \begin{eqnarray*}\label{jth summand in g1 in between -0.5 and 0.5}
        \frac{g^{(j)}(a)}{j!} = {\frac{r + 1}{2} \choose j}\sum_{i = 0}^{p - 1}p^x\lambda_i(a - i)^{\frac{r + 1}{2} - j}\logL(a - i).
    \end{eqnarray*}
    For any $0 \leq a \leq p - 1$, each summand on the right side of the equation above is $0$ mod $\pi$.
    This is clear if $i = a$ and if $i \neq a$, then $a - i$ is a unit and the valuation of the $i^{\mathrm{th}}$ summand on the right side of \eqref{jth summand in g1 in between -0.5 and 0.5} is $x + 1 > 0$.
    Therefore by Lemma \ref{Integers in the lattice}, we get
        $g_1(z) \equiv 0 \mod \pi \latticeL{k}$.
    %
\end{proof}

\begin{theorem}\label{Final theorem for -0.5 < nu < 0.5 for odd weights}
    Let $1 \leq r \leq p - 2$ be odd. If $-0.5 < \nu < 0.5$, then the map $\IZind a^{\frac{r - 1}{2}}d^{\frac{r + 1}{2}} \twoheadrightarrow F_{r-1, \> r}$ factors as
    \[
        \IZind a^{\frac{r - 1}{2}}d^{\frac{r + 1}{2}} \twoheadrightarrow \frac{\IZind a^{\frac{r - 1}{2}}d^{\frac{r + 1}{2}}}{\im T_{-1, 0}} \twoheadrightarrow F_{r - 1, \> r}.
    \]
\end{theorem}
\begin{proof}
    We claim that under the map $\IZind a^{\frac{r - 1}{2}}d^{\frac{r + 1}{2}} \twoheadrightarrow F_{r-1, \> r}$, the element
    \[
        -wT_{-1, 0}\llbracket \id, X^{\frac{r - 1}{2}}Y^{\frac{r + 1}{2}}\rrbracket
    \]
    maps to $0$. The theorem then follows since $\llbracket \id, X^{\frac{r - 1}{2}}Y^{\frac{r + 1}{2}}\rrbracket$ generates $\IZind a^{\frac{r - 1}{2}}d^{\frac{r + 1}{2}}$.
    Indeed, by Proposition \ref{nu in between -0.5 and 0.5}, we have
    \[
        \sum_{a = 0}^{p - 1}(\cL - H_{-} - H_{+})p^{0.5 + x}p^{-0.5}(z - a)^{\frac{r + 1}{2}}\mathbbm{1}_{a + p\Zp} \equiv 0 \mod \pi \latticeL{k}
      \]
    where $x + \nu = -0.5$. Note that $(\cL - H_{-} - H_{+})p^{0.5 + x}$ is a $p$-adic unit.
    And, using \eqref{Formulae for T-10 and T12 in the commutative case},  we
    see that $-wT_{-1, 0}\llbracket \id, X^{\frac{r - 1}{2}}Y^{\frac{r + 1}{2}}\rrbracket$ maps to
        \begin{eqnarray*}
        \sum_{\lambda \in I_1}-\begin{pmatrix}0 & 1 \\ p & -\lambda\end{pmatrix}z^{\frac{r - 1}{2}}\mathbbm{1}_{p\Zp}
        & = & \sum_{\lambda \in I_1} - \frac{1}{p^{r/2}}(z - \lambda)^r\left(\frac{p}{z - \lambda}\right)^{\frac{r - 1}{2}}\mathbbm{1}_{p\Zp}\left(\frac{p}{z - \lambda}\right) \\
        & = & \sum_{a = 0}^{p - 1} p^{-0.5}(z - a)^{\frac{r + 1}{2}}\mathbbm{1}_{a + p\Zp} \mod \pi \latticeL{k}.
    \end{eqnarray*}
    Therefore we have proved the theorem.
\end{proof}

\section{Analysis of $\br{\latticeL{k}}$ around the last point for even weights}\label{SpecEven}

Finally, we assume that $k$ and $r$ are even and describe the behavior of $\br{\latticeL{k}}$
around the last point $\nu = 0$. Some results on Hecke operators in the non-commutative Hecke algebra
that were proved in Section~\ref{Comparison theorem section} are needed, making the analysis trickier.

\subsection{The $\nu \geq 0$ case for even weights}\label{geq section for even weights}
    In this section, we prove that if $\nu \geq 0$, then the map $\IZind a^{\frac{r - 2}{2}}d^{\frac{r + 2}{2}} \twoheadrightarrow F_{r - 2, \> r - 1}$ factors as
    \[
        \IZind a^{\frac{r - 2}{2}}d^{\frac{r + 2}{2}} \twoheadrightarrow \frac{\IZind a^{\frac{r - 2}{2}}d^{\frac{r + 2}{2}}}{\im(\lambda_{r/2}T_{-1, 0} - 1)} \twoheadrightarrow F_{r - 2, \> r - 1},
    \]
    where $\lambda_{r/2} = (-1)^{r/2}\left(\dfrac{r}{2}\right)\left(\dfrac{r + 2}{2}\right)(\cL - H_{-} - H_{+})$.
        Moreover if $\nu = 0$, then there is a surjection
        \[
            \pi(p - 3, \lambda_{r/2}^{-1}, \omega^{\frac{r + 2}{2}}) \twoheadrightarrow F_{r - 2, \> r - 1}.
        \]

    \begin{Proposition}\label{nu geq 0}
        Let $2 \leq r \leq p-1$ be even. Let $\nu \geq 0$. Set
        \[
            g(z) = \sum_{i = 0}^{p - 1}p^{-1}\lambda_i(z - i)^{\frac{r + 2}{2}}\logL(z - i),
        \]
        where the $\lambda_i \in \zp$ are as in Lemma~\ref{Main coefficient identites} (3).
        Then
        \begin{eqnarray*}
            g(z) & \equiv & \sum_{a = 0}^{p - 1}\sum_{\alpha = 0}^{p - 1}\sum_{j = 0}^{r/2}{\frac{r + 2}{2} \choose j}\left(\cL - H_{\frac{r + 2}{2} - j}\right)p^{-1}(\alpha p)^{\frac{r + 2}{2} - j} (z - a - \alpha p)^{j}\mathbbm{1}_{a + \alpha p + p^2\Zp} \\
            && + \sum_{a = 0}^{p - 1}\sum_{j = 0}^{r/2}{\frac{r + 2}{2} \choose j}\frac{a^{\frac{r + 2}{2} - j}}{\frac{r + 2}{2} - j}(z - a)^j\mathbbm{1}_{a + p\Zp} \\
            && + \sum_{a = 0}^{p - 1}\sum_{m = 1}^{\frac{r + 2}{2}}\sum_{j = 0}^{\frac{r + 2}{2} - m}{\frac{r + 2}{2} \choose j}{\frac{r + 2}{2} - j \choose m}(-1)^m\frac{S_m}{p}\frac{a^{\frac{r + 2}{2} - j - m}}{\frac{r + 2}{2} - j}(z - a)^j\mathbbm{1}_{a + p\Zp} \!\!\! \mod \pi \latticeL{k}.
        \end{eqnarray*}
    \end{Proposition}
    \begin{proof}
        Recall that by Lemma~\ref{Main coefficient identites} (3), we have
        $\lambda_i = 1 \!\! \mod p$ for $0 \leq i \leq p - 1$ and
        $\sum_{i = 0}^{p - 1}\lambda_i i^j = 0$ for $0 \leq j \leq \dfrac{r + 2}{2} \leq p - 2$ since $p \geq 5$.         
        Write $g(z) = g(z)\mathbbm{1}_{\Zp} + g(z)\mathbbm{1}_{\Qp \setminus \Zp}$. As usual, we have
        \[
            w \cdot g(z)\mathbbm{1}_{\Qp \setminus \Zp} = \sum_{i = 0}^{p - 1}p^{-1}\lambda_iz^{\frac{r - 2}{2}}(1 - zi)^{\frac{r + 2}{2}}\logL(1 - zi)\mathbbm{1}_{p\Zp}.
        \]
        Using Lemma \ref{qp-zp part is 0}, we see $g(z)\mathbbm{1}_{\Qp \setminus \Zp} \equiv 0 \mod \pi \latticeL{k}$. Therefore $g(z) \equiv g(z)\mathbbm{1}_{\Zp} \! \! \mod \pi \latticeL{k}$.
        
        Now using Lemma \ref{telescoping lemma} (1), we see that $g(z)\mathbbm{1}_{\Zp} \equiv g_2(z) \mod \pi \latticeL{k}$. As usual, we analyze $g_2(z) - g_1(z)$ and $g_1(z)$ separately.

        First consider $g_2(z) - g_1(z)$. For $0 \leq j \leq r/2$, the coefficient of $(z - a - \alpha p)^j\mathbbm{1}_{a + \alpha p + p^2\Zp}$ in $g_2(z) - g_1(z)$ is
        \begin{eqnarray}\label{jth summand in g2 - g1 in geq 0}
            &&{\frac{r + 2}{2} \choose j}\sum_{i = 0}^{p - 1}p^{-1}\lambda_i\Big[(a + \alpha p - i)^{\frac{r + 2}{2} - j}\logL(a + \alpha p - i) - (a - i)^{\frac{r + 2}{2} - j}\logL(a - i) \\
            && {\frac{r + 2}{2} - j \choose 1}(\alpha p)(a - i)^{r/2 - j}\logL(a - i) - \cdots - {\frac{r + 2}{2} - j \choose r/2 - j}(\alpha p)^{r/2 - j}(a - i)\logL(a - i)\Big].\nonumber
        \end{eqnarray}
      Arguing as in previous proofs by analyzing the $i = a$ terms and $i \neq a$ separately
      and using $\sum_{i = 0}^{p - 1}\lambda_i i^j = 0$ for $0 \leq j \leq p - 2$,
      one see that the above coefficient is 
        \[
            {\frac{r + 2}{2} \choose j}p^{-1}\lambda_a(\alpha p)^{\frac{r + 2}{2} - j}\left(\cL - H_{\frac{r + 2}{2} - j}\right) \mod p^{r/2 - j}\pi.
        \]
        Since $\lambda_a \equiv 1 \!\! \mod p$ for $0 \leq a \leq p - 1$,
        using Lemma~\ref{Integers in the lattice}, we see that
        \begin{eqnarray}\label{Penultimate g2 - g1 in geq 0}
            g_2(z) - g_1(z) & \equiv & \sum_{a = 0}^{p - 1}\sum_{\alpha = 0}^{p - 1}\sum_{j = 0}^{r/2}{\frac{r + 2}{2} \choose j}\left(\cL - H_{\frac{r + 2}{2} - j}\right)p^{-1}(\alpha p)^{\frac{r + 2}{2} - j} \\
            && \quad \quad \quad \quad \quad \quad \qquad (z - a - \alpha p)^{j}\mathbbm{1}_{a + \alpha p + p^2\Zp} \mod \pi \latticeL{k}. \nonumber
        \end{eqnarray}

        Next, we simplify $g_1(z)$. Recall that
        \begin{eqnarray}\label{g1 in geq 0}
            g_1(z) = \sum_{a = 0}^{p - 1}\sum_{j = 0}^{r/2}\frac{g^{(j)}(a)}{j!}(z - a)^j\mathbbm{1}_{a + p\Zp},
        \end{eqnarray}
        where
        \begin{eqnarray}\label{jth summand in g1 in geq 0}
            \frac{g^{(j)}(a)}{j!} = {\frac{r + 2}{2} \choose j}\sum_{i = 0}^{p - 1}p^{-1}\lambda_i(a - i)^{\frac{r + 2}{2} - j}\logL(a - i).
        \end{eqnarray}

        The $i = a$ summand is $0$. For $i \neq a$, we write $a - i = [a - i] + pa_i$ for some $a_i \in \Zp$. Then $\logL(a - i) = \logL(1 + [a - i]^{-1}pa_i)$ since $[a - i]$ is a root of unity. Expanding the logarithm,
        \begin{eqnarray*}
            \frac{g^{(j)}(a)}{j!} & \equiv & {\frac{r + 2}{2} \choose j}\sum_{\substack{i = 0 \\ i \neq a}}^{p - 1}p^{-1}\lambda_i[a - i]^{r/2 - j}pa_i \mod \pi \\
            & \equiv & {\frac{r + 2}{2} \choose j}\sum_{\substack{i = 0 \\ i \neq a}}^{p - 1}p^{-1}\lambda_i\left(\frac{(a - i)^{\frac{r + 2}{2} - j} - [a - i]^{\frac{r + 2}{2} - j}}{\frac{r + 2}{2} - j}\right) \mod \pi.
        \end{eqnarray*}
        Recall that $\lambda_i \equiv 1 \!\! \mod p$. Moreover, $0 \leq j \leq r/2$ implies that the power of $[a - i]$ in the last line above is positive. Using \eqref{sum of powers of roots of unity}, we get\footnote{Since $\sum_{i = 0}^{p - 1}[a-i]^{p - 1} = p-1$, the quotient $\dfrac{S_m}{p}$ might possibly
          be replaced by $\dfrac{S_m - (p-1)}{p}$ when $m = p - 1$. However this case only
          arises for odd primes $p$ when $p = 3$, $r = 2$ and $j = 0$ and so is excluded since $p \geq 5$.}
        \[
            \frac{g^{(j)}(a)}{j!} \equiv {\frac{r + 2}{2} \choose j}\sum_{i = 0}^{p - 1}p^{-1}\frac{(a - i)^{\frac{r + 2}{2} - j}}{\frac{r + 2}{2} - j} \equiv {\frac{r + 2}{2} \choose j}\sum_{m = 0}^{\frac{r + 2}{2} - j}{\frac{r + 2}{2} - j \choose m}(-1)^m\frac{S_m}{p}\frac{a^{\frac{r + 2}{2} - j - m}}{\frac{r + 2}{2} - j} \mod \pi.
        \]
        Substituting this in \eqref{g1 in geq 0}, we get
        \[
            g_1(z) \equiv \sum_{a = 0}^{p - 1}\sum_{j = 0}^{r/2}\sum_{m = 0}^{\frac{r + 2}{2} - j}{\frac{r + 2}{2} \choose j}{\frac{r + 2}{2} - j \choose m}(-1)^m\frac{S_m}{p}\frac{a^{\frac{r + 2}{2} - j - m}}{\frac{r + 2}{2} - j}(z - a)^j\mathbbm{1}_{a + p\Zp} \mod \pi \latticeL{k}.
        \]
        We separate the $m = 0$ part from this sum since it will be the important summand. In the $m \neq 0$ part, we interchange the order of the $j$ and $m$ sums to get
        \begin{eqnarray}\label{Intermediate g1 in geq 0}
            g_1(z) \!\! & \equiv & \!\! \sum_{a = 0}^{p - 1}\sum_{j = 0}^{r/2}{\frac{r + 2}{2} \choose j}\frac{a^{\frac{r + 2}{2} - j}}{\frac{r + 2}{2} - j}(z - a)^j\mathbbm{1}_{a + p\Zp} \\
            && \!\! + \sum_{a = 0}^{p - 1}\sum_{m = 1}^{\frac{r + 2}{2}}\sum_{j = 0}^{\frac{r + 2}{2} - m}\!\!{\frac{r + 2}{2} \choose j}\!{\frac{r + 2}{2} - j \choose m}\!(-1)^m\frac{S_m}{p}\frac{a^{\frac{r + 2}{2} - j - m}}{\frac{r + 2}{2} - j}(z - a)^j\mathbbm{1}_{a + p\Zp} \!\!\! \mod \pi \latticeL{k}. \nonumber
        \end{eqnarray}
Adding equations \eqref{Intermediate g1 in geq 0} and \eqref{Penultimate g2 - g1 in geq 0} yields
the proposition.
    \end{proof}

    The next lemma is used in the proof of Theorem~\ref{Final theorem for geq 0}.
    
    \begin{lemma}\label{T12 is 0 in even self-dual}
        Let $2 \leq r \leq p-1$ be even. Let $\nu > -1$. For $0 \leq l \leq \dfrac{r - 2}{2}$, we have
        \[
            \sum_{a = 0}^{p - 1}(z - a)^{l}\mathbbm{1}_{a + p\Zp} \equiv  0 \mod \pi \latticeL{k}.
        \]
    \end{lemma}
    \begin{proof}
    If $r = 2$, then $l$ can only be $0$ and 
        $\sum_{a = 0}^{p - 1} \mathbbm{1}_{a + p\Zp} = \mathbbm{1}_{\Zp}$.
    The equation
    \[
        \begin{pmatrix}p & 0 \\ 0 & 1\end{pmatrix} p\mathbbm{1}_{p\Zp} = \mathbbm{1}_{\Zp}
    \]
    implies that $\mathbbm{1}_{\Zp} \equiv 0 \mod \pi \latticeL{k}$ by Lemma \ref{Integers in the lattice}.
    Therefore the lemma is true for $r = 2$. So we may assume that $r \geq 4$.
    
    First assume $0 \leq l \leq \dfrac{r - 4}{2}$. By Proposition \ref{inductive steps}, we have the shallow inductive step
    \[
        z^{r - n}\mathbbm{1}_{p\Zp} \equiv \sum_{b = 1}^{p - 1}\sum_{j = 1}^{r/2}c(n, j)b^{r - j - n}(z - b)^j\mathbbm{1}_{b + p\Zp} \mod \pi \latticeL{k}
    \]
    for $n = r - l \geq \dfrac{r + 4}{2}$. Applying the operator $\sum\limits_{a = 0}^{p - 1}a^{n - r + l}\begin{pmatrix}1 & 0 \\ -a & 1\end{pmatrix}$ to both sides, 
    substituting $\lambda = a + b$ and summing over $b$ on the right side
    using \eqref{sum over b identity}, 
    replacing $\lambda$ by $a$, 
    and substituting
    the value of $c(r-l, l)$ obtained in Proposition \ref{inductive steps},
    we get
    \[
        0 \equiv \sum_{a = 0}^{p - 1}\frac{l!(r - l)}{(r - l) \cdots (r - 2l)}(z - a)^l\mathbbm{1}_{a + p\Zp} \mod \pi \latticeL{k}.
    \]
    Since $l \leq \dfrac{r - 4}{2} \leq p - 1$ and $r \leq p - 1$, the coefficient of the expression above is a unit so that cancelling it we obtain the lemma for $0 \leq l \leq \dfrac{r - 4}{2}$.

    Next assume that $l = \dfrac{r - 2}{2}$.  The proof of the lemma in this case is similar
    to the proof of Proposition \ref{nu leq}. Set
    \[
            g(z) = p^{-1}\left[\sum_{i \in I} \lambda_i (z - i)^{\frac{r + 4}{2}}\logL(z - i)\right],
        \]
        where $I = \{1,2, \ldots, \dfrac{r+4}{2}, p\}$ and $\lambda_i \in \zp$ for $i \in I$
        are as in Lemma~\ref{Main coefficient identites} (2). Thus
    \begin{eqnarray}
          \label{Sum identity in killing T12 in geq 0}
        \lambda_0 \equiv 1 \!\!\!\!\mod p, \lambda_i \equiv 0 \!\!\!\!\mod p \text{ for } 1 \leq i \leq \dfrac{r + 4}{2}, \lambda_{p} = - 1,  \sum_{i \in I}\lambda_i z_i^j = 0 \text{ for } 0 \leq j \leq \frac{r + 4}{2}.
    \end{eqnarray}
    Write $g(z) = g(z)\mathbbm{1}_{\Zp} + g(z)\mathbbm{1}_{\Qp \setminus \Zp}$. As usual
    \[
        w \cdot g(z)\mathbbm{1}_{\Qp \setminus \Zp} = \sum_{i \in I}p^{-1}\lambda_iz^{\frac{r - 4}{2}}(1 - zi)^{\frac{r + 4}{2}}\logL(1 - zi)\mathbbm{1}_{p\Zp}.
    \]
    Lemma \ref{qp-zp part is 0} shows that $g(z)\mathbbm{1}_{\Qp \setminus \Zp} \equiv 0 \mod \pi \latticeL{k}$. Thus, $g(z) \equiv g(z)\mathbbm{1}_{\Zp} \!\!\! \mod \pi \latticeL{k}$.

    Using Lemma \ref{telescoping lemma} (3), we see $g(z)\mathbbm{1}_{\Zp} \equiv g_2(z) \mod \pi \latticeL{k}$. We analyze $g_2(z) - g_1(z)$ and $g_1(z)$ separately.

    Recall that for $0 \leq j \leq \dfrac{r + 2}{2}$, the coefficient of $(z - a - \alpha p)^j\mathbbm{1}_{a + \alpha p + p^2\Zp}$ in $g_2(z) - g_1(z)$ is
    \begin{eqnarray}\label{jth summand in g2 - g1 in killing T12 in geq 0}
            && {\frac{r + 4}{2} \choose j}\sum_{i \in I}p^{-1}\lambda_i\Big[(a + \alpha p - i)^{\frac{r + 4}{2} - j}\logL(a + \alpha p - i) - (a - i)^{\frac{r + 4}{2} - j}\logL(a - i) \\
            && - {\frac{r + 4}{2} - j \choose 1}(\alpha p)(a - i)^{\frac{r + 2}{2} - j}\logL(a - i) - \cdots - {\frac{r + 4}{2} - j \choose \frac{r + 2}{2} - j}(\alpha p)^{\frac{r + 2}{2} - j}(a - i)\logL(a - i)\Big]. \nonumber
        \end{eqnarray}
    We prove that this coefficient is congruent to $0$ modulo $p^{r/2 - j}\pi$. 
    It follows, by Lemma~\ref{Integers in the lattice}, that
    $g_2(z) - g_1(z) \equiv 0$ modulo $\pi \latticeL{k}$.

    First assume that $a \neq 0$. For $i \neq a$, the $i^{\mathrm{th}}$ summand in expression \eqref{jth summand in g2 - g1 in killing T12 in geq 0} can be written as
        \[
            p^{-1}\lambda_i\big[(a - z_i + \alpha p)^{\frac{r + 4}{2} - j}\logL(1 + (a - z_i)^{-1}\alpha p) + (\alpha p)^{\frac{r + 4}{2} - j}\logL(a - z_i)\big].
        \]
        Since $i \neq a$ and $1 \leq a \leq p - 1$, we see that $a - i$ is a $p$-adic unit. Therefore the valuation of the last term in the expression above is greater than $r/2 - j$. So we drop the last term modulo $p^{r/2 - j}\pi$. Moreover, expanding $(a - i + \alpha p)^{\frac{r + 4}{2} - j}$ and the logarithm, the expression above becomes
        \begin{eqnarray}\label{a neq 0 and zi neq a term in killing T12 in geq 0}
            &&p^{-1}\lambda_i\big[(a - i)^{\frac{r + 2}{2} - j}\alpha p + c_{r/2 - j}(a - i)^{r/2 - j}(\alpha p)^2 + \cdots + c_1(a - i)(\alpha p)^{\frac{r + 2}{2} - j}\big] \\
            && \qquad \qquad \mod p^{r/2 - j}\pi, \nonumber
        \end{eqnarray}
        for some $c_1, \ldots, \> c_{r/2 - j} \in \Zp$.
        Next assume that $i = a$. Then the $i^{\mathrm{th}}$ summand in expression \eqref{jth summand in g2 - g1 in killing T12 in geq 0} can be written as
        \[
            p^{-1}\lambda_i(\alpha p)^{\frac{r + 4}{2} - j}\logL(\alpha p).
        \]
        Since $v_p(\cL) > -1$, the valuation of this term is greater than $r/2 - j$. Therefore the $i = a$ term in expression \eqref{jth summand in g2 - g1 in killing T12 in geq 0} is congruent to $0$ modulo $p^{r/2 - j}\pi$.
        Plugging \eqref{a neq 0 and zi neq a term in killing T12 in geq 0} into \eqref{jth summand in g2 - g1 in killing T12 in geq 0} and using the summation identities in \eqref{Sum identity in killing T12 in geq 0}, we see that the coefficient of $(z - a - \alpha p)^j\mathbbm{1}_{a + \alpha p + p^2 \Zp}$ in $g_2(z) - g_1(z)$ is $0$
        modulo $p^{r/2 - j}\pi$ if $a \neq 0$.

        Now assume that $a = 0$. Using $v_p(\cL) > -1$, we see that each term in expression \eqref{jth summand in g2 - g1 in killing T12 in geq 0} for $i = 0$ and $p$ has valuation greater than $r/2 - j$. Therefore for $a = 0$, the $i = 0$ and $p$ summands in expression \eqref{jth summand in g2 - g1 in killing T12 in geq 0} are $0$ modulo $p^{r/2 - j}\pi$. For $i \neq 0, p$, the $i^{\mathrm{th}}$ summand in expression \eqref{jth summand in g2 - g1 in killing T12 in geq 0} is also given by \eqref{a neq 0 and zi neq a term in killing T12 in geq 0}.
        
        Using the sums in \eqref{Sum identity in killing T12 in geq 0}, we see the coefficient of $(z - a - \alpha p)^j\mathbbm{1}_{a + \alpha p + p^2 \Zp}$ in $g_2(z) - g_1(z)$ is
        \[
            -{\frac{r + 4}{2} \choose j}p^{-1}\lambda_{p}\left[(-p)^{\frac{r + 2}{2} - j}\alpha p + c_{r/2 - j}(-p)^{r/2 - j}(\alpha p)^2 + \cdots + c_1(-p)(\alpha p)^{\frac{r + 2}{2} - j}\right] \mod p^{r/2 - j}\pi.
        \]
        Clearly, each term in the expression above has valuation greater than $r/2 - j$, as desired.

        Next consider
        \begin{eqnarray}\label{g1 in killing T12 in geq 0}
            g_1(z) = \sum_{a = 0}^{p - 1}\sum_{j = 0}^{\frac{r + 2}{2}}\frac{g^{(j)}(a)}{j!}(z - a)^j\mathbbm{1}_{a + p\Zp},
        \end{eqnarray}
        where
        \begin{eqnarray}\label{jth summand in g1 in killing T12 in geq 0}
            \frac{g^{(j)}(a)}{j!} = {\frac{r + 4}{2} \choose j}\sum_{i \in I}p^{-1}\lambda_i(a - i)^{\frac{r + 4}{2} - j}\logL(a - i).
        \end{eqnarray}
        
        First assume that $a \neq 0$. The $i \neq 0, a, p$ terms on the right side of equation \eqref{jth summand in g1 in killing T12 in geq 0} are congruent to $0$ modulo $\pi$ since $a \not \equiv i \!\! \mod p$ and $p \mid \lambda_i$ for $i \neq 0, \> p$ by \eqref{Sum identity in killing T12 in geq 0}. The $i = a$ term is equal to $0$. The sum of the $i = 0$ and $i = p$ summands on the right side of equation \eqref{jth summand in g1 in killing T12 in geq 0} is
        \[
            p^{-1} \left[\lambda_0a^{\frac{r + 4}{2} - j}\logL(a) + \lambda_{p}(a - p)^{\frac{r + 4}{2} - j}\logL(a - p)\right].
        \]
        By \eqref{Sum identity in killing T12 in geq 0}, we have $\lambda_0 = 1 \!\! \mod p$ and $\lambda_{p} = -1$. Therefore the display above is congruent to 
        \[
            p^{-1}\left[a^{\frac{r + 4}{2} - j}\logL(a) - (a - p)^{\frac{r + 4}{2} - j}\logL(a - p)\right] \mod \pi.
        \]
        We expand $(a - p)^{\frac{r + 4}{2} - j}$ and drop the terms divisible by $p$ since $p \> \vert \> \logL(a - p)$ to get
        \[
            p^{-1}(-a^{\frac{r + 4}{2} - j})\logL(1 - a^{-1}p) \mod \pi.
        \]
        Expanding $\logL(1 - a^{-1}p)$ using the usual Taylor series and dropping the terms that are congruent to $0$ mod $\pi$, we see the sum of the $i = 0, \> p$ summands in \eqref{jth summand in g1 in killing T12 in geq 0} is
            $a^{\frac{r + 2}{2} - j} \mod \pi$.
        Therefore for $a \neq 0$, we get
        \begin{eqnarray}\label{jth term in a neq 0 in g1 in killing T12 in geq 0}
            \frac{g^{(j)}(a)}{j!} \equiv {\frac{r + 4}{2} \choose j}a^{\frac{r + 2}{2} - j} \mod \pi.
        \end{eqnarray}

        Next, assume that $a = 0$. Then again the $i \neq 0, \> p$ terms on the right side of equation \eqref{jth summand in g1 in killing T12 in geq 0} are congruent to $0$ modulo $\pi$ by the same reasoning as in the $a \neq 0$ case above. The $i = 0$ term is $0$. The $i = p$ summand in equation \eqref{jth summand in g1 in killing T12 in geq 0} is
         $   -p^{-1}(-p)^{\frac{r + 4}{2} - j}\cL$.
        Therefore for $a = 0$, we have
        \begin{eqnarray}\label{jth term in a = 0 in g1 in killing T12 in geq 0}
            \frac{g^{(j)}(0)}{j!} \equiv (-1)^{\frac{r + 6}{2} - j}{\frac{r + 4}{2} \choose j}p^{\frac{r + 2}{2} - j}\cL \mod \pi.
        \end{eqnarray}
        Putting equations \eqref{jth term in a neq 0 in g1 in killing T12 in geq 0} and \eqref{jth term in a = 0 in g1 in killing T12 in geq 0} in equation \eqref{g1 in killing T12 in geq 0} and using Lemma \ref{Integers in the lattice}, we get
        \begin{eqnarray*}
            g_1(z) & \equiv & \sum_{a = 1}^{p - 1}\left[\sum_{j = 0}^{\frac{r + 2}{2}}{\frac{r + 4}{2} \choose j}a^{\frac{r + 2}{2} - j}(z - a)^j\right]\mathbbm{1}_{a + p\Zp} + \left[\sum_{j = 0}^{\frac{r + 2}{2}}(-1)^{\frac{r + 6}{2} - j}{\frac{r + 4}{2} \choose j}p^{\frac{r + 2}{2} - j}\cL z^j\right]\mathbbm{1}_{p\Zp} \\
            & \equiv & \sum_{a = 1}^{p - 1}a^{-1}\left[z^{\frac{r + 4}{2}} - (z - a)^{\frac{r + 4}{2}}\right]\mathbbm{1}_{a + p\Zp} + \left[\sum_{j = 0}^{\frac{r + 2}{2}}(-1)^{\frac{r + 6}{2} - j}{\frac{r + 4}{2} \choose j}p^{\frac{r + 2}{2} - j}\cL z^j\right]\mathbbm{1}_{p\Zp} \\
            && \qquad \qquad \mod \pi \latticeL{k}.
        \end{eqnarray*}
        Using Lemma \ref{stronger bound for polynomials of large degree with varying radii} for the first sum and for the $j = \dfrac{r + 2}{2}$ summand in the second term, we get
        \[
            g_1(z) \equiv \sum_{a = 1}^{p - 1}a^{-1}z^{\frac{r + 4}{2}}\mathbbm{1}_{a + p\Zp} + \left[\sum_{j = 0}^{r/2}(-1)^{\frac{r + 6}{2} - j}{\frac{r + 4}{2} \choose j}p^{\frac{r + 2}{2} - j}\cL z^j\right]\mathbbm{1}_{p\Zp} \mod \pi \latticeL{k}.
        \]
        Finally, Lemma \ref{Integers in the lattice} implies that the second sum above is $0$ modulo $\pi \latticeL{k}$. Therefore
        \[
            g_1(z) \equiv \sum_{a = 1}^{p - 1}a^{-1}z^{\frac{r + 4}{2}}\mathbbm{1}_{a + p\Zp} \mod \pi \latticeL{k}.
          \]
          
        Since $g(z) \equiv g_1(z) \mod \pi \latticeL{k}$, we obtain
        \[
            0 \equiv \sum_{a = 1}^{p - 1}a^{-1}z^{\frac{r + 4}{2}}\mathbbm{1}_{a + p\Zp} \mod \pi \latticeL{k}.
        \]
        Applying $\beta$, 
        we get
        \[
            0 \equiv \sum_{a = 1}^{p - 1} a^{-1}p^{2}z^{\frac{r - 4}{2}}\mathbbm{1}_{a^{-1}p + p^2\Zp} \mod \pi \latticeL{k}.
        \]
        Expanding $z^{\frac{r - 4}{2}} = \sum_{j = 0}^{\frac{r - 4}{2}}{\frac{r - 4}{2} \choose j}(a^{-1}p)^{\frac{r - 4}{2} - j}(z - a^{-1}p)^j$, replacing $a^{-1}$ by $a$, and including the superfluous $a = 0$ term, we get
        \begin{eqnarray}\label{Main congruence in killing T12 in geq 0}
            0 \equiv \sum_{a = 0}^{p - 1}\sum_{j = 0}^{\frac{r - 4}{2}}{\frac{r - 4}{2} \choose j}p^{r/2 - j}a^{\frac{r - 2}{2} - j}(z - ap)^j\mathbbm{1}_{ap + p^2\Zp} \mod \pi \latticeL{k}.
        \end{eqnarray}

        Next, the shallow inductive steps given in Proposition \ref{inductive steps} are
        \[
            z^{r - n}\mathbbm{1}_{p\Zp} \equiv \sum_{b = 1}^{p - 1}\sum_{j = 1}^{r/2}c(n, j)b^{r - j - n}(z - b)^j\mathbbm{1}_{b + p\Zp} \mod \pi \latticeL{k},
        \]
        for $\dfrac{r + 4}{2} \leq n \leq r$. Applying $\sum\limits_{a = 0}^{p - 1}a^{n - \frac{r + 2}{2}}\begin{pmatrix}1 & 0 \\ -pa & p\end{pmatrix}$ to both sides, 
        substituting $\lambda = a + b$ and evaluating the sum over $b$ using the identity \eqref{sum over b identity} with $i = \dfrac{r + 4}{2}$, 
        replacing $\lambda$ by $a$, and substituting the value of $c(n, j)$ obtained in Proposition \ref{inductive steps}, 
        we get
        \begin{eqnarray}\label{Refined inductive steps in killing T12 in geq 0}
            0 & \equiv & \sum_{a = 0}^{p - 1}\sum_{j = r - n}^{\frac{r-2}{2}}{n - \frac{r + 2}{2} \choose n - r + j}\frac{(r - n)! n}{(r - j)\cdots (n - j)}p^{r/2 - j}a^{\frac{r - 2}{2} - j}(z - ap)^j\mathbbm{1}_{ap + p^2\Zp} \\ 
            && \qquad \qquad \mod \pi \latticeL{k} \nonumber
        \end{eqnarray}
        for $\dfrac{r + 4}{2} \leq n \leq r$. We have dropped the $j = \dfrac{r}{2}$ term since
        the binomial coefficient vanishes.
        
        We wish to write $p(z - ap)^{\frac{r - 2}{2}}\mathbbm{1}_{ap + p^2\Zp}$ as a sum of $x_n \in E$ times the $a^{\mathrm{th}}$ summand on the right side of equation \eqref{Refined inductive steps in killing T12 in geq 0} for $\dfrac{r + 4}{2} \leq n \leq r$ and $x_{\frac{r + 2}{2}} \in E$ times the corresponding summand in equation \eqref{Main congruence in killing T12 in geq 0}. To this end, we compare the coefficients of $p^{r/2 - j}a^{\frac{r - 2}{2} - j}(z - ap)^{j}\mathbbm{1}_{ap + p^2\Zp}$ for $0 \leq j \leq \dfrac{r-2}{2}$ to get the following matrix equation
        \begin{eqnarray}
            \begin{pmatrix}
                {\frac{r - 2}{2} \choose 0}\frac{r}{r} & 0 & 0 & \cdots & {\frac{r - 4}{2} \choose 0} \\
                {\frac{r - 2}{2} \choose 1}\frac{r}{r - 1} & {\frac{r - 4}{2} \choose 0}\frac{r - 1}{(r - 1)(r - 2)} & 0 & \cdots & {\frac{r - 4}{2} \choose 1} \\
                {\frac{r - 2}{2} \choose 2}\frac{r}{r - 2} & {\frac{r - 4}{2} \choose 1}\frac{r - 1}{(r - 2)(r - 3)} & {\frac{r - 6}{2} \choose 0}\frac{2!(r - 2)}{(r - 2)(r - 3)(r - 4)} & \cdots & {\frac{r - 4}{2} \choose 2} \\
                \vdots & \vdots & \vdots & & \vdots \\
                {\frac{r - 2}{2} \choose \frac{r - 4}{2}}\frac{r}{\frac{r + 4}{2}} & {\frac{r - 4}{2} \choose \frac{r - 6}{2}}\frac{r - 1}{(\frac{r + 4}{2})(\frac{r + 2}{2})} & {\frac{r - 6}{2} \choose \frac{r - 8}{2}}\frac{2!(r - 2)}{(\frac{r + 4}{2})(\frac{r + 2}{2})(r/2)} & \cdots & {\frac{r - 4}{2} \choose \frac{r - 4}{2}} \\
                {\frac{r - 2}{2} \choose \frac{r - 2}{2}}\frac{r}{\frac{r + 2}{2}} & {\frac{r - 4}{2} \choose \frac{r - 4}{2}}\frac{r - 1}{(\frac{r + 2}{2})(r/2)} & {\frac{r - 6}{2} \choose \frac{r - 6}{2}}\frac{2!(r - 2)}{(\frac{r + 2}{2})(r/2)(\frac{r - 2}{2})} & \cdots & 0
            \end{pmatrix}
            \begin{pmatrix}
                x_r \\ x_{r - 1} \\ x_{r - 2} \\ \vdots \\ x_{\frac{r + 2}{2}}
            \end{pmatrix}
            = \begin{pmatrix}
                0 \\ 0 \\ 0 \\ \vdots \\ 0 \\ 1
            \end{pmatrix}.
        \end{eqnarray}
        Let $A$ be the coefficient matrix above. We claim that the solution to this equation $x_n \in \Zp$
        are integral for all $\dfrac{r + 2}{2} \leq n \leq r$. To show this, it is enough to prove that $\det A \in \Zp^*$. Note that $A$ is almost lower triangular with all but the last diagonal entry belonging to $\Zp^*$. We show that for $0 \leq j \leq \dfrac{r - 4}{2}$ after making the $\left(j, \dfrac{r - 2}{2}\right)^{\mathrm{th}}$ entries $0$ using column operations, the $\left(\dfrac{r - 2}{2}, \dfrac{r - 2}{2}\right)^{\mathrm{th}}$ entry is a $p$-adic unit.

        First perform the operation $C_{\frac{r - 2}{2}} \to C_{\frac{r - 2}{2}} - C_0$ to make the $\left(0, \dfrac{r - 2}{2}\right)^{\mathrm{th}}$ entry in the last column $0$. For $j \geq 1$, the $\left(j, \dfrac{r - 2}{2}\right)^{\mathrm{th}}$ entry becomes
        \[
            {\frac{r - 4}{2} \choose j} - {\frac{r - 2}{2} \choose j}\frac{r}{r - j} = -{\frac{r - 4}{2} \choose j - 1}\frac{r + \frac{r - 2}{2} - j}{r - j}.
        \]
        Next, perform the operation $C_{\frac{r - 2}{2}} \to C_{\frac{r - 2}{2}} + \frac{(r - 2)\left(r + \frac{r - 2}{2} - 1\right)}{r - 1}C_1$ to make the $\left(1, \dfrac{r - 2}{2}\right)^{\mathrm{th}}$ entry in the last column $0$. Therefore for $j \geq 2$, the $\left(j, \dfrac{r - 2}{2}\right)^{\mathrm{th}}$ entry becomes
        \begin{eqnarray*}
            &&-{\frac{r - 4}{2} \choose j - 1}\frac{r + \frac{r - 2}{2} - j}{r - j} + \frac{(r - 2)\left(r + \frac{r - 2}{2} - 1\right)}{r - 1}{\frac{r - 4}{2} \choose j - 1}\frac{r - 1}{(r - j)(r - j - 1)} \\
            && = {\frac{r - 6}{2} \choose j - 2}\left(\frac{r - 4}{2}\right)\frac{2(r - 1) + \frac{r - 2}{2} - j}{(r - j)(r - j - 1)}.
        \end{eqnarray*}
        Continuing this process, after making the $\left(\dfrac{r - 4}{2}, \dfrac{r - 2}{2}\right)^{\mathrm{th}}$ entry $0$, we see that for $j \geq \dfrac{r - 2}{2}$ the $\left(j, \dfrac{r - 2}{2}\right)^{\mathrm{th}}$ entry in the last column is
        \[
            (-1)^{\frac{r - 2}{2}}{0 \choose j - \frac{r - 2}{2}}\left(\frac{r - 4}{2}\right)!\frac{\left(\frac{r - 2}{2}\right)\left(r - \frac{r - 4}{2}\right) + \left(\frac{r - 2}{2} - j\right)}{(r - j)(r - j - 1) \cdots \left(r - j - \frac{r - 4}{2}\right)}.
        \]
        Therefore the $\left(\dfrac{r - 2}{2}, \dfrac{r - 2}{2}\right)^{\mathrm{th}}$ entry is
        \[
            (-1)^{\frac{r - 2}{2}}\frac{4(r + 4)}{r(r + 2)}
        \]
        For odd primes this fraction fails to be a $p$-adic unit only when $p = 3$ and $r = 2$.
        However, at the beginning of this part of the proof we assumed that $r \geq 4$. Alternatively,
        we note our blanket assumption $p \geq 5$. This shows that $\det A \in \Zp^*$ and
        the $x_n$ are integral for $\dfrac{r+2}{2} \leq n \leq r$, proving the claim.

        Multiplying \eqref{Refined inductive steps in killing T12 in geq 0} by $x_{n}$ and adding
        them up with $x_{\frac{r + 2}{2}}$ times \eqref{Main congruence in killing T12 in geq 0}, we get
        \[
            0 \equiv \sum_{a = 0}^{p - 1}p(z - ap)^{\frac{r - 2}{2}}\mathbbm{1}_{ap + p^2\Zp} \mod \pi \latticeL{k}.
        \]
        Applying the matrix $\begin{pmatrix}p & 0 \\ 0 & 1\end{pmatrix}$, we get the lemma for
        $l = \dfrac{r-2}{2}$ as well.
    \end{proof}

    Now we prove the claim made at the beginning of this section.

    \begin{theorem}\label{Final theorem for geq 0}
        Let $2 \leq r \leq p - 1$ be even. If $\nu \geq 0$, then the map $\IZind a^{\frac{r - 2}{2}}d^{\frac{r + 2}{2}} \twoheadrightarrow F_{r - 2, \> r - 1}$ factors as
        \[
            \IZind a^{\frac{r - 2}{2}}d^{\frac{r + 2}{2}} \twoheadrightarrow \frac{\IZind a^{\frac{r - 2}{2}}d^{\frac{r + 2}{2}}}{\im(\lambda_{r/2}T_{-1, 0} - 1)} \twoheadrightarrow F_{r - 2, \> r - 1}
          \]
        for $\lambda_{r/2} = (-1)^{r/2}\left(\frac{r}{2}\right)\left(\frac{r + 2}{2}\right)(\cL - H_{-} - H_{+})$.
        Moreover if $\nu = 0$, then we get a surjection
        \[
            \pi(p - 3, \lambda_{r/2}^{-1}, \omega^{\frac{r + 2}{2}}) \twoheadrightarrow F_{r - 2, \> r - 1}.
        \]
    \end{theorem}
    \begin{proof}
        Unless specified, all congruences in this proof are in the space $\br{\latticeL{k}}$ modulo the image of the subspace $\IZind \oplus_{j < \frac{r + 2}{2}}\Fq X^{r - j}Y^j$ under $\IZind \SymF{k - 2} \twoheadrightarrow \br{\latticeL{k}}$.
        
        Using Proposition \ref{nu geq 0}, we have\footnote{In the last sum $S_m$ should be replaced
        by $S_m - (p- 1)$ if $m = p - 1$ but since for odd primes this 
        happens only for $p = 3$ it will not concern us by our blanket assumption $p \geq 5$.}
        \begin{eqnarray}\label{Main congruence in geq 0}
            0 & \equiv & \sum_{a = 0}^{p - 1}\sum_{\alpha = 0}^{p - 1}\sum_{j = 0}^{\frac{r - 2}{2}}{\frac{r + 2}{2} \choose j}\left(\cL - H_{\frac{r + 2}{2} - j}\right)p^{-1}(\alpha p)^{\frac{r + 2}{2} - j} (z - a - \alpha p)^{j}\mathbbm{1}_{a + \alpha p + p^2\Zp} \\
            && + \sum_{a = 0}^{p - 1}\sum_{j = 0}^{\frac{r - 2}{2}}{\frac{r + 2}{2} \choose j}\frac{a^{\frac{r + 2}{2} - j}}{\frac{r + 2}{2} - j}(z - a)^j\mathbbm{1}_{a + p\Zp} \nonumber\\
            && + \sum_{a = 0}^{p - 1}\sum_{m = 1}^{\frac{r + 2}{2}}\sum_{j = 0}^{\frac{r + 2}{2} - m - \delta_{m, 1}}{\frac{r + 2}{2} \choose j}{\frac{r + 2}{2} - j \choose m}(-1)^m\frac{S_m}{p}\frac{a^{\frac{r + 2}{2} - j - m}}{\frac{r + 2}{2} - j}(z - a)^j\mathbbm{1}_{a + p\Zp}. \nonumber
        \end{eqnarray}
        We have dropped the $j = r/2$ terms in the equation above because of the statement made at the beginning of this proof (in third sum we need to use the Kronecker delta function $\delta_{m,1}$).
        
        We prove the third term on the right side of the equation above is $0$. Indeed, we
        claim that for each $1 \leq m \leq \dfrac{r + 2}{2}$, the $m^{\mathrm{th}}$ summand in
        this term is an integer multiple of
        \begin{eqnarray*}
            \sum_{a = 0}^{p - 1}(z - a)^{\frac{r + 2}{2} - m - \delta_{m, \> 1}}\mathbbm{1}_{a + p\Zp}.
        \end{eqnarray*}
        This sum is $0$ modulo $\pi \latticeL{k}$ by Lemma \ref{T12 is 0 in even self-dual}. 
        The claim is clearly true if $m = \dfrac{r + 2}{2}$. So assume $1 \leq m \leq r/2$.
        Fixing $a$ and $m$, we write the $(a, \> m)^{\mathrm{th}}$ summand in this sum as
        \begin{eqnarray}\label{Third sum in geq 0}
            \sum_{j = 0}^{\frac{r + 2}{2} - m - \delta_{m, 1}}d_{m, j}a^{\frac{r + 2}{2} - j - m}(z - a)^j\mathbbm{1}_{a + p\Zp},
        \end{eqnarray}
        for $d_{m, j} \in \Zp$. Next, using Proposition \ref{inductive steps} we have the
        following shallow inductive steps
        \[
            z^{r - n}\mathbbm{1}_{p\Zp} \equiv \sum_{b = 1}^{p - 1}\sum_{j = 1}^{\frac{r - 2}{2}}c(n, j)b^{r - j - n}(z - b)^j\mathbbm{1}_{b + p\Zp},
        \]
        for $n = r, \> r - 1, \> \ldots, \> \dfrac{r - 2}{2} + m + \delta_{m, \>1} + 1$. Applying the operator $\sum\limits_{a = 0}^{p - 1}a^{n - \frac{r - 2}{2} - m}\begin{pmatrix}1 & 0 \\ -a & 1\end{pmatrix}$ to both sides, 
        substituting $\lambda = a + b$ and summing over $b$ on the right
        side using \eqref{sum over b identity} with $m = \frac{r}{2} + m$, 
        replacing $\lambda$ by $a$, and substituting
        for $c(n, j)$ from Proposition \ref{inductive steps},
        we get
        \begin{eqnarray}\label{Refined inductive steps for third sum in geq 0}
            0 \equiv \sum_{a = 0}^{p - 1}\sum_{j = r - n}^{\frac{r + 2}{2} - m - \delta_{m, \> 1}}{n - \frac{r - 2}{2} - m \choose n - r + j}\frac{(r - n)!n}{(r - j)\cdots (n - j)}a^{\frac{r + 2}{2} - j - m}(z - a)^j\mathbbm{1}_{a + p\Zp}.
        \end{eqnarray}
        We wish to write \eqref{Third sum in geq 0} as a sum of $x_n \in E$ times the corresponding summand in \eqref{Refined inductive steps for third sum in geq 0} and $x_{\frac{r - 2}{2} + m + \delta_{m, 1}} \in E$ times
        \[
            (z - a)^{\frac{r + 2}{2} - m - \delta_{m, \> 1}}\mathbbm{1}_{a + p\Zp}.
        \]
        Comparing the coefficients of like terms in $a^{\frac{r + 2}{2} - j - m}(z - a)^j\mathbbm{1}_{a + p\Zp}$, we get the following matrix equation
        \[
            \begin{pmatrix}
                {\frac{r + 2}{2} - m \choose 0}\frac{r}{r} & 0 & \cdots & 0 \\
                {\frac{r + 2}{2} - m \choose 1}\frac{r}{r - 1} & {r/2 - m \choose 0}\frac{r - 1}{(r - 1)(r - 2)} & \cdots & 0 \\
                \vdots & \vdots & & \vdots \\
                {\frac{r + 2}{2} - m \choose \frac{r + 2}{2} - m - \delta_{m, 1}}\frac{r}{\frac{r - 2}{2} + m + \delta_{m, 1}} & \cdots & \cdots & 1
            \end{pmatrix}\!\! 
            \begin{pmatrix}
                x_r \\ x_{r - 1} \\ \vdots \\ x_{\frac{r - 2}{2} + m + \delta_{m, \> 1}}
            \end{pmatrix} \! 
            = \! 
            \begin{pmatrix}
                d_{m, 0} \\ d_{m, 1} \\ \vdots \\ d_{m, \frac{r + 2}{2} - m - \delta_{m, 1}}
            \end{pmatrix} \! .
        \]
        Since the determinant of the coefficient matrix above is a $p$-adic unit, the $x_n$ satisfying the equation above are all integral. Hence, as claimed, we can write \eqref{Third sum in geq 0} as an integer
        multiple of
        \[
            (z - a)^{\frac{r + 2}{2} - m - \delta_{m, \> 1}}\mathbbm{1}_{a + p\Zp}.
        \]
        
        Next, we fix $a$ and $\alpha$. We wish to write
        \begin{eqnarray}\label{Crude g2 - g1 in geq 0}
            \sum_{j = 0}^{\frac{r - 2}{2}}{\frac{r + 2}{2} \choose j}\left(\cL - H_{\frac{r + 2}{2} - j}\right)p^{-1}(\alpha p)^{\frac{r + 2}{2} - j} (z - a - \alpha p)^{j}\mathbbm{1}_{a + \alpha p + p^2\Zp}
        \end{eqnarray}
        appearing in the first line on the right side of \eqref{Main congruence in geq 0} as a multiple of
        \[
            p^{-1}(z - a)^{\frac{r + 2}{2}}\mathbbm{1}_{a + p\Zp}.
        \]
        To this end, recall the shallow inductive steps obtained by
        massaging Proposition~\ref{inductive steps} (done many times before)
        \begin{eqnarray}\label{Applied inductive steps for g2 - g1 in geq 0}
            0 \equiv \sum_{\alpha = 0}^{p - 1}\sum_{j = r - n}^{\frac{r - 2}{2}}{n - \frac{r - 2}{2} \choose n - r + j}\frac{(r - n)!n}{(r - j) \cdots (n - j)}p^{-1}(\alpha p)^{\frac{r + 2}{2} - j}(z - a - \alpha p)^j\mathbbm{1}_{a + \alpha p + p^2\Zp}
        \end{eqnarray}
        for $\dfrac{r+4}{2} \leq n \leq r$.
        Note that we have dropped the $j = r/2$ term  because of the statement made at the beginning of this proof.
        Also consider the binomial expansion
        \begin{eqnarray}\label{Trivial equation in g2 - g1 in geq 0}
            p^{-1}(z - a)^{\frac{r + 2}{2}}\mathbbm{1}_{a + p\Zp} \equiv \sum_{\alpha = 0}^{p - 1}\sum_{j = 0}^{\frac{r - 2}{2}}{\frac{r + 2}{2} \choose j}p^{-1}(\alpha p)^{\frac{r + 2}{2} - j}(z - a - \alpha p)^j\mathbbm{1}_{a + \alpha p + p^2\Zp}.
        \end{eqnarray}
        We have dropped the $j = \dfrac{r + 2}{2}$ term by Lemma \ref{stronger bound for polynomials of large degree with varying radii} and the $j = r/2$ term by the statement made at the beginning of this proof.
        We wish to write \eqref{Crude g2 - g1 in geq 0} as a sum of $x_n \in E$ times the corresponding term in \eqref{Applied inductive steps for g2 - g1 in geq 0} and $x_{\frac{r + 2}{2}}$ times the corresponding term in \eqref{Trivial equation in g2 - g1 in geq 0}. Comparing the coefficients of like terms in $p^{-1}(\alpha p)^{\frac{r + 2}{2} - j}(z - a - \alpha p)^{j}\mathbbm{1}_{a + \alpha p + p^2\Zp}$, we get the following matrix equation
        \[
            \!\begin{pmatrix}
                {\frac{r + 2}{2} \choose 0}\frac{r}{r} & 0 & 0 & \cdots & {\frac{r + 2}{2} \choose 0} \\
                {\frac{r + 2}{2} \choose 1}\frac{r}{r - 1} & {r/2 \choose 0}\frac{r - 1}{(r - 1)(r - 2)} & 0 & \cdots & {\frac{r + 2}{2} \choose 1} \\
                {\frac{r + 2}{2} \choose 2}\frac{r}{r - 2} & {r/2 \choose 1}\frac{r - 1}{(r - 2)(r - 3)} & {\frac{r - 2}{2} \choose 0}\frac{2!(r - 2)}{(r - 2)(r - 3)(r - 4)} & \cdots & {\frac{r + 2}{2} \choose 2} \\
                \vdots & \vdots & \vdots & & \vdots \\
                {\frac{r + 2}{2} \choose \frac{r - 2}{2}}\frac{r}{\frac{r + 2}{2}} & {r/2 \choose \frac{r - 4}{2}}\frac{r - 1}{(\frac{r + 2}{2})(r/2)} & {\frac{r - 2}{2} \choose \frac{r - 6}{2}}\frac{2!(r - 2)}{(\frac{r + 2}{2})(r/2)(\frac{r - 2}{2})} & \cdots & {\frac{r + 2}{2} \choose \frac{r - 2}{2}}
            \end{pmatrix} \!\!\! 
            \begin{pmatrix}
                x_r \\ x_{r - 1} \\ x_{r - 2} \\ \vdots \\ x_{\frac{r + 2}{2}}
            \end{pmatrix} \!\!\! 
            = \!\!\! \begin{pmatrix}
                {\frac{r + 2}{2} \choose 0}\left(\cL - H_{\frac{r + 2}{2}}\right) \\
                {\frac{r + 2}{2} \choose 1}\left(\cL - H_{r/2}\right) \\
                {\frac{r + 2}{2} \choose 2}\left(\cL - H_{\frac{r - 2}{2}}\right) \\
                \vdots \\
                {\frac{r + 2}{2} \choose \frac{r - 2}{2}}\left(\cL - H_{2}\right)
            \end{pmatrix}\!\!.
        \]
        It can be checked\footnote{\label{space-time 16}See Appendix~\ref{Footnote 16}.} that $x_{\frac{r + 2}{2}} = \cL - H_{-} - H_{+}$. This shows that \eqref{Main congruence in geq 0} becomes
        \begin{eqnarray*}
            0 & \equiv & \sum_{a = 0}^{p - 1}\left(\cL - H_{-} - H_{+}\right)p^{-1}(z - a)^{\frac{r + 2}{2}}\mathbbm{1}_{a + p\Zp} + \sum_{a = 0}^{p - 1}\sum_{j = 0}^{\frac{r - 2}{2}}{\frac{r + 2}{2} \choose j}\frac{a^{\frac{r + 2}{2} - j}}{\frac{r + 2}{2} - j}(z - a)^j\mathbbm{1}_{a + p\Zp}.
        \end{eqnarray*}
        Proceeding exactly as in the proof of Theorem \ref{Final theorem for geq} with $i = r/2$, the equation above becomes
        \[
            0 \equiv (-1)^{r/2}\left(\frac{r}{2}\right)\left(\frac{r + 2}{2}\right)\left(\cL - H_{-} - H_{+}\right)\left[\sum_{a = 0}^{p - 1}p^{-1}(z - a)^{\frac{r + 2}{2}}\mathbbm{1}_{a + p\Zp}\right] - z^{\frac{r + 2}{2}}\mathbbm{1}_{\Zp}.
        \]

        Using the formula \eqref{Formulae for T-10 and T12 in the commutative case}, 
        we see that
        under the surjection $\IZind a^{\frac{r - 2}{2}}d^{\frac{r + 2}{2}} \twoheadrightarrow F_{r - 2, \> r - 1}$ 
        \begin{itemize}
            \item $-wT_{-1, 0}\llbracket \id, X^{\frac{r - 2}{2}}Y^{\frac{r + 2}{2}}\rrbracket$ maps to
            \begin{eqnarray*}
                \sum_{\lambda \in I_1}-\begin{pmatrix}0 & 1 \\ p & -\lambda\end{pmatrix}z^{\frac{r - 2}{2}}\mathbbm{1}_{p\Zp} & = & \sum_{\lambda \in I_1}-\frac{1}{p^{r/2}}(z - \lambda)^r\left(\frac{p}{z - \lambda}\right)^{\frac{r - 2}{2}}\mathbbm{1}_{p\Zp}\left(\frac{p}{z - \lambda}\right) \\
                & = & \sum_{a = 0}^{p - 1}p^{-1}(z - a)^{\frac{r + 2}{2}}\mathbbm{1}_{a + p\Zp} \mod \latticeL{k}
            \end{eqnarray*}
            \item $-w\llbracket \id, X^{\frac{r - 2}{2}}Y^{\frac{r + 2}{2}}\rrbracket$ maps to
            \begin{eqnarray*}
                -\begin{pmatrix}0 & 1 \\ 1 & 0\end{pmatrix}z^{\frac{r - 2}{2}}\mathbbm{1}_{p\Zp} & = & -z^r\left(\frac{1}{z}\right)^{\frac{r - 2}{2}}\mathbbm{1}_{p\Zp}\left(\frac{1}{z}\right) \\
                & =  & z^{\frac{r + 2}{2}}\mathbbm{1}_{\Zp} \mod \pi \latticeL{k}.
            \end{eqnarray*} 
          \end{itemize}
          This proves the first claim in the theorem.
          The second claim follows as usual using arguments similar to those
          after diagram~\eqref{rectangle}, and \eqref{alt def of pi} (with $r = p-3$).  
    \end{proof}

    \subsection{The $\nu \leq 0$ case for even weights}
    \label{leq section for even weights}
    In this section, we continue to assume that $r$ and $k$ are even.
    We first prove that if $\nu \leq 0$, then the map
    $\IZind a^{r/2}d^{r/2} \twoheadrightarrow F_{r, \> r + 1}$ factors as
    \[
        \IZind a^{r/2}d^{r/2} \twoheadrightarrow \frac{\IZind a^{r/2}d^{r/2}}{\im(\lambda_{r/2}^{-1}T_{1, 2} - 1)} \twoheadrightarrow F_{r, \> r + 1},
    \]
    where $\lambda_{r/2} = (-1)^{r/2}\left(\frac{r}{2}\right)\left(\frac{r + 2}{2}\right)\left(\cL - H_{-} - H_{+}\right)$. To state our second result, we need some notation. Since the inducing character is
    (the restriction of) a power of the determinant character, the non-commutative Hecke algebra
    (Section~\ref{subsubsection non-commutative}) is involved.
    By the second relation in \eqref{Relations in the non-commutative Hecke algebra},
    the operator $-T_{1,2}T_{1,0}$ is a projector and hence induces a decomposition
    \[
        \IZind a^{r/2}d^{r/2} = \im T_{1, 2}T_{1, 0} \oplus \im (1 + T_{1, 2}T_{1, 0}).
    \]
    We may therefore write $F_{r, \> r + 1} = F_{r} \oplus F_{r + 1}$ as a direct sum, where $F_{r}$ is the image of $\im(T_{1, 2}T_{1, 0})$ and $F_{r + 1}$ is the image of $\im (1 + T_{1, 2}T_{1, 0})$. Thus the map $\IZind a^{r/2}d^{r/2} \twoheadrightarrow F_{r, \> r + 1}$ induces
        \begin{eqnarray}\label{Two parts in the last slice}
            \frac{\IZind a^{r/2}d^{r/2}}{\im(1 + T_{1, 2}T_{1, 0})} \twoheadrightarrow \frac{F_{r, \> r + 1}}{F_{r + 1}} \simeq F_{r} \text{ and } \frac{\IZind a^{r/2}d^{r/2}}{\im T_{1, 2}T_{1, 0}} \twoheadrightarrow \frac{F_{r, \> r + 1}}{F_{r}} \simeq F_{r + 1}.
        \end{eqnarray}
        We prove that if $\nu = 0$, then we have a surjection
        \[
            \pi(p - 1, \lambda_{r/2}, \omega^{r/2}) \twoheadrightarrow F_{r}
        \]
        and that $F_{r + 1}$ is not isomorphic to $\pi(0, 0, \omega^{r/2})$.

    \begin{Proposition}\label{leq 0}
        Let $2 \leq r \leq p-1$ be even. Let $\nu \leq 0$. Fix $x \in \bQ$ such that $x + \nu = -1$. Set
        \[
            g(z) = \sum_{i \in I}p^x\lambda_i(z - i)^{\frac{r + 2}{2}}\logL(z - i),
        \]
        where $I = \left\{0, 1, \ldots, \dfrac{r+2}{2},  p \right\}$ and the coefficients $\lambda_i \in \zp$
        are as in Lemma~\ref{Main coefficient identites} (2).
        Then
        \begin{eqnarray*} 
            g(z) & \equiv & \sum_{\alpha = 0}^{p - 1}\sum_{j = 0}^{\frac{r - 2}{2}}\sum_{m = 1}^{r/2 - j}{\frac{r + 2}{2} \choose j}p^{x + \frac{r + 2}{2} - j}t_{j, \> m} (-1)^m \alpha^{\frac{r + 2}{2} - j - m} (z - \alpha p)^{j} \mathbbm{1}_{\alpha p + p^2\Zp} \\
            && + \sum_{a = 1}^{p - 1}\sum_{j = 0}^{r/2}{\frac{r + 2}{2} \choose j}p^{1 + x}a^{r/2 - j}(z - a)^j\mathbbm{1}_{a + p\Zp} + \left(\frac{r + 2}{2}\right)p^{1 + x}\cL z^{r/2}\mathbbm{1}_{p\Zp} \!\!\! \mod \pi \latticeL{k},
        \end{eqnarray*}
        for some $t_{j, \> m} \in \Zp$.
    \end{Proposition}
      
    \begin{proof}
      Recall from Lemma~\ref{Main coefficient identites}(2) that
      $\lambda_0 = 1 \mod p$, $\lambda_i = 0 \mod p$ for $1 \leq i \leq \dfrac{r+2}{2}$, $\lambda_p = -1$ and
      $\sum_{i \in I}\lambda_i i^j = 0$ for $0 \leq j \leq \dfrac{r+2}{2} \leq p-2$ since $p \geq 5$.
      Write $g(z) = g(z)\mathbbm{1}_{\Zp} + g(z)\mathbbm{1}_{\Qp \setminus \Zp}$. As usual, we have
        \[
            w \cdot g(z)\mathbbm{1}_{\Qp \setminus \Zp} = \sum_{i \in I}p^x\lambda_iz^{\frac{r - 2}{2}}(1 - zi)^{\frac{r + 2}{2}}\logL(1 - zi)\mathbbm{1}_{p\Zp}.
        \]
        Lemma \ref{qp-zp part is 0} implies that this function belongs to $\pi \latticeL{k}$.
        Hence $g(z) \equiv g(z)\mathbbm{1}_{\Zp} \!\! \mod\pi \latticeL{k}$.
        
        Next, using Lemma \ref{telescoping lemma} (3), we see that
        $g(z)\mathbbm{1}_{\Zp} \equiv g_2(z) \mod \pi \latticeL{k}$.
        We analyze $g_2(z) - g_1(z)$ and $g_1(z)$ separately.

        For $0 \leq j \leq r/2$, the coefficient of $(z - a - \alpha p)^j\mathbbm{1}_{a + \alpha p + p^2\Zp}$ in $g_2(z) - g_1(z)$ is
        \begin{eqnarray}\label{jth summand in g2 - g1 in leq 0}
            && \!\!\!\!\!\!\!\!\!\!\!\! {\frac{r + 2}{2} \choose j}\sum_{i \in I}p^{x}\lambda_i\Big[(a + \alpha p - i)^{\frac{r + 2}{2} - j}\logL(a + \alpha p - i) - (a - i)^{\frac{r + 2}{2} - j}\logL(a - i) \\
            && \!\!\!\!\!\!\!\!\!\!\!\! - {\frac{r + 2}{2} - j \choose 1}(a - i)^{r/2 - j}(\alpha p) - \ldots - {\frac{r + 2}{2} - j \choose r/2 - j}(a - i)(\alpha p)^{r/2 - j}\logL(a - i)\Big]. \nonumber
        \end{eqnarray}

        First assume $a \neq 0$. For the $i \neq a$ summands we write $\logL(a - i + \alpha p) = \logL(1 + (a - i)^{-1}\alpha p) + \logL(a - i)$. Therefore the $i^{\mathrm{th}}$ term in the display above becomes
        \[
            p^x\lambda_i\left[(a - i + \alpha p)^{\frac{r + 2}{2} - j}\logL(1 + (a - i)^{-1}\alpha p) + (\alpha p)^{\frac{r + 2}{2} - j}\logL(a - i)\right].
        \]
        The conditions $a \neq 0$ and $i \neq a$ imply that $a - i$ is a $p$-adic unit. Therefore $p \mid \logL(a - i)$. So we drop the second term in the display above modulo $p^{r/2 - j}\pi$. Expanding both the factors in the first term, we get
        \begin{eqnarray}\label{a neq 0 i neq a summand in g2 - g1 in leq 0}
            && \!\!\!\!\!\!\!\!\!\!\!\!\!\!\!\! p^x\lambda_i[t_{j, \ r/2 - j}(a - i)^{r/2 - j}(\alpha p) + t_{j, \> \frac{r - 2}{2} - j}(a - i)^{\frac{r - 2}{2} - j}(\alpha p)^2 + \cdots + t_{j, 0}(\alpha p)^{\frac{r + 2}{2} - j}] \\ 
            && \!\!\!\!\!\!\!\!\!\!\!\!\!\!\!\! \qquad \qquad \mod p^{r/2 - j} \pi, \nonumber
        \end{eqnarray}
        where $t_{j, \> 0}, \> \ldots, \> t_{j, \> \frac{r - 2}{2} - j} \in \Zp$ with $t_{j, \> r/2 - j} = 1$ and $t_{j, \> 0} = H_{\frac{r + 2}{2} - j}$.

        Next assume $i = a$. Then the $i^{\mathrm{th}}$ summand in \eqref{jth summand in g2 - g1 in leq 0} is
        \[
            p^x\lambda_a(\alpha p)^{\frac{r + 2}{2} - j}\logL(\alpha p) = p^x\lambda_a(\alpha p)^{\frac{r + 2}{2} - j}\cL + p^x\lambda_a(\alpha p)^{\frac{r + 2}{2} - j}\logL(\alpha).
        \]
        The second term in the display above is $0$ modulo $p^{r/2 - j}\pi$. Indeed, this is clear if $\alpha = 0$ and if $\alpha \neq 0$, then $p \mid \logL(\alpha)$. Therefore the $i = a$ summand in \eqref{jth summand in g2 - g1 in leq 0} is
        \begin{eqnarray}\label{a neq 0 i = a summand in g2 - g1 in leq 0}
            p^x\lambda_a(\alpha p)^{\frac{r + 2}{2} - j}\cL \mod p^{r/2 - j}\pi.
        \end{eqnarray}

        Plugging expressions \eqref{a neq 0 i neq a summand in g2 - g1 in leq 0} and \eqref{a neq 0 i = a summand in g2 - g1 in leq 0} in \eqref{jth summand in g2 - g1 in leq 0} and using the
        vanishing of the sums in the first line of the proof,
        we get
        \[
            p^x\lambda_a(\alpha p)^{\frac{r + 2}{2} - j}\left(\cL - H_{\frac{r + 2}{2} - j}\right) \mod p^{r/2 - j}\pi.
        \]
        Since $p \mid \lambda_a$ for $a \neq 0$, we see that the expression above is $0$ modulo $p^{r/2 - j}\pi$.

        Now assume that $a = 0$. Then the $i \neq a, \> p$ summands in \eqref{jth summand in g2 - g1 in leq 0} are given by \eqref{a neq 0 i neq a summand in g2 - g1 in leq 0}. The sum of the $i = 0$ and $i = p$ summands is
        \begin{eqnarray}\label{Sum of the i = 0 and i = p terms in g2 - g1 in leq 0}
            &&  \!\!\!\!\!\!\!\!\!\!\!\!\!\!\!\! p^x\lambda_0\Big[(\alpha p)^{\frac{r + 2}{2} - j}\logL(\alpha p)\Big] + p^x\lambda_p\Big[(\alpha p - p)^{\frac{r + 2}{2} - j}\logL(\alpha p - p) - (-p)^{\frac{r + 2}{2} - j}\cL \\
            && \!\!\!\!\!\!\!\!\!\!\!\!\!\!\!\!- {\frac{r + 2}{2} - j \choose 1}(-p)^{r/2 - j}(\alpha p) - \cdots - {\frac{r + 2}{2} - j \choose r/2 - j}(-p)(\alpha p)^{r/2 - j}\cL\Big]. \nonumber
        \end{eqnarray}
        Note that $0 \leq j \leq r/2$ implies that 
        \[
            p^x\lambda_p((\alpha - 1)p)^{\frac{r + 2}{2} - j}\logL((\alpha - 1)p) \equiv p^x\lambda_p((\alpha - 1)p)^{\frac{r + 2}{2} - j}\cL \mod p^{r/2 - j}\pi.
        \]
        Indeed, this is clear if $\alpha = 1$. If $\alpha \neq 1$, then $\alpha - 1$ is a unit. So $p \mid \logL(\alpha - 1)$. Similarly, 
        \[
            p^x\lambda_0(\alpha p)^{\frac{r + 2}{2} - j}\logL(\alpha p) \equiv p^x\lambda_0(\alpha p)^{\frac{r + 2}{2} - j}\cL \mod p^{r/2 - j}\pi.
        \]
        Therefore expanding $(\alpha p - p)^{\frac{r + 2}{2} - j}$, we see that \eqref{Sum of the i = 0 and i = p terms in g2 - g1 in leq 0} becomes
        \[
            p^x\lambda_0(\alpha p)^{\frac{r + 2}{2} - j}\cL + p^x\lambda_p(\alpha p)^{\frac{r + 2}{2} - j}\cL \mod p^{r/2 - j}\pi.
        \]
        Since $\lambda_0 \equiv 1 \!\! \mod p$ and $\lambda_p = -1$, the display above is $0$ modulo $p^{r/2 - j}\pi$.

        Using the vanishing of the sums in the first line of the proof, 
        we see that for $a = 0$ the coefficient of $(z - a - \alpha p)^{j}\mathbbm{1}_{a + \alpha p + p^2\Zp}$ in $g_2(z) - g_1(z)$ is
        \begin{eqnarray*}
            && {\frac{r + 2}{2} \choose j}\left[-p^x\lambda_p[t_{j, \> r/2 - j}(-p)^{r/2 - j}(\alpha p) + \ldots + t_{j, \> 0}(\alpha p)^{\frac{r + 2}{2} - j}] - p^x\lambda_0t_{j, \> 0}(\alpha p)^{\frac{r + 2}{2} - j}\right].
        \end{eqnarray*}
        Again using $\lambda_0 \equiv 1 \!\! \mod p$ and $\lambda_p = -1$, the display above becomes
        \begin{eqnarray*}
            &&{\frac{r + 2}{2} \choose j}p^{x + \frac{r + 2}{2} - j}\left[t_{j, \> r/2 - j}(-1)^{r/2 - j}\alpha + \cdots + t_{j, \> 1}(-1)^1\alpha^{r/2 - j}\right] \\
            && \equiv \sum_{m = 1}^{r/2 - j}{\frac{r + 2}{2} \choose j}p^{x + \frac{r + 2}{2} - j}t_{j, \> m}(-1)^m\alpha^{\frac{r + 2}{2} - j - m} \mod p^{r/2 - j}\pi.
        \end{eqnarray*}
        Since for $j = r/2$ there are no terms in the sum above, we see that
        \begin{eqnarray}\label{Crude g2 - g1 in leq 0}
            g_2(z) - g_1(z) & \equiv & \sum_{\alpha = 0}^{p - 1}\sum_{j = 0}^{\frac{r - 2}{2}}\sum_{m = 1}^{r/2 - j}{\frac{r + 2}{2} \choose j}p^{x + \frac{r + 2}{2} - j}t_{j, \> m} (-1)^m \alpha^{\frac{r + 2}{2} - j - m} \\
            && \quad \quad \quad \quad \quad \quad \quad \qquad (z - \alpha p)^{j} \mathbbm{1}_{\alpha p + p^2\Zp} \mod \pi \latticeL{k}. \nonumber
        \end{eqnarray}

        Next consider
        \begin{eqnarray}\label{g1 in leq 0}
            g_1(z) = \sum_{a = 0}^{p - 1}\sum_{j = 0}^{r/2}\frac{g^{(j)}(a)}{j!}(z - a)^j\mathbbm{1}_{a + p\Zp},
        \end{eqnarray}
        where
        \begin{eqnarray}\label{jth summand in g1 in leq 0}
            \frac{g^{(j)}(a)}{j!} = {\frac{r + 2}{2} \choose j}\sum_{i \in I}p^x\lambda_i(a - i)^{\frac{r + 2}{2} - j}\logL(a - i).
        \end{eqnarray}

        First assume $a \neq 0$. Suppose $i \neq 0, \> a, \> p$. These conditions imply that $a - i$ is a $p$-adic unit. Therefore $p \mid \logL(a - i)$. Since $p \mid \lambda_i$ for $1 \leq i \leq \dfrac{r+2}{2}$, we see that the $i^{\mathrm{th}}$
        summand in the equation above is $0$ modulo $\pi$. The $i = a$ summand is clearly $0$. Next, the sum of the $i = 0, \> p$ summands in \eqref{jth summand in g1 in leq 0} is
        \[
            p^x\lambda_0a^{\frac{r + 2}{2} - j}\logL(a) + p^x\lambda_p(a - p)^{\frac{r + 2}{2} - j}\logL(a - p).
        \]
        Recall that $\lambda_0 \equiv 1 \!\!\mod p$, $\lambda_p = -1$. Expanding $(a - p)^{\frac{r + 2}{2} - j}$ and using the fact that $a$ and $a - p$ are $p$-adic units, the display above becomes
        \begin{eqnarray*}
            p^xa^{\frac{r + 2}{2} - j}\logL(a) - p^xa^{\frac{r + 2}{2} - j}\logL(a - p) = 
            -p^xa^{\frac{r + 2}{2} - j}\logL(1 - a^{-1}p)  
            & \equiv &  p^{1 + x}a^{r/2 - j} \!\! \mod \pi.
        \end{eqnarray*}
        Therefore for $a \neq 0$, we have
        \begin{eqnarray}\label{a neq 0 in g1 in leq 0}
            \frac{g^{(j)}(a)}{j!} \equiv {\frac{r + 2}{2} \choose j}p^{1 + x}a^{r/2 - j} \mod \pi.
        \end{eqnarray}

        Next assume that $a = 0$. Suppose $i \neq 0, \> p$. Then again $a - i$ is a $p$-adic unit and $p \mid\lambda_i$. Therefore the $i^{\mathrm{th}}$ summand in \eqref{jth summand in g1 in leq 0} is $0$ modulo $\pi$. The $i = 0$ summand is clearly $0$. The $i = p$ summand is
        \[
            p^x\lambda_p(-p)^{\frac{r + 2}{2} - j}\cL = (-1)^{\frac{r + 4}{2} - j}p^{x + \frac{r + 2}{2} - j}\cL.
        \]
        Therefore for $a = 0$, we have
        \begin{eqnarray}\label{a = 0 in g1 in leq 0}
            \frac{g^{(j)}(a)}{j!} \equiv {\frac{r + 2}{2} \choose j}(-1)^{\frac{r + 4}{2} - j}p^{x + \frac{r + 2}{2} - j}\cL \mod \pi.
        \end{eqnarray}

        Putting equations \eqref{a neq 0 in g1 in leq 0} and \eqref{a = 0 in g1 in leq 0} in \eqref{g1 in leq 0}, we get
        \begin{eqnarray*}
            g_1(z) \!\!\!\! & \equiv & \!\!\!\! \sum_{a = 1}^{p - 1}\left[\sum_{j = 0}^{r/2}{\frac{r + 2}{2} \choose j}p^{1 + x}a^{r/2 - j}(z - a)^j\right] \!\! \mathbbm{1}_{a + p\Zp} + \left[\sum_{j = 0}^{r/2}{\frac{r + 2}{2} \choose j}(-1)^{\frac{r + 4}{2} - j}p^{x + \frac{r + 2}{2} - j}\cL z^{j}\right] \!\! \mathbbm{1}_{p\Zp} \\
            && \qquad \qquad \mod \pi \latticeL{k}.
        \end{eqnarray*}
        Using Lemma \ref{Integers in the lattice}, we drop the $j < r/2$ summands in the second sum above. Therefore
        \begin{eqnarray}\label{Crude g1 in leq 0}
            g_1(z) & \equiv & \sum_{a = 1}^{p - 1}\sum_{j = 0}^{r/2}{\frac{r + 2}{2} \choose j}p^{1 + x}a^{r/2 - j}(z - a)^j\mathbbm{1}_{a + p\Zp} + \left(\frac{r + 2}{2}\right)p^{1 + x}\cL z^{r/2}\mathbbm{1}_{p\Zp} \\
            && \qquad \qquad \mod \pi \latticeL{k}. \nonumber
        \end{eqnarray}
        Since $g(z) \equiv g_2(z) \mod \pi \latticeL{k}$ we add \eqref{Crude g2 - g1 in leq 0} and \eqref{Crude g1 in leq 0} to obtain the proposition.
    \end{proof}

    Now we prove the claims made at the beginning of this section.
    \begin{theorem}\label{Final theorem for leq 0}
        Let $2 \leq r \leq p - 1$ be even. If $\nu \leq 0$, then the map $\IZind a^{r/2}d^{r/2} \twoheadrightarrow F_{r, \> r + 1}$ factors as
        \[
        \IZind a^{r/2}d^{r/2} \twoheadrightarrow \frac{\IZind a^{r/2}d^{r/2}}{\im(\lambda_{r/2}^{-1}T_{1, 2} - 1)} \twoheadrightarrow F_{r, \> r + 1},
        \]
        where
        \[
            \lambda_{r/2} = (-1)^{r/2}\left(\dfrac{r}{2}\right)\left(\dfrac{r + 2}{2}\right)\left(\cL - H_{-} - H_{+}\right).
        \]
            Moreover, if $\nu = 0$, then we have a surjection
            \[
                \pi(p - 1, \lambda_{r/2}, \omega^{r/2}) \twoheadrightarrow F_{r}.
            \]
            Also, $F_{r + 1}$ is not isomorphic to $\pi(0, 0, \omega^{r/2})$.
    \end{theorem}
    \begin{proof}
        Using Proposition \ref{leq 0}, we get
        \begin{eqnarray*}
          0 \!\! & \equiv & \!\! \sum_{\alpha = 0}^{p - 1}\sum_{j = 0}^{\frac{r - 2}{2}}\sum_{m = 1}^{r/2 - j}{\frac{r + 2}{2} \choose j}p^{x + \frac{r + 2}{2} - j}t_{j, \> m} (-1)^m \alpha^{\frac{r + 2}{2} - j - m} (z - \alpha p)^{j} \mathbbm{1}_{\alpha p + p^2\Zp} \\
            && \!\! + \sum_{a = 1}^{p - 1}\sum_{j = 0}^{r/2}p^{1 + x}{\frac{r + 2}{2} \choose j}a^{r/2 - j}(z - a)^{j}\mathbbm{1}_{a + p\Zp} + \left(\frac{r + 2}{2}\right)p^{1 + x}\cL z^{r/2}\mathbbm{1}_{p\Zp} \mod \pi \latticeL{k}.
        \end{eqnarray*}
        First assume that $\nu < 0$. Then $x + \nu = -1$ implies that $x > -1$. Therefore the first two summands on the right side of the equation above are $0$ modulo $\pi \latticeL{k}$ by Lemma \ref{Integers in the lattice}. So
        \[
          0 \equiv \left(\frac{r + 2}{2}\right)p^{1 + x}\cL z^{r/2}\mathbbm{1}_{p\Zp} \mod \pi \latticeL{k}.
        \]
        Since $\left(\dfrac{r + 2}{2}\right)p^{1 + x}\cL$ is a $p$-adic unit, we get
        \[
            0 \equiv z^{r/2}\mathbbm{1}_{p\Zp} \mod \pi \latticeL{k}.
        \]
        This proves the theorem when $\nu < 0$ because $v_p(\lambda_{r/2}^{-1}) > 0$.

        Next assume that $\nu = 0$. Therefore $\cL \in \Zp$ and $x + v_p(\cL) \geq -1$.
        Noting $p^{1+x} = 1$ and interchanging the sums over $j$ and $m$, we get
        \begin{eqnarray*}
            0 & \equiv & \sum_{\alpha = 0}^{p - 1}\sum_{m = 1}^{r/2}\sum_{j = 0}^{r/2 - m}{\frac{r + 2}{2} \choose j}p^{r/2 - j}t_{j, \> m} (-1)^m \alpha^{\frac{r + 2}{2} - j - m} (z - \alpha p)^{j} \mathbbm{1}_{\alpha p + p^2\Zp} \\
            && + \sum_{a = 1}^{p - 1}\sum_{j = 0}^{r/2}{\frac{r + 2}{2} \choose j}a^{r/2 - j}(z - a)^{j}\mathbbm{1}_{a + p\Zp} + \left(\frac{r + 2}{2}\right)\cL z^{r/2}\mathbbm{1}_{p\Zp} \mod \pi \latticeL{k}.
        \end{eqnarray*}
        As in the proof of Theorem \ref{Final theorem for geq 0}, we can prove that the $2 \leq m \leq r/2$ summands in the first sum on the right side of the equation above are $0$ modulo $\pi\latticeL{k}$. So we have
        \begin{eqnarray*}
            0 & \equiv & \sum_{\alpha = 0}^{p - 1}\sum_{j = 0}^{\frac{r - 2}{2}}{\frac{r + 2}{2} \choose j}p^{r/2 - j}t_{j, \> 1}(-1)\alpha^{r/2 - j}(z - \alpha p)^j\mathbbm{1}_{\alpha p + p^2\Zp} \\
            && + \sum_{a = 1}^{p - 1}\sum_{j = 0}^{r/2}{\frac{r + 2}{2} \choose j}a^{r/2 - j}(z - a)^{j}\mathbbm{1}_{a + p\Zp} + \left(\frac{r + 2}{2}\right)\cL z^{r/2}\mathbbm{1}_{p\Zp} \mod \pi \latticeL{k}.
        \end{eqnarray*}
            Adding and subtracting the $a = 0$ summand in the second sum on the right side, we get
        \begin{eqnarray}\label{Main congruence in leq 0}
            0 & \equiv & \sum_{\alpha = 0}^{p - 1}\sum_{j = 0}^{\frac{r - 2}{2}}{\frac{r + 2}{2} \choose j}p^{r/2 - j}t_{j, \> 1}(-1)\alpha^{r/2 - j}(z - \alpha p)^j\mathbbm{1}_{\alpha p + p^2\Zp} - \frac{r + 2}{2}z^{r/2}\mathbbm{1}_{p\Zp} \\
            && + \sum_{a = 0}^{p - 1}\sum_{j = 0}^{r/2}{\frac{r + 2}{2} \choose j}a^{r/2 - j}(z - a)^{j}\mathbbm{1}_{a + p\Zp} + \left(\frac{r + 2}{2}\right)\cL z^{r/2}\mathbbm{1}_{p\Zp} \mod \pi \latticeL{k}. \nonumber
        \end{eqnarray}

        Next massaging the shallow inductive steps in Proposition~\ref{inductive steps} as usual, we have
        \begin{eqnarray}\label{Refined inductive steps in leq 0}
            0 & \equiv & \sum_{\alpha = 0}^{p - 1}\sum_{j = r - n}^{r/2}{n - r/2 \choose n - r + j}\frac{(r - n)! n}{(r - j)\cdots (n - j)}p^{r/2 - j}\alpha^{r/2 - j}(z - \alpha p)^j\mathbbm{1}_{\alpha p + p^2\Zp} \\ 
            && \qquad \qquad \mod \pi \latticeL{k} \nonumber
        \end{eqnarray}
        for each $n = r, \> r - 1, \> \ldots, \> \dfrac{r + 4}{2}$. Also, the binomial expansion gives us
        \begin{eqnarray}\label{Binomial expansion in leq 0}
            z^{r/2}\mathbbm{1}_{p\Zp} = \sum_{\alpha = 0}^{p - 1}\sum_{j = 0}^{r/2}{r/2 \choose j}p^{r/2 - j}\alpha^{r/2 - j}(z - \alpha p)^{j}\mathbbm{1}_{\alpha p + p^2\Zp}.
        \end{eqnarray}

        For a fixed $0 \leq \alpha \leq p - 1$, we wish to write the first term on the right side of \eqref{Main congruence in leq 0} as a sum of $x_n \in E$ times the corresponding term in \eqref{Refined inductive steps in leq 0}, $x_{\frac{r + 2}{2}} \in E$ times the corresponding term in \eqref{Binomial expansion in leq 0} and $x_{r/2} \in E$ times
        \[
            (z - \alpha p)^{r/2}\mathbbm{1}_{\alpha p + p^2\Zp}.
        \]
        Comparing coefficients of like terms in $p^{r/2 - j}\alpha^{r/2 - j}(z - \alpha p)^j\mathbbm{1}_{\alpha p + p^2\Zp}$ and substituting the values of $t_{j, \> 1} =
        (\frac{r+2}{2} - j) (H_{\frac{r+2}{2}-j} - 1)$ using
        \eqref{Small derivative formula for polynomial times logs},
        we get the following $\left(\dfrac{r + 2}{2}\right) \times \left(\dfrac{r + 2}{2}\right)$ matrix equation
        \[
            \begin{pmatrix}
                {r/2 \choose 0}\frac{r}{r} & 0 & \cdots & {r/2 \choose 0} & 0 \\
                {r/2 \choose 1}\frac{r}{r - 1} & {\frac{r - 2}{2} \choose 0}\frac{r - 1}{(r - 1)(r - 2)} & \cdots & {r/2 \choose 1} & 0 \\
                \vdots & \vdots & & \vdots & \vdots \\
                {r/2 \choose r/2}\frac{r}{r/2} & {\frac{r - 2}{2} \choose \frac{r - 2}{2}}\frac{r - 1}{(r/2)(\frac{r - 2}{2})} & \cdots & {r/2 \choose r/2} & 1
            \end{pmatrix}
            \begin{pmatrix}
                x_r \\ x_{r - 1} \\ \vdots \\ x_{\frac{r + 2}{2}} \\ x_{r/2}
            \end{pmatrix}
            =
            \begin{pmatrix}
                -\frac{r + 2}{2}{r/2 \choose 0}\left(H_{\frac{r + 2}{2}} - 1\right) \\
                - \frac{r + 2}{2}{r/2 \choose 1}\left(H_{r/2} - 1\right) \\
                \vdots \\ -\frac{r + 2}{2}{r/2 \choose \frac{r - 2}{2}}\left(H_{2} - 1\right) \\
                -\frac{r + 2}{2}{r/2 \choose r/2}\left(H_{1} - 1\right)
            \end{pmatrix}.
        \]
        It can be checked\footnote{\label{space-time 17}See Appendix~\ref{Footnote 17}.} that
        \[
            x_{\frac{r + 2}{2}} = \frac{r + 2}{2}\left(1 - H_- - H_+\right) \text{ and } x_{r/2} = -(-1)^{r/2}\frac{1}{r/2}.
        \]
        Therefore \eqref{Main congruence in leq 0} becomes
        \begin{eqnarray}\label{Half refined main congruence in leq 0}
            0 & \equiv & \frac{r + 2}{2}\left(1 - H_- - H_+\right)z^{r/2}\mathbbm{1}_{p\Zp} - (-1)^{r/2}\frac{1}{r/2}\sum_{\alpha = 0}^{p - 1}(z - \alpha p)^{r/2}\mathbbm{1}_{\alpha p + p^2\Zp} \\
            && - \frac{r + 2}{2}z^{r/2}\mathbbm{1}_{p\Zp} + \sum_{a = 0}^{p - 1}\sum_{j = 0}^{r/2}{\frac{r + 2}{2} \choose j}a^{r/2 - j}(z - a)^{j}\mathbbm{1}_{a + p\Zp} + \left(\frac{r + 2}{2}\right)\cL z^{r/2}\mathbbm{1}_{p\Zp} \nonumber \\
            && \qquad \qquad \mod \pi \latticeL{k}. \nonumber
        \end{eqnarray}

        Now applying $\begin{pmatrix}p & 0 \\ 0 & 1 \end{pmatrix}$ to \eqref{Refined inductive steps in leq 0} and replacing $\alpha$ by $a$, we get
        \begin{eqnarray}\label{Refined inductive steps in leq 0 for radius 0}
            0 \equiv \sum_{a = 0}^{p - 1}\sum_{j = r - n}^{r/2}{n - r/2 \choose n - r + j}\frac{(r - n)!n}{(r - j) \cdots (n - j)}a^{r/2 - j}(z - a)^j\mathbbm{1}_{a + p\Zp} \mod \pi \latticeL{k}.
        \end{eqnarray}

        Similarly applying $\begin{pmatrix}p & 0 \\ 0 & 1\end{pmatrix}$ to \eqref{Binomial expansion in leq 0} and replacing $\alpha$ by $a$, we get
        \begin{eqnarray}\label{Binomial expansion in leq 0 in radius 0}
            z^{r/2}\mathbbm{1}_{\Zp} \equiv \sum_{a = 0}^{p - 1}\sum_{j = 0}^{r/2}{r/2 \choose j}a^{r/2 - j}(z - a)^j\mathbbm{1}_{a + p\Zp} \mod \pi \latticeL{k}.
        \end{eqnarray}
        We write the fourth term on the right side of \eqref{Half refined main congruence in leq 0} as a sum of $x_n \in E$ times the corresponding term in \eqref{Refined inductive steps in leq 0 for radius 0}, $x_{\frac{r + 2}{2}} \in E$ times the corresponding term in \eqref{Binomial expansion in leq 0 in radius 0} and $x_{r/2} \in E$ times
        \[
            (z - a)^{r/2}\mathbbm{1}_{a + p\Zp}.
        \]
        Comparing coefficients of like terms in $a^{r/2 - j}(z - a)^j\mathbbm{1}_{a + p\Zp}$, we get the following matrix equation
        \[
            \begin{pmatrix}
                {r/2 \choose 0}\frac{r}{r} & 0 & \cdots & {r/2 \choose 0} & 0 \\
                {r/2 \choose 1}\frac{r}{r - 1} & {\frac{r - 2}{2} \choose 0}\frac{r - 1}{(r - 1)(r - 2)} & \cdots & {r/2 \choose 1} & 0 \\
                \vdots & \vdots & & \vdots & \vdots \\
                {r/2 \choose r/2}\frac{r}{r/2} & {\frac{r - 2}{2} \choose \frac{r - 2}{2}}\frac{r - 1}{(r/2)(\frac{r - 2}{2})} & \cdots & {r/2 \choose r/2} & 1
            \end{pmatrix}
            \begin{pmatrix}
                x_r \\ x_{r - 1} \\ \vdots \\ x_{\frac{r + 2}{2}} \\ x_{r/2}
            \end{pmatrix}
            =
            \begin{pmatrix}
                {\frac{r + 2}{2} \choose 0} \\ {\frac{r + 2}{2} \choose 1} \\ \vdots \\ {\frac{r + 2}{2} \choose \frac{r - 2}{2}} \\ {\frac{r + 2}{2} \choose r/2}
            \end{pmatrix}.
        \]
        We solve this matrix equation using Cramer's rule. Let $A$ be the coefficient matrix above, $A_{\frac{r - 2}{2}}$ be the matrix $A$ with the penultimate column replaced by the column vector on the right and $A_{r/2}$ be the matrix $A$ with the last column replaced by the column vector on the right. Then
        \[
            x_{\frac{r + 2}{2}} = \frac{\det A_{\frac{r - 2}{2}}}{\det A} \text{ and } x_{r/2} = \frac{\det A_{r/2}}{\det A}.
        \]
        
        Putting $i = r/2$ in \eqref{Matrix equation in leq}, we see that after making $A$ lower triangular using column operations the $\left(j, \dfrac{r - 2}{2}\right)^{\mathrm{th}}$ entry for $j \geq \dfrac{r - 2}{2}$ becomes
        \[
            (-1)^{\frac{r - 2}{2}}{1 \choose j - \frac{r - 2}{2}}\frac{(r/2)!}{(r - j) \cdots (r/2 - j + 2)}.
        \]
        Next, we reduce $A_{\frac{r - 2}{2}}$ using column operations. Note that
        \[
            A_{\frac{r - 2}{2}} = \begin{pmatrix}
                {r/2 \choose 0}\frac{r}{r} & 0 & \cdots & {r/2 \choose 0}\frac{r + 2}{2}\frac{1}{\frac{r + 2}{2}} & 0 \\
                {r/2 \choose 1}\frac{r}{r - 1} & {\frac{r - 2}{2} \choose 0}\frac{r - 1}{(r - 1)(r - 2)} & \cdots & {r/2 \choose 1}\frac{r + 2}{2}\frac{1}{r/2} & 0 \\
                \vdots & \vdots & & \vdots & \vdots \\
                {r/2 \choose r/2}\frac{r}{r/2} & {\frac{r - 2}{2} \choose \frac{r - 2}{2}}\frac{r - 1}{(r/2)(\frac{r - 2}{2})} & \cdots & {r/2 \choose r/2}\frac{r + 2}{2}\frac{1}{1} & 1
            \end{pmatrix}.
        \]
        Perform the operation $C_{\frac{r - 2}{2}} \to C_{\frac{r - 2}{2}} - C_0$ to make the $\left(0, \dfrac{r - 2}{2}\right)^{\mathrm{th}}$ entry $0$. Then for $j \geq 1$, the $\left(j, \dfrac{r - 2}{2}\right)^{\mathrm{th}}$ entry becomes
        \[
            {r/2 \choose j}\frac{\frac{r + 2}{2}}{\frac{r + 2}{2} - j} - {r/2 \choose j}\frac{r}{r - j} = {\frac{r - 2}{2} \choose j - 1}\left(\frac{r/2}{\frac{r + 2}{2} - j}\right)\left(\frac{r - 2}{2}\right)\frac{1}{r - j}.
        \]
        The $\left(1, \dfrac{r - 2}{2}\right)^{\mathrm{th}}$ entry is
        \[
            \frac{r - 2}{2}\frac{1}{r - 1}.
        \]
        Now perform the operation $C_{\frac{r - 2}{2}} \to C_{\frac{r - 2}{2}} - (r - 2)\left(\dfrac{r - 2}{2}\right)\dfrac{1}{r - 1}C_1$ to make the $\left(1, \dfrac{r - 2}{2}\right)^{\mathrm{th}}$ entry $0$. Then for $j \geq 2$, the $\left(j, \dfrac{r - 2}{2}\right)^{\mathrm{th}}$ entry becomes
        \begin{eqnarray*}
            && {\frac{r - 2}{2} \choose j - 1}\left(\frac{r/2}{\frac{r + 2}{2} - j}\right)\left(\frac{r - 2}{2}\right)\frac{1}{r - j} - \left(\frac{r - 2}{2}\right){\frac{r - 2}{2} \choose j - 1}\frac{r - 2}{(r - j)(r - j - 1)} \\
            && = {\frac{r - 4}{2} \choose j - 2}\left(\frac{\frac{r - 2}{2}}{\frac{r + 2}{2} - j}\right)\left(\frac{r - 2}{2}\right)\left(\frac{r - 4}{2}\right)\frac{1}{(r - j)(r - j - 1)}.
        \end{eqnarray*}
        Continuing this process we see that after making the first $r/2 - 1$ entries in the penultimate
        column $0$, the $\left(j, \dfrac{r - 2}{2}\right)^{\mathrm{th}}$ entry for $j \geq \dfrac{r - 2}{2}$ is
        \[
            {1 \choose j - \frac{r - 2}{2}}\left(\frac{2}{\frac{r + 2}{2} - j}\right)\left(\frac{r - 2}{2}\right)!\frac{1}{(r - j) \cdots (r - j - \frac{r - 4}{2})}.
        \]
        Therefore by Cramer's rule, we see that
        \[
            x_{\frac{r + 2}{2}} = (-1)^{\frac{r - 2}{2}}\frac{1}{r/2}.
        \]
        Using the same column reductions, we also see that
        \[
            x_{r/2} = \frac{1}{r/2}.
        \]

        Therefore \eqref{Half refined main congruence in leq 0} becomes
        \begin{eqnarray*}
            0 & \equiv & \frac{r + 2}{2}(\cL - H_{-} - H_{+})z^{r/2}\mathbbm{1}_{p\Zp} - (-1)^{r/2}\frac{1}{r/2}\sum_{\alpha = 0}^{p - 1}(z - \alpha p)^{r/2}\mathbbm{1}_{\alpha p + p^2\Zp} \\
            && + (-1)^{\frac{r - 2}{2}}\frac{1}{r/2}z^{r/2}\mathbbm{1}_{\Zp} + \frac{1}{r/2}\sum_{a = 0}^{p - 1}(z - a)^{r/2}\mathbbm{1}_{a + p\Zp} \mod \pi \latticeL{k}.
        \end{eqnarray*}
        Multiplying by $(-1)^{r/2}(r/2)$, we get
        \begin{eqnarray}\label{Final congruence in leq 0}
            0 & \equiv & \lambda_{r/2}z^{r/2}\mathbbm{1}_{p\Zp} - \sum_{\alpha = 0}^{p - 1}(z - \alpha p)^{r/2}\mathbbm{1}_{\alpha p + p^2\Zp} \\
            && - z^{r/2}\mathbbm{1}_{\Zp} + (-1)^{r/2}\sum_{a = 0}^{p - 1}(z - a)^{r/2}\mathbbm{1}_{a + p\Zp} \mod \pi \latticeL{k}. \nonumber
        \end{eqnarray}
    
        We first claim that $F_{r}$ is a quotient of $\pi(p - 1, \lambda_{r/2}, \omega^{r/2})$. Note that
        by \eqref{Formulae for T10 and T12 in the non-commutative case},
        under the surjection $\IZind a^{r/2}d^{r/2} \twoheadrightarrow F_{r, \> r + 1}$, the function
        \begin{eqnarray*}
            (1 + T_{1, 2}T_{1, 0})T_{1, 0}\llbracket \id, X^{r/2}Y^{r/2}\rrbracket & = & \llbracket \beta, (-1)^{r/2}X^{r/2}Y^{r/2}\rrbracket + \sum_{\lambda \in I_1}\left\llbracket \begin{pmatrix}1 & 0 \\ -p\lambda & p\end{pmatrix}, X^{r/2}Y^{r/2}\right\rrbracket
        \end{eqnarray*}
        maps to
        \[
            -(-1)^{r/2}z^{r/2}\mathbbm{1}_{p\Zp} + \sum_{a = 0}^{p - 1}(z - ap)^{r/2}\mathbbm{1}_{ap + p^2\Zp}.
        \]
        Applying $(-1)^{r/2}\begin{pmatrix}p & 0 \\ 0 & 1\end{pmatrix}$, we get
        \[
            -z^{r/2}\mathbbm{1}_{\Zp} + (-1)^{r/2}\sum_{a = 0}^{p - 1}(z - a)^{r/2}\mathbbm{1}_{a + p\Zp}.
        \]
        Therefore using \eqref{Final congruence in leq 0}, we have the identity
        \[
            \lambda_{r/2}z^{r/2}\mathbbm{1}_{p\Zp} - \sum_{\alpha = 0}^{p - 1}(z - \alpha p)^{r/2}\mathbbm{1}_{\alpha p + p^2\Zp} = 0
        \]
        in the quotient $F_{r, \> r + 1}/F_{r + 1}$. Applying $-\begin{pmatrix}p & 0 \\ 0 & 1\end{pmatrix}$, we get
        \[
            \sum_{\alpha = 0}^{p - 1}(z - \alpha)^{r/2}\mathbbm{1}_{\alpha + p\Zp} - \lambda_{r/2}z^{r/2}\mathbbm{1}_{\Zp} = 0.
        \]
        As in the proof of Theorem \ref{Final theorem for geq}, we see that the left side of the equation above is the image of $-w(T_{-1, 0} - \lambda_{r/2})\llbracket \id, X^{r/2}Y^{r/2}\rrbracket$ under the first surjection in \eqref{Two parts in the last slice}. Therefore we get a surjection
        \[
          \pi(p - 1, \lambda_{r/2}, \omega^{r/2}) \simeq
          \frac{\IZind a^{r/2}d^{r/2}}{\im(1 + T_{1, 2}T_{1, 0}) + \im (T_{-1, 0} - \lambda_{r/2})} \twoheadrightarrow \frac{F_{r, \> r + 1}}{F_{r + 1}} \simeq F_{r}.
        \]
        Here, the isomorphism follows from the two following facts:
        $\im(1 + T_{1, 2}T_{1, 0}) =  \im ((1 + T_{1, 2}T_{1, 0}) T_{1,0}) = \im(T_{1,0} + T_{1, 2})$ 
        since $T_{1,0}$ is an automorphism and it satisfies the first relation in
        \eqref{Relations in the non-commutative Hecke algebra}, and \eqref{alt def of pi} (with
        $r = p-1$).

        Next, we prove the final claim of the theorem. 
        By  \eqref{Formulae for T10 and T12 in the non-commutative case},
        under the surjection $\IZind a^{r/2}d^{r/2} \twoheadrightarrow F_{r, \> r + 1}$, the function
        \[
            T_{1, 2}T_{1, 0}\llbracket \beta, (-1)^{r/2}X^{r/2}Y^{r/2}\rrbracket = \sum_{\lambda \in I_1} \left\llbracket \begin{pmatrix}1 & 0 \\ -p\lambda & p\end{pmatrix}, X^{r/2}Y^{r/2}\right\rrbracket
        \]
        maps to
        \[
            \sum_{\alpha = 0}^{p - 1}(z - \alpha p)^{r/2}\mathbbm{1}_{\alpha p + p^2\Zp}.
        \]
        Therefore using \eqref{Final congruence in leq 0}, we have the identity
        \[
            (-1)^{r/2}\sum_{a = 0}^{p - 1}(z - a)^{r/2}\mathbbm{1}_{a + p\Zp} - z^{r/2}\mathbbm{1}_{\Zp} + \lambda_{r/2}z^{r/2}\mathbbm{1}_{p\Zp} = 0
        \]
        in the quotient $F_{r, \> r + 1}/F_{r}$. Now under $\IZind a^{r/2}d^{r/2} \twoheadrightarrow F_{r, \> r + 1}$,
        by  \eqref{Formulae for T-10 and T12 in the commutative case},
        \eqref{Formulae for T10 and T12 in the non-commutative case} and the
        $G$-equivariance of $T_{-1,0}$ and $T_{1,0}$, we have
        \begin{eqnarray*}
            (T_{-1, 0} + T_{1, 0})\llbracket w, X^{r/2}Y^{r/2}\rrbracket = \sum_{\lambda \in I_1}\left\llbracket\begin{pmatrix}0 & 1 \\ p & -\lambda\end{pmatrix}, X^{r/2}Y^{r/2}\right\rrbracket + (-1)^{r/2}\left\llbracket \begin{pmatrix}p & 0 \\ 0 & 1\end{pmatrix}, X^{r/2}Y^{r/2}\right\rrbracket
        \end{eqnarray*}
        maps to
        \[
            -\sum_{a = 0}^{p - 1}(z - a)^{r/2}\mathbbm{1}_{a + p\Zp} + (-1)^{r/2}z^{r/2}\mathbbm{1}_{\Zp}.
          \]
          Assume towards a contraction that $F_{r+1}$ is isomorphic to $\pi(0, 0, \omega^{r/2})$.
          We claim that 
        the second map in \eqref{Two parts in the last slice} factors as
        \[
            \frac{\IZind a^{r/2}d^{r/2}}{\im T_{1, 2}T_{1, 0}} \twoheadrightarrow \frac{\IZind a^{r/2}d^{r/2}}{\im T_{1, 2}T_{1, 0} + \im (T_{-1, 0} + T_{1, 0})} \twoheadrightarrow \frac{F_{r, \> r + 1}}{F_{r}} \simeq F_{r + 1}.
          \]
          Indeed, by \cite[Theorem 3.13]{Chi23}, the left hand side is secretly
          spherical induction of the weight $V_0$ twisted
          by $\omega^{r/2}$. By \cite[Proposition 32]{BL94}, any irreducible quotient of
          such a representation is of
          the form $\pi(0,\lambda',\omega^{r/2})$ for some $\lambda'$. Clearly, we must
          have $\lambda' = 0$. So the second map in  \eqref{Two parts in the last slice}
          factors through the spherical
          Hecke operator, which in the Iwahori setting is known to be $T_{-1,0} + T_{1,0}$,
          again by \cite[Theorem 3.13]{Chi23}. This proves the claim. 
          Putting things together, we see that if $F_{r+1} \simeq \pi(0,0,\omega^{r/2})$, then
          $\lambda_{r/2}z^{r/2}\mathbbm{1}_{p\Zp} = 0$
          in $F_{r, \> r + 1}/F_{r}$. Therefore $F_{r, \> r + 1}/F_{r} = 0$,
          a contradiction. 
    \end{proof}

    We end this section with a remark.
        \begin{remark}\label{constant neq +-1 improvement remark}
          If $\lambda_{r/2}^2 \not \equiv 1 \mod \pi$, then we can
          prove that $F_{r + 1} = 0$. It suffices to show that under the surjection $\IZind a^{r/2}d^{r/2} \twoheadrightarrow F_{r, \> r + 1}$, the subspace $\im (1 + T_{1, 2}T_{1, 0})$ maps to $0$. Since $\im T_{1, 2}T_{1, 0} = \im T_{1, 2}$ maps to $0$ in $F_{r, \> r + 1}/F_{r} \simeq F_{r + 1}$ we see that, the function
        \[
            T_{1, 2}\left(-\llbracket \id, X^{r/2}Y^{r/2}\rrbracket + (-1)^{r/2}\begin{pmatrix}p & 0 \\ 0 & 1\end{pmatrix}\llbracket \id, X^{r/2}Y^{r/2}\rrbracket\right)
        \]
        maps to $0$ in $F_{r, \> r + 1}/F_{r}$. But, by \eqref{Formulae for T10 and T12 in the non-commutative case}, this function maps to
        \[
            -\sum_{\alpha = 0}^{p - 1}(z - \alpha p)^{r/2}\mathbbm{1}_{\alpha p + p^2\Zp} + (-1)^{r/2}\sum_{a = 0}^{p - 1}(z - a)^{r/2}\mathbbm{1}_{a + p\Zp}.
        \]
        Using \eqref{Final congruence in leq 0}, we see that
        \[
            (\lambda_{r/2} + w)z^{r/2}\mathbbm{1}_{p\Zp} = 0
        \]
        in $F_{r, \> r + 1}/F_{r}$. Applying the operator $(\lambda_{r/2} - w)$, we get
        \[
            (\lambda_{r/2}^2 - 1)z^{r/2}\mathbbm{1}_{p\Zp} = 0.
        \]
        Since $\lambda_{r/2}^2 \not \equiv 1 \!\! \mod \pi$, we see that $z^{r/2}\mathbbm{1}_{p\Zp} = 0$ in $F_{r, \> r + 1}/F_{r}$. This shows that $F_{r + 1} = 0$.
        \end{remark}

        \section{Reduction mod $p$ of $\latticeL{k}$}
        \label{Section containing the proof of the main theorem}
    In this section, we summarize all the results proved so far and use them to finally prove Theorem \ref{Main theorem in the second part of my thesis}. We show that for $3 \leq k \leq p+1$ (i.e., $1 \leq r \leq p - 1$) and odd primes $p \geq 5$, the reduction mod $p$ of the semi-stable representation $V_{k, \cL}$ is given in terms of the parameter $\nu = v_p(\cL - H_{-} - H_{+})$ by
\[
    \br{V}_{k, \cL} \sim
    \begin{cases}
        \ind (\omega_2^{r+1 + (i-1)(p-1)}), & \text{ if $(i-1) - r/2 < \nu < i - r/2$} \\
        \mu_{\lambda_i}\omega^{r+1-i} \oplus \mu_{\lambda_i^{-1}}\omega^{i}, & \text{ if $\nu = i - r/2$},
    \end{cases}
\]
where $1 \leq i \leq \dfrac{r+1}{2}$ if $r$ is odd and $1 \leq i \leq \dfrac{r+2}{2}$ if $r$ is even and for certain constants $\lambda_i$. We adopt the following conventions:
\begin{itemize}
    \item The first interval $- r/2 < \nu < 1 - r/2$ should be interpreted as $\nu < 1 - r/2$.
    \item If $r$ is odd, then the last case $\nu = 1/2$ should be interpreted as $\nu \geq 1/2$. If $r$ is even, then the interval $0 < \nu < 1$ should be interpreted as $\nu > 0$ and we drop the case $\nu = 1$.
\end{itemize}

\begin{theorem}\label{Main final theorem in the second part of my thesis}
    For $3 \leq k \leq p + 1$ and $p \geq 5$, we have
    \[
    \br{V}_{k, \cL} \sim
    \begin{cases}
        \ind (\omega_2^{r+1 + (i-1)(p-1)}), & \text{ if $(i-1) - r/2 < \nu < i -r/2$} \\
        \mu_{\lambda_i}\omega^{r+1-i} \oplus \mu_{\lambda_i^{-1}}\omega^{i}, & \text{ if $\nu = i -r/2$},
    \end{cases}
    \]
    where $1 \leq i \leq \dfrac{r+1}{2}$ if $r$ is odd and $1 \leq i \leq \dfrac{r+2}{2}$ if $r$ is even. The constants $\lambda_i$ are determined by
    \begin{eqnarray*}
        \lambda_i & = & \br{(-1)^i \> i {r+1-i \choose i}p^{r/2-i}(\cL - H_{-} - H_{+})}, \quad \text{ if } 1 \leq i < \dfrac{r + 1}{2} \\
        \lambda_{i} + \lambda_i^{-1} & = & \br{(-1)^i \> i{r + 1 - i \choose i}p^{r/2 - i}(\cL - H_{-} - H_{+})}, \quad \text{ if } i = \dfrac{r + 1}{2} \text{ and } r \text{ is odd}.
    \end{eqnarray*}
    We follow the conventions stated before this theorem.
\end{theorem}
\begin{proof}
  After the hard work done in Sections \ref{Common section}, \ref{SpecOdd}, \ref{SpecEven},
  the proof of this theorem is a standard application of the
  compatibility with respect to mod $p$ reduction between the $p$-adic and the
  Iwahori mod $p$ LLC (stated in Theorem~\ref{Iwahori mod p LLC}).

  We collect the necessary results from these sections here.
  We first state the common results for odd and even weights from Section \ref{Common section} (cf. Theorem \ref{Final theorem for geq} and Theorem \ref{Final theorem for leq}).
    \begin{enumerate}
        \item For $i = 1, 2, \ldots, \lceil r/2 \rceil - 1$ if $\nu > i - r/2$, then $F_{2i - 2, \> 2i - 1} = 0$.
        \item For $i = 1, 2, \ldots, \lceil r/2 \rceil - 1$ if $\nu = i - r/2$, then $\pi([2i - 2 - r], \lambda_i^{-1}, \omega^{r - i + 1}) \twoheadrightarrow F_{2i - 2, \> 2i - 1}$.
        \item For $i = 1, 2, \ldots, \lceil r/2 \rceil - 1$ if $\nu < i - r/2$, then $F_{2i, \> 2i + 1} = 0$.
        \item For $i = 1, 2, \ldots, \lceil r/2 \rceil - 1$ if $\nu = i - r/2$, then $\pi(r - 2i, \lambda_i, \omega^{i}) \twoheadrightarrow F_{2i, \> 2i + 1}$.
    \end{enumerate}
    Next, we state the extra results for odd weights from Section \ref{SpecOdd} (cf. Theorem \ref{Final theorem for geq 0.5 for odd weights} and Theorem \ref{Final theorem for -0.5 < nu < 0.5 for odd weights}).
    \begin{enumerate}
        \item[5.] If $\nu \geq 0.5$, then $\pi(p - 2, \lambda_{\frac{r + 1}{2}}, \omega^{\frac{r + 1}{2}}) \oplus \pi(p - 2, \lambda_{\frac{r + 1}{2}}^{-1}, \omega^{\frac{r + 1}{2}}) \twoheadrightarrow F_{r - 1, \> r}$.
        \item[6.] If $-0.5 < \nu < 0.5$, then
          $\dfrac{\IZind a^{\frac{r - 1}{2}}d^{\frac{r + 1}{2}}}{\im T_{-1, 0}} \twoheadrightarrow F_{r-1, \> r}$.
          
    \end{enumerate}
    Finally, we state the extra results for even weights from Section \ref{SpecEven} (cf. Theorem \ref{Final theorem for geq 0} and Theorem \ref{Final theorem for leq 0}).
    \begin{enumerate}
        \item[5.] If $\nu > 0$, then $F_{r - 2, \> r - 1} = 0$.
        \item[6.] If $\nu = 0$, then $\pi(p - 3, \lambda_{r/2}^{-1}, \omega^{\frac{r + 2}{2}}) \twoheadrightarrow F_{r - 2, \> r - 1}$.
        \item[7.] If $\nu < 0$, then $F_{r, \> r + 1} = 0$.
        \item[8.] If $\nu = 0$, then $\pi(p - 1, \lambda_{r/2}, \omega^{r/2}) \twoheadrightarrow F_{r}$ and $F_{r + 1}$ is not isomorphic to $\pi(0, 0, \omega^{r/2})$.
    \end{enumerate}

    We complete the proof for $\nu = 0$ and even weights $k$. By $1.$, we see that $F_{0, \> 1}, \ldots, F_{r - 4, \> r - 3}$ $ = 0$. Therefore $\br{\latticeL{k}} = (F_{r - 2, \> r - 1} \oplus F_{r} \oplus F_{r + 1})^{\mathrm{ss}}$. By $6.$ and $8.$, the first two are quotients of principal series (which pair under the Iwahori mod $p$ LLC). If $F_{r + 1}$ is a quotient of a supersingular representation, then by \cite[Proposition 32]{BL94} it must be a quotient of
    $\pi(0, 0, \omega^{r/2})$ and hence by $8.$, it must vanish. Therefore, we are done in this case. On the other hand, if $F_{r + 1}$ is a quotient of a principal series, then by the
    Iwahori mod $p$ LLC, it must be a quotient of $\pi(0, \lambda_{r/2}, \omega^{r/2})$. This forces $(F_{r} \oplus F_{r + 1})^{\mathrm{ss}}$ to be $\pi(0, \lambda_{r/2}, \omega^{r/2})^{\mathrm{ss}}\> \simeq \pi(p - 1, \lambda_{r/2}, \omega^{r/2})^{\mathrm{ss}}$ (note that if $\lambda_{r/2} = \pm 1$, then both the summands can be non-zero, cf. Remark \ref{constant neq +-1 improvement remark})). Again, we are done by the Iwahori mod $p$ LLC.

    The argument for other $\nu$ is similar and easier.
\end{proof}

\vspace{0.3cm}

{\noindent \bf Acknowledgements.} We are grateful to C. Breuil for many encouraging
conversations about the material in this paper over the past few years.
As will be evident to the reader, this paper leans on foundational material developed
by P. Colmez. We also thank A. Jana for useful discussions on
Iwahori-Hecke operators and P. Paule for introducing us to the computer packages
used in the appendix. The second author thanks A. Abbes and E. Ullmo for invitations
to visit IHES in the summers of 2022 and 2023, and S. Morra to visit Sorbonne Paris Nord
in the summer of 2024.

\vspace{.2cm}




\newpage

\appendix

\section*{Appendix}

In the following appendix, we solve for the variables in the 5 matrix equations referred to in footnotes \ref{space-time 10}, \ref{space-time 11}, \ref{space-time 12} (for $r$ odd) and \ref{space-time 16}, \ref{space-time 17} (for $r$ even) in the main text.
We use Cramer's rule. There are two parts to the argument:

\begin{itemize}
  \item A column reduction argument to reduce determinants to 
           certain binomial identities (which also involve partial harmonic sums). 
  \item A verification of these identities using the {\tt Sigma} or {\tt fastZeil} packages in mathematica.
           These packages were developed at the Research Institute of Symbolic Computation (RISC) at 
           Johannes Kepler University in Linz, Austria. 
           We remark that {\tt fastZeil} provides a computer proof of the main
           recursive step in the proof of the identities which can then be verified by hand. So the proofs of these identities are
           ultimately not based on any computer program.
\end{itemize}           
           
            The appendix is divided into $5$ sections, one for each matrix equation. 

\section{Footnote 10}
  \label{Footnote 10}
The first matrix equation we treat arises in footnote~\ref{space-time 10}. We wish to show that in the solution to
\[
    \! \begin{pmatrix}
        {\frac{r + 3}{2} \choose 0}\frac{r}{r} & 0 & 0 & \cdots & 0 \\
        {\frac{r + 3}{2} \choose 1}\frac{r}{r - 1} & {\frac{r + 1}{2} \choose 0}\frac{r - 1}{(r - 1)(r - 2)} & 0 & \cdots & 0 \\
        {\frac{r + 3}{2} \choose 2}\frac{r}{r - 2} & {\frac{r + 1}{2} \choose 1}\frac{r - 1}{(r - 2)(r - 3)} & {\frac{r - 1}{2} \choose 0}\frac{2!(r - 2)}{(r - 2)(r - 3)(r - 4)} & \cdots & 0 \\
        \vdots & \vdots & \vdots & & \vdots \\
        \binom{\frac{r + 3}{2}}{\frac{r - 1}{2}}\frac{r}{\frac{r + 1}{2}} & {\frac{r + 1}{2} \choose \frac{r - 3}{2}}\frac{r - 1}{(\frac{r + 1}{2})(\frac{r - 1}{2})} & {\frac{r - 1}{2} \choose \frac{r - 5}{2}}\frac{2!(r - 2)}{(\frac{r + 1}{2})(\frac{r - 1}{2})(\frac{r - 3}{2})} & \cdots & 1
        \end{pmatrix} \!\! 
        \begin{pmatrix}x_r \\ x_{r - 1} \\ x_{r - 2} \\ \vdots \\ x_{\frac{r+1}{2}} \end{pmatrix} \!\! = \!\! \begin{pmatrix}{\frac{r + 3}{2} \choose 0}\left(\cL - H_{\frac{r + 3}{2}}\right) \\ {\frac{r + 3}{2} \choose 1}\left(\cL - H_{\frac{r + 1}{2}}\right) \\ {\frac{r + 3}{2} \choose 2}\left(\cL - H_{\frac{r - 1}{2}}\right) \\ \vdots \\ {\frac{r + 3}{2} \choose \frac{r - 1}{2}}(\cL - H_2) \end{pmatrix},
\]
we have $x_{\frac{r + 1}{2}} = (-1)^{\frac{r - 1}{2}}\dfrac{r + 3}{4}\left(\cL - H_{\frac{r - 1}{2}} - H_{\frac{r + 1}{2}}\right)$.

Let $A$ be the coefficient matrix, $X$ be the matrix of the indeterminates and write the matrix on the right as $B_1 + B_2 + B_3$, where
\[
    B_1 = \begin{pmatrix}{\frac{r + 3}{2}\choose 0}\cL \\ {\frac{r + 3}{2}\choose 1}\cL \\ {\frac{r + 3}{2} \choose 2}\cL \\ \vdots \\ {\frac{r + 3}{2}\choose \frac{r - 1}{2}}\cL\end{pmatrix}, B_2 = \begin{pmatrix}-{\frac{r + 3}{2}\choose 0}H_{\frac{r + 3}{2}} \\ -{\frac{r + 3}{2}\choose 1}H_{\frac{r + 3}{2}} \\ -{\frac{r + 3}{2} \choose 2}H_{\frac{r + 3}{2}} \\ \vdots \\ -{\frac{r + 3}{2}\choose \frac{r - 1}{2}}H_{\frac{r + 3}{2}}\end{pmatrix} \text{ and } B_3 = \begin{pmatrix}{\frac{r + 3}{2}\choose 0}(0) \\ {\frac{r + 3}{2}\choose 1}(1/\frac{r + 3}{2}) \\ {\frac{r + 3}{2} \choose 2}(1/\frac{r + 1}{2} + 1/\frac{r + 3}{2}) \\ \vdots \\ {\frac{r + 3}{2}\choose \frac{r - 1}{2}}(1/3 + 1/4 + \cdots + 1/\frac{r + 3}{2})\end{pmatrix}.
\]

Let $A_i$ be the matrix $A$ with last column replaced by $B_i$ for $i = 1, 2,$ and $3$. We first claim that
\[
    \frac{\det A_1}{\det A} = (-1)^{\frac{r - 1}{2}}\frac{r + 3}{4}\cL \quad 
    \text{and} \quad
    \frac{\det A_2}{\det A} = (-1)^{\frac{r - 1}{2}}\frac{r + 3}{4}\left(-H_{\frac{r + 3}{2}}\right).
\]
Indeed, this follows immediately  by showing that after making all but the last entry in the last column of the matrix

\[
    \begin{pmatrix}
    {\frac{r + 3}{2} \choose 0}\frac{r}{r} & 0 & 0 & \cdots & {\frac{r + 3}{2} \choose 0} \\
    {\frac{r + 3}{2} \choose 1}\frac{r}{r - 1} & {\frac{r + 1}{2} \choose 0}\frac{r - 1}{(r - 1)(r - 2)} & 0 & \cdots & {\frac{r + 3}{2} \choose 1} \\
    {\frac{r + 3}{2} \choose 2}\frac{r}{r - 2} & {\frac{r + 1}{2} \choose 1}\frac{r - 1}{(r - 2)(r - 3)} & {\frac{r - 1}{2} \choose 0}\frac{2!(r - 2)}{(r - 2)(r - 3)(r - 4)} & \cdots & {\frac{r + 3}{2}\choose 2} \\
    \vdots & \vdots & \vdots & & \vdots \\
    \binom{\frac{r + 3}{2}}{\frac{r - 1}{2}}\frac{r}{\frac{r + 1}{2}} & {\frac{r + 1}{2} \choose \frac{r - 3}{2}}\frac{r - 1}{(\frac{r + 1}{2})(\frac{r - 1}{2})} & {\frac{r - 1}{2} \choose \frac{r - 5}{2}}\frac{2!(r - 2)}{(\frac{r + 1}{2})(\frac{r - 1}{2})(\frac{r - 3}{2})} & \cdots & {\frac{r + 3}{2} \choose \frac{r - 1}{2}}
    \end{pmatrix},
\]
zero using column operations, the last entry becomes $(-1)^{\frac{r - 1}{2}}\dfrac{r + 3}{4}$.

Here are the column operations. The matrix above is a square matrix of size $\dfrac{r + 1}{2}$. Number its columns using the labels $0$ to $\dfrac{r - 1}{2}$.  First perform $C_{\frac{r - 1}{2}} \to C_{\frac{r - 1}{2}} - C_0$ to make the $(0, 0)^{\mth}$ entry $0$. For $j \geq 1$, the $(j, \frac{r - 1}{2})^{\mth}$ entry becomes
\[
    {\frac{r + 3}{2} \choose j} - {\frac{r + 3}{2} \choose j}\frac{r}{r - j} = -\left(\frac{r + 3}{2}\right){\frac{r + 1}{2} \choose j - 1}\frac{1}{r - j}.
\]
So after this operation, the $(1, \frac{r - 1}{2})^{\mth}$ entry is $-\left(\frac{r + 3}{2}\right)\frac{1}{r - 1}$.

Now perform the operation $C_{\frac{r - 1}{2}} \to C_{\frac{r - 1}{2}} + \frac{r - 2}{r - 1}\left(\frac{r + 3}{2}\right)C_1$ to make the $(1, \frac{r - 1}{2})^{\mth}$ entry $0$. For $j \geq 2$, the $(j, \frac{r - 1}{2})^{\mth}$ entry becomes
\begin{eqnarray*}
    && -\left(\frac{r + 3}{2}\right){\frac{r + 1}{2} \choose j - 1}\frac{1}{r - j} + \frac{r - 2}{(r - j)(r - j - 1)}\left(\frac{r + 3}{2}\right){\frac{r + 1}{2} \choose j - 1} \\ 
    && \>\quad = \left(\frac{r + 3}{2}\right)\left(\frac{r + 1}{2}\right){\frac{r - 1}{2} \choose j - 2}\frac{1}{(r - j)(r - j - 1)}.
\end{eqnarray*}

Continuing this process, after making the first $\frac{r - 1}{2}$ entries in the last column $0$, for $j \geq \frac{r - 1}{2}$, the $(j, \frac{r - 1}{2})^{\mth}$ entry becomes
\[
    (-1)^{\frac{r - 1}{2}}\left(\frac{r + 3}{2}\right)\left(\frac{r + 1}{2}\right) \cdots (3) {2 \choose j - \left(\frac{r - 1}{2}\right)}\frac{1}{(r - j)(r - j - 1) \cdots (\frac{r + 3}{2} - j)}.
\]
So the $\left(\frac{r - 1}{2}, \frac{r - 1}{2}\right)^\mth$ entry becomes $(-1)^{\frac{r - 1}{2}}\frac{r + 3}{4}$.

Next, we solve the equation $AX = B_3$ by solving for each fraction $1/(n + 1)$ for $n = 2, 3, \ldots, \frac{r + 1}{2}$ separately. We claim that the $1/(n + 1)$ part of $x_{\frac{r + 1}{2}}$ without $1/(n + 1)$ is 
\[
    \left(\frac{r + 3}{4}\right)\left[\sum_{j = 1}^{n - 1}(-1)^{j - 1}{2j - 1 \choose j}{\frac{r - 1}{2} + j \choose 2j - 1}\right].
\]

This is obvious for $n = 2$ using Cramer's rule since only the last entry in the last column is non-zero. For $n = 3$, the last two entries are non-zero. So we perform the operation $C_{\frac{r - 1}{2}} \to C_{\frac{r - 1}{2}} - {\frac{r + 3}{2} \choose \frac{r - 3}{2}}\frac{\left(\frac{r + 1}{2}\right)\left(\frac{r - 1}{2}\right)}{(2)(1)}C_{\frac{r - 3}{2}}$ to make the second-to-last entry $0$. Therefore the last entry becomes
\begin{eqnarray*}
    && \frac{1}{2}\left(\frac{r + 3}{2}\right)\left(\frac{r + 1}{2}\right) - \frac{1}{6}\left(\frac{r + 3}{2}\right)\left(\frac{r + 1}{2}\right)^2\left(\frac{r - 1}{2}\right)^2\frac{1}{2}3\left(\frac{r + 3}{2}\right)\frac{1}{\left(\frac{r + 1}{2}\right)\left(\frac{r - 1}{2}\right)} \\
    && \>\quad = {\frac{r + 3}{2} \choose 2}\left[1 - \frac{1}{2}\left(\frac{r + 3}{2}\right)\left(\frac{r - 1}{2}\right)\right],
\end{eqnarray*}
verifying the claim in this case. The general case is based on an intelligent guess whose verification we
leave to the reader.

In any case, to show that $x_{\frac{r + 1}{2}} = (-1)^{\frac{r - 1}{2}}\dfrac{r + 3}{4}\left(\cL - H_{\frac{r - 1}{2}} - H_{\frac{r + 1}{2}}\right)$, it suffices to prove the identity
\begin{eqnarray}
  \label{main identity 10}
    \sum_{n = 2}^{\frac{r + 1}{2}}\frac{1}{n + 1}\sum_{j = 1}^{n - 1}(-1)^{j - 1}{2j - 1 \choose j}{\frac{r - 1}{2} + j \choose 2j - 1} = (-1)^{\frac{r - 1}{2}}\left(H_{\frac{r +3}{2}} -H_{\frac{r -1}{2}} -H_{\frac{r +1}{2}} \right)
\end{eqnarray}
Switching the order of $j$ and $n$, the double sum on the left is
\begin{eqnarray}
   \label{reorder n and j}
    \sum_{j = 1}^{\frac{r - 1}{2}} (-1)^{j - 1}{2j - 1 \choose j}{\frac{r - 1}{2} + j \choose 2j - 1} \sum_{n = j+1}^{\frac{r+1}{2}} 
    \frac{1}{n +1} 
    \>\quad & = & \sum_{j = 1}^{\frac{r - 1}{2}} (-1)^{j - 1}{2j - 1 \choose j}{\frac{r - 1}{2} + j \choose 2j - 1} (H_{\frac{r+3}{2}} - H_{j+1}) \nonumber \\
    = \quad H_{\frac{r+3}{2}}  \cdot \sum_{j = 1}^{\frac{r - 1}{2}} (-1)^{j - 1}{2j - 1 \choose j}{\frac{r - 1}{2} + j \choose 2j - 1}  
    &   +  & \quad    \sum_{j = 1}^{\frac{r - 1}{2}} (-1)^{j} {2j - 1 \choose j}{\frac{r - 1}{2} + j \choose 2j - 1} H_{j+1}. 
\end{eqnarray}

The first sum on the right of \eqref{reorder n and j} is easy to compute. Rearranging the binomial coefficients using the identity
${n \choose m} {m \choose j} = {n \choose j} {n-j \choose m-j}$, it is just
\begin{eqnarray*}
     \sum_{j = 1}^{\frac{r - 1}{2}} (-1)^{j - 1} {\frac{r - 1}{2} + j \choose j}{\frac{r - 1}{2}  \choose j-1} & = &  (-1)^\frac{r - 1}{2}  \left(1- {r \choose \frac{r +1}{2}} \right).
\end{eqnarray*}
This can be seen by manipulating some identities in Gould \cite[volumes 4, 5]{Gou10} However,  alternatively,
and more conceptually, we may use the Gauss' summation formula
\begin{eqnarray}
  \label{GSF}
  \prescript{}{2} {F_1} \left({-n, b \atop c}, 1 \right) = \frac{(c-b)_n}{(c)_n}
\end{eqnarray} 
for the Gauss hypergeometric function $\prescript{}{2} {F_1} \left({a, b \atop c}, z \right) = \sum_{k=0}^\infty \frac{(a)_k (b)_k}{(c)_k} \frac{z^k}{k!}$ noting that the series has finitely many terms when $a = -n$ is a negative integer since the rising factorial symbol
$(a)_k = a \cdot (a+1) \cdots (a+k-1)$ eventually dies.
This approach was shown to us by Peter Paule (in the context of Appendix~\ref{Footnote 17}  but we explain all the details in the present context and only recall some details there). 
Changing notation a bit
we must evaluate
\begin{eqnarray*}
\sum_{k = 1}^{n} (-1)^{k- 1}{n+k \choose k}{n \choose k - 1}  
           & =  & \sum_{k' = 0}^{n-1} (-1)^{k'}{n+k'+1 \choose k'+1}{n \choose k'}  \\
           & = & \sum_{k = 0}^{\infty} (-1)^{k}{n+k+1 \choose k+1}{n \choose k} - (-1)^n {2n+1 \choose n+1} \\
           & = & (n+1) \sum_{k = 0}^{\infty} (-1)^{k}{n+k+1 \choose k+1}{n \choose k} / (n+1) - (-1)^n {2n+1 \choose n+1}
\end{eqnarray*}        
for $n = \frac{r-1}{2}$.  
Write $f(n,j) =     (-1)^{j} {n+j+1 \choose j+1}{n \choose j} / (n+1)$. Clearly $f(n,0) = 1$, and by telescoping
\[ f(n,k) = \prod_{j=0}^{k-1} \frac{f(n,j+1)}{f(n,j)}  = \prod_{j=0}^{k-1}  \frac{(-n+j)(n+2+j)}{(j+2)} \cdot \frac{1}{j+1} = \frac{(-n)_k (n+2)_k}{(2)_k} \cdot \frac{1}{k!}. \]
Thus the last sum on the right above is just
$\prescript{}{2}{F_1}\left({-n,n+2 \atop 2}, 1 \right) = \frac{(-n)_n}{(2)_n}$, by \eqref{GSF}.
Thus the first term on the right above is 
\[ (n+1) \cdot \frac{(-n)_n}{(2)_n} = (n+1) \cdot  \frac{(-n)\cdot (-n+1)\cdots (-n+n-1)}{2 \cdot 3 \cdots (2 + n-1)} = (-1)^n \]
as desired.    

The second sum on the right of \eqref{reorder n and j} is trickier. Changing notation we may write it is
as 
\[ \sum_{k = 1}^{n} (-1)^{k}{n+k \choose k}{n \choose k - 1}  H_{k+1} \]
for $n = \frac{r-1}{2}$.
The embedded partial Harmonic sum throws one off. 
However, this sum may be evaluated using the package {\tt Sigma} in mathematica
(developed by Carsten Schneider). Remarkably, the sum satisfies a second order recurrence relation which has a
solution that is an appropriate linear combination of a general solution and a particular solution, much as 
in the theory of $2^{nd}$ order ordinary differential equations. Here are the commands in mathematica
(in text mode):
{\tiny
\begin{eqnarray*}
  \text{In[1]}     & :=  & <<\text{Sigma.m}   \\     
                       &      & \text{Sigma - A summation package by Carsten Schneider -- ? RISC --  
                                           V 2.89 (November 10, 2021)}  \\             
  \text{In[2]}     & :=  & \text{mySum = SigmaSum[SigmaPower[-1,k] SigmaBinomial[n+k,k] SigmaBinomial[n,k-1]     
                                           SigmaHNumber[k+1], \textbraceleft k,1,n\textbraceright]}       \\                                           
  \text{Out[2]}  &  =  & \text{SigmaSum[(1/(k+1) + SigmaHNumber[1, k]) SigmaBinomial[n, -1 + k] 
                                        SigmaBinomial[k + n, k] SigmaPower[-1, k], \textbraceleft k, 1, n\textbraceright]}   \\
  \text{In[3]}     & :=  & \text{rec = GenerateRecurrence[mySum][[1]]}        \\                            
  \text{Out[3]}  &  =  & \text{(1 + n) (5 + 2 n) SUM[n] + 2 (7 + 8 n + 2 n\textsuperscript{2} ) SUM[1 + n] 
                                              + (3 + n) (3 + 2 n)  SUM[2 + n]}   \\
                       &       & \text{   ==  ((1 + 2 n) ((3 + 2 n) (944 + 1847 n + 1298 n\textsuperscript{2}  
                                              + 389 n\textsuperscript{3}  + 42 n\textsuperscript{4} ) 
                                              + (1 + n) (2 + n) (3 + n) (4 + n) (53 + 63 n + 18 n\textsuperscript{2}) 
                                              }   \\
                       &       &  \text{   SigmaHNumber[1, n]) SigmaBinomial[2 n, n] SigmaPower[-1, n]) / 
                                                       ((1 + n)\textsuperscript{2}  (2 + n) (3 + n) (4 + n))}. \\
  \text{In[4]}    &  :=  &  \text{recSol = SolveRecurrence[rec, SUM[n]]}   \\                                   
  \text{Out[4]} &   =  &  \text{\textbraceleft \textbraceleft0, SigmaPower[-1, n]\textbraceright,  
                                               \textbraceleft0, (-1/(n+1) - 2 SigmaHNumber[1, n]) 
                                                                                            SigmaPower[-1, n]\textbraceright,} \\
                      &       &  \text{    \textbraceleft1, ( ((1 + 2 n) (3 + 2 n)) / ((1 + n)\textsuperscript{2}  (2 + n))
                                      +  (1 + 2 n)  SigmaHNumber[1, n] / (1 + n) )} \\
                      &       &  \text{      SigmaBinomial[2 n, n]  SigmaPower[-1, n]\textbraceright  \textbraceright}    \\
  \text{In[5]}    & :=  & \text{FindLinearCombination[recSol, mySum, n, 2]  }           \\       
  \text{Out[5]} &  =  & \text{SigmaPower[-1, n]/(-1-n) - 2 SigmaHNumber[1, n] SigmaPower[-1, n]} \\  
                      &      & \text{   +  ((1 + 2 n) (3 + 2 n)/ ((1 + n)\textsuperscript{2}(2 + n)) SigmaBinomial[2 n, n]         
                                           SigmaPower[-1, n]}   \\
                      &      & \text{   + ((1 + 2 n)/(1+n)) SigmaHNumber[1, n] SigmaBinomial[2 n, n] SigmaPower[-1, n] }
\end{eqnarray*}
}                                                                   
Thus within a few steps {\tt Sigma} yields that the tricky sum above is
\[ 
  (-1)^n (-H_n-H_{n+1}) + (-1)^n \frac{(2n+1)(2n+3)}{(n+1)^2(n+2)}{2n \choose n} + (-1)^n \frac{(2n+1)}{(n+1)} {2n \choose n} H_n 
\]
for $n = \frac{r-1}{2}$. We remark that once the recurrence relation has been produced by {\tt Sigma}, 
it can be checked that the sum above satisfies it (as an example we show how this can be done
for a similar sum and recurrence relation in Appendix~\ref{Footnote 17}). Since the purported answer above 
also satisfies the recurrence relation (a shorter by hand check), and both the sum and the answer above 
agree for small values of $n$,  they must match for all values of $n$. This reasoning makes the 
evaluation of the tricky sum entirely theoretical!

In any case, plugging the evaluations of the easy sum and the tricky sum into \eqref{reorder n and j}, and simplifying noting $H_{\frac{r+3}{2}} = H_n + 1/(n+1) + 1/(n+2)$, we obtain the identity \eqref{main identity 10}.

\section{Footnote 11}\label{Footnote 11}
    To solve
\[
        \begin{pmatrix}
            {\frac{r + 1}{2} \choose 0}\frac{r}{r} & 0 & 0 & \cdots & 0 \\
            {\frac{r + 1}{2}  \choose 1}\frac{r}{r - 1} & {\frac{r - 1}{2} \choose 0}\frac{r - 1}{(r - 1)(r - 2)} & 0 & \cdots & 0 \\
            {\frac{r + 1}{2} \choose 2}\frac{r}{r - 2} & {\frac{r - 1}{2} \choose 1}\frac{r - 1}{(r - 2)(r - 3)} & {\frac{r - 3}{2} \choose 0}\frac{2!(r - 2)}{(r - 2)(r - 3)(r - 4)} & \cdots & 0 \\
            \vdots & \vdots & \vdots & & \vdots \\
            {\frac{r + 1}{2} \choose \frac{r - 1}{2}}\frac{r}{\frac{r + 1}{2}} & {\frac{r - 1}{2} \choose \frac{r - 3}{2}}\frac{r - 1}{(\frac{r + 1}{2})(\frac{r - 1}{2})} & {\frac{r - 3}{2} \choose \frac{r - 5}{2}}\frac{2!(r - 2)}{(\frac{r + 1}{2})(\frac{r - 1}{2})(\frac{r - 3}{2})} & \cdots & 1
        \end{pmatrix}
        \begin{pmatrix}
            x_r \\ x_{r - 1} \\ x_{r - 2} \\ \vdots \\ x_{\frac{r + 1}{2}}
        \end{pmatrix}
        =
        \begin{pmatrix}
            {\frac{r + 1}{2} \choose 0}(\cL - H_{\frac{r + 1}{2}}) \\
            {\frac{r + 1}{2} \choose 1}(\cL - H_{\frac{r - 1}{2}}) \\
            {\frac{r + 1}{2} \choose 2}(\cL - H_{\frac{r - 3}{2}}) \\
            \vdots \\
            {\frac{r + 1}{2} \choose \frac{r - 1}{2}}(\cL - H_{1})
        \end{pmatrix}.
    \]
We have to prove that 
 \begin{eqnarray}
   \label{solution 11}
   x_{\frac{r + 1}{2}} = (-1)^{\frac{r - 1}{2}}(\cL - H_{-} - H_{+}).
\end{eqnarray}
As in Appendix~\ref{Footnote 10}, we separate the matrix on the right as $B_1 + B_2 + B_3$, where
    \[
        B_1 = \begin{pmatrix}{\frac{r + 1}{2} \choose 0}\cL \\ {\frac{r + 1}{2} \choose 1}\cL \\ {\frac{r + 1}{2} \choose 2} \cL \\ \vdots \\ {\frac{r + 1}{2} \choose \frac{r - 1}{2}}\cL\end{pmatrix}, B_2 = \begin{pmatrix}-{\frac{r + 1}{2} \choose 0}H_{\frac{r + 1}{2}} \\ -{\frac{r + 1}{2} \choose 1}H_{\frac{r + 1}{2}} \\ -{\frac{r + 1}{2} \choose 2}H_{\frac{r + 1}{2}} \\ \vdots \\ -{\frac{r + 1}{2} \choose \frac{r - 1}{2}}H_{\frac{r + 1}{2}}\end{pmatrix}, \text{ and } B_3 = \begin{pmatrix}{\frac{r + 1}{2} \choose 0}(0) \\ {\frac{r + 1}{2} \choose 0} (1/\frac{r - 1}{2}) \\ {\frac{r + 1}{2} \choose 0}(1/\frac{r - 1}{2} + 1/\frac{r + 1}{2}) \\ \vdots \\ {\frac{r + 1}{2} \choose \frac{r - 1}{2}}(1/2 + 1/3 + \cdots + 1/\frac{r + 1}{2})\end{pmatrix}.
    \]
As before, we show that
    \[
        \frac{\det A_1}{\det A} = (-1)^{\frac{r - 1}{2}} \cL \quad \text{ and } \quad \frac{\det A_2}{\det A} = (-1)^{\frac{r - 1}{2}} (- H_{\frac{r+1}{2}}).
    \]

    Consider the following column operations on the matrix $A_1$ (with the common $\cL$ term in every entry of the last column dropped and added in at the end). First perform $C_{\frac{r - 1}{2}} \to C_{\frac{r - 1}{2}} - C_0$ to make the $(0, \frac{r - 1}{2})^{\mth}$ entry $0$. For $j \geq 1$, the $(j, \frac{r - 1}{2})^\mth$ entry becomes
    \[
        {\frac{r + 1}{2} \choose j} - {\frac{r + 1}{2} \choose j}\frac{r}{r - j} = -\left(\frac{r + 1}{2}\right)\frac{1}{r - j}{\frac{r - 1}{2} \choose j - 1}.
    \]
    So the $(1, \frac{r - 1}{2})^\mth$ entry is $-\left(\frac{r + 1}{2}\right)\frac{1}{r - 1}$.

    Now perform $C_{\frac{r - 1}{2}} \to C_{\frac{r - 1}{2}} + (r - 2)\left(\frac{r + 1}{2}\right)\frac{1}{r - 1}C_1$. For $j \geq 2$, the $(j, \frac{r - 1}{2})^\mth$ entry becomes
    \begin{eqnarray*}
        && -\left(\frac{r + 1}{2}\right)\frac{1}{r - j}{\frac{r - 1}{2} \choose j - 1} + \left(\frac{r + 1}{2}\right)\frac{r - 2}{(r - j)(r - j - 1)}{\frac{r - 1}{2} \choose j - 1} \\
        && \>\quad = \left(\frac{r + 1}{2}\right)\left(\frac{r - 1}{2}\right)\frac{1}{(r - j)(r - j - 1)}{\frac{r - 3}{2} \choose j - 2}.
    \end{eqnarray*}

    Continuing this process, after making the first $\frac{r - 1}{2}$ entries $0$, for $j \geq \frac{r - 1}{2}$, the $(j, \frac{r - 1}{2})^\mth$ entry becomes
    \[
        (-1)^{\frac{r - 1}{2}}\left(\frac{r + 1}{2}\right)\left(\frac{r - 1}{2}\right) \cdots (2) \frac{1}{(r - j)(r - j - 1)\cdots (\frac{r + 3}{2} - j)}{1 \choose j - \frac{r - 1}{2}}.
    \]
So the $\left(\frac{r - 1}{2}, \frac{r - 1}{2}\right)^\mth$ entry is just $(-1)^{\frac{r - 1}{2}}$. 
Putting the $\cL$ back in we get the first ratio of determinants above is as claimed. The second ratio
follows identically.

    Next, we solve $AX = B_3$ by solving for $1/(n + 1)$ for $n = 1, 2, \ldots, \frac{r - 1}{2}$ separately. We claim that the solution to the $1/(n + 1)$ part is
    \[
        \sum_{j = 1}^{n} (-1)^{j - 1} {2j - 1 \choose j}{\frac{r - 1}{2} + j \choose 2j - 1}.
    \]

    This is easy to check for $n = 1$ as before. For $n = 2$, we use column operations to reduce the last column to its last entry. We perform $C_{\frac{r - 1}{2}} \to C_{\frac{r - 1}{2}} - {\frac{r + 1}{2} \choose \frac{r - 3}{2}}\frac{\left(\frac{r + 1}{2}\right)\left(\frac{r - 1}{2}\right)}{(2)(1)}C_{\frac{r - 3}{2}}$ to make the second-to-last entry in the last column $0$. As a consequence, the last entry becomes
    \[
        {\frac{r + 1}{2}} - \frac{1}{2}\left(\frac{r + 1}{2}\right)\left(\frac{r - 1}{2}\right)\left(\frac{r + 3}{2}\right) = {\frac{r + 1}{2} \choose 1} - 3{\frac{r + 3}{2} \choose 3},
    \]
    verifying the claim when $n = 2$. The case of general $n$ is an intelligent guess left to the reader as
    an exercise.

    To show \eqref{solution 11}, 
    we have to prove the identity
    \[
        \sum_{n = 1}^{\frac{r - 1}{2}} \frac{1}{n + 1} \sum_{j = 1}^{n}(-1)^{j - 1}{2j - 1 \choose j}{\frac{r - 1}{2} + j \choose 2j - 1} = (-1)^{\frac{r - 1}{2}}(-H_{-}).
    \]
As before (cf. \eqref{reorder n and j} above), the sum on the left can be broken into two sums and equals
\[
 H_{\frac{r+1}{2}}  \cdot \sum_{j = 1}^{\frac{r - 1}{2}} (-1)^{j - 1}{2j - 1 \choose j}{\frac{r - 1}{2} + j \choose 2j - 1}  
     \quad  +   \quad    \sum_{j = 1}^{\frac{r - 1}{2}} (-1)^{j} {2j - 1 \choose j}{\frac{r - 1}{2} + j \choose 2j - 1} H_{j+1}. 
\]
Changing notation (and binomial coefficients) a bit this equals
\[
 H_{\frac{r+1}{2}}  \cdot \sum_{k = 1}^{n} (-1)^{k - 1}{n+k \choose k}{n \choose k - 1}  
     \quad  +   \quad    \sum_{k = 1}^{n} (-1)^{k} {n+k \choose k}{n \choose k - 1} H_{k}
\]
with $n = \frac{r-1}{2}$. The first (easy) sum was evaluated above. Using {\tt Sigma} for the
second, the above expression  equals:
\[
  H_{n+1} \cdot (-1)^n \left(1 - {2n+1 \choose n +1} \right) \quad +
  \quad (-1)^n (-H_n - H_{n+1})  
  + (-1)^n H_{n+1} {2n+1 \choose n +1}
\]
which clearly equals  $(-1)^n (-H_n)$, as desired putting $n = \frac{r-1}{2}$.

\section{Footnote 12}
 \label{Footnote 12}

We show that in the matrix equation
\[
\!\!\!\! \begin{pmatrix}
	{\frac{r + 1}{2} \choose 0}\frac{r}{r} & 0 & 0 & \cdots & 0 & 0\\
	{\frac{r + 1}{2}  \choose 1}\frac{r}{r - 1} & {\frac{r - 1}{2} \choose 0}\frac{r - 1}{(r - 1)(r - 2)} & 0 & \cdots & 0 & 0\\
	{\frac{r + 1}{2} \choose 2}\frac{r}{r - 2} & {\frac{r - 1}{2} \choose 1}\frac{r - 1}{(r - 2)(r - 3)} & {\frac{r - 3}{2} \choose 0}\frac{2!(r - 2)}{(r - 2)(r - 3)(r - 4)} & \cdots & 0 & 0\\
	\vdots & \vdots & \vdots & & \vdots & \vdots\\
	{\frac{r + 1}{2} \choose \frac{r - 1}{2}}\frac{r}{\frac{r + 1}{2}} & {\frac{r - 1}{2} \choose \frac{r - 3}{2}}\frac{r - 1}{(\frac{r + 1}{2})(\frac{r - 1}{2})} & {\frac{r - 3}{2} \choose \frac{r - 5}{2}}\frac{2!(r - 2)}{(\frac{r + 1}{2})(\frac{r - 1}{2})(\frac{r - 3}{2})} & \cdots & 1 & 0\\
	{\frac{r + 1}{2} \choose \frac{r + 1}{2}}\frac{r}{\frac{r - 1}{2}} & {\frac{r - 1}{2} \choose \frac{r - 1}{2}}\frac{r - 1}{(\frac{r - 1}{2})(\frac{r - 3}{2})} & {\frac{r - 3}{2} \choose \frac{r - 3}{2}}\frac{2!(r - 2)}{(\frac{r - 1}{2})(\frac{r - 3}{2})(\frac{r - 5}{2})} & \cdots & 0 & 1
\end{pmatrix} \!\!
\begin{pmatrix}
	x_r \\ x_{r - 1} \\ x_{r - 2} \\ \vdots \\ x_{\frac{r + 1}{2}} \\ x_{\frac{r - 1}{2}}
\end{pmatrix} \!\! 
= \!\! 
\begin{pmatrix}
	{\frac{r + 1}{2} \choose 0}(\cL - H_{\frac{r + 1}{2}}) \\
	{\frac{r + 1}{2} \choose 1}(\cL - H_{\frac{r - 1}{2}}) \\
	{\frac{r + 1}{2} \choose 2}(\cL - H_{\frac{r - 3}{2}}) \\
	\vdots \\
	{\frac{r + 1}{2} \choose \frac{r - 1}{2}}(\cL - H_{1}) \\
	0
\end{pmatrix},
\]
we have
$$x_{\frac{r + 1}{2}} = (-1)^{\frac{r - 1}{2}}\left(\cL - H_{-} - H_{+}\right).$$ 

Since the coefficient matrix above is invertible, there is a unique solution. Moreover, 
since the matrix equation above 
contains exactly the same equations as the equations in the matrix 
equation in Appendix~\ref{Footnote 11}  (except for the last equation which may be used to solve for the last variable
in terms of the previous ones),
the formula for  $x_{\frac{r + 1}{2}}$ 
follows immediately from \eqref{solution 11}.

Next, we show that 
 $$x_{\frac{r - 1}{2}} = -\cL + (-1)^{\frac{r - 1}{2}}\frac{2}{r + 1} + (-1)^{\frac{r - 1}{2}}\frac{r + 1}{2}(\cL - H_{-} - H_{+}).$$ 
 This implies the check in footnote~\ref{space-time 12} since this expression reduces mod $\pi$ to the one in that footnote,
noting that $\nu = v_p(\cL-H_{-}-H_{+}) \geq 0.5 > 0$ there.

Write the matrix on the right as $B_1 + B_2 + B_3 + B_4$, where
\[
B_1 = \begin{pmatrix} {\frac{r + 1}{2} \choose 0}\cL \\ {\frac{r + 1}{2} \choose 1}\cL \\ {\frac{r + 1}{2} \choose 2}\cL \\ \vdots \\ {\frac{r + 1}{2} \choose \frac{r + 1}{2}}\cL\end{pmatrix}, B_2 = \begin{pmatrix}0 \\ 0 \\ 0 \\ \vdots \\ -\cL\end{pmatrix}, B_3 = \begin{pmatrix}{\frac{r + 1}{2} \choose 0}(-H_{+}) \\ {\frac{r + 1}{2} \choose 1}(-H_{+}) \\ {\frac{r + 1}{2} \choose 2}(-H_{+}) \\ \vdots \\ {\frac{r + 1}{2} \choose \frac{r + 1}{2}}(-H_{+})\end{pmatrix} \text{ and } B_4 = \begin{pmatrix}{\frac{r + 1}{2} \choose 0}(0) \\ {\frac{r + 1}{2} \choose 1}(1/\frac{r + 1}{2}) \\ {\frac{r + 1}{2} \choose 2}(1/\frac{r - 1}{2} + 1/\frac{r + 1}{2}) \\ \vdots \\ {\frac{r + 1}{2} \choose \frac{r + 1}{2}}(1 + 1/2 + \cdots + 1/ \frac{r + 1}{2})\end{pmatrix}.
\]

We solve the $B_1$ and $B_3$ parts of the matrix equation above together as usual. Consider the matrix obtained by replacing the last column in the coefficient matrix with $B_1$ but without the common $\cL$ term. We perform the operation $C_{\frac{r - 1}{2}} \to C_{\frac{r - 1}{2}} - C_0$ to make the $(0, \frac{r + 1}{2})^\mth$ entry $0$. For $j \geq 1$, the $(j, \frac{r + 1}{2})^\mth$ entry becomes
\[
{\frac{r + 1}{2} \choose j} - {\frac{r + 1}{2} \choose j}\frac{r}{r - j} = {\frac{r + 1}{2} \choose j}\left(\frac{-j}{r - j}\right) = -\left(\frac{r + 1}{2}\right)\frac{1}{r - j}{\frac{r - 1}{2} \choose j - 1}.
\]
So the $(1, \frac{r + 1}{2})^\mth$ entry is $-\left(\frac{r + 1}{2}\right)\frac{1}{r - 1}$. 

Next, we perform $C_{\frac{r + 1}{2}} \to C_{\frac{r + 1}{2}} + (r - 2)\left(\frac{r + 1}{2}\right)\frac{1}{r - 1}C_1$
to make the $(1, \frac{r+1}{2})^\mth$ entry $0$. For $j \geq 2$, the $(j, \frac{r + 1}{2})^\mth$ entry becomes
\begin{eqnarray*}
	&& -\left(\frac{r + 1}{2}\right)\frac{1}{r - j}{\frac{r - 1}{2} \choose j - 1} + (r - 2)\left(\frac{r + 1}{2}\right){\frac{r - 1}{2} \choose j - 1}\frac{1}{(r - j)(r - j - 1)} \\
	&& \>\quad = \left(\frac{r + 1}{2}\right)\left(\frac{r - 1}{2}\right)\frac{1}{(r - j)(r - j - 1)}{\frac{r - 3}{2} \choose j - 2}.
\end{eqnarray*}

Continuing as usual, and after making the first $\frac{r - 1}{2}$ entries in the last column $0$, for $j \geq \frac{r -1}{2}$, the $(j, \frac{r + 1}{2})^\mth$ entry becomes
\[
(-1)^{\frac{r - 1}{2}} \left(\frac{r + 1}{2}\right)\left(\frac{r - 1}{2}\right) \cdots (2)\frac{1}{(r - j)(r - j - 1)\cdots (\frac{r + 3}{2} - j)}{1 \choose j - \frac{r - 1}{2}}.
\]
The above column operations reduce the lower right 2 by 2 matrix to an upper triangular matrix 
with $(\frac{r + 1}{2}, \frac{r + 1}{2})^\mth$ entry equal to $(-1)^{\frac{r - 1}{2}}\frac{r + 1}{2}$. Note that 
the $(\frac{r - 1}{2}, \frac{r - 1}{2})^\mth$ entry remains unchanged and is 1 and the $(\frac{r-1}{2}, \frac{r + 1}{2})^\mth$ entry is irrelevant when computing its determinant. By Cramer's rule (and putting $\cL$ back in)
we see that $B_1$ contributes $(-1)^{\frac{r - 1}{2}} \frac{r + 1}{2} \cL$ to $x_{\frac{r-1}{2}}$.

Arguing similarly, we see that $B_3$ contributes $(-1)^{\frac{r - 1}{2}}\frac{r + 1}{2} (-H_{+})$ to $x_{\frac{r-1}{2}}$. On the other hand, the $B_2$ part of the matrix equation clearly contributes just $-\cL$ to $x_{\frac{r - 1}{2}}$. 
So we have to show that $B_4$ contributes $(-1)^{\frac{r - 1}{2}}\frac{2}{r + 1} + (-1)^{\frac{r - 1}{2}}\frac{r + 1}{2}(-H_{-})$ to $x_{\frac{r-1}{2}}$.

We solve $AX = B_4$ by solving as usual the $1/(n + 1)$ parts separately and adding them up. The claim is that the $1/(n + 1)$ part of $AX = B_4$ without $1/(n + 1)$ is
\[
    1 + \left(\frac{r + 1}{2}\right)\sum_{j = 2}^{n}(-1)^{j - 1}\frac{j - 1}{j}{2j - 1 \choose j}{\frac{r - 1}{2} + j \choose 2j - 1}.
\]
This can easily be seen for $n = 0, 1$ and we do not need to perform any column operations. For $n \geq 2$ however, we have to reduce the last column as follows. Assume $n = 2$. Note that the $(\frac{r - 3}{2}, \frac{r + 1}{2})^\mth$ entry without $1/3$ is ${\frac{r + 1}{2} \choose \frac{r - 3}{2}}$. We have to kill this entry using the $(\frac{r - 3}{2}, \frac{r - 3}{2})^\mth$ entry, which is ${2 \choose 0}\dfrac{(\frac{r - 3}{2})!(\frac{r + 3}{2})}{(\frac{r + 3}{2})(\frac{r + 1}{2}) \cdots (3)} = \frac{2}{(\frac{r + 1}{2})(\frac{r - 1}{2})}$.

We perform the operation $C_{\frac{r + 1}{2}} \to C_{\frac{r + 1}{2}} - (1/2)(\frac{r + 1}{2})(\frac{r - 1}{2})\frac{(\frac{r + 1}{2})(\frac{r - 1}{2})}{2}C_{\frac{r - 3}{2}}$. This makes the matrix into a block lower triangular matrix whose lower right block is upper triangular with diagonal entries $1$ and
\begin{eqnarray*}
    1 - (1/2)\left(\frac{r + 1}{2}\right)\left(\frac{r - 1}{2}\right)\frac{(\frac{r + 1}{2})(\frac{r - 1}{2})}{2}\frac{(\frac{r + 3}{2})}{(\frac{r - 1}{2})} 
    & =  & 1 - \left(\frac{r + 1}{2}\right)\frac{3}{2}{\frac{r + 3}{2} \choose 3},
\end{eqnarray*}
which verifies the identity for $n = 2$. As usual, the verification of the claim for arbitrary $n$ is left to the reader.

The identity that we have to prove is
\begin{eqnarray}
\label{main identity 12}
\sum_{n = 0}^{\frac{r - 1}{2}}\frac{1}{n + 1}\left(1 + \left(\frac{r + 1}{2}\right)\sum_{j = 2}^{n}(-1)^{j - 1}\frac{j - 1}{j}{2j - 1 \choose j}{\frac{r - 1}{2} + j \choose 2j - 1}\right) = (-1)^{\frac{r - 1}{2}}\frac{2}{r + 1} - (-1)^{\frac{r - 1}{2}}\frac{r + 1}{2}H_{-}.
\end{eqnarray}
As before (cf. \eqref{reorder n and j}) we switch the order of $n$ and $j$ (noting that $n = 0, 1$ do not
contribute to the sum on $j$) and break the sum on the left into three parts:
\begin{eqnarray*}
  H_{\frac{r+1}{2}} + \frac{r+1}{2} H_{\frac{r+1}{2}} \sum_{j = 2}^{\frac{r-1}{2}} (-1)^{j - 1}\frac{j - 1}{j}{2j - 1 \choose j}{\frac{r - 1}{2} + j \choose 2j - 1} + \frac{r+1}{2} \sum_{j = 2}^{\frac{r-1}{2}} (-1)^{j}\frac{j - 1}{j}{2j - 1 \choose j}{\frac{r - 1}{2} + j \choose 2j - 1}  H_{j}.
\end{eqnarray*} 
Changing notation (and binomial coefficients) this may be rewritten as
\begin{eqnarray}
  \label{changing notation 12}
  H_{\frac{r+1}{2}} + \frac{r+1}{2} H_{\frac{r+1}{2}} \sum_{k = 2}^{n} (-1)^{k - 1}\frac{k - 1}{k}{n+k \choose k}{n \choose k - 1} + \frac{r+1}{2} \sum_{k = 2}^{n} (-1)^{k}\frac{k - 1}{k}{n+k  \choose k}{n \choose k - 1}  H_{k}
\end{eqnarray} 
with $n = \frac{r - 1}{2}$.
{\tt Sigma} shows that the first sum above is $$-\frac{1}{n+1} + (-1)^n - (-1)^n \frac{n(2n+1)}{(n+1)^2} {2n \choose n}$$ whereas the second is
$$-(-1)^n \frac{n}{(1+n)^2} - 2 (-1)^n H_n + (-1)^n\frac{n(2n+1)}{(n+1)^3} {2n \choose n} + (-1)^n\frac{n(2n+1)}{(n+1)^2} {2n \choose n} H_n. $$
Plugging these into \eqref{changing notation 12}, noting that $H_{n+1} = H_n + \frac{1}{n+1}$, and simplifying,   \eqref{changing notation 12} becomes
$$ (-1)^n \frac{1}{n+1} - (-1)^n (n+1) H_n, $$
which gives the right hand side of \eqref{main identity 12} since $n = \frac{r-1}{2}$.

\section{Footnote 16}
  \label{Footnote 16}

We show that in
\[
\!\begin{pmatrix}
	{\frac{r + 2}{2} \choose 0}\frac{r}{r} & 0 & 0 & \cdots & {\frac{r + 2}{2} \choose 0} \\
	{\frac{r + 2}{2} \choose 1}\frac{r}{r - 1} & {r/2 \choose 0}\frac{r - 1}{(r - 1)(r - 2)} & 0 & \cdots & {\frac{r + 2}{2} \choose 1} \\
	{\frac{r + 2}{2} \choose 2}\frac{r}{r - 2} & {r/2 \choose 1}\frac{r - 1}{(r - 2)(r - 3)} & {\frac{r - 2}{2} \choose 0}\frac{2!(r - 2)}{(r - 2)(r - 3)(r - 4)} & \cdots & {\frac{r + 2}{2} \choose 2} \\
	\vdots & \vdots & \vdots & & \vdots \\
	{\frac{r + 2}{2} \choose \frac{r - 2}{2}}\frac{r}{\frac{r + 2}{2}} & {r/2 \choose \frac{r - 4}{2}}\frac{r - 1}{(\frac{r + 2}{2})(r/2)} & {\frac{r - 2}{2} \choose \frac{r - 6}{2}}\frac{2!(r - 2)}{(\frac{r + 2}{2})(r/2)(\frac{r - 2}{2})} & \cdots & {\frac{r + 2}{2} \choose \frac{r - 2}{2}}
\end{pmatrix} \!\!\! 
\begin{pmatrix}
	x_r \\ x_{r - 1} \\ x_{r - 2} \\ \vdots \\ x_{\frac{r + 2}{2}}
\end{pmatrix} \!\!\! 
= \!\!\! \begin{pmatrix}
	{\frac{r + 2}{2} \choose 0}\left(\cL - H_{\frac{r + 2}{2}}\right) \\
	{\frac{r + 2}{2} \choose 1}\left(\cL - H_{r/2}\right) \\
	{\frac{r + 2}{2} \choose 2}\left(\cL - H_{\frac{r - 2}{2}}\right) \\
	\vdots \\
	{\frac{r + 2}{2} \choose \frac{r - 2}{2}}\left(\cL - H_{2}\right)
\end{pmatrix}\!\!,
\]
we have $x_{\frac{r + 2}{2}} = \cL - H_{-} - H_{+}$. As usual, we separate the matrix on the right as $B_1 + B_2$, where
\[
B_1 = \begin{pmatrix}{\frac{r + 2}{2} \choose 0}(\cL - H_{+}) \\ {\frac{r + 2}{2} \choose 1}(\cL - H_{+}) \\ {\frac{r + 2}{2} \choose 2}(\cL - H_{+}) \\ \vdots \\ {\frac{r + 2}{2} \choose \frac{r - 2}{2}}(\cL - H_{+})\end{pmatrix} \text{ and } B_2 = \begin{pmatrix}{\frac{r + 2}{2} \choose 0}(0) \\ {\frac{r + 2}{2} \choose 1}(1/\frac{r + 2}{2}) \\ {\frac{r + 2}{2} \choose 2}(1/\frac{r}{2} + 1/\frac{r + 2}{2}) \\ \vdots \\ {\frac{r + 2}{2} \choose \frac{r - 2}{2}}(1/3 + 1/4 + \cdots + 1/\frac{r + 2}{2})\end{pmatrix}.
\]

We first claim that the $B_1$ part of $x_{\frac{r + 2}{2}}$ is $\cL - H_{+}$. Indeed, this follows directly since the matrices $A_1$ and $A$ differ only in the last column, and since the last column of $A_1$ is $(\cL - H_{+})$ times the last column of $A$.

Next, we show that the $B_2$ part of $x_{\frac{r + 2}{2}}$ is $-H_{-}$. Firstly, we reduce $A$ to a lower triangular matrix using the following column operations and see that the $(\frac{r - 2}{2}, \frac{r - 2}{2})^\mth$ entry is $(-1)^{\frac{r - 2}{2}}$.
We perform $C_{\frac{r - 2}{2}} \to C_{\frac{r - 2}{2}} - C_0$ to make the $(0, \frac{r - 2}{2})^\mth$ entry $0$. For $j \geq 1$, the $(j, \frac{r - 2}{2})^\mth$ entry becomes
\[
    {\frac{r + 2}{2} \choose j} - {\frac{r + 2}{2} \choose j}\frac{r}{r - j} = -\left(\frac{r + 2}{2}\right){\frac{r}{2} \choose j - 1}\frac{1}{r - j}.
\]
Therefore the $(1, \frac{r - 2}{2})^\mth$ entry becomes
\[
    -\left(\frac{r + 2}{2}\right)\frac{1}{r - 1}.
\]
Now we perform $C_{\frac{r - 2}{2}} \to C_{\frac{r - 2}{2}} + (r - 2)\left(\frac{r + 2}{2}\right)\frac{1}{r - 1}C_1$ to make the $(1, \frac{r - 2}{2})^\mth$ entry $0$. For $j \geq 2$, the $(j, \frac{r - 2}{2})^\mth$ entry becomes
\begin{align*}
    & -\left(\frac{r + 2}{2}\right){\frac{r}{2} \choose j - 1}\frac{1}{r - j} + \left(\frac{r + 2}{2}\right){\frac{r}{2} \choose j - 1}\frac{r - 2}{(r - j)(r - j - 1)} \\
    & \>\quad = \left(\frac{r + 2}{2}\right)\left(\frac{r}{2}\right)\frac{1}{(r - j)(r - j - 1)}{\frac{r - 2}{2} \choose j - 2}.
\end{align*}
Continuing this and making the first $\frac{r - 2}{2}$ entries $0$, for $j \geq \frac{r - 2}{2}$, the $(j, \frac{r - 2}{2})^\mth$ entry becomes
\[
    (-1)^{\frac{r - 2}{2}}\left(\frac{r + 2}{2}\right)\left(\frac{r}{2}\right) \cdots (3)\frac{1}{(r - j)(r - j - 1) \cdots (\frac{r + 4}{2} - j)}{2 \choose j - \frac{r - 2}{2}}.
\]
Therefore the $(\frac{r - 2}{2}, \frac{r - 2}{2})^\mth$ entry is $(-1)^{\frac{r - 2}{2}}$.

Next, we solve the $\frac{1}{n + 1}$ parts of the equation $AX = B_2$ separately. We claim that for $2 \leq n \leq r/2$ the $1/(n + 1)$ part without $1/(n + 1)$ is
\[
    {(-1)^{\frac{r - 2}{2}}}\sum_{j = 1}^{n - 1}(-1)^{j - 1}{2j \choose j - 1}{r/2 + j \choose 2j}.
\]
For $n = 2$, this is easy to see since the corresponding part of $A_2$ and the column reduced form of $A$ found above differ only in the last column and so by Cramer's rule contributes $\frac{{\frac{r + 2}{2} \choose \frac{r - 2}{2}}}{(-1)^{\frac{r - 2}{2}}}$. For $n = 3$, we perform $C_{\frac{r - 2}{2}} \to C_{\frac{r - 2}{2}} - {\frac{r + 2}{2} \choose \frac{r - 4}{2}}\frac{(\frac{r + 2}{2})(\frac{r}{2})(\frac{r - 2}{2})}{6}C_{\frac{r - 4}{2}}$ on the corresponding part of $A_2$. Therefore the $(\frac{r - 2}{2}, \frac{r - 2}{2})^\mth$ entry becomes
\begin{eqnarray*}
    \frac{1}{2}\left(\frac{r + 2}{2}\right)\left(\frac{r}{2}\right) - \frac{1}{6}\left(\frac{r + 2}{2}\right)\left(\frac{r}{2}\right)\left(\frac{r - 2}{2}\right)\left(\frac{r + 4}{2}\right) 
    & = & {\frac{r + 2}{2} \choose 2} - 4{\frac{r + 4}{2} \choose 4},
\end{eqnarray*}
verifying the claim for $n = 3$. The verification of the general claim is left to the reader as usual.

To conclude, we have to prove the identity
\[
(-1)^{\frac{r - 2}{2}}\sum_{n = 2}^{r/2}\frac{1}{n + 1}\sum_{j = 1}^{n - 1}(-1)^{j - 1}{2j \choose j - 1}{r/2 + j \choose 2j} = -H_{-}.
\]
Changing the order of $n$ and $j$ (cf. \eqref{reorder n and j}), and changing notation (and binomial
coefficients) slightly, the left hand side is
\[
H_{n + 1} (-1)^{n} \cdot \sum_{k = 1}^{n-1} (-1)^{k}{n+k \choose k-1}{n + 1 \choose k + 1}  
     \quad  +   \quad    (-1)^{n - 1} \cdot \sum_{k = 1}^{n-1} (-1)^{k} {n+k \choose k-1}{n+1 \choose k + 1} H_{k+1}
\]
with $n = \frac{r}{2}$. As we shall see in the next section, the first sum is $(-1)^{n} - (-1)^{n} {2n \choose n-1}$ (replace $n+1$ by $n$ in \eqref{GS 17}),
and the second sum equals $(-1)^n (H_{n-1}+H_{n+1}) - (-1)^n {2n \choose n-1} H_{n+1}$ 
(by \eqref{main identity 17}). So the 
expression above equals $- H_{n-1} = - H_{-}$.

\section{Footnote 17}
  \label{Footnote 17}

We show that in
\[
\begin{pmatrix}
	{r/2 \choose 0}\frac{r}{r} & 0 & \cdots & {r/2 \choose 0} & 0 \\
	{r/2 \choose 1}\frac{r}{r - 1} & {\frac{r - 2}{2} \choose 0}\frac{r - 1}{(r - 1)(r - 2)} & \cdots & {r/2 \choose 1} & 0 \\
	\vdots & \vdots & & \vdots & \vdots \\
	{r/2 \choose r/2}\frac{r}{r/2} & {\frac{r - 2}{2} \choose \frac{r - 2}{2}}\frac{r - 1}{(r/2)(\frac{r - 2}{2})} & \cdots & {r/2 \choose r/2} & 1
\end{pmatrix}
\begin{pmatrix}
	x_r \\ x_{r - 1} \\ \vdots \\ x_{\frac{r + 2}{2}} \\ x_{r/2}
\end{pmatrix}
=
\begin{pmatrix}
	-\frac{r + 2}{2}{r/2 \choose 0}\left(H_{\frac{r + 2}{2}} - 1\right) \\
	- \frac{r + 2}{2}{r/2 \choose 1}\left(H_{r/2} - 1\right) \\
	\vdots \\ -\frac{r + 2}{2}{r/2 \choose \frac{r - 2}{2}}\left(H_{2} - 1\right) \\
	-\frac{r + 2}{2}{r/2 \choose r/2}\left(H_{1} - 1\right)
\end{pmatrix},
\]
we have
\[
x_{\frac{r + 2}{2}} = \frac{r + 2}{2}\left(1 - H_- - H_+\right) \text{ and } x_{r/2} = -(-1)^{r/2}\frac{1}{r/2}.
\]

We separate the matrix on the right as $B_1 + B_2$, where
\[
B_1 = \begin{pmatrix}\frac{r + 2}{2}{r/2 \choose 0}(1 - H_{+}) \\ \frac{r + 2}{2}{r/2 \choose 1}(1 - H_{+}) \\ \frac{r + 2}{2}{r/2 \choose 2}(1 - H_{+}) \\ \vdots \\ \frac{r + 2}{2}{r/2 \choose r/2}(1 - H_{+})\end{pmatrix} \text{ and } B_2 = \begin{pmatrix}\frac{r + 2}{2}{r/2 \choose 0}(0) \\ \frac{r + 2}{2}{r/2 \choose 1}(1/\frac{r + 2}{2}) \\ \frac{r + 2}{2}{r/2 \choose 2}(1/\frac{r}{2} + 1/\frac{r + 2}{2}) \\ \vdots \\ \frac{r + 2}{2}{r/2 \choose r/2}(1/2 + 1/3 + \cdots + 1/\frac{r + 2}{2})\end{pmatrix}.
\]

First, we solve for $x_{r/2}$. Let $A_1$ and $A_2$ be the matrices obtained by replacing the last column of $A$ with $B_1$ and $B_2$, respectively.   Note that the last two columns of $A_1$ are linearly dependent. Therefore $B_1$ does not contribute towards $x_{r/2}$. We now compute the determinant of $A_2$. Note that $\det A_2$ is $\frac{r + 2}{2}$ times the determinant of

\[
    \begin{pmatrix}
        {r/2 \choose 0}\frac{r}{r} & 0 & \cdots & {r/2 \choose 0} &{r/2 \choose 0}(0) \\
    	{r/2 \choose 1}\frac{r}{r - 1} & {\frac{r - 2}{2} \choose 0}\frac{r - 1}{(r - 1)(r - 2)} & \cdots & {r/2 \choose 1} &{r/2 \choose 1}(1/\frac{r + 2}{2}) \\
	  \vdots & \vdots & & \vdots & \vdots \\
	  {r/2 \choose r/2}\frac{r}{r/2} & {\frac{r - 2}{2} \choose \frac{r - 2}{2}}\frac{r - 1}{(r/2)(\frac{r - 2}{2})} & \cdots & {r/2 \choose r/2} & {r/2 \choose r/2}(1/2 + 1/3 + \cdots + 1/\frac{r + 2}{2})
    \end{pmatrix}.
\]
We first reduce the second-to-last column using the following column operations so that all but the last two entries are $0$. We first perform $C_{\frac{r - 2}{2}} \to C_{\frac{r - 2}{2}} - C_0$ to make the $(0, \frac{r - 2}{2})^\mth$ entry $0$. For $j \geq 1$, the $(j, \frac{r - 2}{2})^\mth$ entry becomes
\[
    {r/2 \choose j} - {r/2 \choose j}\frac{r}{r - j} = -\left(\frac{r}{2}\right){\frac{r - 2}{2} \choose j - 1}\frac{1}{r - j}.
\]
Therefore the $(1, \frac{r - 2}{2})^\mth$ entry is
\[
    -\left(\frac{r}{2}\right)\frac{1}{r - 1}.
\]
Next, we perform $C_{\frac{r - 2}{2}} \to C_{\frac{r - 2}{2}} + (r - 2)\left(\frac{r}{2}\right)\frac{1}{r - 1}C_1$ to make the $(1, \frac{r - 2}{2})^\mth$ entry $0$. For $j \geq 2$, the $(j, \frac{r - 2}{2})^\mth$ entry becomes
\begin{align*}
    & -\left(\frac{r}{2}\right){\frac{r - 2}{2} \choose j - 1}\frac{1}{r - j} + \left(\frac{r}{2}\right){\frac{r - 2}{2} \choose j - 1}\frac{r - 2}{(r - j)(r - j - 1)} \\
    & \>\quad = \left(\frac{r}{2}\right)\left(\frac{r - 2}{2}\right){\frac{r - 4}{2} \choose j - 2}\frac{1}{(r - j)(r - j - 1)}.
\end{align*}
Continuing this and making the first $\frac{r - 2}{2}$ entries $0$, for $j \geq \frac{r - 2}{2}$, the $(j, \frac{r - 2}{2})^\mth$ entry becomes
\[
    (-1)^{\frac{r - 2}{2}}\left(\frac{r}{2}\right)\left(\frac{r - 2}{2}\right)\cdots\left(2\right){1 \choose j - \frac{r - 2}{2}}\frac{1}{(r - j)(r - j - 1) \cdots (\frac{r + 4}{2} - j)}.
\]
Therefore the matrix becomes
\[
    \begin{pmatrix}
        {r/2 \choose 0}\frac{r}{r} & 0 & \cdots & 0 & {r/2 \choose 0}(0) \\
    	{r/2 \choose 1}\frac{r}{r - 1} & {\frac{r - 2}{2} \choose 0}\frac{r - 1}{(r - 1)(r - 2)} & \cdots & 0 & {r/2 \choose 1}(1/\frac{r + 2}{2}) \\
	  \vdots & \vdots & & \vdots & \vdots \\

        {r/2 \choose \frac{r - 2}{2}}\frac{r}{\frac{r + 2}{2}} & {\frac{r - 2}{2} \choose \frac{r - 4}{2}}\frac{r - 1}{(\frac{r + 2}{2})(r/2)} & \cdots & (-1)^{\frac{r - 2}{2}}\frac{2}{\frac{r + 2}{2}} & {r/2 \choose \frac{r - 2}{2}}(1/3 + 1/4 + \cdots + 1/\frac{r + 2}{2})\\
	  {r/2 \choose r/2}\frac{r}{r/2} & {\frac{r - 2}{2} \choose \frac{r - 2}{2}}\frac{r - 1}{(r/2)(\frac{r - 2}{2})} & \cdots & {(-1)^{\frac{r - 2}{2}}} & {r/2 \choose r/2}(1/2 + 1/3 + \cdots + 1/\frac{r + 2}{2})
    \end{pmatrix}.
\]

Next, we claim that for $1 \leq n \leq r/2$, the $\dfrac{1}{n + 1}$ part of $AX = B_2$ {without $1/(n + 1)$} and up to the factor of $\frac{r + 2}{2}$ is 
\[(-1)^{n - 1}n{2n - 1 \choose n - 1}\frac{{r/2 + n \choose r/2 - n}}{{r/2 + 1 \choose r/2 - 1}}.\]
This is easy to see for $n = 1, 2$ using Cramer's rule, noting only the bottom right $2\times 2$ blocks are relevant. For $n = 3$, we kill the $(\frac{r - 4}{2}, \frac{r}{2})^\mth$ entry above using the $(\frac{r - 4}{2}, \frac{r - 4}{2})^\mth$ entry, which is ${2 \choose 0}\frac{(\frac{r - 4}{2})!(\frac{r + 4}{2})}{(\frac{r + 4}{2})(\frac{r + 2}{2}) \cdots (4)} = \frac{6}{(\frac{r + 2}{2})(\frac{r}{2})(\frac{r - 2}{2})}$.
We perform $C_{r/2} \to C_{r/2} - \frac{(\frac{r + 2}{2})(\frac{r}{2})(\frac{r - 2}{2})}{6}\frac{(\frac{r}{2})(\frac{r - 2}{2})}{2}C_{\frac{r - 4}{2}}$. This reduces the $1/4$ part of the matrix above to a block lower triangular matrix whose lower right block is
\[
    \begin{pmatrix}
        (-1)^{\frac{r - 2}{2}}\frac{2}{\frac{r + 2}{2}} & -(r/2)\frac{1}{12}(r^2 + 2r - 20) \\
        (-1)^{\frac{r - 2}{2}} & 1 - 2{\frac{r + 4}{2} \choose 4}
    \end{pmatrix}.
\]
This verifies the claim for $n = 3$ using Cramer's rule. As usual, the verification of the claim for general $n$ is left to the reader.

Therefore, we have to prove the identity
\[
    \sum_{n = 1}^{r/2}\left(\frac{r + 2}{2}\right)(-1)^{n - 1}\frac{n}{n + 1}{2n - 1 \choose n - 1}\frac{{r/2 + n \choose r/2 - n}}{{r/2 + 1 \choose r/2 - 1}} = (-1)^{\frac{r + 2}{2}}\frac{1}{r/2}.
\]

Changing notation slightly, this reduces to proving
\[
    \sum_{k = 1}^{n}(-1)^{k}\frac{2k}{k + 1}{2k - 1 \choose k - 1}{n + k \choose 2k} = (-1)^{n}
\]
for $n = r/2$. Noting that $2{2k - 1 \choose k - 1} = {2k \choose k}$, the left hand side is
\[
    \sum_{k = 1}^{n}(-1)^{k}{2k \choose k}{n + k \choose 2k} - \sum_{k = 1}^{n}(-1)^{k}\frac{1}{k + 1}{2k \choose k}{n + k \choose 2k}.
\]
The first sum is $(-1)^n$ and the second sum is $0$ by \cite[vol 4, Equation 1.4]{Gou10} and \cite[vol 5, Equation 1.22]{Gou10} (taking $y = 1$), respectively.

Next we solve for $x_{\frac{r+2}{2}}$.
Let now $A_1', A_2'$ be the matrices obtained by replacing the second-to-last column in $A$ with $B_1, B_2$, respectively. It is clear that $A_1'$ and $A$ only differ in the second-to-last column, and the second-to-last column in $A_1'$ is $\frac{r + 2}{2}(1 - H_{+})$ times that in $A$. This shows that the $B_1$ part of $x_{\frac{r + 2}{2}}$ is $\frac{r + 2}{2}(1 - H_{+})$. We claim that for $1 \leq n \leq r/2$, the $1/(n + 1)$ part of $x_{\frac{r + 2}{2}}$ in $AX = B_2$ without $(\frac{r + 2}{2})(1/(n + 1))$ is
\[
    (-1)^{\frac{r - 2}{2}}\sum_{j = 1}^{n - 1}(-1)^{j - 1}\frac{j}{j + 1}{r/2 + j \choose 2j}{2j \choose j}.
\]
This is clear for $n = 1, 2$. For $n = 3$, we kill the $(\frac{r - 4}{2}, \frac{r - 2}{2})^\mth$ entry of $A_2'$ with the $(\frac{r - 4}{2}, \frac{r - 4}{2})^\mth$ entry, which is ${2 \choose 0}\frac{(\frac{r - 4}{2})!(\frac{r + 4}{2})}{(\frac{r + 4}{2})(\frac{r + 2}{2}) \cdots (4)} = \frac{6}{(\frac{r + 2}{2})(\frac{r}{2})(\frac{r - 2}{2})}$. We perform $C_{\frac{r - 2}{2}} \to C_{\frac{r - 2}{2}} - \frac{(\frac{r + 2}{2})(\frac{r}{2})(\frac{r - 2}{2})}{6}\frac{(\frac{r}{2})(\frac{r - 2}{2})}{2}C_{\frac{r - 4}{2}}$. Therefore the $(\frac{r - 2}{2}, \frac{r - 2}{2})^\mth$ entry becomes
\[
    \frac{r}{2} - \frac{(\frac{r + 2}{2})(\frac{r}{2})(\frac{r - 2}{2})}{6}\frac{(\frac{r}{2})(\frac{r - 2}{2})}{2}4\frac{{\frac{r + 4}{2}}}{(\frac{r + 2}{2})(\frac{r}{2})(\frac{r - 2}{2})} = \frac{r}{2} - \frac{1}{3}\left(\frac{r + 4}{2}\right)\left(\frac{r}{2}\right)\left(\frac{r - 2}{2}\right).
\]
Using Cramer's rule, we immediately get
\[
    (-1)^{\frac{r - 2}{2}}\left({\frac{r + 2}{2} \choose 2} - 4{\frac{r + 4}{2} \choose 4}\right),
\]
which is the claim for $n = 3$. As usual, the claim for general $n$ is left to the reader.

To find the $B_2$ part of $x_{\frac{r + 2}{2}}$, we have to prove the identity
\[
(-1)^{\frac{r - 2}{2}}\left(\frac{r + 2}{2}\right)\sum_{n = 2}^{r/2}\frac{1}{n + 1}\sum_{j = 1}^{n - 1}(-1)^{j - 1}\frac{j}{j + 1}{r/2 + j \choose 2j}{2j \choose j} = -\left(\frac{r + 2}{2}\right)H_{-}.
\]
Cancel $\frac{r+2}{2}$ from both sides and move the sign on the left to the right side. Changing the order of $n$ and $j$ (cf. \eqref{reorder n and j}) and changing notation a bit we have to show that 
\[ 
  H_{n+1}  \cdot \sum_{k = 1}^{n-1} (-1)^{k - 1}{n+k \choose k-1}{n+1 \choose k + 1}  
     \quad  +   \quad    \sum_{k = 1}^{n-1} (-1)^{k} {n+k \choose k-1}{n+1 \choose k + 1} H_{k+1} = (-1)^n H_{n-1},
\]
for $n = \frac{r}{2}$. The sum on the left is $-(-1)^n + (-1)^n \frac{n}{n+1} {2n \choose n}$. This
follows from \eqref{GS 17} below (replacing $n+1$ with $n$). Plugging this into the first sum and noticing
that the second part of this is the $k = n$ term of the second sum, we are reduced to showing that
\begin{eqnarray}
  \label{main identity 17}
  S(n) := \sum_{k = 1}^{n} (-1)^{k} {n+k \choose k-1}{n+1 \choose k + 1} H_{k+1} = (-1)^n (H_{n-1} + H_{n+1}).
\end{eqnarray}
We can evaluate the `tricky' sum $S(n)$ with {\tt Sigma} but let us use instead {\tt fastZeil}, 
since it also
provides a theoretical proof (and it is perhaps good to give at least one example in this appendix
of how one can verify any given identity theoretically). We thank Peter Paule for explaining
the entire argument below.

Let $L$ be the operator which evaluates functions $f(x)$ at $x=0$, i.e., $Lf(x) = f(0)$. 
Let $D$ be the differentiation operator with respect to $x$, i.e., $Df(x) = f'(x)$. Clearly
\[
  LD {x+k \choose k} = H_k.
\]
Using this, one sees that $S(n) = LD T(x,n)$
where 
\[
  T(x,n) = \sum_{k=0}^n (-1)^k {n \choose k-1} {n + 1 \choose k+1} {x+k+1 \choose k+1}.
\]

The following commands using the package {\tt fastZeil} in mathematica produces a 
recurrence relation for $T(x,n)$:
{\tiny
\begin{eqnarray*}
  \text{In[1]}     & :=  & <<\text{RISC`fastZeil'}   \\     
                       &      & \text{Fast Zeilberger Package version 3.61 written by Peter Paule, Markus Schorn and Axel Riese} \\
                       &       &  \text{Copyright Research Institute for Symbolic Computation (RISC),   
                                           Johannes Kepler University, Linz, Austria}  \\             
  \text{In[2]}     & :=  & \text{Zb}[(-1)^\mathrm{k} \text{ Binomial[n+k, k-1] Binomial[n+1, k+1]     
                                           Binomial[x+k+1, k+1],  \textbraceleft k,1,n\textbraceright, n, 2]}       \\                                                                   
  \text{Out[2]}  &  =  & \text{n\textsuperscript{2} (2 + n)\textsuperscript{2} SUM[n] 
                 - (3 + 2n) (3 + 6 n + 2 n\textsuperscript{2} + 4 x + 6 n x + 2 n\textsuperscript{2}x) SUM[1 + n] 
                                              - (1 + n)\textsuperscript{2} (3 + n)\textsuperscript{2}  SUM[2 + n]}  = 0  \\
  \text{In[3]}     & := & \text{Prove[]}                                              
\end{eqnarray*}
}
\!\!\!\!\! Here Out[2] means that:
\[
  n^2(n+2)^2T(x,n) - (3+2n) (3+6n+2n^2 +(4+6n+2n^2)x) T(x,n+1) - (1+n)^2(3+n)^2 T(x,n+2) = 0
\]
for $n \geq 1$. Applying the operator $LD$ we obtain
\[
  n^2(2+n)^2 S(n) -   (3+2n) (4+6n+2n^2) T(0,n+1) -    (3+2n) (3+6n+2n^2) S(n+1)- (1+n)^2(3+n)^2 S(n+1) = 0.
\]
Now we may evaluate the `easy' sum $T(0,n+1)$ with {\tt Sigma} or by hand using $\prescript{}{2} {F_1}$ as
in the  computation in footnote 10. 
We have:
\begin{eqnarray}
  \label{GS 17}
  T(0,n+1) & = & \sum_{k = 0}^{n+1} (-1)^k {n+1+k \choose k-1} {n+2 \choose k+1}  \nonumber  \\
                 & = & -\frac{1}{2} (1+n)(2+n) \cdot \prescript{}{2} {F_1}\left({-n, 3+n \atop 3}; 1\right) \nonumber \\
                 & {\eqref{GSF} \atop =} & -\frac{1}{2} (1+n)(2+n) \frac{(3-(3+n))_n}{(3)_n}  \nonumber \\
                 & = & -\frac{1}{2} (1+n)(2+n) \frac{(-1)^n  n!}{{\frac{1}{2}}(n+2)!} \nonumber \\
                 & = & (-1)^{n+1}.
\end{eqnarray}
Plugging this into the recursion formula we obtain:
\[
  n^2 (2+n)^2 S(n) - (3+2n) (4+6n+2n^2) (-1)^{n+1}  - (3+2n) (3+6n+2n^2) S(n+1)- (1+n)^2 (3+n)^2 S(n+2) = 0.
\]
This gives a recurrence relation for the left hand side $S(n)$ of \eqref{main identity 17}.
Since an easy by hand check shows that the right hand side of \eqref{main identity 17} also satisfies the 
above recurrence relation and both sides of \eqref{main identity 17} agree for small values 
of $n$ (the common 
values are $-3/2, 17/6, -43/12, \ldots$ for $n = 1,2,3, \ldots$), we see that \eqref{main identity 17}
holds for all $n \geq 1$.

It remains to give a theoretical proof of the recurrence relation for $T(x,n)$ above. This is produced
by the command in `In[3]' above. The output is a `Computer Theorem' with proof which we reproduce
more or less identically below.

\begin{theorem}[Computer Theorem]
  Let $$F(k,n) = (-1)^k {n \choose k-1} {n + 1 \choose k+1} {x+k+1 \choose k+1}$$ and
  $$\mathrm{SUM}(n) = \sum_{k=1}^n F(k,n).$$ Then
  $$n^2 (2 + n)^{2} \mathrm{SUM}[n] 
                 - (3 + 2n) (3 + 6 n + 2 n^{2} + 4 x + 6 n x + 2 n^{2}x) \mathrm{SUM}[1 + n] 
                                              - (1 + n)^{2} (3 + n)^{2} \mathrm{SUM}[2 + n] = 0.$$ 
\end{theorem}

\begin{proof}
  Let $\Delta_k$ denote the forward difference operator in $k$ and define 
  $$R(k,n) = \frac{2(k-1)(k+1)^2(n+1)(n+2)(n+3)}{(-k+n+1)(-k+n+2)}.$$ 
  Then the theorem follows from summing the equation 
  \begin{eqnarray*}
    n^2 (2 + n)^{2} F(k,n)
                 - (3 + 2n) (3 + 6 n + 2 n^{2} + 4 x + 6 n x + 2 n^{2}x) F(k,1 + n) \qquad  \\
                                              - (1 + n)^{2} (3 + n)^{2} F(k,2 + n) \quad =   \quad 
                                              \Delta_k(F(k,n)R(k,n)) & 
   \end{eqnarray*}
  over $k$ from $1$ to $n+1$ (since the right hand side telescopes to $0$ since at the endpoints $F(n+1,n) = 0 = R(1,n)$). This equation is routine
  verifiable by dividing both sides by $F(k,n)$ and checking the resulting rational equation holds.
\end{proof}  

We did indeed carry out the last check by hand, thereby removing all
computer programs from the proof of the recurrence relation for $T(x,n)$ 
and hence the identity \eqref{main 
identity 17}. Though, of course, without the computer's help in producing the auxiliary rational
function $R(k,n)$ above, one would not even know where to start.

\end{document}